\documentclass{article}
\usepackage[utf8]{inputenc}
\usepackage{color}
\usepackage{refstyle}
\usepackage{float}
\usepackage{mathrsfs}
\usepackage{mathtools}
\usepackage{amsmath}
\usepackage{amsthm}
\usepackage{amssymb}
\usepackage{graphicx}
\usepackage{setspace}
\PassOptionsToPackage{normalem}{ulem}
\usepackage{ulem}
\usepackage{nomencl}
\providecommand{\printnomenclature}{\printglossary}
\providecommand{\makenomenclature}{\makeglossary}
\makenomenclature
\usepackage[unicode=true,
 bookmarks=true,bookmarksnumbered=true,bookmarksopen=true,bookmarksopenlevel=1,
 breaklinks=false,pdfborder={0 0 1},backref=false,colorlinks=true]
 {hyperref}

\makeatletter

\AtBeginDocument{\providecommand\Subsecref[1]{\ref{Subsec:#1}}}
\AtBeginDocument{\providecommand\Eqref[1]{\ref{Eq:#1}}}
\AtBeginDocument{\providecommand\Secref[1]{\ref{Sec:#1}}}
\AtBeginDocument{\providecommand\Corref[1]{\ref{Cor:#1}}}
\AtBeginDocument{\providecommand\Thmref[1]{\ref{Thm:#1}}}
\AtBeginDocument{\providecommand\subsecref[1]{\ref{subsec:#1}}}
\AtBeginDocument{\providecommand\Lemref[1]{\ref{Lem:#1}}}
\AtBeginDocument{\providecommand\Factref[1]{\ref{Fact:#1}}}
\AtBeginDocument{\providecommand\Blackboxref[1]{\ref{Blackbox:#1}}}
\AtBeginDocument{\providecommand\corref[1]{\ref{cor:#1}}}
\AtBeginDocument{\providecommand\Defref[1]{\ref{Def:#1}}}
\AtBeginDocument{\providecommand\secref[1]{\ref{sec:#1}}}
\providecommand{\tabularnewline}{\\}
\RS@ifundefined{subsecref}
  {\newref{subsec}{name = \RSsectxt}}
  {}
\RS@ifundefined{thmref}
  {\def\RSthmtxt{theorem~}\newref{thm}{name = \RSthmtxt}}
  {}
\RS@ifundefined{lemref}
  {\def\RSlemtxt{lemma~}\newref{lem}{name = \RSlemtxt}}
  {}

\theoremstyle{plain}
\newtheorem*{fact*}{\protect\factname}
\theoremstyle{plain}
\newtheorem{thm}{\protect\theoremname}
\theoremstyle{definition}
\newtheorem{defn}[thm]{\protect\definitionname}
\theoremstyle{plain}
\newtheorem{fact}[thm]{\protect\factname}
\theoremstyle{plain}
\newtheorem{lem}[thm]{\protect\lemmaname}
\theoremstyle{plain}
\newtheorem{cor}[thm]{\protect\corollaryname}
\theoremstyle{remark}
\newtheorem*{rem*}{\protect\remarkname}
\theoremstyle{plain}
\newtheorem{blackbox}[thm]{\protect\blackboxname}
\theoremstyle{definition}
\newtheorem*{defn*}{\protect\definitionname}
\theoremstyle{definition}
\newtheorem*{problem*}{\protect\problemname}

\@ifundefined{date}{}{\date{}}
\usepackage{fullpage}
\usepackage{graphicx}
\usepackage{amsfonts}
\usepackage{amsthm}
\allowdisplaybreaks
\usepackage{cases}
\usepackage{tikz-cd}
\usepackage{cancel}
\usepackage{faktor}

\DeclareMathOperator{\Imm}{Im}
\DeclareMathOperator{\Ree}{Re}

\DeclareMathOperator{\vol}{vol}
\let\div\relax
\DeclareMathOperator{\div}{div}
\DeclareMathOperator{\Div}{div}
\DeclareMathOperator{\Ric}{Ric}
\DeclareMathOperator{\grad}{grad}
\DeclareMathOperator{\colim}{colim}
\DeclareMathOperator{\Ker}{Ker}
\DeclareMathOperator{\supp}{supp}
\DeclareMathOperator{\tr}{tr}

\DeclareMathOperator{\avg}{avg}

\usepackage{refstyle}
\newref{cor}{name = corollary~, names = corollaries~, Name = Corollary~, Names = Corollaries~}
\newref{conj}{name= conjecture~, names = conjectures~, Name = Conjecture~, Names = Conjectures~}
\newref{exa}{name= example~, names = examples~, Name = Example~, Names = Examples~}
\newref{rem}{name= remark~, names = remarks~, Name = Remark~, Names = Remarks~} 
\newref{thm}{name= theorem~, names = theorems~, Name = Theorem~, Names = Theorems~}
\newref{lem}{name= lemma~, names = lemmas~, Name = Lemma~, Names = Lemmas~}
\newref{def}{name= definition~, names = definitions~, Name = Definition~, Names = Definitions~}
\newref{fact}{name= fact~, names = facts~, Name = Fact~, Names = Facts~}
\newref{blackbox}{name= blackbox~, names = blackboxes~, Name = Blackbox~, Names = Blackboxes~}
\newref{sec}{name= section~, names = sections~, Name = Section~, Names = Sections~}
\newref{subsec}{name= subsection~, names = subsections~, Name = Subsection~, Names = Subsections~}

\usepackage[framemethod=TikZ]{mdframed}
 
\surroundwithmdframed[
   topline=false,
   rightline=false,
   bottomline=false,
   leftmargin=\parindent,
   skipabove=\medskipamount,
   skipbelow=\medskipamount
]{proof}

\def\Xint#1{\mathchoice
{\XXint\displaystyle\textstyle{#1}}%
{\XXint\textstyle\scriptstyle{#1}}%
{\XXint\scriptstyle\scriptscriptstyle{#1}}%
{\XXint\scriptscriptstyle\scriptscriptstyle{#1}}%
\!\int}
\def\XXint#1#2#3{{\setbox0=\hbox{$#1{#2#3}{\int}$ }
\vcenter{\hbox{$#2#3$ }}\kern-.6\wd0}}

\def\dashint{\Xint-}

\usepackage{accents}

\providetoggle{nomsort}
\settoggle{nomsort}{true} % true = sort by use, false = sort as usual

\makeatletter
\iftoggle{nomsort}{%
    \let\old@@@nomenclature=\@@@nomenclature        
        \newcounter{@nomcount} \setcounter{@nomcount}{0}%
        \renewcommand\the@nomcount{\two@digits{\value{@nomcount}}}% Ensure 10>01
        \def\@@@nomenclature[#1]#2#3{% Taken from package documentation
          \addtocounter{@nomcount}{1}%
        \def\@tempa{#2}\def\@tempb{#3}%
          \protected@write\@nomenclaturefile{}%
          {\string\nomenclatureentry{\the@nomcount\nom@verb\@tempa @[{\nom@verb\@tempa}]%
          \begingroup\nom@verb\@tempb\protect\nomeqref{\theequation}%
          |nompageref}{\thepage}}%
          \endgroup
          \@esphack}%
      }{}
\makeatother

\newcommand\iso{\xrightarrow{
   \,\smash{\raisebox{-0.65ex}{\ensuremath{\scriptstyle\sim}}}\,}}

\usepackage{tensind}
\tensordelimiter{?}

\hypersetup{citecolor=purple,linkcolor=blue}

\newcommand\restr[2]{{% we make the whole thing an ordinary symbol
  \left.\kern-\nulldelimiterspace % automatically resize the bar with \right
  #1 % the function
  \vphantom{\big|} % pretend it's a little taller at normal size
  \right|_{#2} % this is the delimiter
  }}

\usepackage{extarrows}

\makeatother

\usepackage[bibstyle=alphabetic,citestyle=alphabetic-verb]{biblatex}

\providecommand{\blackboxname}{Blackbox}
\providecommand{\corollaryname}{Corollary}
\providecommand{\definitionname}{Definition}
\providecommand{\factname}{Fact}
\providecommand{\lemmaname}{Lemma}
\providecommand{\problemname}{Problem}
\providecommand{\remarkname}{Remark}
\providecommand{\theoremname}{Theorem}

\begin{document}
\title{Hodge-theoretic analysis on manifolds with boundary, heatable currents,
and Onsager's conjecture in fluid dynamics}
\author{Khang Manh Huynh}
\maketitle
\begin{abstract}
We use Hodge theory and functional analysis to develop a clean approach
to heat flows and Onsager's conjecture on Riemannian manifolds with
boundary, where the weak solution lies in the trace-critical Besov
space $B_{3,1}^{\frac{1}{3}}$. We also introduce heatable currents
as the natural analogue to tempered distributions and justify their
importance in Hodge theory.
\end{abstract}

\subsection*{Acknowledgments}

The author thanks his advisor Terence Tao for the invaluable guidance
and patience as the project quickly outgrew its original scale. The
author is also grateful to Jochen Glück for recommending Stephan Fackler's
insightful PhD thesis \cite{Stephan_thesis}.

\tableofcontents{}

\section{Introduction}

\subsection{Onsager's conjecture\label{subsec:Onsager's-conjecture_intro}}

Recall the incompressible Euler equation in fluid dynamics:

\begin{equation}
\left\{ \begin{array}{rll}
\partial_{t}V+\Div\left(V\otimes V\right) & =-\grad p & \text{ in }M\\
\Div V & =0 & \text{ in }M\\
\left\langle V,\nu\right\rangle  & =0 & \text{ on }\partial M
\end{array}\right.\label{eq:Euler}
\end{equation}

where $\left\{ %
\begin{tabular}{l}
$(M,g)$ is an oriented, compact smooth Riemannian manifold with smooth
boundary, dimension $\geq2$\tabularnewline
$\nu$ is the outwards unit normal vector field on $\partial M$.\tabularnewline
$I\subset\mathbb{R}$ is an open interval, $V:I\to\mathfrak{X}M$,
$p:I\times M\to\mathbb{R}$.\tabularnewline
\end{tabular}\right.$

Observe that the \textbf{Neumann condition} $\left\langle V,\nu\right\rangle =0$
means $V\in\mathfrak{X}_{N}$, where $\mathfrak{X}_{N}$ is the set
of vector fields on $M$ which are tangent to the boundary. Note that
when $V$ is not smooth, we need the trace theorem to define the condition
(see \Subsecref{On-domains}).

Roughly speaking, Onsager's conjecture says that the energy $\left\Vert V(t,\cdot)\right\Vert _{L^{2}}$
is a.e. constant in time when $V$ is a weak solution whose regularity
is at least $\frac{1}{3}$. Making that statement precise is part
of the challenge.

In the boundaryless case, the ``positive direction'' (conservation
when regularity is at least $\frac{1}{3})$ has been known for a long
time \parencite{Eyink1994_prelim,constantin1994,Cheskidov2008}. The
``negative direction'' (failure of energy conservation when regularity
is less than $\frac{1}{3}$) is substantially harder \parencite{DeLellis2012_Onsager,DeLellis2014_Onsager},
and was finally settled by Isett in his seminal paper \cite{PhilipIsett2018}
(see the survey in \parencite{DeLellis2019_survey} for more details
and references).

Since then more attention has been directed towards the case with
boundary, and its effects in the generation of turbulence. In \cite{Titi_Onsager_bounded_domain},
the ``positive direction'' was proven in the case $M$ is a bounded
domain in $\mathbb{R}^{n}$ and $V\in L_{t}^{3}C^{0,\alpha}\mathfrak{X}_{N}$
($\alpha>\frac{1}{3}$). The result was then improved in various ways
\parencite{Drivas_2018_Onsager_anomalous,Bardos_2018_Gwiazda_conservation,Bardos_2019_Onsager_Holder_int}.
In \cite{Nguyen_Onsager_bounded_domain}, the conjecture was proven
for $V$ in $L_{t}^{3}B_{3,\infty}^{\alpha}\mathfrak{X}$ ($\alpha>\frac{1}{3}$)
along with some ``strip decay'' conditions for $V$ and $p$ near
the boundary (more details in \Subsecref{Searching-for-the}). Most
recently, the conjecture was proven as part of a more general conservation
of entropy law in \parencite{Titi_Gwiazda_conservation_entropy},
where $M$ is a domain in $\mathbb{R}^{n}$, $V\in L_{t}^{3}\underline{B}_{3,\mathrm{VMO}}^{1/3}\mathfrak{X}$
(where $\underline{B}_{3,\mathrm{VMO}}^{1/3}\mathfrak{X}$ is a VMO-type
subspace of $B_{3,\infty}^{1/3}\mathfrak{X}$), along with a ``strip
decay'' condition involving both $V$ and $p$ near the boundary
(see \Subsecref{Searching-for-the}).

Much less is known about the conjecture on general Riemannian manifolds.
The key arguments on flat spaces rely on the nice properties of convolution,
such as $\Div\left(T*\phi_{\varepsilon}\right)=\Div\left(T\right)*\phi_{\varepsilon}$
where $T$ is a tensor field and $\phi_{\varepsilon}\xrightarrow{\varepsilon\downarrow0}\delta_{0}$
is a mollifier, or that mollification is essentially local. This ``local
approach'' by convolution does not generalize well to Riemannian
manifolds. In \parencite{Isett2015_heat} -- the main inspiration
for this paper -- Isett and Oh used the heat flow to prove the conjecture
on compact Riemannian manifolds without boundary, for $V\in L_{t}^{3}B_{3,c\left(\mathbb{N}\right)}^{\frac{1}{3}}\mathfrak{X}$
(where $B_{3,c\left(\mathbb{N}\right)}^{\frac{1}{3}}\mathfrak{X}$
is the $B_{3,\infty}^{\frac{1}{3}}$-closure of compactly supported
smooth vector fields). The situation becomes more complicated when
the boundary is involved. Most notably, the covariant derivative behaves
badly on the boundary (e.g. the second fundamental form), and it is
difficult to avoid boundary terms that come from integration by parts.
Even applying the heat flow to a distribution might no longer be well-defined.
This requires a finer understanding of analysis involving the boundary,
as well as the properties of the heat flow.

In this paper, we will see how we can resolve these issues, and that
the conjecture still holds true with the boundary:
\begin{fact*}
Assuming $M$ as in \Eqref{Euler}, conservation of energy is true
when $\left(V,p\right)$ is a weak solution with $V\in L_{t}^{3}B_{3,1}^{\frac{1}{3}}\mathfrak{X}_{N}$.
\end{fact*}
It is not a coincidence that this is also the lowest regularity where
the trace theorem holds. We also note a very curious fact that no
``strip decay'' condition involving $p$ (which is present in different
forms for the results on flat spaces) seems to be necessary, and we
only need $p\in L_{\mathrm{loc}}^{1}\left(I\times M\right)$ (see
\Subsecref{Justification} for details). One way to explain this minor
improvement is that the ``strip decay'' condition involving $V$
naturally originates from the trace theorem (see \Subsecref{Justification}),
and is therefore included in the condition $V\in L_{t}^{3}B_{3,1}^{\frac{1}{3}}\mathfrak{X}_{N}$,
while the presence of $p$ is more of a technical artifact arising
from localization (see \parencite[Section 4]{Titi_Gwiazda_conservation_entropy}),
which typically does not respect the Leray projection. By using the
trace theorem and the heat flow, our approach becomes global in nature,
and thus avoids the artifact. Another approach is to formulate the
conjecture in terms of Leray weak solutions like in \parencite{Skipper2018_onsager_no_pressure},
without mentioning $p$ at all, and we justify how this is possible
in \Subsecref{Justification}.

A more local approach, where we assume $V\in L_{t}^{3}B_{3,c\left(\mathbb{N}\right)}^{\frac{1}{3}}\mathfrak{X}$
as in \parencite{Isett2015_heat}, and the ``strip decay'' condition
as in \parencite[Equation 4.9]{Titi_Gwiazda_conservation_entropy},
would be a good topic for another paper. Nevertheless, $B_{3,1}^{\frac{1}{3}}\mathfrak{X}_{N}$
is an interesting space with its own unique results, which keep the
exposition simple and allow the boundary condition to be natural.

\subsection{Modularity}

The paper is intended to be modular: the part dealing with Onsager's
conjecture (\Secref{Onsager's-conjecture}) is relatively short, while
the rest is to detail the tools for harmonic analysis on manifolds
we will need (and more). As we will summarize the tools in \Secref{Onsager's-conjecture},
they can be read independently.

\subsection{Motivation behind the approach\label{subsec:Motivation-behind-the}}

Riemannian manifolds (and their semi-Riemannian counterparts) are
among the most important natural settings for modern geometric PDEs
and physics, where the objects for analysis are often vector bundles
and differential forms. The two fundamental tools for a harmonic analyst
-- \textbf{mollification} and \textbf{Littlewood-Paley projection}
via the Fourier transform -- do not straightforwardly carry over
to this setting, especially when the boundary is involved. Even in
the case of scalar functions on bounded domains in $\mathbb{R}^{n}$,
mollification arguments often need to stay away from the boundary,
which can present a problem when the trace is nonzero. Consider, however,
the idea of a special kind of Littlewood-Paley projection which preserves
the boundary conditions and commutes with important operators such
as divergence and the \textbf{Leray projection}, or using the principles
of harmonic analysis without translation invariance. It is one among
a vast constellation of ideas which have steadily become more popular
over the years, with various approaches proposed (and we can not hope
to fully recount here).

For our discussion, the starting point of interest is perhaps \parencite{Strichartz1983},
in which Strichartz introduced to analysts what had long been known
to geometers, the rich setting of complete Riemannian manifolds, where
harmonic analysis (and the \textbf{Riesz transform }in particular)
can be done via the Laplacian and the \textbf{heat semigroup} $e^{t\Delta}$,
constructed by \textbf{dissipative operators} and Yau's lemma. Then
in \cite{Klainerman2006}, Klainerman and Rodnianski defined the $L^{2}$-heat
flow by the \textbf{spectral theorem} and used it to get the Littlewood-Paley
projection on compact 2-surfaces. In \cite{Isett2015_heat}, Isett
and Oh successfully tackled Onsager's conjecture on Riemannian manifolds
without boundary by using Strichartz's heat flow. These results hint
at the central importance of the heat flow for analysis on manifolds.
But it is not enough to settle the case with boundary, especially
when derivatives are involved. Some pieces of the puzzle are still
missing.

To paraphrase James Arthur (in his introduction to the trace formula
and the Langlands program), there is an intimate link between geometric
objects and ``spectral'' phenomena, much like how the shape of a
drum affects its sounds. For a Riemannian manifold, that link is better
known as the Laplacian -- the generator of the heat flow -- and
\textbf{Hodge theory} is the study of how the Laplacian governs the
cohomology of a Riemannian manifold. An oversimplified description
of Fourier analysis on $\mathbb{R}^{n}$ would be ``the spectral
theory of the Laplacian'' \parencite{strichartz1989harmonic_spectral_laplacian},
where the heat kernel is the Gaussian function, invariant under the
Fourier transform and a possible choice of mollifier. Additionally,
the \textbf{Helmholtz decomposition}, originally discovered in a hydrodynamic
context, turned out to be a part of Hodge theory. It should therefore
be no surprise that Hodge theory is the natural framework in which
we formulate harmonic analysis on manifolds, heat flows and Onsager's
conjecture. Wherever there is the Laplacian, there is harmonic analysis.
Historically, Milgram managed to establish a subset of Hodge theory
by heat flow methods \parencite{Milgram1951heat}. Here, however,
we will establish Hodge theory by standard elliptic estimates, from
which we develop analysis on manifolds and construct the heat flow.
Most notably, Hodge theory greatly simplifies some crucial approximation
steps involving the boundary (\Corref{crucial_boundary_approx_homN}),
and helps predict some key results Onsager's conjecture would require
(\Thmref{different_conditions_generalized}, \Subsecref{Distributions-and-adjoints},
\Subsecref{Interpolation-and-B-analyticity}). That such leaps of
faith turn out to be true only further underscore how well-made the
conjecture is in its anticipation of undiscovered mathematics.

For those familiar with the smoothing properties of Littlewood-Paley
projection as well as \textbf{Bernstein inequalities} \parencite[Appendix A]{tao2006nonlinear},
the rough picture is that $e^{t\Delta}\approx P_{\leq\frac{1}{\sqrt{t}}}$.
While the introduction of curvature necessitates the change of constants
in estimates, and the boundary requires its own considerations, it
is remarkable how far we can go with this analogy. Regarding the properties
we will need for Onsager's conjecture, there is a satisfying explanation:
the theory of \textbf{sectorial operators} in functional analysis.
This, together with Hodge theory, the theory of \textbf{Besov spaces}
and \textbf{interpolation theory}, allows us to build a basic foundation
for global analysis on Riemannian manifolds in general, which will
be more than enough to handle Onsager's conjecture.

Hodge theory and sectorial operators, in their various forms, have
been used in fluid dynamics for a long time by Fujita, Kato, Giga,
Miyakawa \emph{et al. }(cf. \parencite{Fujita1964_Kato_Navier_Stokes,Miyakawa_Lp_analyticity_Navier,Giga1981_Analyticity_stokes_op,Giga1985_Lp_Stokes_solution,Baba2016}
and their references). Although we will not use them for this paper,
we also ought to mention the results regarding bisectorial operators,
$H^{\infty}$ functional calculus, and Hodge theory on rough domains
developed by Alan McIntosh, Marius Mitrea, Sylvie Monniaux\emph{ et
al.} (cf. \parencite{Alan_Hinf_calculus_86,Duong_Alan_Functional_calculi_1996,Fabes1998_Mitrea_Boundary_layers_Helmholtz,AlanMcIntosh2004_Lip_domain,Mitrea2008_Stokes_Fujita_Kato_approach,Mitrea2009_Sylvie_Neumann_Laplacian,mitrea2009_Nonlinear_hodge_NS,Gol_dshtein_2010_Mitrea_Hodge_mixed_formalism,Shen2012_Stokes_Lipschitz,mcintosh_hal_Hodge_dirac}
and their references), which generalize many Hodge-theoretic results
in this paper. Alternative formalizations of Littlewood-Paley theory
also exist (cf. \parencite{Han_2008_Besov_metric_Carnot,Kerkyacharian_2014_Heat_Dirichlet_decomposition,Feichtinger_2016_wavelet,Kriegler_2016_sectorial_Paley,DuongThinh_Weighted_besov_operator,Taniguchi_2018_Neumann_Besov}
and their references). Here, we are mainly focused on the analogy
between the heat flow and the Littlewood-Paley projection on $L^{p}$
spaces of differential forms (over manifolds with boundary), as well
as the interplay with Hodge theory.

Lastly, we also introduce \textbf{heatable currents} -- the largest
space on which the heat flow can be profitably defined -- as the
analogue to tempered distributions on manifolds (\Subsecref{Distributions-and-adjoints}).
In doing so, we will realize that the energy-conserving weak solution
in Onsager's conjecture solves the Euler equation in the sense of
heatable currents. This is an elegant insight that helps show how
interconnected these subjects are. Much like how learning the language
of measure theory can shed light on problems in calculus and familiarity
with differential geometry simplifies many calculations in fluid dynamics,
the cost of learning ostensibly complicated formalism is often dwarfed
by the benefits in clarity it brings. That being said, accessibility
is also important, and besides providing a gentle introduction to
the theory with copious references, this paper also hopes to convince
the reader of the naturality behind the formalism.

\subsection{Blackboxes}

Since we draw upon many areas, the paper is intended to be as self-contained
as possible, but we will assume familiarity with basic elements of
functional analysis, harmonic analysis and complex analysis. Some
familiarity with differential and Riemannian geometry is certainly
needed (cf. \parencite{jeffrey_lee2009manifolds,chavel2006riemannian}),
as well as \textbf{Penrose notation} (cf. \parencite[Section 2.4]{wald1984general}).
In addition, a number of blackbox theorems will be borrowed from the
following sources:
\begin{enumerate}
\item For interpolation theory: \citetitle{Lofstrom} \cite{Lofstrom} and
\citetitle{Voigt1992} \cite{Voigt1992}
\item For harmonic analysis and elements of functional analysis:
\begin{itemize}
\item \citetitle{eliasstein1971_singular_integrals} \parencite{eliasstein1971_singular_integrals}
\item \citetitle{Taylor_PDE1} \cite{Taylor_PDE1}
\item \citetitle{rieusset2002_recent_developments} \parencite{rieusset2002_recent_developments}
\end{itemize}
\item For Besov spaces: \citetitle{triebel_function_I,triebel_function_II}
\parencite{triebel_function_I,triebel_function_II}
\item For Hodge theory: \citetitle{Schwarz1995} \cite{Schwarz1995}
\item For semigroups and sectorial operators: \citetitle{engel2000one-parameter}
\cite{engel2000one-parameter} and \citetitle{Arendt_Laplace} \cite{Arendt_Laplace}
\end{enumerate}
The first three categories should be familiar with harmonic analysts.

\subsection{For the specialists}

Some noteworthy characteristics of our approach:
\begin{itemize}
\item An alternative development of the (absolute Neumann) heat flow. In
particular, the extrapolation of analyticity to $L^{p}$ spaces does
not involve establishing the resolvent estimate in \emph{Yosida's
half-plane criterion} (\Thmref{yosida_half_plane}), either via ``Agmon's
trick'' \parencite{Agmon_trick} as done in \parencite{Miyakawa_Lp_analyticity_Navier}
or manual estimates as in \parencite{Baba2016}. Instead, by abstract
Stein interpolation, we only need the local boundedness of the heat
flow on $L^{p}$, which can follow cleanly from Gronwall and integration
by parts (\Thmref{local_boundedness_Lp}). In short, functional analysis
does the heavy lifting. We also managed to attain $W^{1,p}$-analyticity
assuming the Neumann condition (\Subsecref{W1p-analyticity}), and
$B_{p,1}^{\frac{1}{p}}$-analyticity via the Leray projection (\Subsecref{Interpolation-and-B-analyticity}).
\item We do not focus on the \textbf{Stokes operator} in this paper, but
our results (\Subsecref{W1p-analyticity}, \Subsecref{Interpolation-and-B-analyticity})
do contain the case of the Stokes operator corresponding to the ``Navier-type''
/ ``free'' boundary condition, as discussed in \parencite{Miyakawa_Lp_analyticity_Navier,Giga1982_Nonstationary_first_order,Mitrea2009_Sylvie_Neumann_Laplacian,mitrea2009_Nonlinear_hodge_NS,Baba2016}
and others. This should not be confused with the Stokes operator corresponding
to the ``no-slip'' boundary condition, as discussed in \parencite{Fujita1964_Kato_Navier_Stokes,Giga1985_Lp_Stokes_solution,Mitrea2008_Stokes_Fujita_Kato_approach}
and others. See \parencite{Hieber2018_Stokes_operator_survey} for
more references.
\item For simplicity, we stay within the smooth and compact setting, which,
as Hilbert would say, is that special case containing all the germs
of generality. An effort has also been made to keep the material concrete
(as opposed to, for instance, using Hilbert complexes).
\item Heatable currents are introduced as the analogue to tempered distributions,
and we show how they naturally appear in the characterization of the
adjoints of $d$ and $\delta$ (\Subsecref{Distributions-and-adjoints}).
\item A refinement of a special case of the fractional Leibniz rule, with
the supports of functions taken into account, is given in \Thmref{product_estimate}.
\item For the proof of Onsager's conjecture, there are some subtle, but
substantial differences with \parencite{Isett2015_heat}:
\begin{itemize}
\item In \parencite{Isett2015_heat}, Besov spaces are \emph{defined} by
the heat flow, and compatibility with the usual scalar Besov spaces
is proven when $M$ is $\mathbb{R}^{n}$ or $\mathbb{T}^{n}$. Here
we will use the standard scalar Besov spaces as defined by Triebel
in \parencite{triebel_function_I,triebel_function_II}, and prove
the appropriate estimates for the heat flow by interpolation.
\item The heat flow used by Isett \& Oh (constructed by Strichartz using
dissipative operators) is generated by the \textbf{Hodge Laplacian},
which is self-adjoint in the no-boundary case. In the case with boundary,
there are four different self-adjoint versions for the Hodge Laplacian
(see \Thmref{4_version_gaffney}), and we choose the \textbf{absolute
Neumann} version. There are also heat flows generated by the \textbf{connection
Laplacian}, but we do not use them in this paper since the connection
Laplacian does not commute with the \textbf{exterior derivative} and
the Leray projection etc. The theory of dissipative operators is also
not sufficient to establish $L^{p}$-analyticity and $W^{1,p}$-analyticity
for all $p\in\left(1,\infty\right)$, so we instead use the theory
of sectorial operators, which is made for this purpose.
\item The commutator we will use is a bit different from that in \parencite{Isett2015_heat}.
This will help us eliminate some boundary terms. We will also avoid
the explicit formula and computations in \parencite[Lemma 4.4]{Isett2015_heat},
as they also lead to various boundary terms. Generally speaking, the
covariant derivative behaves badly on the boundary.
\end{itemize}
\item A calculation of the pressure by negative-order Hodge-Sobolev spaces
(\Subsecref{Calc_pressure}).
\item More results will be proven for analysis on manifolds than needed
for Onsager's conjecture, as they are of independent interest. For
the sake of accessibility, we will also review most of the relevant
background material, with the assumption that the reader is a harmonic
analyst who knows some differential geometry.
\end{itemize}
It is hard to overstate our indebtedness to all the mathematicians
whose work our theory will build upon, from harmonic analysis to Hodge
theory and sectorial operators, and yet hopefully each will be able
to find within this paper something new and interesting.

\section{Common notation\label{sec:Common-notation}}

It might not be an exaggeration to say the main difficulty in reading
a paper dealing with Hodge theory is understanding the notation, and
an effort has been made to keep our notation as standard and self-explanatory
as possible.

Some common notation we use:
\begin{itemize}
\item $A\lesssim_{x,\neg y}B$ means $A\leq CB$ where $C>0$ depends on
$x$ and not $y$. Similarly, $A\sim_{x,\neg y}B$ means $A\lesssim_{x,\neg y}B$
and $B\lesssim_{x,\neg y}A$. When the dependencies are obvious by
context, we do not need to make them explicit.
\item $\mathbb{N}_{0},\mathbb{N}_{1}:$ the set of natural numbers, starting
with $0$ and $1$ respectively.
\item DCT: dominated convergence theorem, FTC: fundamental theorem of calculus,
PTAS: passing to a subsequence, WLOG: without loss of generality.
\item TVS: topological vector space, NVS: normed vector space, SOT: strong
operator topology.
\item For TVS $X$, $Y\leq X$ means $Y$ is a subspace of $X$.
\item $\mathcal{L}(X,Y):$ the space of continuous linear maps from TVS
$X$ to $Y$. Also $\mathcal{L}(X)=\mathcal{L}(X,X)$.
\item $C^{0}(S\to Y)$: the space of bounded, continuous functions from
metric space $S$ to normed vector space $Y$. Not to be confused
with $C_{\text{loc}}^{0}(S\to Y)$, which is the space of locally
bounded, continuous functions.
\item $\left\Vert x\right\Vert _{D(A)}=\left\Vert x\right\Vert _{X}+\left\Vert Ax\right\Vert _{X}$
and $\left\Vert x\right\Vert _{D(A)}^{*}=\left\Vert Ax\right\Vert _{X}$
where $A$ is an unbounded operator on (real/complex) Banach space
$X$ and $x\in D(A).$ Note that $\left\Vert \cdot\right\Vert _{D(A)}^{*}$
is not always a norm. Also define $D(A^{\infty})=\cap_{k\in\mathbb{N}_{1}}D(A^{k}).$
\item For $\delta\in(0,\pi]$, define the open sector $\Sigma_{\delta}^{+}=\{z\in\mathbb{C}\backslash\{0\}:|\arg z|<\delta\}$,
$\Sigma_{\delta}^{-}=-\Sigma_{\delta}^{+}$, $\mathbb{D}=\{z\in\mathbb{C}:|z|<1\}.$
Also define $\Sigma_{0}^{+}=(0,\infty)$ and $\Sigma_{0}^{-}=-\Sigma_{0}^{+}$.
\item $B(x,r)$: the open ball of radius $r$ centered at $x$ in a metric
space.
\item $\mathcal{S}(\mathbb{R}^{n})$: the space of Schwartz functions on
$\mathbb{R}^{n}$, $\mathcal{S}(\overline{\Omega})$: restrictions
of Schwartz functions to the domain $\Omega\subset\mathbb{R}^{n}$.
\end{itemize}
There is also a list of other symbols we will use at the end of the
paper.

\section{Onsager's conjecture\label{sec:Onsager's-conjecture}}

\subsection{Summary of preliminaries}

At the cost of some slight duplication of exposition, we will quickly
summarize the key tools we need for the proof, and leave the development
of such tools for the rest of the paper. Alternatively, the reader
can read the theory first and come back to this section later.
\begin{defn}
For the rest of the paper, unless otherwise stated, let $M$ be a
compact, smooth, Riemannian $n$-dimensional manifold, with no or
smooth boundary. We also let $I\subset\mathbb{R}$ be an open time
interval. We write $M_{<r}=\{x\in M:\mathrm{dist}(x,\partial M)<r\}$
for $r>0$ small. Similarly define $M_{\geq r},M_{<r},M_{[r_{1},r_{2}]}$
etc. Let $\accentset{\circ}{M}$ denote the interior of $M$.

By musical isomorphism, we can consider $\mathfrak{X}M$ (the space
of \textbf{smooth vector fields}) mostly the same as $\Omega^{1}(M)$
(the space of \textbf{smooth $1$-forms}), \emph{mutatis mutandis}.
We note that $\mathfrak{X}M$, $\mathfrak{X}\left(\partial M\right)$
and $\restr{\mathfrak{X}M}{\partial M}$ are different. Unless otherwise
stated, \uline{let the implicit domain be \mbox{$M$}}, so $\mathfrak{X}$
stands for $\mathfrak{X}M$, and similarly $\Omega^{k}$ for $\Omega^{k}M$.
For $X\in\mathfrak{X}$, we write $X^{\flat}$ as its dual $1$-form.
For $\omega\in\Omega^{1}$, we write $\omega^{\sharp}$ as its dual
vector field.

Let $\mathfrak{X}_{00}\left(M\right)$ denote the set of smooth vector
fields of compact support in $\accentset{\circ}{M}$. Define $\Omega_{00}^{k}\left(M\right)$
similarly (smooth differential forms with compact support in $\accentset{\circ}{M}$).

Let $\nu$ denote the outwards unit normal vector field on $\partial M$.
$\nu$ can be extended via geodesics to a smooth vector field $\widetilde{\nu}$
which is of unit length near the boundary (and cut off at some point
away from the boundary).

For $X\in\mathfrak{X}M,$ define $\mathbf{n}X=\left\langle X,\nu\right\rangle \nu\in\left.\mathfrak{X}M\right|_{\partial M}$
(the \textbf{normal part}) and $\mathbf{t}X=\left.X\right|_{\partial M}-\mathbf{n}X$
(the \textbf{tangential part}). We note that $\mathbf{t}X$ and $\mathbf{n}X$
only depend on $\restr{X}{\partial M}$, so $\mathbf{t}$ and $\mathbf{n}$
can be defined on $\restr{\mathfrak{X}M}{\partial M}$, and $\mathbf{t}\left(\left.\mathfrak{X}M\right|_{\partial M}\right)\iso\mathfrak{X}(\partial M)$.

For $\omega\in\Omega^{k}\left(M\right),$ define $\mathbf{t}\omega$
and $\mathbf{n}\omega$ by 
\[
\mathbf{t}\omega(X_{1},...,X_{k}):=\omega(\mathbf{t}X_{1},...,\mathbf{t}X_{k})\;\;\forall X_{j}\in\mathfrak{X}M,j=1,...,k
\]
and $\mathbf{n}\omega=\left.\omega\right|_{\partial M}-\mathbf{t}\omega$.
Note that $\left(\mathbf{n}X\right)^{\flat}=\mathbf{n}X^{\flat}\;\forall X\in\mathfrak{X}$.

Let $\nabla$ denote the \textbf{Levi-Civita connection}, $d$ the
\textbf{exterior derivative}, $\delta$ the \textbf{codifferential},
and $\Delta=-\left(d\delta+\delta d\right)$ the \textbf{Hodge-Laplacian},
which is defined on vector fields by the musical isomorphism.

Familiar scalar function spaces such as $L^{p},W^{m,p}$ (\textbf{Lebesgue-Sobolev
spaces}), $B_{p,q}^{s}$ (\textbf{Besov spaces}), $C^{0,\alpha}$
(\textbf{Holder spaces}) (see \Secref{Scalar-function-spaces} for
precise definitions) can be defined on $M$ by partitions of unity
and given a unique topology (\Subsecref{On-domains}, \Subsecref{Compatibility-with-scalar}).
Similarly, we define such function spaces for \textbf{tensor fields}
and \textbf{differential forms} on $M$ by partitions of unity and
local coordinates (see \subsecref{The-setting}). For instance, we
can define $L^{2}\mathfrak{X}$ or $B_{3,1}^{\frac{1}{3}}\mathfrak{X}$.
\end{defn}

\begin{fact}
\label{fact:Fact_embed} $\forall\alpha\in\left(\frac{1}{3},1\right),\forall p\in\left(1,\infty\right):W^{1,p}\mathfrak{X}\hookrightarrow B_{p,1}^{\frac{1}{p}}\mathfrak{X\hookrightarrow}L^{p}\mathfrak{X}$
and $C^{0,\alpha}\mathfrak{X}=B_{\infty,\infty}^{\alpha}\mathfrak{X}\hookrightarrow B_{3,\infty}^{\alpha}\mathfrak{X}\hookrightarrow B_{3,1}^{\frac{1}{3}}\mathfrak{X}$
 (cf. \Subsecref{On-domains}, \Subsecref{Interpolation-=000026-embedding})
\end{fact}

\begin{defn}
We write $\left\langle \cdot,\cdot\right\rangle $ to denote the \textbf{Riemannian
fiber metric} for tensor fields on $M$. We also define the dot product
\[
\left\langle \left\langle \sigma,\theta\right\rangle \right\rangle =\int_{M}\left\langle \sigma,\theta\right\rangle \vol
\]
where $\sigma$ and $\theta$ are tensor fields of the same type,
while $\vol$ is the \textbf{Riemannian volume form}. When there is
no possible confusion, we will omit writing $\vol$.

We define $\mathfrak{X}_{N}=\{X\in\mathfrak{X}:\begin{array}{c}
\mathbf{n}X=0\end{array}\}$ (\textbf{Neumann condition}). Similarly, we can define $\Omega_{N}^{k}$.
In order to define the Neumann condition for less regular vector fields
(and differential forms), we need to use the \textbf{trace theorem}.
\end{defn}

\begin{fact}
\label{fact:fact_trace}(\Subsecref{On-domains}, \Subsecref{Compatibility-with-scalar})
Let $p\in[1,\infty)$. Then
\begin{itemize}
\item $B_{p,1}^{\frac{1}{p}}\left(M\right)\twoheadrightarrow L^{p}\left(\partial M\right)$
and $B_{p,1}^{\frac{1}{p}}\mathfrak{X}M\twoheadrightarrow L^{p}\restr{\mathfrak{X}M}{\partial M}$
are continuous surjections.
\item $\forall m\in\mathbb{N}_{1}:B_{p,1}^{m+\frac{1}{p}}\mathfrak{X}M\twoheadrightarrow B_{p,1}^{m}\restr{\mathfrak{X}M}{\partial M}\hookrightarrow W^{m,p}\restr{\mathfrak{X}M}{\partial M}$
is continuous.
\end{itemize}
\end{fact}

Also closely related is the \textbf{coarea formula}:
\begin{fact}
\label{fact:fact_coarea}(\Thmref{coarea}) Let $p\in[1,\infty)$,
$r>0$ be small and $f$ be in $B_{p,1}^{\frac{1}{p}}(M)$:
\begin{enumerate}
\item $\left([0,r)\to\mathbb{R},\rho\mapsto\left\Vert f\right\Vert _{L^{p}(\partial M_{>\rho})}\right)$
is continuous and bounded by $C\left\Vert f\right\Vert _{B_{p,1}^{\frac{1}{p}}}$
for some $C>0$.
\item $\left|M_{<r}\right|\sim_{M,\neg r}\left|\partial M\right|r$ and
$\left\Vert f\right\Vert _{L^{p}(M_{\leq r})}\sim_{\neg r}\left\Vert \left\Vert f\right\Vert _{L^{p}(\partial M_{>\rho})}\right\Vert _{L_{\rho}^{p}((0,r))}$.
\item $\left\Vert f\right\Vert _{L^{p}(M_{\leq r},\avg)}\lesssim_{\neg r}\left\Vert f\right\Vert _{B_{p,1}^{\frac{1}{p}}(M)}$
and $\left\Vert f\right\Vert _{L^{p}(M_{\leq r},\avg)}\xrightarrow{r\downarrow0}\left\Vert f\right\Vert _{L^{p}(\partial M,\avg)}$,
where $\mathrm{\avg}$ means normalizing the measure to make it a
probability measure.
\item Let $\mathfrak{f}\in L^{p}(I\to B_{p,1}^{\frac{1}{p}}(M)),$ then
$\left\Vert \mathfrak{f}\right\Vert _{L_{t}^{p}B_{p,1}^{\frac{1}{p}}(M)}\gtrsim_{\neg r}\left\Vert \mathfrak{f}\right\Vert _{L_{t}^{p}L^{p}(M_{\leq r},\avg)}\xrightarrow{r\downarrow0}\left\Vert \mathfrak{f}\right\Vert _{L_{t}^{p}L^{p}(\partial M,\avg)}$.
\end{enumerate}
Analogous results hold if $f\in B_{p,1}^{\frac{1}{p}}\mathfrak{X}$.
(\Subsecref{Compatibility-with-scalar})
\end{fact}

Therefore, we can define spaces such as $B_{3,1}^{\frac{1}{3}}\mathfrak{X}_{N}=\{X\in B_{3,1}^{\frac{1}{3}}\mathfrak{X}:\begin{array}{c}
\mathbf{n}X=0\end{array}\}$ and $W^{1,3}\mathfrak{X}_{N}$. However, something like $L^{2}\mathfrak{X}_{N}$
would not make sense since the trace map does not continuously extend
to $L^{2}\mathfrak{X}$.
\begin{defn}
We define $\mathbb{P}$ as the \textbf{Leray projection} (constructed
in \Thmref{Friedrichs_Decomposition}), which projects $\mathfrak{X}$
onto $\mathrm{Ker}\left(\restr{\mathrm{div}}{\mathfrak{X}_{N}}\right)$.
Note that the Neumann condition is enforced by $\mathbb{P}$.
\end{defn}

\begin{fact}
$\forall m\in\mathbb{N}_{0},\forall p\in\left(1,\infty\right)$, \textbf{$\mathbb{P}$
}is continuous on $W^{m,p}\mathfrak{X}$ and $\mathbb{P}\left(W^{m,p}\mathfrak{X}\right)=W^{m,p}\text{-}\mathrm{cl}\left(\mathrm{Ker}\left(\restr{\mathrm{div}}{\mathfrak{X}_{N}}\right)\right)$
(closure in the $W^{m,p}$-topology). (\Subsecref{Hodge-decomposition})
\end{fact}

We collect some results regarding our heat flow in one place:
\begin{fact}[Absolute Neumann heat flow]
\label{fact:Fact_heat} There exists a semigroup of operators $\left(S(t)\right)_{t\geq0}$
acting on $\cup_{p\in\left(1,\infty\right)}L^{p}\mathfrak{X}$ such
that
\begin{enumerate}
\item $S\left(t_{1}\right)S\left(t_{2}\right)=S\left(t_{1}+t_{2}\right)\;\forall t_{1},t_{2}\geq0$
and $S\left(0\right)=1$.
\item (\Subsecref{Lp-analyticity}) $\forall p\in\left(1,\infty\right),\forall X\in L^{p}\mathfrak{X}:$
\begin{enumerate}
\item $S(t)X\in\mathfrak{X}_{N}$ and $\partial_{t}\left(S(t)X\right)=\Delta S(t)X$
$\forall t>0$.
\item $S(t)X\xrightarrow[t\to t_{0}]{C^{\infty}}S\left(t_{0}\right)X$ $\forall t_{0}>0$.
\item $\left\Vert S(t)X\right\Vert _{W^{m,p}}\lesssim_{m,p}\left(\frac{1}{t}\right)^{\frac{m}{2}}\left\Vert X\right\Vert _{L^{p}}$
$\forall m\in\mathbb{N}_{0},\forall t\in\left(0,1\right)$.
\item $S(t)X\xrightarrow[t\to0]{L^{p}}X$.
\end{enumerate}
\item (\Subsecref{W1p-analyticity}) $\forall p\in\left(1,\infty\right),\forall X\in W^{1,p}\mathfrak{X}_{N}:$
\begin{enumerate}
\item $\left\Vert S(t)X\right\Vert _{W^{m+1,p}}\lesssim_{m,p}\left(\frac{1}{t}\right)^{\frac{m}{2}}\left\Vert X\right\Vert _{W^{1,p}}$
$\forall m\in\mathbb{N}_{0},\forall t\in\left(0,1\right)$.
\item $S(t)X\xrightarrow[t\to0]{W^{1,p}}X$.
\end{enumerate}
\item (\Thmref{compat_Helmholtz_leray}) $S\left(t\right)\mathbb{P}=\mathbb{P}S\left(t\right)$
on $W^{m,p}\mathfrak{X}$ $\forall m\in\mathbb{N}_{0},\forall p\in\left(1,\infty\right),\forall t\geq0$.
\item (\Subsecref{Lp-analyticity}) $\left\langle \left\langle S(t)X,Y\right\rangle \right\rangle =\left\langle \left\langle X,S(t)Y\right\rangle \right\rangle \forall t\geq0,\forall p\in\left(1,\infty\right),\forall X\in L^{p}\mathfrak{X},\forall Y\in L^{p'}\mathfrak{X}$.
\end{enumerate}
\end{fact}

These estimates precisely fit the analogy $e^{t\Delta}\approx P_{\leq\frac{1}{\sqrt{t}}}$
where $P$ is the \textbf{Littlewood-Paley projection}. We also stress
that the heat flow preserves the space of tangential, divergence-free
vector fields (the range of $\mathbb{P}$), and is intrinsic (with
no dependence on choices of local coordinates).

Analogous results hold for scalar functions and differential forms
(\Secref{Heat-flow}). We also have commutativity with the exterior
derivative and codifferential in the case of differential forms (\Thmref{commute_derivatives_I}).
Loosely speaking, this allows the heat flow to preserve the overall
Hodge structure on the manifold. All these properties would not be
possible under standard mollification via partitions of unity.

Note that for $X\in\mathfrak{X}$, $X\otimes X$ is not dual to a
differential form. As our heat flow is generated by the Hodge Laplacian,
it is less useful in mollifying general tensor fields (for which the
connection Laplacian is better suited). Fortunately, we will never
actually have to do so in this paper.

We observe some basic identities (cf. \Thmref{integrate_tensor_forms_by_parts}):
\begin{itemize}
\item Using \textbf{Penrose abstract index notation} (see \Subsecref{Differential-forms-=000026}),
for any smooth tensors $T_{a_{1}...a_{k}}$, we define $\left(\nabla T\right)_{ia_{1}...a_{k}}=\nabla_{i}T_{a_{1}...a_{k}}$
and $\div T=\nabla^{i}T_{ia_{2}...a_{k}}.$
\item For all smooth tensors $T_{a_{1}...a_{k}}$ and $Q_{a_{1}...a_{k+1}}$:
\[
\int_{M}\nabla_{i}\left(T_{a_{1}...a_{k}}Q^{ia_{1}...a_{k}}\right)=\int_{M}\nabla_{i}T_{a_{1}...a_{k}}Q^{ia_{1}...a_{k}}+\int_{M}T_{a_{1}...a_{k}}\nabla_{i}Q^{ia_{1}...a_{k}}=\int_{\partial M}\nu_{i}T_{a_{1}...a_{k}}Q^{ia_{1}...a_{k}}
\]
\item For $X\in\mathfrak{X}_{N},Y\in\mathfrak{X},f\in C^{\infty}(M):$
\begin{enumerate}
\item $\int_{M}Xf=\int_{M}\mathrm{div}\left(fX\right)-\int_{M}f\mathrm{div}\left(X\right)=\int_{\partial M}\left\langle fX,\nu\right\rangle -\int_{M}f\mathrm{div}X=-\int_{M}f\mathrm{div}X$
\item $\int_{M}\left\langle \mathrm{div}(X\otimes X),Y\right\rangle =-\int_{M}\left\langle X\otimes X,\nabla Y\right\rangle $
\end{enumerate}
\item $\left(\nabla_{a}\nabla_{b}-\nabla_{b}\nabla_{a}\right)T^{ij}{}_{kl}=-R_{ab\sigma}{}^{i}T^{\sigma j}{}_{kl}-R_{ab\sigma}{}^{j}T^{i\sigma}{}_{kl}+R_{abk}{}^{\sigma}T^{ij}{}_{\sigma l}+R_{abl}{}^{\sigma}T^{ij}{}_{k\sigma}$
for any tensor $T^{ij}{}_{kl}$, where $R$ is the \textbf{Riemann
curvature tensor}.\textbf{ }Similar identities hold for other types
of tensors. When we do not care about the exact indices and how they
contract, we can just write the \textbf{schematic identity }$\left(\nabla_{a}\nabla_{b}-\nabla_{b}\nabla_{a}\right)T^{ij}{}_{kl}=R*T.$
As $R$ is bounded on compact $M$, interchanging derivatives is a
zeroth-order operation on $M$. In particular, we have the \textbf{Weitzenbock
formula}:
\begin{equation}
\Delta X=\nabla_{i}\nabla^{i}X+R*X\;\forall X\in\mathfrak{X}M\label{eq:Weitzen_schematic}
\end{equation}
\item For $X\in\mathbb{P}L^{2}\mathfrak{X},Y\in\mathfrak{X},Z\in\mathfrak{X},f\in C^{\infty}\left(M\right):$
\begin{enumerate}
\item $\int_{M}Xf=0$
\item $\int_{M}\left\langle \nabla_{X}Y,Z\right\rangle =-\int_{M}\left\langle Y,\nabla_{X}Z\right\rangle $
.
\end{enumerate}
\end{itemize}
There is an elementary lemma which is useful for convergence (the
proof is straightforward and omitted):
\begin{lem}[Dense convergence]
\label{lem:dense_conv} Let $X,Y$ be (real/complex) Banach spaces
and $X_{0}\leq X$ be norm-dense. Let $(T_{j})_{j\in\mathbb{N}}$
be bounded in $\mathcal{L}(X,Y)$ and $T\in\mathcal{L}(X,Y)$.

If $T_{j}x_{0}\to Tx_{0}\;\forall x_{0}\in X_{0}$ then $T_{j}x\to Tx\;\forall x\in X$.
\end{lem}

\begin{defn}[Heatable currents]

As the heat flow does not preserve compact supports in $\accentset{\circ}{M}$,
it is not defined on distributions. This inspires the formulation
of \textbf{heatable currents}. Define:
\begin{itemize}
\item $\mathscr{D}\Omega^{k}=\Omega_{00}^{k}=\colim\{\left(\Omega_{00}^{k}\left(K\right),C^{\infty}\text{ topo}\right):K\subset\accentset{\circ}{M}\text{ compact}\}$
as the space of \textbf{test $k$-forms} with \textbf{Schwartz's topology}\footnote{Confusingly enough, ``Schwartz's topology'' refers to the topology
on the space of distributions, not the topology for Schwartz functions.} (colimit in the category of locally convex TVS).
\item $\mathscr{D}'\Omega^{k}=\left(\mathscr{D}\Omega^{k}\right)^{*}$ as
the space of\textbf{ $k\text{-}$currents} (or \textbf{distributional
$k$}-\textbf{forms}), equipped with the weak{*} topology.
\item $\mathscr{D}_{N}\Omega^{k}=\{\omega\in\Omega^{k}:\mathbf{n}\Delta^{m}\omega=0,\mathbf{n}d\Delta^{m}\omega=0\;\forall m\in\mathbb{N}_{0}\}$
as the space of \textbf{heated $k$-forms }with the Frechet $C^{\infty}$
topology and $\mathscr{D}'_{N}\Omega^{k}=\left(\mathscr{D}_{N}\Omega^{k}\right)^{*}$
as the space of\textbf{ heatable $k$-currents} (or \textbf{heatable
distributional $k$}-\textbf{forms}) with the weak{*} topology.
\item \textbf{Spacetime test forms}: $\mathscr{D}\left(I,\Omega^{k}\right)=C_{c}^{\infty}\left(I,\Omega_{00}^{k}\right)=\colim\{\left(C_{c}^{\infty}\left(I_{1},\Omega_{00}^{k}(K)\right),C^{\infty}\text{ topo}\right):I_{1}\times K\subset I\times\accentset{\circ}{M}\text{ compact}\}$
and $\mathscr{D}_{N}\left(I,\Omega^{k}\right)=\colim\{\left(C_{c}^{\infty}\left(I_{1},\mathscr{D}_{N}\Omega^{k}\right),C^{\infty}\text{ topo}\right):I_{1}\subset I\text{ compact}\}$.
\item \textbf{Spacetime distributions} $\mathscr{D}'\left(I,\Omega^{k}\right)=\mathscr{D}\left(I,\Omega^{k}\right)^{*}$,
$\mathscr{D}'_{N}\left(I,\Omega^{k}\right)=\mathscr{D}_{N}\left(I,\Omega^{k}\right)^{*}.$
\end{itemize}
In particular, $\mathscr{D}_{N}\mathfrak{X}$ is defined from $\mathscr{D}_{N}\Omega^{1}$
by the musical isomorphism, and it is invariant under our heat flow
(much like how the space of Schwartz functions $\mathcal{S}(\mathbb{R}^{n})$
is invariant under the Littlewood-Paley projection). By that analogy,
heatable currents are tempered distributions on manifolds, and we
can write 
\[
\left\langle \left\langle S(t)\Lambda,X\right\rangle \right\rangle =\left\langle \left\langle \Lambda,S\left(t\right)X\right\rangle \right\rangle \;\forall\Lambda\in\mathscr{D}'_{N}\mathfrak{X},\forall X\in\mathscr{D}_{N}\mathfrak{X},\forall t\geq0
\]
where the dot product $\left\langle \left\langle \cdot,\cdot\right\rangle \right\rangle $
is simply abuse of notation.
\end{defn}

\begin{fact}
\label{fact:D_N-basic-properties}Some basic properties of $\mathscr{D}_{N}\mathfrak{X}$
and $\mathscr{D}'_{N}\mathfrak{X}$:
\begin{itemize}
\item $\left\langle \left\langle \Delta X,Y\right\rangle \right\rangle =\left\langle \left\langle X,\Delta Y\right\rangle \right\rangle \;\forall X,Y\in\mathscr{D}_{N}\mathfrak{X}$.
(\Thmref{integrate_tensor_forms_by_parts})
\item $S(t)\Lambda\in\mathscr{D}_{N}\mathfrak{X}$ $\forall t>0,\forall\Lambda\in\mathscr{D}'_{N}\mathfrak{X}$.
(\Subsecref{Distributions-and-adjoints}, a heatable current becomes
heated once the heat flow is applied)
\item $\mathfrak{X}_{00}\subset\mathscr{D}_{N}\mathfrak{X}$ and is dense
in $L^{p}\mathfrak{X}$ $\forall p\in[1,\infty)$. Also, $L^{p}\mathfrak{X}\hookrightarrow\mathscr{D}'_{N}\mathfrak{X}$
is continuous $\forall p\in[1,\infty]$.
\item $\mathbb{P}B_{3,1}^{\frac{1}{3}}\mathfrak{X}=\mathbb{P}B_{3,1}^{\frac{1}{3}}\mathfrak{X}_{N}$,
$\mathbb{P}W^{1,p}\mathfrak{X}=\mathbb{P}W^{1,p}\mathfrak{X}_{N}$
and $\mathbb{P}\mathscr{D}_{N}\mathfrak{X}\leq\mathscr{D}_{N}\mathfrak{X}$.
(\Subsecref{Hodge-decomposition})
\item $W^{1,p}\text{-}\mathrm{cl}\left(\mathscr{D}_{N}\mathfrak{X}\right)=W^{1,p}\mathfrak{X}_{N}$
$\forall p\in\left(1,\infty\right)$ (\Subsecref{W1p-analyticity}),
$B_{3,1}^{\frac{1}{3}}\text{-}\mathrm{cl}\left(\mathbb{P}\mathscr{D}_{N}\mathfrak{X}\right)=\mathbb{P}B_{3,1}^{\frac{1}{3}}\mathfrak{X}_{N}$
(\Subsecref{Interpolation-and-B-analyticity})
\item $\forall X\in\mathscr{D}_{N}\mathfrak{X}:S(t)X\xrightarrow[t\downarrow0]{C^{\infty}}X$
and $\partial_{t}\left(S(t)X\right)=\Delta S(t)X=S(t)\Delta X\;\forall t\geq0.$
(\Thmref{Sobolev_tower}, \Subsecref{Lp-analyticity})
\item (\Subsecref{Lp-analyticity}, \Subsecref{Interpolation-and-B-analyticity})
$\forall t\in(0,1),\forall m,m'\in\mathbb{N}_{0},\forall p\in(1,\infty),\forall X\in\mathscr{D}_{N}\mathfrak{X}:$
\begin{enumerate}
\item $\left\Vert S(t)X\right\Vert _{W^{m+m',p}}\lesssim\left(\frac{1}{t}\right)^{\frac{m'}{2}}\left\Vert X\right\Vert _{W^{m,p}}$.
\item $\left\Vert S(t)X\right\Vert _{B_{p,1}^{m+m'+\frac{1}{p}}}\lesssim\left(\frac{1}{t}\right)^{\frac{1}{2p}+\frac{m'}{2}}\left\Vert X\right\Vert _{W^{m,p}}$.
\item $t^{\frac{1}{2}\left(m-\frac{1}{p}\right)}\left\Vert S(t)X\right\Vert _{W^{m,p}}+\left\Vert S(t)X\right\Vert _{B_{p,1}^{\frac{1}{p}}}\lesssim\left\Vert X\right\Vert _{B_{p,1}^{\frac{1}{p}}}$
when $m\geq1$ and $X\in\mathbb{P}\mathscr{D}_{N}\mathfrak{X}$. \\
By dense convergence (\Lemref{dense_conv}), this means $S(t)X\xrightarrow[t\downarrow0]{B_{3,1}^{\frac{1}{3}}}X$
$\forall X\in\mathbb{P}B_{3,1}^{\frac{1}{3}}\mathfrak{X}_{N}$.
\end{enumerate}
\end{itemize}
\end{fact}

\begin{cor}[Vanishing]
 \label{cor:vanishing_property}$\forall X\in\mathbb{P}B_{3,1}^{\frac{1}{3}}\mathfrak{X}_{N}:$
$s^{\frac{1}{3}}\left\Vert S(s)X\right\Vert _{W^{1,3}}\xrightarrow{s\downarrow0}0$.
\end{cor}

\begin{rem*}
So, for $\mathcal{U}\in L_{t}^{3}\mathbb{P}B_{3,1}^{\frac{1}{3}}\mathfrak{X}_{N}$:
$\text{\ensuremath{\left\Vert \mathcal{U}(t)\right\Vert _{L_{t}^{3}B_{3,1}^{\frac{1}{3}}}}\ensuremath{\ensuremath{\gtrsim}}}\left\Vert \left\Vert \sigma^{\frac{1}{3}}\left\Vert S(\sigma)\mathcal{U}(t)\right\Vert _{W^{1,3}}\right\Vert _{L_{\sigma}^{\infty}\left([0,s]\right)}\right\Vert _{L_{t}^{3}}\xrightarrow[DCT]{s\downarrow0}0$.\\
This pointwise vanishing property becomes important for the commutator
estimate in Onsager's conjecture at the critical regularity level
$\frac{1}{3}$, while higher regularity levels have enough room for
vanishing in norm (which is better).
\end{rem*}
\begin{proof}
For $Y\in\mathbb{P}\mathscr{D}_{N}\mathfrak{X}$, as $s>0$ small:
$s^{\frac{1}{3}}\left\Vert S(s)Y\right\Vert _{W^{1,3}}\lesssim s^{\frac{1}{3}}\left\Vert Y\right\Vert _{W^{1,3}}\xrightarrow{s\downarrow0}0$.
Then note $s^{\frac{1}{3}}\left\Vert S(s)X\right\Vert _{W^{1,3}}\lesssim\left\Vert X\right\Vert _{B_{3,1}^{\frac{1}{3}}}\forall X\in\mathbb{P}B_{3,1}^{\frac{1}{3}}\mathfrak{X}_{N}$,
so we can apply dense convergence (\Lemref{dense_conv}).
\end{proof}

\subsection{Searching for the proper formulation\label{subsec:Searching-for-the}}

Onsager's conjecture states that energy is conserved when $\mathcal{V}$
has enough regularity, with appropriate conditions near the boundary.
But making this statement precise is half of the challenge.
\begin{defn}
We say $\left(\mathcal{V},\mathfrak{p}\right)$ is a \textbf{weak
solution }to the Euler equation when
\begin{itemize}
\item $\mathcal{V}\in L_{\mathrm{loc}}^{2}\left(I,\mathbb{P}L^{2}\mathfrak{X}\right)$,
$\mathfrak{p}\in L_{\mathrm{loc}}^{1}(I\times M)$
\item $\forall\mathcal{X}\in C_{c}^{\infty}\left(I,\mathfrak{X}_{00}\right):$$\iint_{I\times M}\left\langle \mathcal{V},\partial_{t}\mathcal{X}\right\rangle +\left\langle \mathcal{V}\otimes\mathcal{V},\nabla\mathcal{X}\right\rangle +\mathfrak{p}\Div\mathcal{X}=0$.
\end{itemize}
\end{defn}

The last condition means $\partial_{t}\mathcal{V}+\mathrm{div}(\mathcal{V}\otimes\mathcal{V})+\grad\mathfrak{p}=0$
as spacetime distributions. Note that $\mathcal{V}\otimes\mathcal{V}\in L_{\mathrm{loc}}^{1}\left(I,L^{1}\mathfrak{X}\right)$
so it is a distribution.

The keen reader should notice we use a different font for time-dependent
vector fields.

There is not enough time-regularity for FTC, and we cannot say
\[
\left\langle \left\langle \mathcal{V}\left(t_{1}\right),X\right\rangle \right\rangle -\left\langle \left\langle \mathcal{V}\left(t_{0}\right),X\right\rangle \right\rangle =\int_{t_{0}}^{t_{1}}\left\langle \left\langle \mathcal{V}\otimes\mathcal{V},\nabla X\right\rangle \right\rangle +\int_{t_{0}}^{t_{1}}\int_{M}\mathfrak{p}\Div X\;\forall X\in\mathfrak{X}_{00}
\]
But we can still use approximation to the identity (in the time variable)
near $t_{0}$,$t_{1}$, as well as Lebesgue differentiation to get
something similar for a.e. $t_{0},t_{1}$. By using dense convergence
(\Lemref{dense_conv}) and modifying $I$ into $I_{0}\subset I$ such
that $\left|I\backslash I_{0}\right|=0$, we can say $\mathcal{V}\in C_{\mathrm{loc}}^{0}\left(I_{0},\left(L^{2}\mathfrak{X},\mathrm{weak}\right)\right)\leq L_{\mathrm{loc}}^{\infty}\left(I,L^{2}\mathfrak{X}\right).$

We do not have $\mathcal{V}\in C_{\mathrm{loc}}^{0}\left(I,L^{2}\mathfrak{X}\right)$,
so energy conservation only means $\partial_{t}\left(\left\Vert \mathcal{V}(t)\right\Vert _{L^{2}\mathfrak{X}}^{2}\right)=0$
as a distribution. In other words, the goal is to show 
\[
\int_{I}\eta'(t)\left\langle \left\langle \mathcal{V}(t),\mathcal{V}(t)\right\rangle \right\rangle \mathrm{d}t=0\;\forall\eta\in C_{c}^{\infty}(I)
\]

Next, having the test vector field $\mathcal{X}\in C_{c}^{\infty}\left(I,\mathfrak{X}_{00}\right)$
can be quite restrictive, since the heat flow (much like the Littlewood-Paley
projection) does not preserve compact supports in $\accentset{\circ}{M}$.
We need a notion that is more in tune with our theory.
\begin{defn}
We say $\left(\mathcal{V},\mathfrak{p}\right)$ is a \textbf{Hodge
weak solution }to the Euler equation when $\mathcal{V}\in L_{\mathrm{loc}}^{2}\left(I,\mathbb{P}L^{2}\mathfrak{X}\right)$,
$\mathfrak{p}\in L_{\mathrm{loc}}^{1}(I\times M)$ and
\[
\forall\mathcal{X}\in C_{c}^{\infty}\left(I,\mathfrak{X}_{N}\right):\iint_{I\times M}\left\langle \mathcal{V},\partial_{t}\mathcal{X}\right\rangle +\left\langle \mathcal{V\otimes V},\nabla\mathcal{X}\right\rangle +\mathfrak{p}\Div\mathcal{X}=0
\]
Now this looks better, since $\mathfrak{X}_{N}$ is invariant under
the heat flow. However, this is a leap of faith we will need to justify
later (cf. \Subsecref{Justification}).

As $\mathbb{P}\mathfrak{X}\leq\mathfrak{X}_{N}$, we can go further
and say $\mathcal{V}$ is a \textbf{Hodge-Leray weak solution }to
the Euler equation when $\mathcal{V}\in L_{\mathrm{loc}}^{2}\left(I,\mathbb{P}L^{2}\mathfrak{X}\right)$
and
\[
\forall\mathcal{X}\in C_{c}^{\infty}\left(I,\mathbb{P}\mathfrak{X}\right):\iint_{I\times M}\left\langle \mathcal{V},\partial_{t}\mathcal{X}\right\rangle +\left\langle \mathcal{V\otimes V},\nabla\mathcal{X}\right\rangle =0
\]
This would help give a formulation of Onsager's conjecture that does
not depend on the pressure, similar to \parencite{Skipper2018_onsager_no_pressure}.
\end{defn}

Next, we look at the conditions for $\mathcal{V}$ and $\mathfrak{p}$
near $\partial M$. In \parencite{Titi_Onsager_bounded_domain}, they
assumed $\mathcal{V}\in L_{t}^{3}C^{0,\alpha}\mathfrak{X}_{N}$ with
$\alpha\in\left(\frac{1}{3},1\right)$. In \parencite{Nguyen_Onsager_bounded_domain},
they assumed $\mathcal{V}\in L_{t}^{3}B_{3,\infty}^{\alpha}\mathfrak{X}$
($\alpha\in\left(\frac{1}{3},1\right)$) with a more general ``strip
decay'' condition:
\begin{itemize}
\item $\left\Vert \mathcal{V}\right\Vert _{L_{t}^{3}L^{3}\left(M_{<r},\mathrm{avg}\right)}^{2}\left\Vert \left\langle \mathcal{V},\widetilde{\nu}\right\rangle \right\Vert _{L_{t}^{3}L^{3}\left(M_{<r},\mathrm{avg}\right)}\xrightarrow{r\downarrow0}0$
\item $\left\Vert \mathfrak{p}\right\Vert _{L_{t}^{\frac{3}{2}}L^{\frac{3}{2}}\left(M_{<r},\mathrm{avg}\right)}\left\Vert \left\langle \mathcal{V},\widetilde{\nu}\right\rangle \right\Vert _{L_{t}^{3}L^{3}\left(M_{<r},\mathrm{avg}\right)}\xrightarrow{r\downarrow0}0$.
\end{itemize}
In \parencite{Titi_Gwiazda_conservation_entropy} (the most recent
result), they assumed $\mathcal{V}\in L_{t}^{3}\underline{B}_{3,\mathrm{VMO}}^{1/3}\mathfrak{X}$
(see the paper for the full definition), along with a minor relaxation
for the ``strip decay'' condition:

\[
\left\Vert \left(\frac{\left|\mathcal{V}\right|^{2}}{2}+\mathfrak{p}\right)\left\langle \mathcal{V},\widetilde{\nu}\right\rangle \right\Vert _{L_{t}^{1}L^{1}\left(M_{[\frac{r}{4},\frac{r}{2}]},\mathrm{avg}\right)}\xrightarrow{r\downarrow0}0
\]

When $\mathcal{V}\in L_{t}^{3}B_{3,1}^{\frac{1}{3}}\mathfrak{X}$,
$\left\Vert \left\langle \mathcal{V},\widetilde{\nu}\right\rangle \right\Vert _{L_{t}^{3}L^{3}\left(M_{<r},\mathrm{avg}\right)}\xrightarrow{r\downarrow0}\left\Vert \left\langle \mathcal{V},\nu\right\rangle \right\Vert _{L_{t}^{3}L^{3}\left(\partial M,\mathrm{avg}\right)}$
by \Factref{fact_coarea}. This motivates our formulation later in
\Subsecref{Proof-of-Onsager}, where we put $\mathcal{V}\in L_{t}^{3}\mathbb{P}B_{3,1}^{\frac{1}{3}}\mathfrak{X}_{N}$.

\subsection{Justification of formulation\label{subsec:Justification}}

We define the cutoffs 
\begin{equation}
\psi_{r}(x)=\Psi_{r}\left(\mathrm{dist}\left(x,\partial M\right)\right)\label{eq:cutoff}
\end{equation}
where $r>0$ small, $\Psi_{r}\in C^{\infty}([0,\infty),[0,\infty))$
such that $\mathbf{1}_{[0,\frac{3}{4}r)}\geq\Psi_{r}\geq\mathbf{1}_{[0,\frac{r}{2}]}$
and $\left\Vert \Psi'_{r}\right\Vert _{\infty}\lesssim\frac{1}{r}.$\\
Then $\nabla\psi_{r}(x)=f_{r}(x)\widetilde{\nu}(x)$ where $|f_{r}(x)|\lesssim\frac{1}{r}$
and $\supp\psi_{r}\subset M_{<r}$. \nomenclature{$\psi_{r},f_{r}$}{cutoffs on $M$ living near the boundary\nomrefpage}

Let $(\mathcal{V},\mathfrak{p})$ be a weak solution to the Euler
equation and $\alpha\in(\frac{1}{3},1)$. Define different conditions:
\begin{enumerate}
\item $\mathcal{V}\in L_{t}^{3}C^{0,\alpha}\mathfrak{X}_{N}$.
\item $\mathcal{V}\in L_{t}^{3}B_{3,\infty}^{\alpha}\mathfrak{X}$ and $\left\Vert \mathcal{V}\right\Vert _{L_{t}^{3}L^{3}\left(M_{<r},\mathrm{avg}\right)}^{2}\left\Vert \left\langle \mathcal{V},\widetilde{\nu}\right\rangle \right\Vert _{L_{t}^{3}L^{3}\left(M_{<r},\mathrm{avg}\right)}\xrightarrow{r\downarrow0}0$.
\item $\mathcal{V}\in L_{t}^{3}B_{3,1}^{\frac{1}{3}}\mathfrak{X}_{N}$.
\item $(\mathcal{V},\mathfrak{p})$ is a Hodge weak solution.
\item $\mathcal{V}$ is a Hodge-Leray weak solution.
\end{enumerate}
\begin{thm}
\label{thm:different_conditions_generalized} We have $(1)\implies(2)\implies(3)\implies\left(4\right)\implies(5)$.
\end{thm}

\begin{proof}
By \Factref{Fact_embed}, $C^{0,\alpha}\mathfrak{X}_{N}=B_{\infty,\infty}^{\alpha}\mathfrak{X}_{N}\hookrightarrow B_{3,\infty}^{\alpha}\mathfrak{X}_{N}\hookrightarrow B_{3,1}^{\frac{1}{3}}\mathfrak{X}_{N}$.
Then by the coarea formula,
\[
\left\Vert \left\langle \mathcal{V},\widetilde{\nu}\right\rangle \right\Vert _{L_{t}^{3}L^{3}\left(M_{<r},\mathrm{avg}\right)}^{3}\lesssim\left\Vert \mathcal{V}\right\Vert _{L_{t}^{3}L^{3}\left(M_{<r},\mathrm{avg}\right)}^{2}\left\Vert \left\langle \mathcal{V},\widetilde{\nu}\right\rangle \right\Vert _{L_{t}^{3}L^{3}\left(M_{<r},\mathrm{avg}\right)}\lesssim\left\Vert \mathcal{V}\right\Vert _{L_{t}^{3}B_{3,1}^{\frac{1}{3}}\mathfrak{X}}^{2}\left\Vert \left\langle \mathcal{V},\widetilde{\nu}\right\rangle \right\Vert _{L_{t}^{3}L^{3}\left(M_{<r},\mathrm{avg}\right)}
\]

So for $\mathcal{V}\in L_{t}^{3}B_{3,1}^{\frac{1}{3}}\mathfrak{X}$:
$\left\Vert \mathcal{V}\right\Vert _{L_{t}^{3}L^{3}\left(M_{<r},\mathrm{avg}\right)}^{2}\left\Vert \left\langle \mathcal{V},\widetilde{\nu}\right\rangle \right\Vert _{L_{t}^{3}L^{3}\left(M_{<r},\mathrm{avg}\right)}\xrightarrow{r\downarrow0}0\iff\left\Vert \left\langle \mathcal{V},\nu\right\rangle \right\Vert _{L_{t}^{3}L^{3}\left(\partial M\right)}=0\iff\mathbf{n}\mathcal{V}=0$.

As $\left(4\right)\implies(5)$ is obvious, the only thing left is
to show $(3)\implies\left(4\right)$. Recall the cutoffs $\psi_{r}$
from \Eqref{cutoff}.

Let $I_{1}\subset I$ be bounded and $\mathcal{X}\in C_{c}^{\infty}\left(I_{1},\mathfrak{X}_{N}\right)$,
then $\left(1-\psi_{r}\right)\mathcal{X}\in C_{c}^{\infty}\left(I,\mathfrak{X}_{00}\right)$,
and so by the definition of weak solution:
\begin{align*}
\begin{aligned}0= & \iint_{I\times M}\left(1-\psi_{r}\right)\left\langle \mathcal{V},\partial_{t}\mathcal{X}\right\rangle +\left\langle \mathcal{V},\nabla_{\mathcal{V}}\left(\left(1-\psi_{r}\right)\mathcal{X}\right)\right\rangle +\mathfrak{p}\Div\left(\left(1-\psi_{r}\right)\mathcal{X}\right)\\
= & \iint_{I\times M}\left(1-\psi_{r}\right)\left(\left\langle \mathcal{V},\partial_{t}\mathcal{X}\right\rangle +\left\langle \mathcal{V},\nabla_{\mathcal{V}}\mathcal{X}\right\rangle +\mathfrak{p}\Div\mathcal{X}\right)-\iint_{I\times M}\left(\left\langle \mathcal{V},\nabla\psi_{r}\right\rangle \left\langle \mathcal{V},\mathcal{X}\right\rangle +\mathfrak{p}\left\langle \mathcal{X},\nabla\psi_{r}\right\rangle \right)
\end{aligned}
\end{align*}
We are done if the first term goes to zero as $r\downarrow0$ . So
we only need to show the second term goes to zero. Since $\nabla\psi_{r}=f_{r}\widetilde{\nu}$
and $\supp\psi_{r}\subset M_{<r}$, we only need to bound
\begin{align*}
 & \left|\iint_{I_{1}\times M_{<r}}f_{r}\left\langle \mathcal{V},\widetilde{\nu}\right\rangle \left\langle \mathcal{V},\mathcal{X}\right\rangle +\mathfrak{p}f_{r}\left\langle \mathcal{X},\widetilde{\nu}\right\rangle \right|\\
\lesssim & \frac{1}{r}\left\Vert \mathcal{V}\right\Vert _{L_{t}^{3}L^{3}\left(M_{<r}\right)}\left\Vert \left\langle \mathcal{V},\widetilde{\nu}\right\rangle \right\Vert _{L_{t}^{3}L^{3}\left(M_{<r}\right)}\left\Vert \mathcal{X}\right\Vert _{L_{t}^{3}L^{3}\left(M_{<r}\right)}+\frac{1}{r}\left\Vert \mathfrak{p}\right\Vert _{L^{1}(I_{1}\times M_{<r})}\left\Vert \left\langle \mathcal{X},\widetilde{\nu}\right\rangle \right\Vert _{L_{t}^{\infty}L^{\infty}\left(M_{<r}\right)}\\
\lesssim & \left\Vert \mathcal{V}\right\Vert _{L_{t}^{3}L^{3}\left(M_{<r},\mathrm{avg}\right)}\left\Vert \left\langle \mathcal{V},\widetilde{\nu}\right\rangle \right\Vert _{L_{t}^{3}L^{3}\left(M_{<r},\mathrm{avg}\right)}\left\Vert \mathcal{X}\right\Vert _{L_{t}^{3}L^{3}\left(M_{<r},\mathrm{avg}\right)}+\left\Vert \mathfrak{p}\right\Vert _{L^{1}(I_{1}\times M_{<r})}\left\Vert \left\langle \mathcal{X},\widetilde{\nu}\right\rangle \right\Vert _{L_{t}^{\infty}C^{0,1}\left(M_{<r}\right)}\\
\lesssim & \left\Vert \mathcal{V}\right\Vert _{L_{t}^{3}B_{3,1}^{\frac{1}{3}}\mathfrak{X}}\left\Vert \left\langle \mathcal{V},\widetilde{\nu}\right\rangle \right\Vert _{L_{t}^{3}L^{3}\left(M_{<r},\mathrm{avg}\right)}\left\Vert \mathcal{X}\right\Vert _{L_{t}^{3}B_{3,1}^{\frac{1}{3}}\mathfrak{X}}+\left\Vert \mathfrak{p}\right\Vert _{L^{1}(I_{1}\times M_{<r})}\left\Vert \left\langle \mathcal{X},\widetilde{\nu}\right\rangle \right\Vert _{L_{t}^{\infty}C^{0,1}\left(M_{<r}\right)} & \xrightarrow{r\downarrow0}0
\end{align*}
We used the estimate $\left\Vert \left\langle \mathcal{X},\widetilde{\nu}\right\rangle \right\Vert _{L^{\infty}\left(M_{<r}\right)}\lesssim r\left\Vert \left\langle \mathcal{X},\widetilde{\nu}\right\rangle \right\Vert _{C^{0,1}\left(M_{<r}\right)}$
since $\left\langle \mathcal{X},\nu\right\rangle =0$ on $\partial M$.
\end{proof}
\begin{rem*}
Interestingly, as \Subsecref{Proof-of-Onsager} will show, no ``strip
decay'' condition involving $\mathfrak{p}$ seems to be necessary.
See the end of \Subsecref{Onsager's-conjecture_intro} for a discussion
of this minor improvement.

We briefly note that when $\partial M=\emptyset,$ it is customary
to set $\mathrm{dist}\left(x,\partial M\right)=\infty$, and $\psi_{r}=0$,
$M_{>r}=M=\accentset{\circ}{M}$, $M_{<r}=\emptyset$, and $\mathscr{D}_{N}\mathfrak{X}M=\mathscr{D}\mathfrak{X}M=\mathfrak{X}M$.
\end{rem*}

\subsection{Heating the nonlinear term}

\label{subsec:heating_nonlinear} Let $U,V\in B_{3,1}^{\frac{1}{3}}\mathfrak{X}$.
Then $U\otimes V\in L^{1}\mathfrak{X}$ and $\Div\left(U\otimes V\right)$
is defined as a distribution. To apply the heat flow to $\Div\left(U\otimes V\right)$,
we need to define $\left(\Div\left(U\otimes V\right)\right)^{\flat}$
so that it is heatable.\\
Recall integration by parts:
\[
\left\langle \left\langle \Div\left(Y\otimes Z\right),X\right\rangle \right\rangle =-\left\langle \left\langle Y\otimes Z,\nabla X\right\rangle \right\rangle +\int_{\partial M}\left\langle \nu,Y\right\rangle \left\langle Z,X\right\rangle \;\forall X,Y,Z\in\mathfrak{X}\left(M\right)
\]
Observe that for $X\in\mathfrak{X}$, even though $\left\langle \left\langle \Div\left(U\otimes V\right),X\right\rangle \right\rangle $
is not defined, $\int_{\partial M}\left\langle \nu,U\right\rangle \left\langle V,X\right\rangle -\left\langle \left\langle U\otimes V,\nabla X\right\rangle \right\rangle $
is well-defined by the trace theorem. So we will define the heatable
$1$-current $\left(\Div\left(U\otimes V\right)\right)^{\flat}$ by
\[
\left\langle \left\langle \Div\left(U\otimes V\right),X\right\rangle \right\rangle =-\left\langle \left\langle U\otimes V,\nabla X\right\rangle \right\rangle +\int_{\partial M}\left\langle \nu,U\right\rangle \left\langle V,X\right\rangle \;\forall X\in\mathscr{D}_{N}\mathfrak{X}\;(X\text{ is heated})
\]
It is continuous on $\mathscr{D}_{N}\mathfrak{X}$ since $\left|\left\langle \left\langle \Div\left(U\otimes V\right),X\right\rangle \right\rangle \right|\lesssim\left\Vert U\right\Vert _{B_{3,1}^{\frac{1}{3}}}\left\Vert V\right\Vert _{B_{3,1}^{\frac{1}{3}}}\left\Vert X\right\Vert _{B_{3,1}^{\frac{1}{3}}}+\left\Vert U\right\Vert _{L^{3}}\left\Vert V\right\Vert _{L^{3}}\left\Vert \nabla X\right\Vert _{L^{3}}$.
By the same formula and reasoning, we see that $\left(\Div\left(U\otimes V\right)\right)^{\flat}$
is not just heatable, but also a continuous linear functional on $\left(\mathfrak{X}\left(M\right),C^{\infty}\text{ topo}\right)$.

On the other hand, we can get away with less regularity by assuming
$U\in\mathbb{P}L^{2}\mathfrak{X}.$ Then we simply need to define
$\left\langle \left\langle \Div\left(U\otimes V\right),X\right\rangle \right\rangle =-\left\langle \left\langle U\otimes V,\nabla X\right\rangle \right\rangle \;\forall X\in\mathfrak{X}$.

In short, $\left(\Div\left(U\otimes V\right)\right)^{\flat}$ is heatable
when $U\in\mathbb{P}L^{2}\mathfrak{X}$ and $V\in L^{2}\mathfrak{X}$.
Consequently, by \Thmref{different_conditions_generalized}, when
$\left(\mathcal{V},\mathfrak{p}\right)$ is a weak solution to the
Euler equation and $\mathcal{V}\in L_{t}^{3}B_{3,1}^{\frac{1}{3}}\mathfrak{X}_{N}$:
$\left(\mathcal{V},\mathfrak{p}\right)$ is a Hodge weak solution
and
\begin{equation}
\boxed{\partial_{t}\mathcal{V}+\mathrm{div}(\mathcal{V}\otimes\mathcal{V})+\grad\mathfrak{p}=0\text{ in }\mathscr{D}_{N}'\left(I,\mathfrak{X}\right).}\label{eq:sense_heatable_current}
\end{equation}

\subsection{Proof of Onsager's conjecture\label{subsec:Proof-of-Onsager}}

For the rest of the proof, we will write $e^{t\Delta}$ for $S(t)$,
as we will not need another heat flow. For $\varepsilon>0$ and vector
field $X$, we will write $X^{\varepsilon}$ for $e^{\varepsilon\Delta}X$.
\nomenclature{$e^{t\Delta}$}{the absolute Neumann heat flow, defined for the proof of Onsager's conjecture\nomrefpage}

We opt to formulate the conjecture without mentioning the pressure
(see \Subsecref{Justification} for the justification).
\begin{thm}[Onsager's conjecture]
 Let $M$ be a compact, oriented Riemannian manifold with no or smooth
boundary. Let $\mathcal{V}\in L_{t}^{3}\mathbb{P}B_{3,1}^{\frac{1}{3}}\mathfrak{X}_{N}$
such that $\forall\mathcal{X}\in C_{c}^{\infty}\left(I,\mathbb{P}\mathfrak{X}\right):$$\iint_{I\times M}\left\langle \mathcal{V},\partial_{t}\mathcal{X}\right\rangle +\left\langle \mathcal{V}\otimes\mathcal{V},\nabla\mathcal{X}\right\rangle =0$
(Hodge-Leray weak solution).

Then we can show
\[
\int_{I}\eta'(t)\left\langle \left\langle \mathcal{V}(t),\mathcal{V}(t)\right\rangle \right\rangle \mathrm{d}t=0\;\forall\eta\in C_{c}^{\infty}(I)
\]
Consequently, $\left\langle \left\langle \mathcal{V}(t),\mathcal{V}(t)\right\rangle \right\rangle $
is constant for a.e. $t\in I$.
\end{thm}

As usual, there is a \textbf{commutator estimate} which we will leave
for later:
\begin{align}
 & \int_{I}\eta\left\langle \left\langle \Div\left(\mathcal{U}\otimes\mathcal{U}\right)^{2\varepsilon},\mathcal{U}^{2\varepsilon}\right\rangle \right\rangle -\int_{I}\eta\left\langle \left\langle \Div\left(\mathcal{U}^{2\varepsilon}\otimes\mathcal{U}^{2\varepsilon}\right),\mathcal{U}^{2\varepsilon}\right\rangle \right\rangle \nonumber \\
= & \int_{I}\eta\left\langle \left\langle \Div\left(\mathcal{U}\otimes\mathcal{U}\right)^{3\varepsilon},\mathcal{U}^{\varepsilon}\right\rangle \right\rangle -\int_{I}\eta\left\langle \left\langle \Div\left(\mathcal{U}^{2\varepsilon}\otimes\mathcal{U}^{2\varepsilon}\right)^{\varepsilon},\mathcal{U}^{\varepsilon}\right\rangle \right\rangle \xrightarrow{\varepsilon\downarrow0}0\label{eq:commutator_est}
\end{align}
for all $\mathcal{U}\in L_{t}^{3}\mathbb{P}B_{3,1}^{\frac{1}{3}}\mathfrak{X}_{N},\eta\in C_{c}^{\infty}\left(I\right)$.

Notation: we write $\Div\left(\mathcal{U}\otimes\mathcal{U}\right)^{\varepsilon}$
for $\left(\Div\left(\mathcal{U}\otimes\mathcal{U}\right)\right)^{\varepsilon}$
and $\nabla\mathcal{U}^{\varepsilon}$ for $\nabla\left(\mathcal{U}^{\varepsilon}\right)$
(recall that the heat flow does not work on tensors $\mathcal{U}\otimes\mathcal{U}$
and $\nabla\mathcal{U}$). Compared with \parencite{Isett2015_heat},
our commutator estimate looks a bit different, to ease some integration
by parts procedures down the line.
\begin{rem*}
For any $U$ in $\mathbb{P}L^{2}\mathfrak{X}$, $\Div\left(U\otimes U\right)^{\flat}$
is a heatable $1$-current (see \Subsecref{heating_nonlinear}). In
particular, for $\varepsilon>0$, $\Div\left(U\otimes U\right)^{\varepsilon}$
is smooth and
\begin{equation}
\left\langle \left\langle \Div\left(U\otimes U\right)^{\varepsilon},Y\right\rangle \right\rangle =-\left\langle \left\langle U\otimes U,\nabla\left(Y^{\varepsilon}\right)\right\rangle \right\rangle \;\forall Y\in\mathfrak{X}\label{eq:distribution_parts_ok}
\end{equation}
Consequently, \Eqref{commutator_est} is well-defined.
\end{rem*}
\begin{thm}[Onsager]
 Assume \Eqref{commutator_est} is true. Then $\int_{I}\eta'(t)\left\langle \left\langle \mathcal{V}(t),\mathcal{V}(t)\right\rangle \right\rangle \mathrm{d}t=0$.
\end{thm}

\begin{proof}
Let $\Phi\in C_{c}^{\infty}(\mathbb{R})$ and $\Phi_{\tau}\xrightarrow{\tau\downarrow0}\delta_{0}$
be a radially symmetric mollifier. Write $\mathcal{V}^{\varepsilon}$
for $e^{\varepsilon\Delta}\mathcal{V}$ (spatial mollification) and
$\mathcal{V}_{\tau}$ for $\Phi_{\tau}*\mathcal{V}$ (temporal mollification).
First, we mollify in time and space
\[
\frac{1}{2}\int_{I}\eta'\left\langle \left\langle \mathcal{V},\mathcal{V}\right\rangle \right\rangle \stackrel{\text{DCT}}{=}\lim_{\varepsilon\downarrow0}\lim_{\tau\downarrow0}\frac{1}{2}\int_{I}\eta'\left\langle \left\langle \mathcal{V}_{\tau}^{\varepsilon},\mathcal{V}_{\tau}^{\varepsilon}\right\rangle \right\rangle 
\]
Then we want to get rid of the time derivative:
\begin{align*}
\frac{1}{2}\int_{I}\eta'\left\langle \left\langle \mathcal{V}_{\tau}^{\varepsilon},\mathcal{V}_{\tau}^{\varepsilon}\right\rangle \right\rangle  & =-\int_{I}\eta\left\langle \left\langle \partial_{t}\mathcal{V}_{\tau}^{\varepsilon},\mathcal{V}_{\tau}^{\varepsilon}\right\rangle \right\rangle =-\int_{I}\left\langle \left\langle \partial_{t}\left(\eta\mathcal{V}_{\tau}^{\varepsilon}\right),\mathcal{V}_{\tau}^{\varepsilon}\right\rangle \right\rangle +\int_{I}\eta'\left\langle \left\langle \mathcal{V}_{\tau}^{\varepsilon},\mathcal{V}_{\tau}^{\varepsilon}\right\rangle \right\rangle 
\end{align*}
Then we use the definition of Hodge-Leray weak solution, and exploit
the commutativity between spatial and temporal operators:
\begin{align*}
\frac{1}{2}\int_{I}\eta'\left\langle \left\langle \mathcal{V}_{\tau}^{\varepsilon},\mathcal{V}_{\tau}^{\varepsilon}\right\rangle \right\rangle  & =\int_{I}\left\langle \left\langle \partial_{t}\left(\eta\mathcal{V}_{\tau}^{\varepsilon}\right),\mathcal{V}_{\tau}^{\varepsilon}\right\rangle \right\rangle =\int_{I}\left\langle \left\langle \partial_{t}\left[\left(\eta\mathcal{V}_{\tau}^{2\varepsilon}\right)_{\tau}\right],\mathcal{V}\right\rangle \right\rangle \\
 & =-\int_{I}\left\langle \left\langle \nabla\left[\left(\eta\mathcal{V}_{\tau}^{2\varepsilon}\right)_{\tau}\right],\mathcal{V}\otimes\mathcal{V}\right\rangle \right\rangle  & \text{ as }\left(\eta\mathcal{V}_{\tau}^{2\varepsilon}\right)_{\tau}\in C_{c}^{\infty}\left(I,\mathbb{P}\mathfrak{X}\right)\\
 & =-\int_{I}\left\langle \left\langle \left[\eta\left(\nabla\mathcal{V}_{\tau}^{2\varepsilon}\right)\right]_{\tau},\mathcal{V}\otimes\mathcal{V}\right\rangle \right\rangle \\
 & =-\int_{I}\eta\left\langle \left\langle \left(\nabla\mathcal{V}^{2\varepsilon}\right)_{\tau},\left(\mathcal{V}\otimes\mathcal{V}\right)_{\tau}\right\rangle \right\rangle 
\end{align*}
As there is no longer a time derivative on $\mathcal{V}$, we get
rid of $\tau$ by letting $\tau\downarrow0$ (fine as $\mathcal{V}$
is $L^{3}$ in time). Recall \Eqref{distribution_parts_ok}:
\begin{align*}
\frac{1}{2}\int_{I}\eta'\left\langle \left\langle \mathcal{V}^{\varepsilon},\mathcal{V}^{\varepsilon}\right\rangle \right\rangle  & =-\int_{I}\eta\left\langle \left\langle \nabla\left(\mathcal{V}^{2\varepsilon}\right),\mathcal{V}\otimes\mathcal{V}\right\rangle \right\rangle =\int_{I}\eta\left\langle \left\langle \mathcal{V}^{\varepsilon},\Div\left(\mathcal{V}\otimes\mathcal{V}\right)^{\varepsilon}\right\rangle \right\rangle \\
 & =\int_{I}\eta\left\langle \left\langle \mathcal{V}^{\varepsilon},\Div\left(\mathcal{V}^{\varepsilon}\otimes\mathcal{V^{\varepsilon}}\right)\right\rangle \right\rangle +o_{\varepsilon}(1) & \text{(commutator estimate)}\\
 & =\int_{I}\eta\left\langle \left\langle \mathcal{V}^{\varepsilon},\nabla_{\mathcal{V}^{\varepsilon}}\mathcal{V^{\varepsilon}}\right\rangle \right\rangle +o_{\varepsilon}(1)=\int_{I}\eta\int_{M}\mathcal{V}^{\varepsilon}\left(\frac{\left|\mathcal{V}^{\varepsilon}\right|^{2}}{2}\right)+o_{\varepsilon}(1)=o_{\varepsilon}(1) & \text{as }\mathcal{V}^{\varepsilon}\in\mathbb{P}\mathfrak{X}
\end{align*}
So $\frac{1}{2}\int_{I}\eta'\left\langle \left\langle \mathcal{V},\mathcal{V}\right\rangle \right\rangle =\lim_{\varepsilon\downarrow0}\lim_{\tau\downarrow0}\frac{1}{2}\int_{I}\eta'\left\langle \left\langle \mathcal{V}_{\tau}^{\varepsilon},\mathcal{V}_{\tau}^{\varepsilon}\right\rangle \right\rangle =\lim_{\varepsilon\downarrow0}\frac{1}{2}\int_{I}\eta'\left\langle \left\langle \mathcal{V}^{\varepsilon},\mathcal{V}^{\varepsilon}\right\rangle \right\rangle =0$.
\end{proof}
The proof is short and did not much use the Besov regularity of $\mathcal{V}$.
It is the commutator estimate that presents the main difficulty. We
proceed similarly as in \parencite{Isett2015_heat}.

Let $\mathcal{U}\in L_{t}^{3}\mathbb{P}B_{3,1}^{\frac{1}{3}}\mathfrak{X}_{N}$.
By setting $\mathcal{U}(t)$ to $0$ for $t$ in a null set, WLOG
$\mathcal{U}(t)\in\mathbb{P}B_{3,1}^{\frac{1}{3}}\mathfrak{X}_{N}\;\forall t\in I$.
Define the commutator 
\[
\mathcal{W}(t,s)=\Div\left(\mathcal{U}(t)\otimes\mathcal{U}(t)\right)^{3s}-\Div\left(\mathcal{U}\left(t\right)^{2s}\otimes\mathcal{U}\left(t\right)^{2s}\right)^{s}
\]

When $t$ and $s$ are implicitly understood, we will not write them.
As $\Div\left(\mathcal{U}(t)\otimes\mathcal{U}(t)\right)^{3s}$ solves
$\left(\partial_{s}-3\Delta\right)\mathcal{X}=0$, we define $\mathcal{N}=\left(\partial_{s}-3\Delta\right)\mathcal{W}$.
Then $\mathcal{W}$ and $\mathcal{N}$ obey the Duhamel formula:
\begin{lem}[Duhamel formulas]
$ $
\begin{enumerate}
\item $\mathcal{W}(t,s)\xrightarrow{s\downarrow0}0$ in $\mathscr{D}'_{N}\mathfrak{X}$
and therefore in $\mathscr{D}'\mathfrak{X}$. Furthermore, $\mathcal{W}(\cdot,s)\xrightarrow{s\downarrow0}0$
in $\mathscr{D}'_{N}\left(I,\mathfrak{X}\right)$ and therefore in
$\mathscr{D}'\left(I,\mathfrak{X}\right)$ (spacetime distribution).
\item For fixed $t_{0}\in I$ and $s>0$: $\int_{\varepsilon}^{s}\mathcal{N}\left(t_{0},\sigma\right)^{3(s-\sigma)}\mathrm{d}\sigma\xrightarrow{\varepsilon\downarrow0}\mathcal{W}\left(t_{0},s\right)$
in $\mathscr{D}'_{N}\mathfrak{X}$.
\end{enumerate}
\end{lem}

\begin{proof}
$ $
\begin{enumerate}
\item Let $X\in\mathscr{D}_{N}\mathfrak{X},\mathcal{X}\in C_{c}^{\infty}\left(I,\mathscr{D}_{N}\mathfrak{X}\right).$
It is trivial to check (with DCT)
\[
\begin{aligned} & \left\langle \left\langle \mathcal{U}(t)\otimes\mathcal{U}(t),\nabla\left(X^{3s}\right)\right\rangle \right\rangle -\left\langle \left\langle \mathcal{U}(t)^{2s}\otimes\mathcal{U}(t)^{2s},\nabla\left(X^{s}\right)\right\rangle \right\rangle \xrightarrow{s\downarrow0}0\\
 & \int_{I}\left\langle \left\langle \mathcal{U}\otimes\mathcal{U},\nabla\left(\mathcal{X}^{3s}\right)\right\rangle \right\rangle -\int_{I}\left\langle \left\langle \mathcal{U}^{2s}\otimes\mathcal{U}^{2s},\nabla\left(\mathcal{X}^{s}\right)\right\rangle \right\rangle \xrightarrow{s\downarrow0}0
\end{aligned}
\]
\item Let $\varepsilon>0$. By the smoothing effect of $e^{s\Delta}$, $\mathcal{W}(t_{0},\cdot)$
and $\mathcal{N}(t_{0},\cdot)$ are in $C_{\mathrm{loc}}^{0}\left((0,1],\mathscr{D}_{N}\mathfrak{X}\right)$.
As $\left(e^{s\Delta}\right)_{s\geq0}$ is a $C_{0}$ semigroup on
$\left(H^{m}\text{-}\mathrm{cl}\left(\mathscr{D}_{N}\mathfrak{X}\right),\left\Vert \cdot\right\Vert _{H^{m}}\right)$
$\forall m\in\mathbb{N}_{0}$, and a semigroup basically corresponds
to an ODE (cf. \parencite[Appendix A, Proposition 9.10 \& 9.11]{Taylor_PDE1}),
from $\partial_{s}\mathcal{W}=3\Delta\mathcal{W}+\mathcal{N}$ for
$s\geq\varepsilon$ we get the Duhamel formula
\[
\forall s>\varepsilon:\mathcal{W}(t_{0},s)=\mathcal{W}\left(t_{0},\varepsilon\right)^{3(s-\varepsilon)}+\int_{\varepsilon}^{s}\mathcal{N}\left(t_{0},\sigma\right)^{3\left(s-\sigma\right)}\mathrm{d}\sigma
\]
So we only need to show $\mathcal{W}\left(t_{0},\varepsilon\right)^{3(s-\varepsilon)}\xrightarrow[\varepsilon\downarrow0]{\mathscr{D}'_{N}\mathfrak{X}}$0.
Let $X\in\mathscr{D}_{N}\mathfrak{X}.$ 
\begin{align*}
\left\langle \left\langle X,\mathcal{W}\left(t_{0},\varepsilon\right)^{3(s-\varepsilon)}\right\rangle \right\rangle  & =\left\langle \left\langle X^{3(s-\varepsilon)},\Div\left(\mathcal{U}\left(t_{0}\right)\otimes\mathcal{U}\left(t_{0}\right)\right)^{3\varepsilon}\right\rangle \right\rangle -\left\langle \left\langle X^{3(s-\varepsilon)},\Div\left(\mathcal{U}\left(t_{0}\right)^{2\varepsilon}\otimes\mathcal{U}\left(t_{0}\right)^{2\varepsilon}\right)^{\varepsilon}\right\rangle \right\rangle \\
 & =-\left\langle \left\langle \nabla\left(X^{3s}\right),\mathcal{U}\left(t_{0}\right)\otimes\mathcal{U}\left(t_{0}\right)\right\rangle \right\rangle +\left\langle \left\langle \nabla\left(X^{3s-2\varepsilon}\right),\mathcal{U}\left(t_{0}\right)^{2\varepsilon}\otimes\mathcal{U}\left(t_{0}\right)^{2\varepsilon}\right\rangle \right\rangle \xrightarrow{\varepsilon\downarrow0}0.
\end{align*}
\end{enumerate}
\end{proof}
From now on, we write $\int_{0+}^{s}$ for $\lim_{\varepsilon\downarrow0}\int_{\varepsilon}^{s}$.
Then 
\[
\int_{I}\mathrm{d}t\;\eta\left(t\right)\left\langle \left\langle \mathcal{W}\left(t,s\right),\mathcal{U}\left(t\right)^{s}\right\rangle \right\rangle =\int_{I}\mathrm{d}t\;\eta\left(t\right)\int_{0+}^{s}\mathrm{d}\sigma\left\langle \left\langle \mathcal{N}\left(t,\sigma\right)^{3(s-\sigma)},\mathcal{U}\left(t\right)^{s}\right\rangle \right\rangle 
\]

To clean up the algebra, we will classify the terms that are going
to appear but are actually negligible in the end. The following estimates
lie at the heart of the problem, showing why the regularity needs
to be at least $\frac{1}{3}$, and that our argument barely holds
thanks to the pointwise vanishing property (\Corref{vanishing_property}).
\begin{lem}[3 error estimates]
 Define the $k$-jet fiber norm $\left|X\right|_{J^{k}}=\left(\sum\limits _{j=0}^{k}\left|\nabla^{\left(j\right)}X\right|^{2}\right)^{\frac{1}{2}}\;\forall X\in\mathfrak{X}$
(more details in \Subsecref{Vector-bundles}). Then we have
\begin{enumerate}
\item $\int_{I}\left|\eta\right|\int_{0+}^{s}\mathrm{d}\sigma\int_{M}\left|\mathcal{U}^{2\sigma}\right|_{J^{1}}^{2}\left|\mathcal{U}^{4s-2\sigma}\right|_{J^{1}}\xrightarrow{s\downarrow0}0$
\item $\int_{I}\left|\eta\right|\int_{0+}^{s}\mathrm{d}\sigma\int_{\partial M}\left|\mathcal{U}^{2\sigma}\right|^{2}\left|\mathcal{U}^{4s-2\sigma}\right|_{J^{2}}\xrightarrow{s\downarrow0}0$
\item $\int_{I}\left|\eta\right|\int_{0+}^{s}\mathrm{d}\sigma\int_{\partial M}\left|\mathcal{U}^{2\sigma}\right|\left|\mathcal{U}^{2\sigma}\right|_{J^{1}}\left|\mathcal{U}^{4s-2\sigma}\right|_{J^{1}}\xrightarrow{s\downarrow0}0$
\end{enumerate}
\end{lem}

\begin{proof}
Define $A\left(t,s\right)=s^{\frac{1}{3}}\left\Vert \mathcal{U}\left(t\right)^{\frac{s}{2}}\right\Vert _{W^{1,3}}$.
Then for $s>0$ small: $\left\Vert \mathcal{U}\left(t\right)^{s}\right\Vert _{B_{3,1}^{1+\frac{1}{3}}}\lesssim\left(\frac{1}{s}\right)^{\frac{1}{6}}\left\Vert \mathcal{U}\left(t\right)^{\frac{s}{2}}\right\Vert _{W^{1,3}}\lesssim\left(\frac{1}{s}\right)^{\frac{1}{2}}A\left(t,s\right)$
and $\left\Vert \left\Vert A\left(t,\sigma\right)\right\Vert _{L_{\sigma\leq s}^{\infty}}\right\Vert _{L_{t}^{3}}\xrightarrow{s\downarrow0}0$
by \Corref{vanishing_property}. We also note that $\left\Vert \mathcal{U}\left(t\right)^{s}\right\Vert _{B_{3,1}^{2+\frac{1}{3}}}\lesssim\left(\frac{1}{s}\right)^{\frac{2}{3}}\left\Vert \mathcal{U}\left(t\right)^{\frac{s}{2}}\right\Vert _{W^{1,3}}\lesssim\left(\frac{1}{s}\right)A\left(t,s\right)$.

Now we can prove the error estimates go to 0:
\begin{enumerate}
\item 
\begin{align*}
 & \int_{I}\left|\eta\right|\int_{0+}^{s}\mathrm{d}\sigma\int_{M}\left|\mathcal{U}^{2\sigma}\right|_{J^{1}}^{2}\left|\mathcal{U}^{4s-2\sigma}\right|_{J^{1}}\lesssim\int_{I}\left|\eta\right|\int_{0+}^{s}\mathrm{d}\sigma\left\Vert \mathcal{U}^{2\sigma}\right\Vert _{W^{1,3}}^{2}\left\Vert \mathcal{U}^{4s-2\sigma}\right\Vert _{W^{1,3}}\\
\lesssim & \int_{I}\mathrm{d}t\left|\eta(t)\right|\int_{0+}^{s}\mathrm{d}\sigma\left(\frac{1}{\sigma}\right)^{\frac{2}{3}}\left(\frac{1}{2s-\sigma}\right)^{\frac{1}{3}}A\left(t,2\sigma\right)^{2}A\left(t,4s-2\sigma\right)\\
\stackrel{\sigma\mapsto s\sigma}{=} & \int_{I}\mathrm{d}t\left|\eta(t)\right|\int_{0+}^{1}\mathrm{d}\sigma\;\left(\frac{1}{\sigma}\right)^{\frac{2}{3}}\left(\frac{1}{2-\sigma}\right)^{\frac{1}{3}}A\left(t,2s\sigma\right)^{2}A\left(t,4s-2s\sigma\right)\lesssim\int_{I}\mathrm{d}t\;\left|\eta(t)\right|\left\Vert A\left(t,\sigma\right)\right\Vert _{L_{\sigma\leq4s}^{\infty}}^{3}\xrightarrow{s\downarrow0}0.
\end{align*}
\item 
\begin{align*}
 & \int_{I}\left|\eta\right|\int_{0+}^{s}\mathrm{d}\sigma\int_{\partial M}\left|\mathcal{U}^{2\sigma}\right|^{2}\left|\mathcal{U}^{4s-2\sigma}\right|_{J^{2}}\lesssim\int_{I}\left|\eta\right|\int_{0+}^{s}\mathrm{d}\sigma\left\Vert \mathcal{U}^{2\sigma}\right\Vert _{L^{3}\mathfrak{X}M|_{\partial M}}^{2}\left\Vert \mathcal{U}^{4s-2\sigma}\right\Vert _{W^{2,3}\mathfrak{X}M|_{\partial M}}\\
\stackrel{\text{Trace}}{\lesssim} & \int_{I}\left|\eta\right|\int_{0+}^{s}\mathrm{d}\sigma\left\Vert \mathcal{U}^{2\sigma}\right\Vert _{B_{3,1}^{\frac{1}{3}}\mathfrak{X}M}^{2}\left\Vert \mathcal{U}^{4s-2\sigma}\right\Vert _{B_{3,1}^{2+\frac{1}{3}}\mathfrak{X}M}\\
\lesssim & \int_{I}\mathrm{d}t\;\left|\eta\left(t\right)\right|\left\Vert \mathcal{U}\left(t\right)\right\Vert _{B_{3,1}^{\frac{1}{3}}\mathfrak{X}M}^{2}\int_{0+}^{s}\mathrm{d}\sigma\left(\frac{1}{2s-\sigma}\right)A\left(t,4s-2\sigma\right)\\
\stackrel{\sigma\mapsto s\sigma}{=} & \int_{I}\mathrm{d}t\;\left|\eta\left(t\right)\right|\left\Vert \mathcal{U}\left(t\right)\right\Vert _{B_{3,1}^{\frac{1}{3}}\mathfrak{X}M}^{2}\int_{0+}^{1}\mathrm{d}\sigma\left(\frac{1}{2-\sigma}\right)A\left(t,4s-2s\sigma\right)\lesssim\int_{I}\mathrm{d}t\;\left|\eta\left(t\right)\right|\left\Vert \mathcal{U}\left(t\right)\right\Vert _{B_{3,1}^{\frac{1}{3}}\mathfrak{X}M}^{2}\left\Vert A\left(t,\sigma\right)\right\Vert _{L_{\sigma\leq4s}^{\infty}}\\
\lesssim & \left\Vert \mathcal{U}\right\Vert _{L_{t}^{3}B_{3,1}^{\frac{1}{3}}\left(M\right)}^{2}\left\Vert \left\Vert A\left(t,\sigma\right)\right\Vert _{L_{\sigma\leq4s}^{\infty}}\right\Vert _{L_{t}^{3}}\xrightarrow{s\downarrow0}0
\end{align*}
\item 
\begin{align*}
 & \int_{I}\left|\eta\right|\int_{0+}^{s}\mathrm{d}\sigma\int_{\partial M}\left|\mathcal{U}^{2\sigma}\right|\left|\mathcal{U}^{2\sigma}\right|_{J^{1}}\left|\mathcal{U}^{4s-2\sigma}\right|_{J^{1}}\\
\stackrel{\text{Trace}}{\lesssim} & \int_{I}\left|\eta\right|\int_{0+}^{s}\mathrm{d}\sigma\left\Vert \mathcal{U}^{2\sigma}\right\Vert _{B_{3,1}^{\frac{1}{3}}\mathfrak{X}M}\left\Vert \mathcal{U}^{2\sigma}\right\Vert _{B_{3,1}^{1+\frac{1}{3}}\mathfrak{X}M}\left\Vert \mathcal{U}^{4s-2\sigma}\right\Vert _{B_{3,1}^{1+\frac{1}{3}}\mathfrak{X}M}\\
\lesssim & \int_{I}\mathrm{d}t\;\left|\eta\left(t\right)\right|\left\Vert \mathcal{U}\left(t\right)\right\Vert _{B_{3,1}^{\frac{1}{3}}}\int_{0+}^{s}\mathrm{d}\sigma\left(\frac{1}{\sigma}\right)^{\frac{1}{2}}\left(\frac{1}{2s-\sigma}\right)^{\frac{1}{2}}A\left(t,2\sigma\right)A\left(t,4s-2\sigma\right)\\
\stackrel{\sigma\mapsto s\sigma}{=} & \int_{I}\mathrm{d}t\;\left|\eta\left(t\right)\right|\left\Vert \mathcal{U}\left(t\right)\right\Vert _{B_{3,1}^{\frac{1}{3}}}\int_{0+}^{1}\mathrm{d}\sigma\left(\frac{1}{\sigma}\right)^{\frac{1}{2}}\left(\frac{1}{2-\sigma}\right)^{\frac{1}{2}}A\left(t,2s\sigma\right)A\left(t,4s-2s\sigma\right)\\
\lesssim & \int_{I}\mathrm{d}t\;\left|\eta\left(t\right)\right|\left\Vert \mathcal{U}\left(t\right)\right\Vert _{B_{3,1}^{\frac{1}{3}}}\left\Vert A\left(t,\sigma\right)\right\Vert _{L_{\sigma\leq4s}^{\infty}}^{2}\lesssim\left\Vert \mathcal{U}\right\Vert _{L_{t}^{3}B_{3,1}^{\frac{1}{3}}}\left\Vert \left\Vert A\left(t,\sigma\right)\right\Vert _{L_{\sigma\leq4s}^{\infty}}\right\Vert _{L_{t}^{3}}^{2}\xrightarrow{s\downarrow0}0
\end{align*}
\end{enumerate}
\end{proof}
Note that
\begin{align*}
\mathcal{N}\left(t,\sigma\right) & =\left(\partial_{\sigma}-3\Delta\right)\left(-\Div\left(\mathcal{U}^{2\sigma}\otimes\mathcal{U}^{2\sigma}\right)^{\sigma}\right)=-2\Div\left(\Delta\mathcal{U}^{2\sigma}\otimes\mathcal{U}^{2\sigma}\right)^{\sigma}-2\Div\left(\mathcal{U}^{2\sigma}\otimes\Delta\mathcal{U}^{2\sigma}\right)^{\sigma}+2\Delta\Div\left(\mathcal{U}^{2\sigma}\otimes\mathcal{U}^{2\sigma}\right)^{\sigma}
\end{align*}

Finally, we will show 
\[
\int_{I}\eta\left\langle \left\langle \mathcal{W}(s),\mathcal{U}^{s}\right\rangle \right\rangle =\int_{I}\mathrm{d}t\;\eta\left(t\right)\left\langle \left\langle \mathcal{W}(t,s),\mathcal{U}\left(t\right)^{s}\right\rangle \right\rangle \xrightarrow{s\downarrow0}0
\]

\begin{proof}
Integrate by parts into 3 components:
\begin{align*}
\int_{I}\eta\left\langle \left\langle \mathcal{W}(s),\mathcal{U}^{s}\right\rangle \right\rangle  & =\int_{I}\mathrm{d}t\;\eta\left(t\right)\int_{0+}^{s}\mathrm{d}\sigma\left\langle \left\langle \mathcal{N}\left(t,\sigma\right)^{3(s-\sigma)},\mathcal{U}\left(t\right)^{s}\right\rangle \right\rangle =\int_{I}\mathrm{d}t\;\eta\left(t\right)\int_{0+}^{s}\mathrm{d}\sigma\left\langle \left\langle \mathcal{N}\left(t,\sigma\right),\mathcal{U}\left(t\right)^{4s-3\sigma}\right\rangle \right\rangle \\
 & =2\int_{I}\eta\int_{0+}^{s}\mathrm{d}\sigma\left\langle \left\langle \Delta\mathcal{U}^{2\sigma}\otimes\mathcal{U}^{2\sigma},\nabla\left(\mathcal{U}^{4s-2\sigma}\right)\right\rangle \right\rangle +2\int_{I}\eta\int_{0+}^{s}\mathrm{d}\sigma\left\langle \left\langle \mathcal{U}^{2\sigma}\otimes\Delta\mathcal{U}^{2\sigma},\nabla\left(\mathcal{U}^{4s-2\sigma}\right)\right\rangle \right\rangle \\
 & \;-2\int_{I}\eta\int_{0+}^{s}\mathrm{d}\sigma\left\langle \left\langle \mathcal{U}^{2\sigma}\otimes\mathcal{U}^{2\sigma},\nabla\left(\Delta\mathcal{U}^{4s-2\sigma}\right)\right\rangle \right\rangle 
\end{align*}
Note that for the third component, we used some properties from \Factref{D_N-basic-properties}
to move the Laplacian. It also explains our choice of $\mathcal{W}$.

We now use Penrose notation to estimate the 3 components. To clean
up the notation, we only focus on the integral on $M$, with the other
integrals $2\int_{I}\eta\int_{0+}^{s}\mathrm{d}\sigma\left(\cdot\right)$
in variables $t$ and $\sigma$ implicitly understood. We also use\textbf{
schematic identities} for linear combinations of similar-looking tensor
terms where we do not care how the indices contract (recall \Eqref{Weitzen_schematic}).
By the error estimates above, all the terms with $R$ or $\nu$ will
be negligible as $s\downarrow0$, and interchanging derivatives will
be a free action. We write $\approx$ to throw the negligible error
terms away. Also, when we write $\left(\nabla_{j}\mathcal{U}_{l}\right)^{4s-2\sigma},$
we mean the heat flow is applied to $\mathcal{U}$, not $\nabla\mathcal{U}$
(which is not possible anyway).

First component:
\begin{align*}
 & \int_{M}\left\langle \Delta\mathcal{U}^{2\sigma}\otimes\mathcal{U}^{2\sigma},\nabla\left(\mathcal{U}^{4s-2\sigma}\right)\right\rangle =\cancel{\int_{M}R*\mathcal{U}^{2\sigma}*\mathcal{U}^{2\sigma}*\nabla\left(\mathcal{U}^{4s-2\sigma}\right)}+\int_{M}\begin{array}{c}
\left(\nabla_{i}\nabla^{i}\mathcal{U}^{j}\right)^{2\sigma}\left(\mathcal{U}^{l}\right)^{2\sigma}\left(\nabla_{j}\mathcal{U}_{l}\right)^{4s-2\sigma}\end{array}\\
\approx & \cancel{\int_{\partial M}\left(\nu_{i}\nabla^{i}\mathcal{U}^{j}\right)^{2\sigma}\left(\mathcal{U}^{l}\right)^{2\sigma}\left(\nabla_{j}\mathcal{U}_{l}\right)^{4s-2\sigma}}-\cancel{\int_{M}\begin{array}{c}
\left(\nabla^{i}\mathcal{U}^{j}\right)^{2\sigma}\left(\nabla_{i}\mathcal{U}^{l}\right)^{2\sigma}\left(\nabla_{j}\mathcal{U}_{l}\right)^{4s-2\sigma}\end{array}}\\
 & -\int_{M}\begin{array}{c}
\left(\nabla^{i}\mathcal{U}^{j}\right)^{2\sigma}\left(\mathcal{U}^{l}\right)^{2\sigma}\left(\nabla_{i}\nabla_{j}\mathcal{U}_{l}\right)^{4s-2\sigma}\end{array}
\end{align*}
Second component:
\begin{align*}
 & \int_{M}\left\langle \mathcal{U}^{2\sigma}\otimes\Delta\mathcal{U}^{2\sigma},\nabla\left(\mathcal{U}^{4s-2\sigma}\right)\right\rangle =\cancel{\int_{M}\mathcal{U}^{2\sigma}*R*\mathcal{U}^{2\sigma}*\nabla\left(\mathcal{U}^{4s-2\sigma}\right)}+\int_{M}\begin{array}{c}
\left(\mathcal{U}^{j}\right)^{2\sigma}\left(\nabla_{i}\nabla^{i}\mathcal{U}^{l}\right)^{2\sigma}\left(\nabla_{j}\mathcal{U}_{l}\right)^{4s-2\sigma}\end{array}\\
\approx & \cancel{\int_{\partial M}\left(\mathcal{U}^{j}\right)^{2\sigma}\left(\nu_{i}\nabla^{i}\mathcal{U}^{l}\right)^{2\sigma}\left(\nabla_{j}\mathcal{U}_{l}\right)^{4s-2\sigma}}-\cancel{\int_{M}\begin{array}{c}
\left(\nabla_{i}\mathcal{U}^{j}\right)^{2\sigma}\left(\nabla^{i}\mathcal{U}^{l}\right)^{2\sigma}\left(\nabla_{j}\mathcal{U}_{l}\right)^{4s-2\sigma}\end{array}}\\
 & -\int_{M}\begin{array}{c}
\left(\mathcal{U}^{j}\right)^{2\sigma}\left(\nabla^{i}\mathcal{U}^{l}\right)^{2\sigma}\left(\nabla_{i}\nabla_{j}\mathcal{U}_{l}\right)^{4s-2\sigma}\end{array}
\end{align*}
For the third component, note $\nabla\left(R*U\right)=\nabla R*U+R*\nabla U$
\begin{align*}
 & -\int_{M}\left\langle \mathcal{U}^{2\sigma}\otimes\mathcal{U}^{2\sigma},\nabla\left(\Delta\mathcal{U}^{4s-2\sigma}\right)\right\rangle =-\cancel{\int_{M}\mathcal{U}^{2\sigma}*\mathcal{U}^{2\sigma}*\nabla\left(R*\mathcal{U}^{4s-2\sigma}\right)}-\int_{M}\begin{array}{c}
\left(\mathcal{U}^{j}\right)^{2\sigma}\left(\mathcal{U}^{l}\right)^{2\sigma}\left(\nabla_{j}\nabla^{i}\nabla_{i}\mathcal{U}_{l}\right)^{4s-2\sigma}\end{array}\\
\approx & -\int_{M}\begin{array}{c}
\left(\mathcal{U}^{j}\right)^{2\sigma}\left(\mathcal{U}^{l}\right)^{2\sigma}\left(\cancel{R*\nabla\left(\mathcal{U}^{4s-2\sigma}\right)}+\nabla^{i}\nabla_{j}\nabla_{i}\mathcal{U}_{l}^{4s-2\sigma}\right)\end{array}\\
\approx & -\int_{M}\begin{array}{c}
\left(\mathcal{U}^{j}\right)^{2\sigma}\left(\mathcal{U}^{l}\right)^{2\sigma}\left(\cancel{\nabla\left(R*\mathcal{U}^{4s-2\sigma}\right)}+\nabla^{i}\nabla_{i}\nabla_{j}\mathcal{U}_{l}^{4s-2\sigma}\right)\end{array}\\
\approx & -\cancel{\int_{\partial M}\left(\mathcal{U}^{j}\right)^{2\sigma}\left(\mathcal{U}^{l}\right)^{2\sigma}\left(\nu^{i}\nabla_{i}\nabla_{j}\mathcal{U}_{l}\right)^{4s-2\sigma}}+\int_{M}\begin{array}{c}
\left(\nabla^{i}\mathcal{U}^{j}\right)^{2\sigma}\left(\mathcal{U}^{l}\right)^{2\sigma}\left(\nabla_{i}\nabla_{j}\mathcal{U}_{l}\right)^{4s-2\sigma}\end{array}\\
 & +\int_{M}\begin{array}{c}
\left(\mathcal{U}^{j}\right)^{2\sigma}\left(\nabla^{i}\mathcal{U}^{l}\right)^{2\sigma}\left(\nabla_{i}\nabla_{j}\mathcal{U}_{l}\right)^{4s-2\sigma}\end{array}
\end{align*}
Add them up, and we get $0$ as $2\int_{I}\eta\int_{0+}^{s}\mathrm{d}\sigma\left(\cdot\right)\xrightarrow{s\downarrow0}0$.
\end{proof}
So we are done and the rest of the paper is to develop the tools we
have borrowed for the proof.

\section{Functional analysis}

\subsection{Common tools}

We note a useful inequality:
\begin{thm}[Ehrling's inequality]
\label{thm:Ehrling} Let $X,Y,\widetilde{X}$ be (real/complex) Banach
spaces such that $X$ is reflexive and $X\hookrightarrow\widetilde{X}$
is a continuous injection. Let $T:X\to Y$ be a linear compact operator.
Then $\forall\varepsilon>0,\exists C_{\varepsilon}>0$: 
\[
\left\Vert Tx\right\Vert _{Y}\leq\varepsilon\left\Vert x\right\Vert _{X}+C_{\varepsilon}\left\Vert x\right\Vert _{\widetilde{X}}\;\forall x\in X
\]
\end{thm}

\begin{rem*}
Usually, $X$ is some higher-regularity space than $\widetilde{X}$
(e.g. $H^{1}$ and $L^{2}$). The inequality is useful when the higher-regularity
norm is expensive. We will need this for the $L^{p}$-analyticity
of the heat flow (\Thmref{local_boundedness_Lp}).
\end{rem*}
\begin{proof}
Proof by contradiction: Assume $\varepsilon>0$ and there is $\left(x_{j}\right)_{j\in\mathbb{N}}$
such that $\left\Vert x_{j}\right\Vert _{X}=1$ and $\left\Vert Tx_{j}\right\Vert _{Y}>\varepsilon+j\left\Vert x_{j}\right\Vert _{\widetilde{X}}$.
Since $X$ is reflexive, by Banach-Alaoglu and PTAS, WLOG assume $x_{j}\xrightharpoonup{X}x_{\infty}$.
Then $Tx_{j}\xrightharpoonup{Y}Tx_{\infty}$ and $x_{j}\xrightharpoonup{\widetilde{X}}x_{\infty}$.
As $T$ is compact, PTAS, WLOG $Tx_{j}\to Tx_{\infty}$. So $\left\Vert Tx_{\infty}\right\Vert _{Y}\geq\limsup_{j\to\infty}\left(\varepsilon+j\left\Vert x_{j}\right\Vert _{\widetilde{X}}\right)>0$
and $x_{j}\xrightarrow{\widetilde{X}}0$. Then $x_{j}\xrightharpoonup{\widetilde{X}}0$
and $x_{\infty}=0$, contradicting $\left\Vert Tx_{\infty}\right\Vert _{Y}>0$.
\end{proof}
\begin{defn}[Banach-valued holomorphic functions]
 Let $\Omega\subset\mathbb{C}$ be an open set and $X$ be a complex
Banach space. Then a function $f:\Omega\to X$ is said to be \textbf{holomorphic
}(or \textbf{analytic}) when $\forall z\in\Omega:f'(z):=\lim_{\left|h\right|\to0}\frac{f(z+h)-f(z)}{h}$
exists. The words ``holomorphic''\textbf{ }and ``analytic'' are
mostly interchangeable, but ``analytic'' stresses the existence
of power series expansion and can also describe functions on $\mathbb{R}$
for which analytic continuation into the complex plane exists.
\end{defn}

\begin{thm}[Identity theorem]
 Let $X$ be a complex Banach space and $X_{0}\leq X$ closed. Let
$\Omega\subset\mathbb{C}$ be connected, open and $f:\Omega\to X$
holomorphic. Assume there is a sequence $\left(z_{j}\right)_{j\in\mathbb{N}}$
such that $z_{j}\to z\in\Omega$ and $f(z_{j})\in X_{0}\;\forall j$.
Then $f(\Omega)\subset X_{0}$.
\end{thm}

\begin{proof}
Let $\Lambda\in X^{*}$ such that $\Lambda(X_{0})=0$. Reduce this
to the scalar version in complex analysis.
\end{proof}
In fact, many theorems from scalar complex analysis similarly carry
over via linear functionals (cf. \parencite[Theorem 3.31]{Rudin_FuncAnal}).

\subsection{Interpolation theory}

We will quickly review the theory of complex and real interpolation,
and state the abstract Stein interpolation theorem. Interpolation
theory can be seen as vast generalizations of the Marcinkiewicz and
Riesz-Thorin interpolation theorems.
\begin{defn}
An \textbf{interpolation couple} of (real/complex) Banach spaces is
a pair $(X_{0},X_{1})$ of Banach spaces with a Hausdorff TVS $\mathcal{X}$
such that $X_{0}\hookrightarrow\mathcal{X},X_{1}\hookrightarrow\mathcal{X}$
are continuous injections. Then $X_{0}\cap X_{1}$ and $X_{0}+X_{1}$
are Banach spaces under the norms 
\[
\left\Vert x\right\Vert _{X_{0}\cap X_{1}}=\max\left(\left\Vert x\right\Vert _{X_{0}},\left\Vert x\right\Vert _{X_{1}}\right)\text{ and }\left\Vert x\right\Vert _{X_{0}+X_{1}}=\inf_{x=x_{0}+x_{1},x_{j}\in X_{j}}\left\Vert x_{0}\right\Vert _{X_{0}}+\left\Vert x_{1}\right\Vert _{X_{1}}
\]

Let $(Y_{0},Y_{1})$ be another interpolation couple. We say $T:\left(X_{0},X_{1}\right)\to\left(Y_{0},Y_{1}\right)$
is a \textbf{morphism} when $T\in\mathcal{L}\left(X_{0}+X_{1},Y_{0}+Y_{1}\right)$
and $T\in\mathcal{L}\left(X_{j},Y_{j}\right)$ for $j=0,1$ under
domain restriction. That implies $T\in\mathcal{L}\left(X_{0}\cap X_{1},Y_{0}\cap Y_{1}\right)$
and we write $T\in\mathcal{L}\left(\left(X_{0},X_{1}\right),\left(Y_{0},Y_{1}\right)\right)$.
We also write $\mathcal{L}\left(\left(X_{0},X_{1}\right)\right)=\mathcal{L}\left(\left(X_{0},X_{1}\right),\left(X_{0},X_{1}\right)\right)$.
\nomenclature{$\mathcal{L}\left(\left(X_{0},X_{0}\right),\left(Y_{0},Y_{1}\right)\right)$}{morphisms between interpolation couples\nomrefpage}

Let $P\in\mathcal{L}(\left(X_{0},X_{1}\right))$ such that $P^{2}=P$.
Then we call $P$ a\textbf{ projection }on the interpolation couple
$(X_{0},X_{1})$.
\end{defn}

\begin{defn}
Let $(X_{0},X_{1})$ be an interpolation couple of (real/complex)
Banach spaces. Then define the $J$-functional:

$$%
\begin{tabular}{ccl}
$J:$ & $(0,\infty)\times X_{0}\cap X_{1}$ & $\longrightarrow\mathbb{R}$\tabularnewline
 & $(t,x)$ & $\longmapsto\left\Vert x\right\Vert _{X_{0}}+t\left\Vert x\right\Vert _{X_{1}}$\tabularnewline
\end{tabular}$$

For $\theta\in(0,1),q\in[1,\infty],$ define the \textbf{real interpolation
space}

\[
\left(X_{0},X_{1}\right){}_{\theta,q}=\left\{ \sum_{j\in\mathbb{Z}}u_{j}:u_{j}\in X_{0}\cap X_{1},\left(2^{-j\theta}J(2^{j},u_{j})\right)_{j\in\mathbb{Z}}\in l_{j}^{q}(\mathbb{Z})\right\} 
\]
\nomenclature{$\left(X_{0},X_{1}\right)_{\theta,q}$}{real interpolation\nomrefpage}which
is Banach under the norm $\left\Vert x\right\Vert _{\left(X_{0},X_{1}\right)_{\theta,q}}=\inf\limits _{x=\sum\limits _{j\in\mathbb{Z}}u_{j}}\text{\ensuremath{\left\Vert 2^{-j\theta}J(2^{j},u_{j})\right\Vert }}_{l_{j}^{q}}$.
Note that $\sum_{j\in\mathbb{Z}}u_{j}$ denotes a series that converges
in $X_{0}+X_{1}$.
\end{defn}

\begin{itemize}
\item When $q\in[1,\infty]$ and $x\in X_{0}\cap X_{1},$ note that $\forall j\in\mathbb{Z}:x=\sum_{k\in\mathbb{Z}}\delta_{kj}x$
and
\[
\left\Vert x\right\Vert _{\left(X_{0},X_{1}\right){}_{\theta,q}}\leq\inf_{j\in\mathbb{Z}}\left|2^{-j\theta}J(2^{j},x)\right|=\inf_{j\in\mathbb{Z}}\left|2^{-j\theta}\left\Vert x\right\Vert _{X_{0}}+2^{j(1-\theta)}\left\Vert x\right\Vert _{X_{1}}\right|\sim_{\neg\theta,\neg q}\left\Vert x\right\Vert _{X_{0}}^{1-\theta}\left\Vert x\right\Vert _{X_{1}}^{\theta}
\]
The last estimate comes from AM-GM and shifting $j$ so that $\left\Vert x\right\Vert _{X_{0}}\sim2^{j}\left\Vert x\right\Vert _{X_{1}}$.
Note that the implied constants do not depend on $\theta$ and $q$.
\item By considering the finite partial sums $\sum_{|j|<j_{0}}u_{j},$ we
conclude that $X_{0}\cap X_{1}$ is dense in $\left(X_{0},X_{1}\right)_{\theta,q}$
when $q\in[1,\infty)$.
\item Let $(Y_{0},Y_{1})$ be another interpolation couple and $T\in\mathcal{L}\left(\left(X_{0},X_{1}\right),\left(Y_{0},Y_{1}\right)\right)$.
For $\theta\in(0,1),q\in[1,\infty]$, define $X_{\theta,q}=\left(X_{0},X_{1}\right)_{\theta,q},Y_{\theta,q}=\left(Y_{0},Y_{1}\right)_{\theta,q}$.
Then $T\in\mathcal{L}(X_{\theta,q},Y_{\theta,q})$ and
\[
\left\Vert T\right\Vert _{\mathcal{L}(X_{\theta,q},Y_{\theta,q})}\lesssim_{\neg\theta,\neg q,\neg T}\left\Vert T\right\Vert _{\mathcal{L}(X_{0},Y_{0})}^{1-\theta}\left\Vert T\right\Vert _{\mathcal{L}(X_{1},Y_{1})}^{\theta}
\]
where the implied constant does not depend on $\theta$ and $q$.
This can be proved by a simple shifting argument.
\item If $P$ is a projection on $\left(X_{0},X_{1}\right)$ then $\left(PX_{0},PX_{1}\right)_{\theta,q}=P\left(X_{0},X_{1}\right)_{\theta,q}$.
\end{itemize}
\begin{rem*}
There is also an equivalent characterization by the $K$-functional,
which we shall omit. This theory can also be extended to quasi-Banach
spaces. We refer to \parencite{Lofstrom,triebel_function_I} for more
details.
\end{rem*}
\begin{defn}
Let $(X_{0},X_{1})$ be an interpolation couple of complex Banach
spaces.

Let $\Omega=\{z\in\mathbb{C}:0<\Ree z<1\}$. We then define the Banach
space of vector-valued holomorphic/analytic functions on the strip:
\[
\mathcal{F}_{X_{0},X_{1}}=\{f\in C^{0}(\overline{\Omega}\to X_{0}+X_{1}):f\text{ holomorphic in }\Omega,\;\left\Vert f(it)\right\Vert _{X_{0}}+\left\Vert f(1+it)\right\Vert _{X_{1}}\xrightarrow{|t|\to\infty}0\}
\]
 with the norm $\left\Vert f\right\Vert _{\mathcal{F}_{X_{0},X_{1}}}=\max\left(\sup_{t\in\mathbb{R}}\left\Vert f(it)\right\Vert _{X_{0}},\sup_{t\in\mathbb{R}}\left\Vert f(1+it)\right\Vert _{X_{1}}\right)$.

For $\theta\in[0,1],$ define the \textbf{complex interpolation space}
$[X_{0},X_{1}]_{\theta}=\{f(\theta):f\in\mathcal{F}_{X_{0},X_{1}}\}$,
which is Banach under the norm \nomenclature{$\left[X_{0},X_{1}\right]_{\theta}$}{complex interpolation\nomrefpage}
\[
\left\Vert x\right\Vert _{[X_{0},X_{1}]_{\theta}}=\inf_{\substack{f\in\mathcal{F}_{X_{0},X_{1}}\\
f(\theta)=x
}
}\left\Vert f\right\Vert _{\mathcal{F}_{X_{0},X_{1}}}
\]
\end{defn}

\begin{itemize}
\item When $x\in X_{0}\cap X_{1}\backslash\{0\}$ , $\theta\in[0,1]$, $\varepsilon>0,$
define $f_{\varepsilon}(z)=e^{\varepsilon(z^{2}-\theta^{2})}\frac{x}{\left\Vert x\right\Vert _{X_{0}}^{1-z}\left\Vert x\right\Vert _{X_{1}}^{z}}$.
By the freedom in choosing $\varepsilon$, we conclude 
\[
\left\Vert x\right\Vert _{[X_{0},X_{1}]_{\theta}}\leq\inf_{\varepsilon>0}\left\Vert f_{\varepsilon}\right\Vert _{\mathcal{F}_{X_{0},X_{1}}}\left\Vert x\right\Vert _{X_{0}}^{1-\theta}\left\Vert x\right\Vert _{X_{1}}^{\theta}\leq\inf_{\varepsilon>0}\max\left(e^{\varepsilon(1-\theta^{2})},e^{-\epsilon\theta^{2}}\right)\left\Vert x\right\Vert _{X_{0}}^{1-\theta}\left\Vert x\right\Vert _{X_{1}}^{\theta}=\left\Vert x\right\Vert _{X_{0}}^{1-\theta}\left\Vert x\right\Vert _{X_{1}}^{\theta}
\]
\item When $\theta\in[0,1]$, by Poisson summation and Fourier series, we
can prove that
\[
\mathcal{F}_{X_{0},X_{1}}^{0}=\{e^{Cz^{2}}\sum_{j=1}^{N}e^{\lambda_{j}z}x_{j}:N\in\mathbb{N},C>0,\lambda_{j}\in\mathbb{R},x_{j}\in X_{0}\cap X_{1}\}
\]
is dense in $\mathcal{F}_{X_{0},X_{1}}$ (cf. \parencite[Lemma 4.2.3]{Lofstrom}).
This implies $X_{0}\cap X_{1}$ is dense in $[X_{0},X_{1}]_{\theta}$.\\
There is a simple extension of the above density result. Let $U$
be dense in $X_{0}\cap X_{1}$ and define $A(\Omega)=\{\phi\in C^{0}(\overline{\Omega}\to\mathbb{C}):\phi\text{ holomorphic in }\Omega\}.$
Then
\[
\mathcal{F}_{X_{0},X_{1}}^{U}=\{e^{Cz^{2}}\sum_{j=1}^{N}\phi_{j}(z)u_{j}:N\in\mathbb{N},C>0,\phi_{j}\in A(\Omega),u_{j}\in U\}
\]
is dense in $\mathcal{F}_{X_{0},X_{1}}$. This will lead to the abstract
Stein interpolation theorem.
\item Let $(Y_{0},Y_{1})$ be another interpolation couple and $T\in\mathcal{L}\left(\left(X_{0},X_{1}\right),\left(Y_{0},Y_{1}\right)\right)$.
Then for $\theta\in[0,1]$, almost by the definitions, we conclude
\[
\left\Vert T\right\Vert _{\mathcal{L}(\left[X_{0},X_{1}\right]_{\theta},\left[Y_{0},Y_{1}\right]_{\theta})}\leq\left\Vert T\right\Vert _{\mathcal{L}(X_{0},Y_{0})}^{1-\theta}\left\Vert T\right\Vert _{\mathcal{L}(X_{1},Y_{1})}^{\theta}
\]
\item If $P$ is a projection on $\left(X_{0},X_{1}\right)$ then $\left[PX_{0},PX_{1}\right]{}_{\theta}=P\left[X_{0},X_{1}\right]{}_{\theta}$
\end{itemize}
\begin{rem*}
A keen reader would notice that we use square brackets for complex
interpolation, and parentheses for real interpolation. One reason
is that the real interpolation methods easily extend to quasi-Banach
spaces, while the complex interpolation method does not. There is
a version of complex interpolation for special quasi-Banach spaces,
which is denoted by parentheses (cf. \parencite[Section 2.4.4]{triebel_function_I}),
but we shall omit it for simplicity.
\end{rem*}
\begin{blackbox}[Abstract Stein interpolation]
 Let $(X_{0},X_{1})$ and $(Y_{0},Y_{1})$ be interpolation couples
of complex Banach spaces and $U$ dense in $X_{0}\cap X_{1}.$ Let
$\Omega=\{z\in\mathbb{C}:0<\Ree z<1\}$ and $\left(T(z)\right)_{z\in\overline{\Omega}}$
be a family of linear mappings $T(z):U\to Y_{0}+Y_{1}$ such that
\begin{enumerate}
\item $\forall u\in U:\left(\overline{\Omega}\to Y_{0}+Y_{1},z\mapsto T(z)u\right)$
is continuous, bounded and analytic in $\Omega$.
\item For $j=0,1$ and $u\in U$: $\left(\mathbb{R}\to Y_{j},t\mapsto T(j+it)u\right)$
is continuous and bounded by $M_{j}\left\Vert u\right\Vert _{X_{j}}$
for some $M_{j}>0$.
\end{enumerate}
Then for $\theta\in[0,1],$ we can conclude 
\[
\left\Vert T(\theta)u\right\Vert _{\left[Y_{0},Y_{1}\right]_{\theta}}\leq M_{0}^{1-\theta}M_{1}^{\theta}\left\Vert u\right\Vert _{\left[X_{0},X_{1}\right]_{\theta}}\;\forall u\in U
\]
Consequently, by unique extension, we have $T(\theta)\in\mathcal{L}(\left[X_{0},X_{1}\right]_{\theta},\left[Y_{0},Y_{1}\right]_{\theta})$.
\end{blackbox}

\begin{proof}
See \parencite{Voigt1992}, which is a very short read.
\end{proof}
\begin{rem*}
We will only use Stein interpolation in \Subsecref{Stein-extrapolation-of}.
\end{rem*}

\subsection{Stein extrapolation of analyticity of semigroups\label{subsec:Stein-extrapolation-of}}

We are inspired by \cite[Theorem 3.1.1]{Stephan_thesis} (Stein extrapolation)
and \cite[Theorem 3.1.10]{Stephan_thesis} (Kato-Beurling extrapolation),
and wish to create variants for our own use. We will focus on Stein
extrapolation, since it is simpler to deal with.

There exists a subtle, but very important criterion to establish analyticity/holomorphicity:

\begin{blackbox}[Holo on total]
\label{blackbox:Holo_on_total} Let $\Omega\subset\mathbb{C}$ be
open and $X$ complex Banach. Let $f:\Omega\to X$ be a function.
Assume $N\leq X^{*}$ is \textbf{total} (separating points) and f
is locally bounded.

Then $f$ is analytic iff $\Lambda f$ is analytic $\forall\Lambda\in N$.
\end{blackbox}

\begin{proof}
This is a consequence of Krein-Smulian and the Vitali holomorphic
convergence theorem, and we refer to \cite[Theorem A.7]{Arendt_Laplace}.
\end{proof}
\begin{rem*}
It will quickly become obvious how crucial this criterion is for the
rest of the paper. Let us briefly note that an improvement has just
been discovered by Arendt \emph{et al.} \parencite{Arendt_holo_total}
(the author thanks Stephan Fackler for bringing this news).
\end{rem*}
\begin{cor}[Inheritance of analyticity]
\label{cor:Inherit_of_anal} Let $\Omega\subset\mathbb{C}$ be open
and $X,Y$ be complex Banach spaces where $j:X\hookrightarrow Y$
is a continuous injection. Let $f:\Omega\to X$ be locally bounded.
Then f is analytic iff $j\circ f$ is analytic.
\end{cor}

\begin{proof}
$\text{Im}(j^{*})$ is weak$^{*}$-dense, therefore total.
\end{proof}
\begin{cor}[Evaluation on dense set]
\label{cor:Eval_on_dense_set} Let $X,Y$ be complex Banach spaces
with $X_{0}\leq X$. Let $\Omega\subset\mathbb{C}$ be open and $f:\Omega\to\mathcal{L}(X,Y)$
be a function. Assume $X_{0}\leq X$ is weakly dense and f is locally
bounded.

Then $f$ is analytic $\iff$$\forall x_{0}\in X_{0},$ $f(\cdot)x_{0}:\Omega\to Y$
is analytic.
\end{cor}

\begin{proof}
Consider $N_{X_{0}}=\text{span}\{y^{*}\circ\text{ev}_{x_{0}}:x_{0}\in X_{0},y^{*}\in Y^{*}\}\leq\mathcal{L}(X,Y)^{*}$.
It is total as $X_{0}$ is weakly dense. Use \Blackboxref{Holo_on_total}.
\end{proof}

\subsubsection{Semigroup definitions}

As mentioned before, we assume the reader is familiar with basic elements
of functional analysis, including semigroup theory as covered in \parencite[Appendix A.9]{Taylor_PDE1}.

Unfortunately, definitions vary depending on the authors, so we need
to be careful about which ones we are using.
\begin{defn}
For $\delta\in(0,\pi]$, define $\Sigma_{\delta}^{+}=\{z\in\mathbb{C}\backslash\{0\}:|\arg z|<\delta\}$,
$\Sigma_{\delta}^{-}=-\Sigma_{\delta}^{+}$, $\mathbb{D}=\{z\in\mathbb{C}:|z|<1\}.$
Also define $\Sigma_{0}^{+}=(0,\infty)$ and $\Sigma_{0}^{-}=-\Sigma_{0}^{+}$.

Let $X$ be a complex Banach space.

$(T(t))_{t\geq0}\subset\mathcal{L}(X)$ is called:
\begin{itemize}
\item a\textbf{ semigroup} when $T:[0,\infty)\to\mathcal{L}(X)$ is a monoid
homomorphism ($T(0)=1,T(t_{1}+t_{2})=T(t_{1})T(t_{2}))$
\item \textbf{degenerate} when $T:(0,\infty)\to\mathcal{L}(X)$ is continuous
in the SOT (\textbf{strong operator topology}).
\item \textbf{immediately norm-continuous }when $T:(0,\infty)\to\mathcal{L}(X)$
is norm-continuous.
\item \textbf{$C_{0}$ }(\textbf{strongly continuous})\textbf{ }when $T:[0,\infty)\to\mathcal{L}(X)$
is continuous in the SOT.
\item \textbf{bounded} when $T([0,\infty))$ is bounded in $\mathcal{L}(X)$,
and \textbf{locally bounded} when $T(K)$ is bounded $\forall K\subset[0,\infty)$
bounded. (so $C_{0}$ implies local boundedness by Banach-Steinhaus,
and the semigroup property implies we just need to test $K\subset[0,1)$)
\end{itemize}
$(T(z))_{z\in\Sigma_{\delta}^{+}\cup\{0\}}\subset\mathcal{L}(X)$
is called
\begin{itemize}
\item a\textbf{ semigroup }when $T:\Sigma_{\delta}^{+}\cup\{0\}\to\mathcal{L}(X)$
is a monoid homomorphism.
\item \textbf{$C_{0}$ }when $\forall\delta'\in(0,\delta),T:\Sigma_{\delta'}^{+}\cup\{0\}\to\mathcal{L}(X)$
is continuous in the SOT.
\item \textbf{bounded }when $T\left(\Sigma_{\delta'}^{+}\right)$ is bounded
$\forall\delta'\in(0,\delta)$ and \textbf{locally bounded} when $T(K)$
is bounded $\forall K\subset\Sigma_{\delta'}^{+}$ bounded. (so $C_{0}$
implies local boundedness, and the semigroup property implies we just
need to test $K\subset\mathbb{D}\cap\Sigma_{\delta'}^{+}$)
\item \textbf{analytic} when $T:\Sigma_{\delta}^{+}\to\mathcal{L}(X)$ is
analytic
\end{itemize}
We say $(T(t))_{t\geq0}$ is \textbf{analytic }of angle $\delta\in(0,\frac{\pi}{2}]$
if there is an extension $(T(z))_{z\in\Sigma_{\delta}^{+}\cup\{0\}}\subset\mathcal{L}(X)$
which is analytic and \uline{locally bounded}. If furthermore $(T(z))_{z\in\Sigma_{\delta}^{+}\cup\{0\}}$
is bounded, we say $(T(t))_{t\geq0}$ is \textbf{boundedly analytic
}of angle $\delta$.
\end{defn}

\begin{rem*}
A subtle problem is that when $(T(t))_{t\geq0}$ is bounded and analytic,
we cannot conclude $(T(t))_{t\geq0}$ is boundedly analytic (cf. \cite[Definition 3.7.3]{Arendt_Laplace}).
\end{rem*}
\begin{blackbox}
\label{blackbox:analytic_extension} If $(T(t))_{t\geq0}$ is a $C_{0}$
semigroup which is (boundedly) analytic of angle $\delta\in(0,\frac{\pi}{2}]$,
then $(T(z))_{z\in\Sigma_{\delta}^{+}\cup\{0\}}$ is a $C_{0},$ (bounded)
semigroup.
\end{blackbox}

\begin{proof}
The semigroup property comes from the identity theorem, and $C_{0}$
comes from the Vitali holomorphic convergence theorem. We refer to
\cite[Proposition 3.7.2]{Arendt_Laplace}.
\end{proof}
\begin{thm}[Sobolev tower]
 \label{thm:Sobolev_tower}Let $(e^{tA})_{t\geq0}$ be a $C_{0}$
semigroup on a (real/complex) Banach space $X$ with generator $A$
(implying $A$ is closed and densely defined). Then $\forall m\in\mathbb{N}_{1}$,
$D(A^{m})$ is a Banach space under the norm $\left\Vert x\right\Vert _{D(A^{m})}=\left\Vert x\right\Vert _{X}+\sum_{k=1}^{m}\left\Vert A^{k}x\right\Vert _{X}$,
and $D(A^{m})$ is dense in $X$.

As $e^{tA}$ and $A$ commute on $D(A)$, we conclude that $\left(e^{tA}\right)_{t\geq0}$,
after domain restriction, is also a $C_{0}$ semigroup on $D(A^{m})$
and $\left\Vert e^{tA}\right\Vert _{\mathcal{L}\left(D(A^{m})\right)}\leq\left\Vert e^{tA}\right\Vert _{\mathcal{L}\left(X\right)}\;\forall t\geq0$.

Lastly, if $X$ is a complex Banach space and $\left(e^{tA}\right)_{t\geq0}$
is (boundedly) analytic on $X$, $\left(e^{tA}\right)_{t\geq0}$ is
also (boundedly) analytic on $D(A^{m})$ after domain restriction.
\end{thm}

\begin{proof}
Most are just the basics of semigroup theory (cf. \parencite[Appendix A.9]{Taylor_PDE1}).
We only prove the last assertion. All we need is commutativity: if
$\left(e^{tA}\right)_{t\geq0}$ is extended to $(e^{zA})_{z\in\Sigma_{\delta}^{+}\cup\{0\}}$,
we want to show $e^{zA}A=Ae^{zA}$ on $D(A)$.

By \Blackboxref{analytic_extension}, $(e^{zA})_{z\in\Sigma_{\delta}^{+}\cup\{0\}}$
is a $C_{0}$ semigroup. Therefore $\forall x\in D(A),\forall z\in\Sigma_{\delta}^{+}:$
\[
e^{zA}Ax=e^{zA}\left(X\text{-}\lim_{t\downarrow0}\frac{e^{tA}-1}{t}x\right)=X\text{-}\lim_{t\downarrow0}e^{zA}\frac{e^{tA}-1}{t}x=X\text{-}\lim_{t\downarrow0}\frac{e^{tA}-1}{t}e^{zA}x
\]
The last term implies $e^{zA}x\in D\left(A\right)$ and $e^{zA}Ax=Ae^{zA}x$.
Then use \Corref{Inherit_of_anal} and \Corref{Eval_on_dense_set}
to get analyticity.
\end{proof}

\subsubsection{Simple extrapolation (with core)}
\begin{lem}
\label{lem:from_core_U}Let $U$, X be complex Banach spaces and $U\hookrightarrow X$
be a continuous injection with dense image.
\begin{enumerate}
\item Let $(T(t))_{t\geq0}\subset\mathcal{L}(X)$ be locally bounded and
$T(t)U\leq U\;\forall t\geq0.$ Assume $(T(t))_{t\geq0}$ is a $C_{0}$
semigroup on $U$. Then $(T(t))_{t\geq0}$ on $X$ is also a $C_{0}$
semigroup.
\item Let $(T(z))_{z\in\Sigma_{\delta}^{+}\cup\{0\}}\subset\mathcal{L}(X)$
(where $\delta\in(0,\frac{\pi}{2}]$) be locally bounded and $T(z)U\leq U\;\forall z\in\Sigma_{\delta}^{+}$.
Assume $(T(z))_{z\in\Sigma_{\delta}^{+}\cup\{0\}}$ is a $C_{0},$
analytic semigroup on $U$. Then $(T(z))_{z\in\Sigma_{\delta}^{+}\cup\{0\}}$
on $X$ is also a $C_{0}$, analytic semigroup.
\end{enumerate}
\end{lem}

\begin{rem*}
The assumption of local boundedness on $X$ is important. We will
also use this result in \Subsecref{W1p-analyticity} to establish
the $W^{1,p}$-analyticity of the heat flow.
\end{rem*}
\begin{proof}
The semigroup property comes from the density of $U$ in $X$.

To get $C_{0}$ on $X$, use the local boundedness on $X$ and dense
convergence (\Lemref{dense_conv}).

For analyticity in (2), use \Corref{Eval_on_dense_set}.
\end{proof}
\begin{lem}[Core]
\label{lem:core} Let $A$ be an unbounded linear operator on a (real/complex)
Banach space $X$ and $E\leq D(A)$. E is called a \textbf{core }when
$E$ is dense in $\left(D(A),\left\Vert \cdot\right\Vert _{D(A)}\right)$.

If $A$ is the generator of a  $C_{0}$ semigroup on $X$, $E$ is
dense in $X$ and $e^{tA}E\leq E$, then E is a core.
\end{lem}

\begin{proof}
Let $x\in D(A)$. Then there is $\left(x_{j}\right)_{j\in\mathbb{N}}$
in $E$ such that $x_{j}\xrightarrow{X}x$. It is trivial to check
\[
\frac{1}{t}\int_{0}^{t}e^{sA}x_{j}\;\mathrm{d}s\xrightarrow[j\to\infty]{\left\Vert \cdot\right\Vert _{D(A)}}\frac{1}{t}\int_{0}^{t}e^{sA}x\;\mathrm{d}s\xrightarrow[t\downarrow0]{\left\Vert \cdot\right\Vert _{D(A)}}x
\]
 as $\left(\restr{e^{sA}}{D(A)}\right)_{s\geq0}$ is $\left\Vert \cdot\right\Vert _{D(A)}$-continuous.
Note that $\int_{0}^{t}e^{sA}x_{j}\;\mathrm{d}s$ is in the $\left\Vert \cdot\right\Vert _{D(A)}$-closure
of $E$ by the Riemann integral.
\end{proof}
\begin{thm}[Simple extrapolation with core]
\label{thm:core_extrapol}\label{thm:simple_core_extrapol} Let $(X_{0},X_{1})$
be an interpolation couple of complex Banach spaces and $X_{\theta}=[X_{0},X_{1}]_{\theta}$
for $\theta\in(0,1]$.

Let $\left(T(t)\right)_{t\geq0}\subset\mathcal{L}\left(\left(X_{0},X_{1}\right)\right)$
. Assume that on $X_{0}$, $\left(T(t)\right)_{t\geq0}$ is bounded.

Assume that on $X_{1}$, $\left(T(t)\right)_{t\geq0}$ is a $C_{0}$
semigroup, boundedly analytic of angle $\delta\in(0,\frac{\pi}{2}]$
with generator $A_{1}$.

Assume $\exists m\in\mathbb{N}_{1}:\;\left(D(A_{1}^{m}),\left\Vert \cdot\right\Vert _{D(A_{1}^{m})}\right)\hookrightarrow\left(X_{0}\cap X_{1},\left\Vert \cdot\right\Vert _{X_{0}\cap X_{1}}\right)\hookrightarrow X_{0}$
are continuous injections with dense images.

Then on $X_{\theta}$, $\left(T(t)\right)_{t\geq0}$ is a $C_{0}$
semigroup, and boundedly analytic of angle $\theta\delta$.
\end{thm}

\begin{rem*}
The existence of a convenient core like $D(A_{1}^{m})$ is usually
a trivial consequence of Sobolev embedding. We can replace bounded
analyticity on $X_{1}$ and $X_{\theta}$ with analyticity, and boundedness
on $X_{0}$ with local boundedness via the usual rescaling argument
($\forall\delta'\in(0,\delta)\subset(0,\frac{\pi}{2}),\exists C_{\delta'}>0:\left\Vert e^{-C_{\delta'}z}T(z)\right\Vert _{\mathcal{L}(X_{1})}\lesssim_{\delta'}1\;\forall z\in\Sigma_{\delta'}^{+}$).

The existence of a core allows conditions on $X_{0}$ and $X_{1}$
to be more general than those in \cite[Theorem 3.1.1]{Stephan_thesis}
(which requires immediate norm-continuity on $X_{0}$), and actually
be equivalent to those in \cite[Theorem 3.1.10]{Stephan_thesis} (though
Kato-Beurling covers more than just complex interpolation). Once again,
the assumption of (local) boundedness on $X_{0}$ is important.

We will use this result to establish the $L^{p}$-analyticity of the
heat flow in \Subsecref{Lp-analyticity}.
\end{rem*}
\begin{proof}
Let $U=D(A_{1}^{m})$. Then $U$ is Banach as $A_{1}$ is closed.
Obviously $U\hookrightarrow X_{\theta}$ is a continuous injection
with dense image, and $(T(z))_{z\in\Sigma_{\delta}^{+}\cup\{0\}}$
is a $C_{0}$, bounded, analytic semigroup on $U$ (via Sobolev tower).

By \Lemref{from_core_U}, $\left(T(t)\right)_{t\geq0}$ is a $C_{0}$,
bounded semigroup on $X_{0}$. Also by \Lemref{from_core_U}, to get
the desired conclusion, we only need to show $\left(T(z)\right)_{z\in\Sigma_{\theta\delta}^{+}\cup\{0\}}$
is locally bounded in $\mathcal{L}(X_{\theta})$.

Fix $\delta'\in(0,\delta)$. We use abstract Stein interpolation.
Define the strip $\Omega=\{0<\text{Re}<1\}$. Let $\alpha\in(-\delta',\delta')$,
$\rho>0,u\in U$ and 
\[
L(z)=T(\rho e^{i\alpha z})u\;\forall z\in\overline{\Omega}
\]

Note that $U\leq X_{0}\cap X_{1}$ is dense. We check the other conditions
for interpolation:
\begin{itemize}
\item As $U\hookrightarrow X_{0}$ and $U\hookrightarrow X_{1}$ are continuous,
$(\overline{\Omega}\to X_{0}+X_{1},z\mapsto L(z)u)$ is continuous,
bounded on $\overline{\Omega}$ and analytic on $\Omega$ (as $L(z)u\in X_{1}\hookrightarrow X_{0}+X_{1}$).
\item For $j=0,1$ $\left(\mathbb{R}\to X_{j},s\mapsto L(j+is)u\right)$
is
\begin{itemize}
\item continuous since $U\hookrightarrow X_{j}$ is continuous.
\item bounded by $C_{j,T}\left\Vert u\right\Vert _{X_{j}}$ for some $C_{j,T}>0$
since $\left(T(t)\right)_{t\geq0}$ is bounded on $X_{0}$ and $\left(T(te^{i\alpha})\right)_{t\geq0}$
is bounded on $X_{1}$.
\end{itemize}
\end{itemize}
Then by Stein interpolation, we conclude $\{T(\rho e^{i\theta\alpha}):\rho>0,\alpha\in(-\delta',\delta')\}=T(\Sigma_{\theta\delta'}^{+})\subset\mathcal{L}(X_{\theta})$
is bounded.
\end{proof}

\subsubsection{Coreless version}

There is an alternative version which we will not use, but is of independent
interest:
\begin{thm}[Coreless extrapolation]
\label{thm:coreless_extrapol} Let $(X_{0},X_{1})$ be an interpolation
couple of complex Banach spaces and $X_{\theta}=[X_{0},X_{1}]_{\theta}$
for $\theta\in(0,1]$.

Let $\left(T(t)\right)_{t\geq0}\subset\mathcal{L}\left(\left(X_{0},X_{1}\right)\right)$
be \uline{a semigroup}. Assume that on $X_{0},$ $\left(T(t)\right)_{t\geq0}$
is bounded and \uline{degenerate}.

Assume that on $X_{1}$, $\left(T(t)\right)_{t\geq0}$ is a $C_{0}$
semigroup, boundedly analytic of angle $\delta\in(0,\frac{\pi}{2}]$
with generator $A_{1}$.

Then on $X_{\theta}$, $\left(T(t)\right)_{t\geq0}$ is a $C_{0}$
semigroup, boundedly analytic of angle $\theta\delta$.
\end{thm}

\begin{rem*}
The differences with the previous version are underlined. Again, via
rescaling we can replace bounded analyticity on $X_{1}$ and $X_{\theta}$
with analyticity, and boundedness on $X_{0}$ with local boundedness.
The conditions on $X_{0}$ and $X_{1}$ are still a bit more general
than those in \cite[Theorem 3.1.1]{Stephan_thesis}, which requires
immediate norm-continuity on $X_{0}$. In practice local boundedness
on $X_{0}$ can usually come from global analysis, while degeneracy
can come from Sobolev embedding and dense convergence (\Lemref{dense_conv}).
Immediate norm-continuity is harder to establish.

Note that \Thmref{coreless_extrapol} is not as general as \cite[Theorem 3.1.10]{Stephan_thesis}
(which removes the need for degeneracy and covers more than just complex
interpolation), though it is markedly easier to prove.
\end{rem*}
\begin{proof}
By interpolation, $\left(T(t)\right)_{t\geq0}$ is a bounded semigroup
on $X_{\theta}$.

Let $U=X_{0}\cap X_{1}$. Obviously $\left(T(t)\right)_{t\geq0}$
is a bounded semigroup on $U$.

Then observe that $\forall u\in U,\forall t,t_{0}\geq0:$
\[
\left\Vert \left(T(t)-T(t_{0})\right)u\right\Vert _{X_{\theta}}\le\left\Vert \left(T(t)-T(t_{0})\right)u\right\Vert _{X_{0}}^{1-\theta}\left\Vert \left(T(t)-T(t_{0})\right)u\right\Vert _{X_{1}}^{\theta}\lesssim\left\Vert \left(T(t)-T(t_{0})\right)u\right\Vert _{X_{1}}^{\theta}
\]
Since $\theta\neq0$, we have $T(t)u\xrightarrow[t\to t_{0}]{X_{\theta}}T(t_{0})u.$
As $\left(T(t)\right)_{t\geq0}$ is bounded on $X_{\theta}$ and $U$
is dense in $X_{\theta}$, we conclude $\left(T(t)\right)_{t\geq0}$
is $C_{0}$ on $X_{\theta}$ by dense convergence (\Lemref{dense_conv}).

Fix $\delta'\in(0,\delta)$. We use abstract Stein interpolation.
Define the strip $\Omega=\{0<\text{Re}<1\}$. Let $\alpha\in(-\delta',\delta')$,
$\rho>0,u\in U$ and 
\[
L(z)=T(\rho e^{i\alpha z})u\;\forall z\in\overline{\Omega}
\]

Note that $U=X_{0}\cap X_{1}$. We check the other conditions for
interpolation:
\begin{itemize}
\item As $U\hookrightarrow X_{0}$ and $U\hookrightarrow X_{1}$ are continuous,
$(\overline{\Omega}\to X_{0}+X_{1},z\mapsto L(z)u)$ is continuous,
bounded on $\overline{\Omega}$ and analytic on $\Omega$ (as $L(z)u\in X_{1}\hookrightarrow X_{0}+X_{1}$).
\item For $j=0,1$ $\left(\mathbb{R}\to X_{j},s\mapsto L(j+is)u\right)$
is
\begin{itemize}
\item continuous since $\left(T(t)\right)_{t\geq0}$ is degenerate on $X_{0}$
and $\left(T(te^{i\alpha})\right)_{t\geq0}$ is $C_{0}$ on $X_{1}$.
\item bounded by $C_{j,T}\left\Vert u\right\Vert _{X_{j}}$ for some $C_{j,T}>0$
since $\left(T(t)\right)_{t\geq0}$ is bounded on $X_{0}$ and $\left(T(te^{i\alpha})\right)_{t\geq0}$
is bounded on $X_{1}$.
\end{itemize}
\end{itemize}
By Stein interpolation, $\{T(\rho e^{i\theta\alpha}):\rho>0,\alpha\in(-\delta',\delta')\}=T(\Sigma_{\theta\delta'}^{+})\subset\mathcal{L}(X_{\theta})$
is bounded.

Finally, we just need to show $\left(T(z)\right)_{z\in\Sigma_{\theta\delta}^{+}\cup\{0\}}$
is analytic on $X_{\theta}$. Let $u\in U$. Then 
\[
\left(\Sigma_{\delta}^{+}\to X_{1}\hookrightarrow X_{0}+X_{1},z\mapsto T(z)u\right)
\]
 is analytic. Therefore $\left(\Sigma_{\theta\delta}^{+}\to X_{\theta}\hookrightarrow X_{0}+X_{1},z\mapsto T(z)u\right)$
is analytic, while $\left(\Sigma_{\theta\delta}^{+}\to X_{\theta},z\mapsto T(z)u\right)$
is locally bounded, so we can use \Corref{Inherit_of_anal} to conclude
$\left(\Sigma_{\theta\delta}^{+}\to X_{\theta},z\mapsto T(z)u\right)$
is analytic. As $U$ is dense in $X_{\theta}$, by \corref{Eval_on_dense_set},
we conclude $\left(\Sigma_{\theta\delta}^{+}\to\mathcal{L}\left(X_{\theta}\right),z\mapsto T(z)\right)$
is analytic.
\end{proof}

\subsection{Sectorial operators}

Recall that if $\left(T(t)\right)_{t\geq0}$ is a $C_{0}$ semigroup
on a complex Banach space $X$, then it has a closed, densely defined
generator $A$, and $T(t)=e^{tA}$ is exponentially bounded: $\left\Vert e^{tA}\right\Vert \lesssim_{\neg t}e^{Ct}$
for some $C>0.$ Then $\forall\zeta\in\{\text{Re}>C\}:\zeta\in\rho(A)$
and

\[
\frac{1}{\zeta-A}x=\int_{0}^{\infty}e^{-\zeta t}e^{tA}x\;\text{d}t\;\forall x\in X
\]

(cf. \parencite[Appendix A, Proposition 9.2]{Taylor_PDE1})

This means that the resolvent $\frac{1}{\zeta-A}$ is the Laplace
transform of the semigroup $e^{tA}$. This naturally leads to the
question when we can perform the inverse Laplace transform, to recover
the semigroup from the resolvent. This motivates the definition of
sectorial operators, which includes the Laplacian.

Unfortunately, there are wildly different definitions currently in
use by authors. The reader should study the definitions closely whenever
they consult any literature on sectorial operators (e.g. \parencite{lunardi1995analytic,Haase2006_thesis,Arendt_Laplace,engel2000one-parameter}).
\begin{defn}
Let $A$ be an unbounded operator on a complex Banach space $X$.
For $\theta\in[0,\pi),$ we say $A$ is
\begin{itemize}
\item \textbf{sectorial} of angle $\theta$ ($A\in\text{Sect}(\theta)$)
when $\left\{ %
\begin{tabular}{l}
$\sigma(A)\subset\overline{\Sigma_{\theta}^{-}}$\tabularnewline
$\forall\omega\in[0,\pi-\theta):M(A,\omega):=\sup\limits _{\lambda\in\Sigma_{\omega}^{+}}\left\Vert \frac{\lambda}{\lambda-A}\right\Vert <\infty$\tabularnewline
\end{tabular}\right.$
\item \textbf{quasi-sectorial} when $\exists a\in\mathbb{R}:A-a$ is sectorial.
\item \textbf{acutely sectorial} when $A\in\text{Sect}(\theta)$ for some
$\theta\in[0,\frac{\pi}{2})$
\item \textbf{acutely quasi-sectorial} when $\exists a\in\mathbb{R}:A-a$
is acutely sectorial.
\end{itemize}
For $r>0,\eta\in(\frac{\pi}{2},\pi)$, we define the (counterclockwise-oriented)
\textbf{Mellin curve} 
\[
\gamma_{r,\eta}=e^{i\eta}[r,\infty)\cup e^{-i\eta}[r,\infty)\cup re^{i[-\eta,\eta]}
\]
\end{defn}

\begin{rem*}
Depending on the author, ``sectorial'' can mean any of those four,
and that is not taking sign conventions into account (some authors
want $-\Delta$ to be sectorial), as well as whether $A$ should be
densely defined. The term ``quasi-sectorial'' is taken from \parencite{Haase2006_thesis}.

In particular, letting the spectrum be in the left half-plane means
we agree with \parencite{engel2000one-parameter,lunardi1995analytic}
and disagree with \parencite{Arendt_Laplace,Alan_Hinf_calculus_86,Haase2006_thesis}.
This is simply a personal preference, of being able to say ``the
Laplacian is sectorial'', or ``generators of $C_{0}$ analytic semigroups
are acutely sectorial''. Also, for bounded holomorphic calculus,
$e^{t\Delta}$ morally comes from $\left(e^{tz}\right)_{z\in\sigma(\Delta)}$
which is bounded in the left half-plane.
\end{rem*}
In keeping with tradition, here is the usual visualization:

\begin{figure}[H]
\centering{}\includegraphics[width=0.5\textwidth]{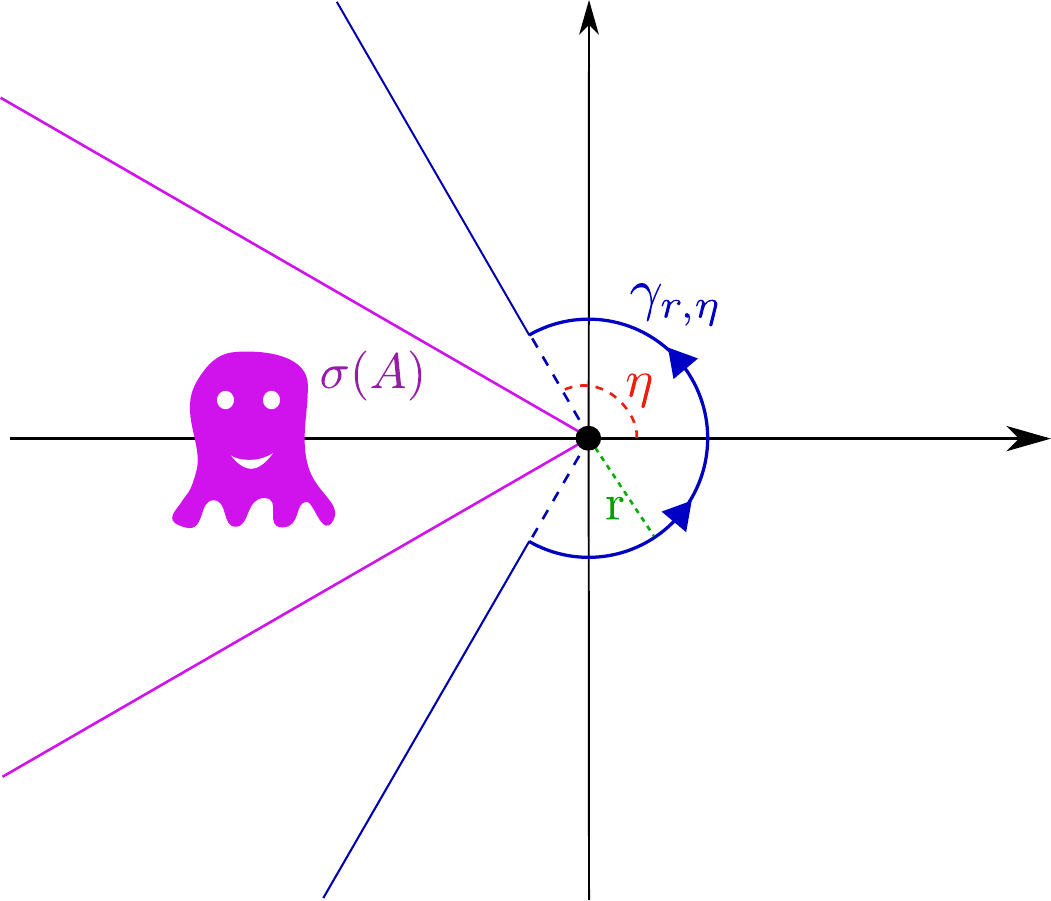}\caption{Acutely sectorial operators}
\end{figure}

\begin{blackbox}
\label{blackbox:analyticity}$A$ generates a $C_{0}$, boundedly
analytic semigroup on complex Banach space $X$ if and only if $A$
is densely defined and acutely sectorial.

When that happens, $\exists\delta\in\left(0,\frac{\pi}{2}\right)$
and $\eta\in(\frac{\pi}{2},\pi)$ such that $\left(e^{tA}\right)_{t\geq0}$
extends to $\left(e^{\zeta A}\right)_{\zeta\in\Sigma_{\delta}^{+}\cup\{0\}}$
and

\[
e^{\zeta A}=\frac{1}{2\pi i}\int_{\gamma_{r,\eta}}e^{\zeta z}\frac{1}{z-A}\;\text{d}z\;\;\;\forall\zeta\in\Sigma_{\delta}^{+},\forall r>0
\]

Also $\forall t>0,\forall k\in\mathbb{N}_{1}:e^{tA}(X)\leq D(A^{\infty})$,
$\left\Vert A^{k}e^{tA}\right\Vert \lesssim_{\neg t,\neg k}\frac{k^{k}}{t^{k}}$
and $\partial_{t}^{k}(e^{tA}x)=A^{k}e^{tA}x\;\forall x\in X$.
\end{blackbox}

\begin{rem*}
This is the aforementioned inverse Laplace transform. The Mellin curve
and the resolvent estimate in the definition of sectoriality ensure
sufficient decay for the integral to make sense. As it is a complex
line integral and the resolvent is analytic, the semigroup becomes
analytic.

A trivial consequence is that $D(A^{\infty})$ is dense in $X$ and
therefore a core.

When $A$ is densely defined and acutely quasi-sectorial, a simple
rescaling $e^{t(A-a)}=e^{-ta}e^{tA}$ implies $\left(e^{tA}\right)_{t\geq0}$
is a $C_{0}$, analytic semigroup.
\end{rem*}
\begin{proof}
See \parencite[Section II.4.a]{engel2000one-parameter}. The curious
figure $\frac{k^{k}}{t^{k}}$ comes from $A^{k}e^{tA}=\left(Ae^{\frac{t}{k}A}\right)^{k}$.
\end{proof}
\begin{thm}[Yosida's half-plane criterion]
 \label{thm:yosida_half_plane}$A$ is acutely quasi-sectorial if
and only if $\exists C>0$ such that
\begin{itemize}
\item $\{\Ree>C\}\subset\rho(A)$
\item $\sup\limits _{\lambda\in\{\Ree>C\}}\left\Vert \frac{\lambda}{\lambda-A}\right\Vert <\infty$
\end{itemize}
\end{thm}

\begin{rem*}
This is how the $L^{p}$-analyticity of the heat flow is traditionally
established. Yet proving the resolvent estimate is nontrivial, as
it is quite a refinement of elliptic estimates, so we choose not to
do so. Interestingly, we will instead use this for the $B_{3,1}^{\frac{1}{3}}$-analyticity
of the heat flow in \Subsecref{Interpolation-and-B-analyticity},
though that case is especially easy since we already have analyticity
at the two endpoints $L^{3}$ and $W^{1,3}$.
\end{rem*}
\begin{proof}
We only need to prove $\Leftarrow$. Recall the proof of how $\rho(A)$
is open: $\forall\lambda\in\rho(A),B\left(\lambda,\left\Vert \frac{1}{\lambda-A}\right\Vert ^{-1}\right)\subset\rho(A)$.
Applying this allows us to open up $\{\text{Re}>C\}$ and get $C+\Sigma_{\eta}^{+}\subset\rho(A)$
for some $\eta\in(\frac{\pi}{2},\pi)$. By choosing $\eta$ near $\frac{\pi}{2}$,
the resolvent estimate is retained.
\end{proof}
\begin{defn}
Let $A$ be an unbounded operator on a Hilbert space $X$. Then $A$
is called
\begin{itemize}
\item \textbf{symmetric} when $\left\langle Ax,y\right\rangle =\left\langle x,Ay\right\rangle \;\forall x,y\in D(A)$,
or equivalently, $A\subset A^{*}$ (where $A$ and $A^{*}$ are identified
with their graphs).
\item \textbf{self-adjoint} when $A=A^{*}$. This implies $\sigma(A)\subset\mathbb{R}$
(cf. \parencite[Appendix A, Proposition 8.5]{Taylor_PDE1}).
\item \textbf{dissipative }when $\text{Re}\left\langle Ax,x\right\rangle \leq0\;\forall x\in D(A)$.
\end{itemize}
When $A$ is dissipative, $\forall\lambda\in\{\Ree>0\},\forall x\in D(A):\Ree\left\langle \left(\lambda-A\right)x,x\right\rangle \geq\Ree\left\langle \lambda x,x\right\rangle $
so $\left\Vert \left(\lambda-A\right)x\right\Vert \geq\Ree\lambda\left\Vert x\right\Vert $.

Recall how $\rho(A)$ is proved to be open: $\forall\lambda\in\rho(A),B\left(\lambda,\left\Vert \frac{1}{\lambda-A}\right\Vert ^{-1}\right)\subset\rho(A)$.
Consequently, if $A$ is dissipative and $\exists\lambda_{0}\in\{\text{Re}>0\}\cap\rho(A)$,
we can conclude $\{\text{Re}>0\}\subset\rho(A)$.
\end{defn}

\begin{thm}[Dissipative sectoriality]
\label{thm:Dissipative} Assume $X$ is a complex Hilbert space and
$A$ is an unbounded, self-adjoint, dissipative operator on X. Then
$A$ is acutely sectorial of angle $0$.
\end{thm}

\begin{rem*}
Though standard, this might be the most elegant theorem in the theory,
and later on will instantly imply the $L^{2}$-analyticity of the
heat flow in \Subsecref{L2-analyticity}. The theorem can also be
proved by Euclidean geometry. When $X$ is separable, we can also
use the spectral theorem for unbounded operators.
\end{rem*}
\begin{proof}
As $A$ is self-adjoint, $\mathbb{C}\backslash\mathbb{R}\subset\rho(A)$.
By dissipativity, we conclude $\sigma(A)\subset(-\infty,0].$ Also
by self-adjointness, $\text{Re}\left\langle Ax,x\right\rangle =\left\langle Ax,x\right\rangle \leq0\;\forall x\in D(A)$.

Arbitrarily pick $\theta\in(\frac{\pi}{2},\pi)$. We want to show
$\left\Vert \frac{z}{z-A}\right\Vert \lesssim_{\theta}1\;\forall z\in\Sigma_{\theta}^{+}$.

Let $x\in X$ and $u=\frac{1}{z-A}x$. As $\left|\left\langle u,x\right\rangle \right|\leq\left\Vert u\right\Vert _{X}\left\Vert x\right\Vert _{X}$,
we want to show $\left\Vert u\right\Vert _{X}^{2}\lesssim_{\theta}\left|\frac{1}{z}\left\langle u,x\right\rangle \right|$.
Note that 
\[
\frac{1}{z}\left\langle u,x\right\rangle =\frac{1}{z}\left\langle u,(z-A)u\right\rangle =\left\langle u,u\right\rangle -\frac{1}{z}\left\langle Au,u\right\rangle 
\]

WLOG assume $\left\Vert u\right\Vert _{X}=1$. Then we want $1\lesssim_{\theta}\left|1-\frac{1}{z}\left\langle Au,u\right\rangle \right|$.
Note that $-\left\langle Au,u\right\rangle \geq0$ and $-\frac{1}{z}\left\langle Au,u\right\rangle \in\Sigma_{\theta}^{+}$.
Then we are done since 
\[
\left|1-\frac{1}{z}\left\langle Au,u\right\rangle \right|\geq\text{dist}(0,1+\Sigma_{\theta}^{+})>0.
\]

By Euclidean geometry, we can even calculate $\text{dist}(0,1+\Sigma_{\theta}^{+})$.
We will not need it though.
\end{proof}

\section{Scalar function spaces\label{sec:Scalar-function-spaces}}

Throughout this section, we work with complex-valued functions.

\subsection{On $\mathbb{R}^{n}$}
\begin{defn}
Here we recall the various (inhomogeneous) function spaces which are
particularly suitable for interpolation. They are defined as subspaces
of $\mathcal{S}'(\mathbb{R}^{n})$ with certain norms being finite:
\begin{enumerate}
\item \textbf{Lebesgue-Sobolev spaces:} for $m\in\mathbb{N}_{0},p\in[1,\infty]$:
$\left\Vert f\right\Vert _{W^{m,p}(\mathbb{R}^{n})}\sim\sum_{k=0}^{m}\left\Vert \nabla^{k}f\right\Vert _{p}$
where $\nabla^{k}f\in L^{p}$ are tensors defined by distributions.\nomenclature{$W^{m,p}$}{Sobolev spaces\nomrefpage}
It is customary to write $H^{m}$ for $W^{m,2}$.
\item \textbf{Bessel potential spaces}: for $s\in\mathbb{R},p\in[1,\infty]$:
$\left\Vert f\right\Vert _{H^{s,p}(\mathbb{R}^{n})}\sim\left\Vert \left\langle \nabla\right\rangle ^{s}f\right\Vert _{p}$
where $\left\langle \nabla\right\rangle ^{s}=\left(1-\Delta\right)^{\frac{s}{2}}$
is the Bessel potential.
\item \textbf{Besov spaces}: for $s\in\mathbb{R},p\in[1,\infty],q\in\left[1,\infty\right]$:
$\left\Vert f\right\Vert _{B_{p,q}^{s}(\mathbb{R}^{n})}\sim\left\Vert P_{\leq1}f\right\Vert _{p}+\left\Vert N^{s}\left\Vert P_{N}f\right\Vert _{p}\right\Vert _{l_{N>1}^{q}}$
where $P_{N}$ and $P_{\leq N}$ (for $N\in2^{\mathbb{Z}}$) are the
standard Littlewood-Paley projections (cf. \parencite[Appendix A]{tao2006nonlinear}).\nomenclature{$B_{p,q}^{s}$}{Besov spaces\nomrefpage}
\item \textbf{Triebel-Lizorkin spaces}: for $s\in\mathbb{R},p\in[1,\infty),q\in\left[1,\infty\right]$:
$\left\Vert f\right\Vert _{F_{p,q}^{s}(\mathbb{R}^{n})}\sim\left\Vert P_{\leq1}f\right\Vert _{p}+\left\Vert N^{s}\left\Vert P_{N}f\right\Vert _{l_{N>1}^{q}}\right\Vert _{p}$\nomenclature{$F_{p,q}^{s}$}{Triebel-Lizorkin spaces\nomrefpage}
\end{enumerate}
\end{defn}

\begin{rem*}
As there are multiple characterizations for the same spaces, we only
define up to equivalent norms. Of course, the topologies induced by
equivalent norms are the same.

In the literature, ``Fractional Sobolev spaces'' like $W^{s,p}$
could either refer to $B_{p,p}^{s}$ (\textbf{Sobolev--Slobodeckij
spaces}) or $H^{s,p}$. We shall avoid using the term at all. There
are also some delicate issues with $F_{\infty,q}^{s}$ which we do
not need to discuss here (cf. \parencite[Section 2.3.4]{triebel_function_I}).
\end{rem*}
\begin{blackbox}
Recall from harmonic analysis (cf.\parencite[Section 2.5.6, 2.3.3, 2.11.2]{triebel_function_I}
and \parencite[Part 1, Chapter 3.1]{rieusset2002_recent_developments}):
\begin{itemize}
\item $W^{m,p}(\mathbb{R}^{n})=H^{m,p}(\mathbb{R}^{n})$ for $m\in\mathbb{N}_{0},p\in(1,\infty).$
\item $F_{p,2}^{s}(\mathbb{R}^{n})=H^{s,p}(\mathbb{R}^{n})$ for $s\in\mathbb{R},p\in(1,\infty)$.
\item $B_{p,1}^{m}(\mathbb{R}^{n})\hookrightarrow W^{m,p}(\mathbb{R}^{n})\hookrightarrow B_{p,\infty}^{m}(\mathbb{R}^{n})$
for $m\in\mathbb{N}_{0},p\in[1,\infty].$
\item $\mathcal{S}(\mathbb{R}^{n})$ is dense in $W^{m,p}(\mathbb{R}^{n})$,
$B_{p,q}^{s}(\mathbb{R}^{n})$ and $F_{p,q}^{s}(\mathbb{R}^{n})$
for $m\in\mathbb{N}_{0},\;s\in\mathbb{R},p\in[1,\infty),q\in[1,\infty)$.
\item $B_{p,\min(p,q)}^{s}(\mathbb{R}^{n})\hookrightarrow F_{p,q}^{s}(\mathbb{R}^{n})\hookrightarrow B_{p,\max(p,q)}^{s}(\mathbb{R}^{n})$
for $s\in\mathbb{R},p\in[1,\infty),q\in[1,\infty]$.
\item $\left(B_{p,q}^{s}\left(\mathbb{R}^{n}\right)\right)^{*}=B_{p',q'}^{-s}\left(\mathbb{R}^{n}\right)$
for $s\in\mathbb{R},p\in[1,\infty),q\in[1,\infty)$.\\
$\left(F_{p,q}^{s}\left(\mathbb{R}^{n}\right)\right)^{*}=F_{p',q'}^{-s}\left(\mathbb{R}^{n}\right)$
for $s\in\mathbb{R},p\in(1,\infty),q\in(1,\infty)$.
\end{itemize}
\end{blackbox}

\subsection{On domains\label{subsec:On-domains}}
\begin{defn}
\label{def:domain}A \textbf{$C^{\infty}$ domain} $\Omega$ in $\mathbb{R}^{n}$
is defined as an open subset of $\mathbb{R}^{n}$ with smooth boundary,
and scalar function spaces are then defined on $\Omega$. If $\Omega\subset S\subset\overline{\Omega}$,
let function spaces on $S$ implicitly refer to function spaces on
$\Omega$. This will make it possible to discuss function spaces on,
for example, $\overline{\mathbb{R}_{+}^{n}}\cap B_{\mathbb{R}^{n}}(0,1)$,
or compact Riemannian manifolds with boundary.
\end{defn}

Obviously, Sobolev spaces are still defined on domains by distributions.
The big question is finding a good characterization for $B_{p,q}^{s}$
and $F_{p,q}^{s}$ on domains, when the Fourier transform is no longer
available. This is among the main topics of Triebel's seminal books.
Let us review the results:
\begin{defn}
Let $\Omega$ be either $\mathbb{R}^{n},$ or the half-space $\mathbb{R}_{+}^{n}$,
or a bounded $C^{\infty}$ domain in $\mathbb{R}^{n}$.

Then $B_{p,q}^{s}(\Omega)$ and $F_{p,q}^{s}(\Omega)$ can simply
be defined as the restrictions of $B_{p,q}^{s}(\mathbb{R}^{n})$ and
$F_{p,q}^{s}(\mathbb{R}^{n})$ to $\Omega$ and
\[
\left\Vert f\right\Vert _{B_{p,q}^{s}(\Omega)}=\inf\{\left\Vert F\right\Vert _{B_{p,q}^{s}(\mathbb{R}^{n})}:F\in B_{p,q}^{s}(\mathbb{R}^{n}),\left.F\right|_{\Omega}=f\}\;\text{ for }s\in\mathbb{R};p,q\in[1,\infty]
\]
\[
\left\Vert f\right\Vert _{F_{p,q}^{s}(\Omega)}=\inf\{\left\Vert F\right\Vert _{F_{p,q}^{s}(\mathbb{R}^{n})}:F\in F_{p,q}^{s}(\mathbb{R}^{n}),\left.F\right|_{\Omega}=f\}\;\text{ for }s\in\mathbb{R},p\in[1,\infty),q\in[1,\infty]
\]
A more useful characterization is via BMD (\textbf{ball mean difference}).
Let $\tau_{h}f(x)=f(x+h)$ be the translation operator and $\Delta_{h}f=\tau_{h}f-f$
be the difference operator. Then for $m\in\mathbb{N}_{1}$, we can
define $\Delta_{h}^{m}=\left(\Delta_{h}\right)^{m}$ as the $m$-th
difference operator. As we need to stay on the domain $\Omega$, define
\[
V^{m}(x,t)=\frac{1}{m}\left(B(x,mt)\cap\Omega-x\right)\;\text{ for }x\in\Omega,t>0,m\in\mathbb{N}_{1}
\]
So $V^{m}(x,t)\subset B(0,t)$, $x+mV^{m}(x,t)\subset\Omega$ and
$\Delta_{h}^{l}f(x)$ is well-defined when $h\in V^{m}(x,t)$. Also
note for $t\in\left(0,1\right)$: $|V^{m}(x,t)|\sim_{\Omega,m}t^{n}$.
Then by \parencite[Section 3.5.3, 5.2.2]{triebel_function_II}:
\begin{enumerate}
\item For $m\in\mathbb{N}_{1},p\in[1,\infty],q\in[1,\infty],s\in(0,m),r\in[1,p]:$
\begin{equation}
\left\Vert f\right\Vert _{B_{p,q}^{s}(\Omega)}\sim\left\Vert f\right\Vert _{p}+\left\Vert t^{-s}\left\Vert \left\Vert \Delta_{h}^{m}f(x)\right\Vert _{L_{h}^{r}\left(\frac{1}{t^{n}}\mathrm{d}h,V^{m}(x,t)\right)}\right\Vert _{L_{x}^{p}(\Omega)}\right\Vert _{L^{q}\left(\frac{1}{t}\mathrm{d}t,(0,1)\right)}\label{eq:BMD}
\end{equation}
We carefully note here that $m>s$ (the difference operator must be
strictly higher-order than the regularity), and that the variable
$t$ is small, which will play a big role in \Thmref{product_estimate}.
We also note that this is different from the classical characterization
via differences (\parencite[Section 3.4.2]{triebel_function_I}, \parencite[Section 1.10.3]{triebel_function_II})
which analysts might be more familiar with:
\[
\left\Vert f\right\Vert _{B_{p,q}^{s}(\Omega)}\sim\left\Vert f\right\Vert _{p}+\left\Vert \left|h\right|^{-s}\left\Vert \Delta_{h,\Omega}^{m}f(x)\right\Vert _{L_{x}^{p}(\Omega)}\right\Vert _{L^{q}\left(\frac{\text{d}h}{|h|^{n}},B(0,1)\right)}
\]
where $\Delta_{h,\Omega}^{m}f(x)$ is the same as $\Delta_{h}^{m}f(x)$,
but zero wherever undefined, and $m\in\mathbb{N}_{1},p\in(1,\infty),q\in[1,\infty],s\in(0,m)$.
\item For $m\in\mathbb{N}_{1},p\in[1,\infty),q\in[1,\infty],s\in(0,m),r\in[1,p]:$
\[
\left\Vert f\right\Vert _{F_{p,q}^{s}(\Omega)}\sim\left\Vert f\right\Vert _{p}+\left\Vert t^{-s}\left\Vert \left\Vert \Delta_{h}^{m}f(x)\right\Vert _{L_{h}^{r}\left(\frac{\text{d}h}{t^{n}},V^{m}(x,t)\right)}\right\Vert _{L^{q}\left(\frac{\text{d}t}{t},(0,1)\right)}\right\Vert _{L_{x}^{p}(\Omega)}
\]
\end{enumerate}
\end{defn}

\begin{blackbox}[Diffeomorphisms and smooth multipliers]

\label{blackbox:diffeomorphism-multiplier} Every diffeomorphism on
$\mathbb{R}^{n}$ preserves (under pullback) the topology of
\begin{itemize}
\item $W^{k,p}(\mathbb{R}^{n})$ for $k\in\mathbb{N}_{0},p\in[1,\infty]$
\item $B_{p,q}^{s}(\mathbb{R}^{n})$ for $s\in\mathbb{R},p\in[1,\infty],q\in[1,\infty]$
\item $F_{p,q}^{s}(\mathbb{R}^{n})$ for $s\in\mathbb{R},p\in[1,\infty),q\in[1,\infty]$
\end{itemize}
Also on the same spaces, for $\phi\in C_{c}^{\infty}(\mathbb{R}^{n}),$
$f\mapsto\phi f$ is a bounded linear map .
\end{blackbox}

\begin{rem*}
This allows us to trivially define function spaces on compact Riemannian
manifolds with boundary via partitions of unity and give them unique
topologies.
\end{rem*}
\begin{proof}
For $W^{k,p}$ it is trivial. For $B_{p,q}^{s}$ and $F_{p,q}^{s}$,
see \parencite[Section 4.3, 4.2.2]{triebel_function_II} and \parencite[Section 2.8.2]{triebel_function_I}.
\end{proof}
\begin{blackbox}[Extension and trace]
\label{blackbox:trace_ext}

Let $\Omega$ be either the half-space $\mathbb{R}_{+}^{n}$ or a
bounded $C^{\infty}$ domain in $\mathbb{R}^{n}$.
\begin{enumerate}
\item \textbf{Stein extension}: There exists a common (continuous linear)
extension operator $\mathfrak{E}:W^{k,p}(\Omega)\hookrightarrow W^{k,p}(\mathbb{R}^{n})$
for all $k\in\mathbb{N}_{0},p\in[1,\infty]$
\item \textbf{Triebel extension}: For any $N\in\mathbb{N}_{1},$ there exists
a common (continuous linear) extension operator $\mathfrak{E}^{N}$
such that
\begin{enumerate}
\item $\mathfrak{E}^{N}:B_{p,q}^{s}(\Omega)\hookrightarrow B_{p,q}^{s}(\mathbb{R}^{n})$
for all $|s|<N,p\in[1,\infty],q\in[1,\infty]$
\item $\mathfrak{E}^{N}:F_{p,q}^{s}(\Omega)\hookrightarrow F_{p,q}^{s}(\mathbb{R}^{n})$
for all $|s|<N,p\in[1,\infty),q\in[1,\infty]$
\end{enumerate}
\item \textbf{Trace theorems}: Let $n\geq2$.
\begin{enumerate}
\item For $p\in[1,\infty],q\in[1,\infty],s>\frac{1}{p}:$ $B_{p,q}^{s}(\Omega)\twoheadrightarrow B_{p,q}^{s-\frac{1}{p}}(\partial\Omega)$
is a \textbf{retraction} (continuous surjection with a bounded linear
section as a right inverse).
\item For $p\in[1,\infty),q\in[1,\infty],s>\frac{1}{p}:$ $F_{p,q}^{s}(\Omega)\twoheadrightarrow B_{p,p}^{s-\frac{1}{p}}(\partial\Omega)$
is a retraction.
\item (Limiting case) For $p\in[1,\infty),$ $B_{p,1}^{\frac{1}{p}}(\Omega)\twoheadrightarrow L^{p}(\partial\Omega)$
and $W^{1,1}(\Omega)\twoheadrightarrow L^{1}(\partial\Omega)$ are
continuous surjections.
\end{enumerate}
\end{enumerate}
\end{blackbox}

\begin{rem*}
It is important to note that we do not have the trace theorem for,
say, $B_{3,2}^{\frac{1}{3}}(\Omega)$ (cf. \parencite[Section 3]{Schneider2011_trace})
\end{rem*}
\begin{proof}
$\;$
\begin{enumerate}
\item See \parencite[Section VI.3]{eliasstein1971_singular_integrals}.
\item See \parencite[Section 4.5, 5.1.3]{triebel_function_II}.
\item See \parencite[Section 2.7.2, 3.3.3]{triebel_function_I} and the
remarks.
\end{enumerate}
\end{proof}
\begin{cor}
\label{cor:sobolev_and_density}Let $\Omega$ be either the half-space
$\mathbb{R}_{+}^{n}$ or a bounded $C^{\infty}$ domain in $\mathbb{R}^{n}$.
\begin{itemize}
\item $F_{p,2}^{m}(\Omega)=W^{m,p}(\Omega)$ for $m\in\mathbb{N}_{0},p\in(1,\infty)$.
\item $B_{p,1}^{m}(\Omega)\hookrightarrow W^{m,p}(\Omega)\hookrightarrow B_{p,\infty}^{m}(\Omega)$
for $m\in\mathbb{N}_{0},p\in[1,\infty].$
\item $\mathcal{S}(\overline{\Omega})$ is dense in $W^{m,p}(\Omega)$,
$F_{p,q}^{s}(\Omega)$ and $B_{p,q}^{s}(\Omega)$ for $m\in\mathbb{N}_{0},\;s\in\mathbb{R},p\in[1,\infty),q\in[1,\infty)$.
\item $B_{p,\min(p,q)}^{s}(\Omega)\hookrightarrow F_{p,q}^{s}(\Omega)\hookrightarrow B_{p,\max(p,q)}^{s}(\Omega)$
for $s\in\mathbb{R},p\in[1,\infty),q\in[1,\infty]$
\end{itemize}
\end{cor}

\begin{rem*}
When $\Omega$ is a bounded $C^{\infty}$ domain, $\mathcal{S}(\overline{\Omega})=C^{\infty}(\overline{\Omega})$.
\end{rem*}
\begin{proof}
Use Triebel and Stein extensions.
\end{proof}

\subsection{Holder \& Zygmund spaces}
\begin{defn}
Let $\Omega$ be either $\mathbb{R}^{n}$, the half-space $\mathbb{R}_{+}^{n}$
or a bounded $C^{\infty}$ domain in $\mathbb{R}^{n}$. Recall some
$L^{\infty}$ type spaces:
\begin{itemize}
\item \textbf{Holder spaces}: for $k\in\mathbb{N}_{0},\alpha\in(0,1]$,
\[
\left\Vert f\right\Vert _{C^{k,\alpha}(\Omega)}=\left\Vert f\right\Vert _{C^{k}(\Omega)}+\max_{|\beta|=k}\left[D^{\beta}f\right]_{C^{0,\alpha}(\Omega)}
\]
where $\left[g\right]_{C^{0,\alpha}(\Omega)}=\sup_{x\neq y}\frac{\left|g(x)-g(y)\right|}{\left|x-y\right|^{\alpha}}$
\item \textbf{Zygmund spaces}: for $s>0,$ define $\mathfrak{C}^{s}(\Omega)=B_{\infty,\infty}^{s}(\Omega)$
. Then for $m\in\mathbb{N}_{1},s\in(0,m)$:
\[
\left\Vert f\right\Vert _{\mathfrak{C}^{s}(\Omega)}\sim\left\Vert f\right\Vert _{L^{\infty}(\Omega)}+\left\Vert t^{-s}\left\Vert \left\Vert \Delta_{h}^{m}f(x)\right\Vert _{L_{h}^{\infty}\left(V^{m}(x,t)\right)}\right\Vert _{L_{x}^{\infty}(\Omega)}\right\Vert _{L_{t}^{\infty}\left((0,1)\right)}\sim\sup|f|+\sup_{0<|h|\leq1,x\in\Omega}\frac{\left|\Delta_{h,\Omega}^{m}f(x)\right|}{|h|^{s}}
\]
\end{itemize}
\nomenclature{$\mathfrak{C}^{s}(\Omega)$}{Zygmund spaces\nomrefpage}It
is well-known (cf. \parencite[Section 2.2.2, 2.5.7, 2.5.12, 2.8.3]{triebel_function_I})
that
\begin{itemize}
\item $\left\Vert f\right\Vert _{\mathfrak{C}^{k+\alpha}(\Omega)}\sim\text{\ensuremath{\left\Vert f\right\Vert }}_{C^{k}}+\max_{|\beta|=k}\left\Vert D^{\beta}f\right\Vert _{\mathfrak{C}^{\alpha}(\Omega)}$
for $k\in\mathbb{N}_{0},\alpha\in(0,1]$
\item $\left\Vert f\right\Vert _{\mathfrak{C}^{k+\alpha}(\Omega)}\sim\text{\ensuremath{\left\Vert f\right\Vert }}_{C^{k,\alpha}}$
for $k\in\mathbb{N}_{0},\alpha\in(0,1)$.
\item $\left\Vert fg\right\Vert _{\mathfrak{C}^{s}(\Omega)}\lesssim\text{\ensuremath{\left\Vert f\right\Vert }}_{\mathfrak{C}^{s}}\text{\ensuremath{\left\Vert g\right\Vert }}_{\mathfrak{C}^{s}}$
for $s>0$.
\end{itemize}
Note that $C^{0,1},C^{1}$ and $\mathfrak{C}^{1}$ are different.
\end{defn}

\subsection{Interpolation \& embedding\label{subsec:Interpolation-=000026-embedding}}
\begin{blackbox}[Interpolation]
Let $\Omega$ be either $\mathbb{R}^{n}$, the half-space $\mathbb{R}_{+}^{n}$
or a bounded $C^{\infty}$ domain in $\mathbb{R}^{n}$. Throughout
the theorem, always assume $\theta\in(0,1),s_{\theta}=(1-\theta)s_{0}+\theta s_{1}$.
\begin{enumerate}
\item $\left(B_{p,q_{0}}^{s_{0}}(\Omega),B_{p,q_{1}}^{s_{1}}(\Omega)\right)_{\theta,q}=B_{p,q}^{s_{\theta}}(\Omega)$
for $s_{0}\neq s_{1},s_{j}\in\mathbb{R},p\in[1,\infty],q_{j},q\in[1,\infty]$.\\
\\
$\left(F_{p,q_{0}}^{s_{0}}(\Omega),F_{p,q_{1}}^{s_{1}}(\Omega)\right)_{\theta,q}=B_{p,q}^{s_{\theta}}(\Omega)$
for $s_{0}\neq s_{1},s_{j}\in\mathbb{R},p\in[1,\infty),q_{j},q\in[1,\infty]$.
\item $\left(B_{p_{0},q_{0}}^{s_{0}}(\Omega),B_{p_{1},q_{1}}^{s_{1}}(\Omega)\right)_{\theta,p_{\theta}}=B_{p_{\theta},p_{\theta}}^{s_{\theta}}(\Omega)$
for $s_{0}\neq s_{1},s_{j}\in\mathbb{R},p_{j}\in[1,\infty],q_{j}\in[1,\infty],\frac{1}{p_{\theta}}=\frac{1-\theta}{p_{0}}+\frac{\theta}{p_{1}}=\frac{1-\theta}{q_{0}}+\frac{\theta}{q_{1}}$
\item $\left[B_{p_{0},q_{0}}^{s_{0}}(\Omega),B_{p_{1},q_{1}}^{s_{1}}(\Omega)\right]_{\theta}=B_{p_{\theta},q_{\theta}}^{s_{\theta}}(\Omega)$
and $\left[F_{p_{0},q_{0}}^{s_{0}}(\Omega),F_{p_{1},q_{1}}^{s_{1}}(\Omega)\right]_{\theta}=F_{p_{\theta},q_{\theta}}^{s_{\theta}}(\Omega)$
\\
\\
for $s_{j}\in\mathbb{R},p_{j}\in(1,\infty),q_{j}\in(1,\infty),\frac{1}{p_{\theta}}=\frac{1-\theta}{p_{0}}+\frac{\theta}{p_{1}},\frac{1}{q_{\theta}}=\frac{1-\theta}{q_{0}}+\frac{\theta}{q_{1}}$.
\item $\left[L^{p_{0}}(\Omega),L^{p_{1}}(\Omega)\right]_{\theta}=L^{p_{\theta}}(\Omega)$
for $p_{j}\in\left[1,\infty\right],\frac{1}{p_{\theta}}=\frac{1-\theta}{p_{0}}+\frac{\theta}{p_{1}}$.
\item $\left(W^{m_{0},p}(\Omega),W^{m_{1},p}(\Omega)\right)_{\theta,q}=B_{p,q}^{m_{\theta}}(\Omega)$
for $m_{j}\in\mathbb{N}_{0},m_{0}\neq m_{1},p\in[1,\infty],q\in[1,\infty]$,
$m_{\theta}=\left(1-\theta\right)m_{0}+\theta m_{1}$.
\end{enumerate}
\end{blackbox}

\begin{proof}
$\;$
\begin{enumerate}
\item Extension operators and \parencite[Section 2.4.2]{triebel_function_I}.
\item Extension operators and \parencite[Theorem 6.4.5]{Lofstrom}.
\item Extension operators and \parencite[Section 2.4.7]{triebel_function_I}.
\item Extension by zero and \parencite[Section 5.1.1]{Lofstrom}
\item Recall $B_{p,1}^{m}(\Omega)\hookrightarrow W^{m,p}(\Omega)\hookrightarrow B_{p,\infty}^{m}(\Omega)$
for $m\in\mathbb{N}_{0},p\in[1,\infty].$ Then apply 1.
\end{enumerate}
\end{proof}
\begin{blackbox}[Embedding]
 Let $\Omega$ be a bounded $C^{\infty}$ domain in $\mathbb{R}^{n}$.
Assume $\infty>s_{0}>s_{1}>-\infty.$ Then
\begin{enumerate}
\item $B_{p_{0},q_{0}}^{s_{0}}(\Omega)\hookrightarrow B_{p_{1},q_{1}}^{s_{1}}(\Omega)$
is compact for $p_{j}\in[1,\infty],q_{j}\in[1,\infty],\frac{1}{p_{1}}>\frac{1}{p_{0}}-\frac{s_{0}-s_{1}}{n}$\\
\\
$F_{p_{0},q_{0}}^{s_{0}}(\Omega)\hookrightarrow F_{p_{1},q_{1}}^{s_{1}}(\Omega)$
is compact for $p_{j}\in[1,\infty),q_{j}\in[1,\infty],\frac{1}{p_{1}}>\frac{1}{p_{0}}-\frac{s_{0}-s_{1}}{n}$
\item $B_{p_{0},q}^{s_{0}}(\Omega)\hookrightarrow B_{p_{1},q}^{s_{1}}(\Omega)$
is continuous for $p_{j}\in[1,\infty],q\in[1,\infty],\frac{1}{p_{1}}=\frac{1}{p_{0}}-\frac{s_{0}-s_{1}}{n}$\\
\\
$F_{p_{0},q_{0}}^{s_{0}}(\Omega)\hookrightarrow F_{p_{1},q_{1}}^{s_{1}}(\Omega)$
is continuous for $p_{j}\in[1,\infty),q_{j}\in[1,\infty],\frac{1}{p_{1}}=\frac{1}{p_{0}}-\frac{s_{0}-s_{1}}{n}$
\end{enumerate}
\end{blackbox}

\begin{proof}
$\;$
\begin{enumerate}
\item See \parencite[Section 4.3.2, Remark 1]{triebel_function_I} and \parencite[Section 3.3.1]{triebel_function_I}.
\item See \parencite[Section 3.3.1]{triebel_function_I}.
\end{enumerate}
\end{proof}
\begin{cor}
Let $\Omega$ be a bounded $C^{\infty}$ domain in $\mathbb{R}^{n}$.
Then
\begin{enumerate}
\item For $m_{j}\in\mathbb{N}_{0},m_{0}>m_{1},p_{j}\in[1,\infty],\frac{1}{p_{1}}>\frac{1}{p_{0}}-\frac{m_{0}-m_{1}}{n}$:\\
 $W^{m_{0},p_{0}}(\Omega)\hookrightarrow B_{p_{0},\infty}^{m_{0}}(\Omega)\hookrightarrow B_{p_{1},1}^{m_{1}}(\Omega)\hookrightarrow W^{m_{1},p_{1}}(\Omega)$
is compact.
\item For $m_{j}\in\mathbb{N}_{0},m_{0}>m_{1},p_{0}\in[1,\infty],\alpha\in\left(0,1\right),0>\frac{1}{p_{0}}-\frac{m_{0}-\left(m_{1}+\alpha\right)}{n}$:\\
 $W^{m_{0},p_{0}}(\Omega)\hookrightarrow B_{p_{0},\infty}^{m_{0}}(\Omega)\hookrightarrow B_{\infty,\infty}^{m_{1}+\alpha}(\Omega)=C^{m_{1},\alpha}(\overline{\Omega})$
is compact.
\item For $m\in\mathbb{N}_{1},p\in(1,\infty):W^{m,p}(\Omega)\hookrightarrow B_{p,\infty}^{m}(\Omega)\hookrightarrow B_{p,1}^{\frac{1}{p}}(\Omega)\twoheadrightarrow L^{p}(\partial\Omega)$
is compact.
\end{enumerate}
\end{cor}

\begin{rem*}
These include the Rellich-Kondrachov embeddings found in \parencite[Theorem 6.3]{adams2003sobolev},
so the Besov embeddings generalize Sobolev embeddings.
\end{rem*}

\subsection{Strip decay\label{subsec:Strip-decay}}

Some notation first: let $\Omega$ be a $C^{\infty}$ domain in $\mathbb{R}^{n}$
or a compact Riemannian manifold with or without boundary. Define
$\Omega_{>r}=\{x\in\Omega:\text{dist}(x,\partial\Omega)>r\}$ where
$\text{dist}(x,\partial\Omega)=\infty$ if $\partial\Omega=\emptyset$.
Similarly define $\Omega_{\geq r},\Omega_{<r},\Omega_{[r_{1},r_{2}]}$.\nomenclature{$\Omega_{<r}$}{$\{x\in\Omega:\text{dist}(x,\partial\Omega)<r\}$ \nomrefpage}

When $|\Omega|<\infty$ and $p\in[1,\infty)$, we write
\[
\left\Vert f\right\Vert _{L^{p}(\Omega,\text{avg})}=\left\Vert f(x)\right\Vert _{L_{x}^{p}(\frac{\text{d}x}{\left|\Omega\right|},\Omega)}=\left(\dashint_{\Omega}|f|^{p}\right)^{\frac{1}{p}}=\frac{1}{|\Omega|^{1/p}}\left(\int_{\Omega}|f|^{p}\right)^{\frac{1}{p}}
\]
By convention, we set $\left\Vert f\right\Vert _{L^{\infty}(\Omega,\text{avg})}=\left\Vert f(x)\right\Vert _{L_{x}^{\infty}(\frac{\text{d}x}{\left|\Omega\right|},\Omega)}=\left\Vert f\right\Vert _{L^{\infty}(\Omega)}$.
The implicit measure is of course the Riemannian measure. In such
mean integrals, the domain becomes a probability space.\nomenclature{$\left\Vert f\right\Vert _{L^{p}(\Omega,\avg)}$}{integration on probability space \nomrefpage}
\begin{thm}[Coarea formula]
\label{thm:coarea}$ $
\begin{enumerate}
\item For any $h\in\mathbb{R}^{n}$, the translation semigroup $\left(\tau_{th}\right)_{t\geq0}$
is a $C_{0}$ semigroup on $W^{m,p}(\mathbb{R}^{n})$, $B_{p,q}^{s}(\mathbb{R}^{n})$
and $F_{p,q}^{s}(\mathbb{R}^{n})$ for $m\in\mathbb{N}_{0},\;s\in\mathbb{R},p\in[1,\infty),q\in[1,\infty)$.
Consequently, for $p\in[1,\infty)$ and $f\in B_{p,1}^{\frac{1}{p}}(\mathbb{R}^{n})$
, 
\[
\left([0,\infty)\to L^{p}(\mathbb{R}^{n-1}),t\mapsto\left.\tau_{th}f\right|_{\mathbb{R}^{n-1}}\right)
\]
is continuous and bounded by $C\left\Vert f\right\Vert _{B_{p,1}^{\frac{1}{p}}(\mathbb{R}^{n})}$
for some $C>0$. 
\item Let $\Omega$ be a bounded $C^{\infty}$ domain in $\mathbb{R}^{n}$
(or a compact Riemannian manifold with boundary). Let $p\in[1,\infty)$.
Then for $f\in B_{p,1}^{\frac{1}{p}}(\Omega)$ and $r>0$ \uline{small}:
\begin{enumerate}
\item $\left([0,r)\to\mathbb{R},\rho\mapsto\left\Vert f\right\Vert _{L^{p}(\partial\Omega_{>\rho})}\right)$
is continuous and bounded by $C\left\Vert f\right\Vert _{B_{p,1}^{\frac{1}{p}}(\Omega)}$
for some $C>0$.
\item $\left\Vert f\right\Vert _{L^{p}(\Omega_{\leq r})}\sim_{\neg r}\left\Vert \left\Vert f\right\Vert _{L^{p}(\partial\Omega_{>\rho})}\right\Vert _{L_{\rho}^{p}((0,r))}$
\item $\left\Vert f\right\Vert _{L^{p}(\Omega_{\leq r},\avg)}\lesssim_{\neg r}\left\Vert f\right\Vert _{B_{p,1}^{\frac{1}{p}}(\Omega)}$
and $\left\Vert f\right\Vert _{L^{p}(\Omega_{\leq r},\avg)}\xrightarrow{r\downarrow0}\left\Vert f\right\Vert _{L^{p}(\partial\Omega,\avg)}$.
\item Let $I\subset\mathbb{R}$ be an open interval and $\mathfrak{f}\in L^{p}(I\to B_{p,1}^{\frac{1}{p}}(\Omega)),$then
$\left\Vert \mathfrak{f}\right\Vert _{L_{t}^{p}B_{p,1}^{\frac{1}{p}}(\Omega)}\gtrsim_{\neg r}\left\Vert \mathfrak{f}\right\Vert _{L_{t}^{p}L^{p}(\Omega_{\leq r},\avg)}\xrightarrow{r\downarrow0}\left\Vert \mathfrak{f}\right\Vert _{L_{t}^{p}L^{p}(\partial\Omega,\avg)}$. 
\end{enumerate}
\item Let $p\in[1,\infty),f\in W^{1,p}(\Omega),$ show that $\left\Vert f\right\Vert _{L^{p}\left(\Omega_{<r}\right)}\lesssim_{\neg r}r\left\Vert f\right\Vert _{W^{1,p}\left(\Omega_{<r}\right)}+r^{\frac{1}{p}}\left\Vert f\right\Vert _{L^{p}\left(\partial\Omega\right)}$
for $r>0$ small.
\end{enumerate}
\end{thm}

\begin{proof}
$\;$ 
\begin{enumerate}
\item Use the density of $\mathcal{S}(\mathbb{R}^{n})$ and \Lemref{dense_conv}. 
\item $\;$ 
\begin{enumerate}
\item By partition of unity, geodesic normals, diffeomorphisms and the smallness
of $r$, reduce the problem to the half-space case, which is just
1).
\item Approximate $f$ in $B_{p,1}^{\frac{1}{p}}$ by $C^{\infty}(\overline{\Omega})$
functions. This is the well-known coarea formula, which corresponds
to Fubini's theorem in the half-space case. Note that $\left\Vert f\right\Vert _{L^{p}(\partial\Omega_{>\rho})}$
is defined by the trace theorem. See \parencite[Section III.5]{chavel2006riemannian}
for more details. 
\item For $r$ small, $\left|\Omega_{<r}\right|\sim\left|\partial\Omega\right|r$
and $\left|\partial\Omega_{>r}\right|\sim\left|\partial\Omega\right|$,
so 
\[
\left\Vert f\right\Vert _{L^{p}(\Omega_{\leq r},\text{avg})}\sim\left\Vert \left\Vert f\right\Vert _{L^{p}(\partial\Omega_{>\rho},\text{avg})}\right\Vert _{L_{\rho}^{p}((0,r),\text{avg})}\leq\sup_{\rho<r}\left\Vert f\right\Vert _{L^{p}(\partial\Omega_{>\rho},\text{avg})}
\]
and $\left\Vert \left\Vert f\right\Vert _{L^{p}(\partial\Omega_{>\rho},\text{avg})}\right\Vert _{L_{\rho}^{p}((0,r),\text{avg})}\xrightarrow{r\downarrow0}\left\Vert f\right\Vert _{L^{p}(\partial\Omega,\text{avg})}$
by continuity in a). 
\item Dominated convergence. 
\end{enumerate}
\item By the trace theorem, WLOG $f\in C^{\infty}(\overline{\Omega})$.
By partition of unity and diffeomorphisms, WLOG $\overline{\Omega}=\overline{\mathbb{R}_{+}^{n}}=\{(\mathbf{x},y):\mathbf{x}\in\mathbb{R}^{n-1},y\geq0\}$.
Then 
\begin{align*}
\left\Vert f\right\Vert _{L^{p}\left(\Omega_{<r}\right)} & \sim\left\Vert \left\Vert f(\mathbf{x},y)\right\Vert _{L_{y}^{p}\left([0,r]\right)}\right\Vert _{L_{\mathbf{x}}^{p}}\lesssim\left\Vert \left\Vert \left\Vert \partial_{y}f(\mathbf{x},\rho)\right\Vert _{L_{\rho}^{1}\left([0,y]\right)}+\left|f(\mathbf{x},0)\right|\right\Vert _{L_{y}^{p}\left([0,r]\right)}\right\Vert _{L_{\mathbf{x}}^{p}}\\
 & \lesssim\left\Vert \left\Vert \left\Vert \partial_{y}f(\mathbf{x},\rho)\right\Vert _{L_{\rho}^{p}\left([0,y]\right)}\left|y\right|^{\frac{1}{p'}}\right\Vert _{L_{y}^{p}\left([0,r]\right)}\right\Vert _{L_{\mathbf{x}}^{p}}+r^{\frac{1}{p}}\left\Vert f(\mathbf{x},0)\right\Vert _{L_{\mathbf{x}}^{p}}
\end{align*}
The first term $\lesssim\left\Vert \left\Vert \partial_{y}f(\mathbf{x},\rho)\right\Vert _{L_{\rho}^{p}\left([0,r]\right)}\left\Vert \left|y\right|^{\frac{1}{p'}}\right\Vert _{L_{y}^{p}\left([0,r]\right)}\right\Vert _{L_{\mathbf{x}}^{p}}\lesssim r\left\Vert \partial_{y}f\right\Vert _{L^{p}\left(\Omega_{<r}\right)}$.
So we are done.
\end{enumerate}
\end{proof}
\begin{thm}[Product estimate]
 \label{thm:product_estimate} Let $M$ be a bounded $C^{\infty}$
domain in $\mathbb{R}^{n}$ (or a compact Riemannian manifold with
boundary). Assume $r>0$ small, $f_{r}\in C^{\infty}(\overline{M})$
with support in $M_{<r}$. Then for $p\in\left(1,\infty\right)$,
$g\in B_{p,1}^{\frac{1}{p}}(M)$:
\[
\Vert f_{r}g\Vert_{B_{p,1}^{1/p}(M)}\lesssim_{M,\neg r}\Vert f_{r}\Vert_{B_{\infty,1}^{1/p}(M)}\Vert g\Vert_{L^{p}(M_{<4r})}+\Vert f_{r}\Vert_{L^{\infty}(M_{<r})}\Vert g\Vert_{B_{p,1}^{1/p}(M)}
\]
\end{thm}

\begin{rem*}
The theory of product and commutator estimates (Kato-Ponce, Coifman-Meyer
etc.) has a long and rich history which we will not recount here (cf.
\parencite{Kato1988_Ponce_commutator,tao2004lecture_coifman,Grafakos2014_commutator,Naibo2019_coifman_meyer}).
However, for our intended application, $f_{r}$ has very small support
and we want to use $\Vert g\Vert_{L^{p}(M_{<4r})}$ instead of $\Vert g\Vert_{L^{p}(M)}$
to control the product. Unfortunately there does not seem to be much,
if at all, literature on this issue. This theorem will only be used
for \Thmref{B1ppN0_is_B1ppN}, and is not necessary for Onsager's
conjecture.
\end{rem*}
\begin{proof}
By diffeomorphisms, partition of unity, and geodesic normals, WLOG
assume $M=\overline{\mathbb{R}_{+}^{n}}$ with $M_{<r}=\{x\in\mathbb{R}^{n}:0\leq x_{n}<r\}$.\\
Recall $\Vert g\Vert_{L^{p}(x_{n}=a)}\lesssim\Vert g\Vert_{B_{p,1}^{1/p}(\mathbb{R}_{+}^{n})}\;\forall0\leq a<\infty$
where $\Vert g\Vert_{L^{p}(x_{n}=a)}:=\Vert g\Vert_{L^{p}(\{x\in\mathbb{R}^{n}:x_{n}=a\})}$
is defined by the trace theorem.

WLOG, assume $\left\Vert f_{r}\right\Vert _{\infty}\leq1.$ Recall
the characterization of Besov spaces by ball mean difference (BMD)
and write $V(x,t)$ for $V^{1}(x,t)$ (see \Eqref{BMD}). Then
\[
\Vert f_{r}g\Vert_{B_{p,1}^{1/p}(M)}\sim\Vert f_{r}g\Vert_{L^{p}(M)}+\left\Vert t^{-\frac{1}{p}-n}\left\Vert \Vert\Delta_{h}(f_{r}g)(x)\Vert_{L_{h}^{1}(V(x,t))}\right\Vert _{L_{x}^{p}(M)}\right\Vert _{L_{t}^{1}(\frac{\text{d}t}{t},(0,1))}
\]
The term $\Vert f_{r}g\Vert_{L^{p}(M)}$ is easily bounded and thrown
away. For the remaining term, we use the identity $\Delta_{h}(f_{r}g)=\Delta_{h}f_{r}g+\tau_{h}f_{r}\Delta_{h}g$
to bound it by
\[
\left\Vert t^{-\frac{1}{p}-n}\left\Vert \Vert\Delta_{h}f_{r}(x)\Vert_{L_{h}^{1}(V(x,t))}g(x)\right\Vert _{L_{x}^{p}(M)}\right\Vert _{L^{1}(\frac{\text{d}t}{t},(0,1))}+\left\Vert t^{-\frac{1}{p}-n}\left\Vert \left\Vert f_{r}\right\Vert _{\infty}\Vert\Delta_{h}g(x)\Vert_{L_{h}^{1}(V(x,t))}\right\Vert _{L_{x}^{p}(M)}\right\Vert _{L^{1}(\frac{\text{d}t}{t},(0,1))}
\]
The second term here is just $\Vert f_{r}\Vert_{L^{\infty}}\Vert g\Vert_{B_{p,1}^{1/p}(M)}$,
so throw it away. For the remaining term, by using $\left\Vert \cdot\right\Vert _{L^{p}(M)}\lesssim\left\Vert \cdot\right\Vert _{L^{p}(M_{<4r})}+\left\Vert \cdot\right\Vert _{L^{p}(M_{>4r})}$
and
\[
\left\Vert \Vert\Delta_{h}f_{r}(x)\Vert_{L_{h}^{1}(V(x,t))}g(x)\right\Vert _{L_{x}^{p}(M_{<4r})}\lesssim\left\Vert \Vert\Delta_{h}f_{r}(x)\Vert_{L_{h}^{1}(V(x,t))}\right\Vert _{L_{x}^{\infty}(M_{<4r})}\left\Vert g(x)\right\Vert _{L_{x}^{p}(M_{<4r})}
\]
we are left with

\[
\Vert f_{r}\Vert_{B_{\infty,1}^{1/p}(M)}\Vert g\Vert_{L^{p}(M_{<4r})}+\left\Vert t^{-\frac{1}{p}-n}\left\Vert \Vert\Delta_{h}f_{r}(x)\Vert_{L_{h}^{1}(V(x,t))}g(x)\right\Vert _{L_{x}^{p}(M_{>4r})}\right\Vert _{L^{1}(\frac{\text{d}t}{t},(0,1))}
\]
Throwing away the first term, we have arrived at the important estimate:
what happens on $M_{>4r}.$ It will turn out that the values of $g$
on $M_{>4r}$ are well-controlled by $\Vert g\Vert_{B_{p,1}^{1/p}(M)}$.
To begin, recall $f_{r}$ is supported on $M_{<r}$ and use the crude
geometric estimate 
\[
t^{-n}\Vert\Delta_{h}f_{r}(x)\Vert_{L_{h}^{1}(V(x,t))}=t^{-n}\Vert f_{r}(x+h)\Vert_{L_{h}^{1}(V(x,t))}\lesssim\frac{\left|B(x,t)\cap M_{<r}\right|}{\left|B(x,t)\right|}\lesssim\frac{r}{x_{n}}\mathbf{1}_{t>x_{n}-r}\;\forall x\in M_{>4r},\forall t\in(3r,1)
\]
Note that $t>3r$ comes from $t>x_{n}-r>4r-r.$ So we have used the
``room'' from $4r$ to get an $O(r)$-lower bound for $t$. By $x_{n}<r+t$,
we now only need to bound
\[
\left\Vert t^{-\frac{1}{p}}\left\Vert g(x)\frac{r}{x_{n}}\right\Vert _{L_{x}^{p}(M_{\left[4r,r+t\right]})}\right\Vert _{L^{1}(\frac{\text{d}t}{t},(3r,1))}
\]
Obviously, we will integrate $g$ on $x_{n}$-slices (using $p>1$):
\[
\left\Vert g(x)\frac{r}{x_{n}}\right\Vert _{L_{x}^{p}(M_{\left[4r,r+t\right]})}=r\left\Vert \frac{1}{\rho}\left\Vert g\right\Vert _{L^{p}(x_{n}=\rho)}\right\Vert _{L_{\rho}^{p}([4r,r+t])}\lesssim\Vert g\Vert_{B_{p,1}^{1/p}(M)}r\left\Vert \frac{1}{\rho}\right\Vert _{L_{\rho}^{p}([4r,\infty))}\lesssim r^{\frac{1}{p}}\Vert g\Vert_{B_{p,1}^{1/p}(M)}
\]
Then we are done (using $p<\infty$):
\[
r^{\frac{1}{p}}\left\Vert t^{-\frac{1}{p}}\right\Vert _{L^{1}(\frac{\text{d}t}{t},(3r,1))}=\left\Vert \left(\frac{r}{t}\right)^{\frac{1}{p}}\right\Vert _{L^{1}(\frac{\text{d}t}{t},(3r,1))}=\left\Vert \left(\frac{1}{t}\right)^{\frac{1}{p}}\right\Vert _{L^{1}(\frac{\text{d}t}{t},(3,\frac{1}{r}))}\lesssim_{\neg r}1
\]
\end{proof}

\section{Hodge theory}

We stick closely to the terminology and symbols of \parencite{Schwarz1995},
with some careful exceptions.

\subsection{The setting\label{subsec:The-setting}}
\begin{defn}
Define a \textbf{$\partial$-manifold} as a paracompact, Hausdorff,
metric-complete, oriented, smooth manifold, with no or smooth boundary.
\end{defn}

Note that this means $B_{\mathbb{R}^{n}}(0,1)$ is not a $\partial$-manifold
(as it is not complete), but $\overline{B_{\mathbb{R}^{n}}(0,1)}$
is.

\emph{For the rest of this paper}, unless mentioned otherwise, we
work on $M$ which is a compact Riemannian $n$-dimensional $\partial$-manifold
(where $n\geq2$), and use $\nu$ \nomenclature{$\nu$}{outwards unit normal vector field on $\partial M$\nomrefpage}
to denote the outwards unit normal vector field on $\partial M$.

As before, define $M_{>r}=\{x\in M:\text{dist}(x,\partial M)>r\}$,
and similarly for $M_{\geq r},M_{<r},M_{[r_{1},r_{2}]}$ etc.

For $r>0$ small, the map $\left(\partial M\times[0,r)\to M_{<r},(x,t)\mapsto\exp_{x}(-t\nu)\right)$
is a diffeomorphism, which we call a \textbf{Riemannian collar}. Then
$\nu$ can be extended via geodesics to a smooth vector field $\widetilde{\nu}$
which is of unit length near the boundary (cut off at some point away
from the boundary, but we only care about the area near the boundary).
\nomenclature{$\widetilde{\nu}$}{extension of $\nu$ near $\partial M$\nomrefpage}

Let $\vol$ stand for the Riemannian volume form orienting $M$ and
$\vol_{\partial}$ for that of $\partial M$. Let $\jmath:\partial M\hookrightarrow M$
be the smooth inclusion map \nomenclature{$\jmath$}{$\jmath:\partial M\hookrightarrow M$  is the smooth inclusion map\nomrefpage}
and $\iota$ stand for interior product (contraction) of differential
forms.\nomenclature{$\iota$}{interior product (contraction) of differential forms\nomrefpage}
Note that for a smooth differential form $\omega$, $\jmath^{*}\omega$
only depends on $\restr{\omega}{\partial M}$, so by abuse of notation,
we can write
\[
\vol_{\partial}=\jmath^{*}(\iota_{\nu}\vol)
\]
\nomenclature{$\vol_\partial$}{volume form of $\partial M$\nomrefpage}where
$\iota_{\nu}\vol\in\restr{\Omega^{n-1}\left(M\right)}{\partial M}$.
Additionally, the Stokes theorem reads $\int_{M}d\omega=\int_{\partial M}\jmath^{*}\omega$
for $\omega\in\Omega^{n-1}(M)$.

\subsubsection{Vector bundles\label{subsec:Vector-bundles}}

Let $\mathbb{F}$ be a real vector bundle over $M$ with a Riemannian
fiber metric $\left\langle \cdot,\cdot\right\rangle _{\mathbb{F}}$.

Define
\begin{itemize}
\item $\Gamma(\mathbb{F})$ : the space of smooth sections of $\mathbb{F}$
\item $\Gamma_{c}(\mathbb{F})$ : smooth sections with compact support (so
$\Gamma_{c}(\mathbb{F})=\Gamma(\mathbb{F})$ since $M$ is compact)
\item $\Gamma_{00}(\mathbb{F}$) : smooth sections with compact support
in $\accentset{\circ}{M}$ (the interior of $M$).\nomenclature{$\Gamma(\mathbb{F}),\;\Gamma_{c}(\mathbb{F}),\;\Gamma_{00}(\mathbb{F})$}{the space of smooth sections of $\mathbb{F}$ with different support conditions\nomrefpage}
\end{itemize}
\begin{rem*}
We are following \parencite{Schwarz1995}, where Hodge theory is also
formulated for non-compact $M$. In the book, $\Gamma_{0}\mathbb{F}$
is used instead of $\Gamma_{00}\mathbb{F}$ to denote compact support
in $\accentset{\circ}{M}$. As that can be confused with having zero
trace, we opt to write $\Gamma_{00}\mathbb{F}$ instead.
\end{rem*}
Then on $\Gamma_{c}(\mathbb{F}$), define the dot product

\[
\left\langle \left\langle \sigma,\theta\right\rangle \right\rangle =\int_{M}\left\langle \sigma,\theta\right\rangle _{\mathbb{F}}\vol
\]
\nomenclature{$\left\langle \left\langle \sigma,\theta\right\rangle \right\rangle$}{dot product on $\Gamma(\mathbb{F})$\nomrefpage}

and $\left|\sigma\right|_{\mathbb{F}}=\sqrt{\left\langle \sigma,\sigma\right\rangle _{\mathbb{F}}}.$
Then for $p\in[1,\infty)$, $L^{p}\Gamma(\mathbb{F})$ is the completion
of $\Gamma_{c}(\mathbb{F})$ under the norm 
\[
\left\Vert \sigma\right\Vert _{L^{p}\Gamma(\mathbb{F})}=\left\Vert \left|\sigma\right|_{\mathbb{F}}\right\Vert _{L^{p}(M)}
\]

Let $\nabla^{\mathbb{F}}$ be a connection on $\mathbb{F}$. Then
for $\sigma\in\Gamma(\mathbb{F}),\nabla^{\mathbb{F}}\sigma\in\Gamma(T^{\ast}M\otimes\mathbb{F})$
and we can define the fiber metric
\[
\left\langle \alpha\otimes\sigma,\beta\otimes\theta\right\rangle _{T^{*}M\otimes\mathbb{F}}=\left\langle \alpha,\beta\right\rangle _{T^{*}M}\left\langle \sigma,\theta\right\rangle _{\mathbb{F}}
\]

In local coordinates (Einstein notation):
\[
\left\langle \nabla^{\mathbb{F}}\sigma,\nabla^{\mathbb{F}}\theta\right\rangle _{T^{*}M\otimes\mathbb{F}}=\left\langle dx^{i}\otimes\nabla_{i}^{\mathbb{F}}\sigma,dx^{j}\otimes\nabla_{j}^{\mathbb{F}}\sigma\right\rangle _{\mathbb{F}}=\left\langle dx^{i},dx^{j}\right\rangle _{T^{*}M}\left\langle \nabla_{i}^{\mathbb{F}}\sigma,\nabla_{j}^{\mathbb{F}}\theta\right\rangle _{\mathbb{F}}=g^{ij}\left\langle \nabla_{i}^{\mathbb{F}}\sigma,\nabla_{j}^{\mathbb{F}}\theta\right\rangle _{\mathbb{F}}
\]

For higher derivatives, define the $k$-jet fiber metric
\[
\left\langle \sigma,\theta\right\rangle _{J^{k}\mathbb{F}}=\sum_{0\leq j\leq k}\left\langle \left(\nabla^{\mathbb{F}}\right)^{(j)}\sigma,\left(\nabla^{\mathbb{F}}\right)^{(j)}\theta\right\rangle _{\left(\bigotimes^{j}T^{*}M\right)\otimes\mathbb{F}}
\]

and $\left|\sigma\right|_{J^{k}\mathbb{F}}=\sqrt{\left\langle \sigma,\sigma\right\rangle _{J^{k}\mathbb{F}}}.$
Then we have Cauchy-Schwarz: $\left|\left\langle \sigma,\theta\right\rangle _{J^{k}\mathbb{F}}\right|\leq\left|\sigma\right|_{J^{k}\mathbb{F}}\left|\theta\right|_{J^{k}\mathbb{F}}.$

Then for $m\in\mathbb{N}_{0},p\in[1,\infty),$ we define the Sobolev
space $W^{m,p}\Gamma(\mathbb{F})$ as the completion of $\Gamma_{c}(\mathbb{F})$
under the norm 
\[
\left\Vert \sigma\right\Vert _{W^{m,p}\Gamma\left(\mathbb{F}\right)}=\left\Vert \left|\sigma\right|_{J^{m}\mathbb{F}}\right\Vert _{L^{p}(M)}
\]

It is worth noting that $\left|\sigma\right|_{J^{m}\mathbb{F}}$,
up to some constants, does not depend on $\nabla^{\mathbb{F}}$. Indeed,
assume there is another connection $\widetilde{\nabla}^{\mathbb{F}}$,
then $\nabla^{\mathbb{F}}-\widetilde{\nabla}^{\mathbb{F}}$ is tensorial:
\[
\left(\nabla_{X}^{\mathbb{F}}-\widetilde{\nabla}_{X}^{\mathbb{F}}\right)(f\sigma)=f\left(\nabla_{X}^{\mathbb{F}}-\widetilde{\nabla}_{X}^{\mathbb{F}}\right)(\sigma)=\left(\nabla_{fX}^{\mathbb{F}}-\widetilde{\nabla}_{fX}^{\mathbb{F}}\right)(\sigma)\;\text{for }f\in C^{\infty}(M),\sigma\in\Gamma(\mathbb{F}),X\in\mathfrak{X}M
\]

So there is a $C^{\infty}(M)$-multilinear map $A:\mathfrak{X}M\otimes_{C^{\infty}(M)}\Gamma(\mathbb{F})\to\Gamma(\mathbb{F})$
such that $\left(\nabla_{X}^{\mathbb{F}}-\widetilde{\nabla}_{X}^{\mathbb{F}}\right)(\sigma)=A(X,\sigma)$.
By the compactness of $M$ and the boundedness of $A$, we conclude
$\left|\sigma\right|_{J^{m}\mathbb{F},\nabla^{\mathbb{F}}}\sim\left|\sigma\right|_{J^{m}\mathbb{F},\widetilde{\nabla}^{\mathbb{F}}}$.
Therefore the topology of $W^{m,p}\Gamma(\mathbb{F})$ is uniquely
defined.
\begin{defn}[Distributions]
 Set $\mathscr{D}\Gamma\left(\mathbb{F}\right)=\Gamma_{00}\left(\mathbb{F}\right)$
as the space of \textbf{test sections} and $\mathscr{D}'\Gamma\left(\mathbb{F}\right)=\left(\mathscr{D}\Gamma\left(\mathbb{F}\right)\right)^{*}$
the space of \textbf{distributional sections}. As usual, in the category
of locally convex TVS, $\mathscr{D}\Gamma\left(\mathbb{F}\right)$
is given \textbf{Schwartz's topology} as the colimit of $\{\Gamma\left(\mathbb{F}\right)_{K}:K\subset\accentset{\circ}{M}\text{ compact}\}$,
where $\Gamma\left(\mathbb{F}\right)_{K}:=\{\sigma\in\Gamma\left(\mathbb{F}\right):\supp\sigma\subset K\}$
has the \textbf{Frechet} $C^{\infty}$ topology.
\end{defn}

\subsubsection{Compatibility with scalar function spaces\label{subsec:Compatibility-with-scalar}}

We aim to show that the global definitions of Sobolev spaces in \Subsecref{Vector-bundles}
are compatible with the definitions of Sobolev spaces by local coordinates.

Let $\left(\psi_{\alpha},U_{\alpha}\right)_{\alpha}$ be a finite
partition of unity, where $U_{\alpha}$ is open in $M$ and $\psi_{\alpha}$
is supported in $U_{\alpha}$. Normally in differential geometry,
$U_{\alpha}$ is diffeomorphic to either $\overline{\mathbb{R}_{+}^{n}}\cap B_{\mathbb{R}^{n}}(0,1)$
or $B_{\mathbb{R}^{n}}(0,1).$ However, it is problematic that the
half-ball does not have $C^{\infty}$ boundary, so we use some piecewise-linear
functions and mollification to create a bounded $C^{\infty}$ domain.

\begin{figure}[H]
\begin{centering}
\includegraphics{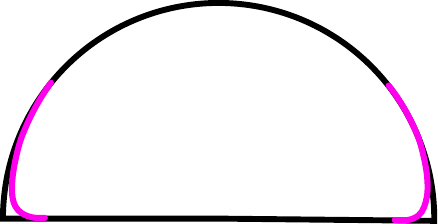}
\par\end{centering}
\caption{Smoothing the corners}

\end{figure}

So WLOG, $\overline{U_{\alpha}}$ is diffeomorphic to the closure
of a bounded $C^{\infty}$ domain in $\mathbb{R}^{n}$, and scalar
function spaces are well-defined on $U_{\alpha}$ (recall \Defref{domain}).
Note that $\supp\psi_{\alpha}$ might intersect with $\partial M$.

For $U_{\alpha}$ chosen small enough, the bundle $\mathbb{F}$ on
$U_{\alpha}$ is diffeomorphic to $U_{\alpha}\times F$ (where $F$
is the typical fiber of $\mathbb{F}$).

Let $\left(e_{\beta}^{\alpha}\right)_{\beta}$ be the coordinate sections
on $\supp\psi_{\alpha}$, and cut off such that $\supp\psi_{\alpha}\subset\accentset{\circ}{\left(\supp e_{\beta}^{\alpha}\right)}\subset\supp e_{\beta}^{\alpha}\subset U_{\alpha}$.
Let $\sigma\in\Gamma\left(\mathbb{F}\right).$ Then there exist $c_{\beta}^{\alpha}(\sigma)\in C_{c}^{\infty}(U_{\alpha})$
such that $\supp c_{\beta}^{\alpha}(\sigma)\subset\supp\psi_{\alpha}$,
$\psi_{\alpha}\sigma=\sum_{\beta}c_{\beta}^{\alpha}(\sigma)e_{\beta}^{\alpha}$
and
\[
\sigma=\sum_{\alpha,\beta}c_{\beta}^{\alpha}(\sigma)e_{\beta}^{\alpha}
\]

Now, observe that $\left|\sigma\right|_{\mathbb{F}}\sim\sum_{\alpha}\left|\psi_{\alpha}\sigma\right|_{\mathbb{F}}$
and 
\[
\left|\psi_{\alpha}\sigma\right|_{\mathbb{F}}=\left(\sum_{\beta,\beta'}c_{\beta}^{\alpha}(\sigma)c_{\beta'}^{\alpha}(\sigma)\left\langle e_{\beta}^{\alpha},e_{\beta'}^{\alpha}\right\rangle _{\mathbb{F}}\right)^{\frac{1}{2}}\sim\left(\sum_{\beta}\left|c_{\beta}^{\alpha}(\sigma)\right|^{2}\right)^{\frac{1}{2}}
\]
To see this, let $x\in\supp\psi_{\alpha}$ and $\left\langle e_{\beta}^{\alpha},e_{\beta'}^{\alpha}\right\rangle _{\mathbb{F}}(x)=B_{\beta\beta'}(x)$.
Then $B_{x}(u,v):=\sum_{\beta,\beta'}u_{\beta}v_{\beta'}B_{\beta\beta'}(x)$
is a positive-definite inner product, which induces a norm on a finite-dimensional
vector space, where all norms are equivalent. Then simply note $B_{x}(u,u)$
is continuous in variable $x\in\supp\psi_{\alpha}$.

Also, in local coordinates, there are $s_{i\beta}^{\gamma}\in C_{c}^{\infty}(U_{\alpha})$
such that $\nabla_{i}^{\mathbb{F}}e_{\beta}^{\alpha}=\sum_{\gamma}s_{i\beta}^{\gamma}e_{\gamma}^{\alpha}$
on $\supp\psi_{\alpha}.$ Then 
\[
\nabla_{i}^{\mathbb{F}}(\psi_{\alpha}\sigma)=\sum_{\beta}\partial_{i}c_{\beta}^{\alpha}(\sigma)e_{\beta}^{\alpha}+\sum_{\beta,\gamma}c_{\beta}^{\alpha}(\sigma)s_{i\beta}^{\gamma}e_{\gamma}^{\alpha}=\sum_{\beta}d_{i\beta}^{\alpha}(\sigma)e_{\beta}^{\alpha}\quad\text{where }d_{i\beta}^{\alpha}(\sigma)=\partial_{i}c_{\beta}^{\alpha}(\sigma)+\sum_{\gamma}c_{\gamma}^{\alpha}(\sigma)s_{i\gamma}^{\beta}
\]
So $\left|\sigma\right|_{J^{1}\mathbb{F}}\sim\sum_{\alpha,\beta}|c_{\beta}^{\alpha}(\sigma)|+\sum_{\alpha,\beta,i}|d_{i\beta}^{\alpha}(\sigma)|\sim\sum_{\alpha,\beta}|c_{\beta}^{\alpha}(\sigma)|+\sum_{\alpha,\beta,i}|\partial_{i}c_{\beta}^{\alpha}(\sigma)|$.

Similarly $\left|\sigma\right|_{J^{m}\mathbb{F}}\sim\sum_{\alpha,\beta}\sum_{k\leq m}\left|\nabla^{(k)}c_{\beta}^{\alpha}(\sigma)\right|$.

So for $m\in\mathbb{N}_{0},p\in[1,\infty),$
\[
\left\Vert \sigma\right\Vert _{W^{m,p}}\sim\sum_{\alpha,\beta}\left\Vert c_{\beta}^{\alpha}(\sigma)\right\Vert _{W^{m,p}(U_{\alpha},\mathbb{R})}
\]

Now define $S\sigma=\left(c_{\beta}^{\alpha}(\sigma)\right)_{\alpha,\beta}$
and $R\left(c_{\beta}^{\alpha}\right)_{\alpha,\beta}=\sum_{\alpha,\beta}c_{\beta}^{\alpha}e_{\beta}^{\alpha}$.
Then $RS=1$ on $\Gamma\mathbb{\left(F\right)}$ and $P:=SR$ is a
projection on $\prod_{\alpha,\beta}C^{\infty}(\overline{U_{\alpha}})$.
Note that $P$ depends on the choice of partition of unity. By looking
into the definitions of $R$ and $S$, we can extend this to have
$P=SR$ as a continuous projection on $\prod_{\alpha,\beta}L^{1}(U_{\alpha})$
and 
\[
\left\Vert P\left(c_{\beta}^{\alpha}\right)_{\alpha,\beta}\right\Vert _{\prod_{\alpha,\beta}W^{m,p}(U_{\alpha})}\lesssim\sum_{\alpha,\beta}\left\Vert c_{\beta}^{\alpha}\right\Vert _{W^{m,p}(U_{\alpha})}\text{ for }m\in\mathbb{N}_{0},p\in\left[1,\infty\right],c_{\beta}^{\alpha}\in W^{m,p}(U_{\alpha})
\]

The keen reader should have noticed we never mentioned the case $p=\infty$
in \Subsecref{Vector-bundles} as we defined $W^{m,p}\Gamma\mathbb{\left(F\right)}$
by the completion of smooth sections, and $C^{\infty}(M)$ is not
dense in $W^{m,\infty}(M)$. Now, however, by using local coordinates,
we are justified in defining $W^{m,p}\Gamma\mathbb{\left(F\right)}=\{\sum_{\alpha,\beta}c_{\beta}^{\alpha}e_{\beta}^{\alpha}:c_{\beta}^{\alpha}\in W^{m,p}(U_{\alpha})\}$
for $m\in\mathbb{N}_{0},p\in[1,\infty]$ with the norm defined (up
to equivalent norms) as
\[
\left\Vert \sum_{\alpha,\beta}c_{\beta}^{\alpha}e_{\beta}^{\alpha}\right\Vert _{W^{m,p}\Gamma\mathbb{F}}:=\left\Vert S\sum_{\alpha,\beta}c_{\beta}^{\alpha}e_{\beta}^{\alpha}\right\Vert _{\prod_{\alpha,\beta}W^{m,p}(U_{\alpha})}
\]

Then $B_{p,q}^{s}\Gamma\left(\mathbb{F}\right)$ and $F_{p,q}^{s}\Gamma\left(\mathbb{F}\right)$
can be defined similarly. In other words, for $m\in\mathbb{N}_{0},p\in[1,\infty],q\in[1,\infty],s\geq0$:
\begin{itemize}
\item $W^{m,p}\Gamma\left(\mathbb{F}\right)\simeq P\prod_{\alpha,\beta}W^{m,p}(U_{\alpha})$
\item $B_{p,q}^{s}\Gamma\left(\mathbb{F}\right)\simeq P\prod_{\alpha,\beta}B_{p,q}^{s}(U_{\alpha})$
\item $F_{p,q}^{s}\Gamma\left(\mathbb{F}\right)\simeq P\prod_{\alpha,\beta}F_{p,q}^{s}(U_{\alpha}),\;p\neq\infty$
\end{itemize}
By using \Blackboxref{diffeomorphism-multiplier}, we can show the
Banach topologies of these spaces are uniquely defined (independent
of the choices of $\psi_{\alpha},U_{\alpha}$). For convenience (such
as working with Holder's inequality), we still use the Sobolev norms
$W^{m,p}$ ($m\in\mathbb{N}_{0},p\in[1,\infty)$) defined globally
in \Subsecref{Vector-bundles}.

All theorems from \secref{Scalar-function-spaces} that worked on
bounded $C^{\infty}$ domains carry over to our setting on $M$, \emph{mutatis
mutandis}. For instance, $B_{3,1}^{\frac{1}{3}}\Gamma(\mathbb{F})\twoheadrightarrow L^{3}\left.\Gamma(\mathbb{F})\right|_{\partial M}$
is a continuous surjection and
\[
B_{3,1}^{\frac{1}{3}}\Gamma(\mathbb{F})=\left(L^{3}\Gamma(\mathbb{F}),W^{1,3}\Gamma(\mathbb{F})\right)_{\frac{1}{3},1}
\]

Moreover, for $p\in(1,\infty)$, $L^{p}\Gamma(\mathbb{F})$ is reflexive.
By Holder's inequality, $\left(L^{p}\Gamma\left(\mathbb{F}\right)\right)^{*}=L^{p'}\Gamma\left(\mathbb{F}\right)$
for $p\in(1,\infty)$.

\subsubsection{Complexification issue}

A small step which we omitted is complexification. As $\mathbb{F}$
is a real vector bundle, the previous definitions only give $W^{m,p}\Gamma\left(\mathbb{F}\right)\simeq P\prod_{\alpha,\beta}W^{m,p}(U_{\alpha},\mathbb{R})$
for $m\in\mathbb{N}_{0},p\in[1,\infty]$. In working with real manifolds,
differential forms/tensors and their dot products, we always assume
real-valued coefficients for sections, but whenever we need to use
theorems involving complex Banach spaces or the theory of function
spaces, we assume an implicit complexification step. Fortunately,
no complications arise from complexification (see \Secref{Complexification}
for the full reasoning), so for the rest of the paper we can ignore
this detail. When we want to be explicit, we will specify the scalars
we are using, e.g. $\mathbb{R}W^{m,p}\Gamma(\mathbb{F})$ versus $\mathbb{C}W^{m,p}\Gamma(\mathbb{F})$.\nomenclature{$\mathbb{R}W^{m,p},\;\mathbb{C}W^{m,p}$}{real and complexified versions of function space\nomrefpage}

\subsection{Differential forms \& boundary\label{subsec:Differential-forms-=000026}}

Unless mentioned otherwise, the metric is the Riemannian metric, and
the connection is the \textbf{Levi-Civita connection}.

For $X\in\mathfrak{X}M,$ define $\mathbf{n}X=\left\langle X,\nu\right\rangle \nu\in\left.\mathfrak{X}M\right|_{\partial M}$
(the \textbf{normal part}) and $\mathbf{t}X=\left.X\right|_{\partial M}-\mathbf{n}X$
(the \textbf{tangential part}). We note that $\mathbf{t}X$ and $\mathbf{n}X$
only depend on $\restr{X}{\partial M}$, so $\mathbf{t}$ and $\mathbf{n}$
can be defined on $\restr{\mathfrak{X}M}{\partial M}$, and by abuse
of notation, $\mathbf{t}\left(\left.\mathfrak{X}M\right|_{\partial M}\right)\iso\mathfrak{X}(\partial M)$.
\nomenclature{$\mathfrak{X}M$}{set of smooth vector fields on $M$\nomrefpage}
\nomenclature{$\mathbf{t}$}{tangential part\nomrefpage}\nomenclature{$\mathbf{n}$}{normal part\nomrefpage}

For $\omega\in\Omega^{k}\left(M\right),$ define $\mathbf{t}\omega$
and $\mathbf{n}\omega$ by 
\[
\mathbf{t}\omega(X_{1},...,X_{k}):=\omega(\mathbf{t}X_{1},...,\mathbf{t}X_{k})\;\;\forall X_{j}\in\mathfrak{X}M,j=1,...,k
\]
and $\mathbf{n}\omega=\left.\omega\right|_{\partial M}-\mathbf{t}\omega$.
\nomenclature{$\Omega^{k}\left(M\right)$}{set of smooth differential forms on $M$\nomrefpage}
By abuse of notation, we similarly observe that $\mathbf{t}\left(\left.\Omega^{k}\left(M\right)\right|_{\partial M}\right)\iso\Omega^{k}(\partial M)=\jmath^{*}\left(\left.\Omega^{k}\left(M\right)\right|_{\partial M}\right)=\jmath^{*}\left(\Omega^{k}\left(M\right)\right)$.

Recall the \textbf{musical isomorphism}: $X_{p}^{\flat}(Y_{p})=\left\langle X_{p},Y_{p}\right\rangle $
and $\left\langle \omega_{p}^{\sharp},Y_{p}\right\rangle =\omega_{p}(Y_{p})$
for $p\in M,\omega_{p}\in T_{p}^{*}M,X_{p}\in T_{p}M,Y_{p}\in T_{p}M$.\nomenclature{$X_{p}^{\flat},\;\omega_{p}^{\sharp}$}{musical isomorphism\nomrefpage}

Recall the usual \textbf{Hodge star operator }$\star:\Omega^{k}(M)\iso\Omega^{n-k}(M)$,
\textbf{exterior derivative} $d:\Omega^{k}(M)\to\Omega^{k+1}(M)$,
\textbf{codifferential} $\delta:\Omega^{k}(M)\to\Omega^{k-1}(M)$,
and \textbf{Hodge Laplacian} $\Delta=-\left(d\delta+\delta d\right)$
(cf. \parencite[Section 2.10]{Taylor_PDE1} and \parencite[Definition 1.2.2]{Schwarz1995})\nomenclature{$\star$}{Hodge star\nomrefpage}\nomenclature{$\delta$}{codifferential\nomrefpage}\nomenclature{$\Delta$}{Hodge Laplacian\nomrefpage}.

We will often use \textbf{Penrose abstract index notation} (cf. \parencite[Section 2.4]{wald1984general}),
which should not be confused with the similar-looking \textbf{Einstein
notation} for local coordinates, or the similar-sounding \textbf{Penrose
graphical notation}. In Penrose notation, we collect the usual identities
in differential geometry (cf. \parencite{jeffrey_lee2009manifolds}):
\begin{itemize}
\item For any tensor $T_{a_{1}...a_{k}}$, define $\left(\nabla T\right)_{ia_{1}...a_{k}}=\nabla_{i}T_{a_{1}...a_{k}}$
and $\div T=\nabla^{i}T_{ia_{2}...a_{k}}.$
\item $\left(d\omega\right)_{ba_{1}...a_{k}}=\left(k+1\right)\widetilde{\nabla}_{[b}\omega_{a_{1}...a_{k}]}\;\forall\omega\in\Omega^{k}(M)$
where $\widetilde{\nabla}$ is any torsion-free connection.
\item $\left(\delta\omega\right)_{a_{1}...a_{k-1}}=-\nabla^{b}\omega_{ba_{1}...a_{k-1}}=-(\div w)_{a_{1}...a_{k-1}}\forall\omega\in\Omega^{k}(M)$
\item $\left(\nabla_{a}\nabla_{b}-\nabla_{b}\nabla_{a}\right)T^{ij}{}_{kl}=-R_{ab\sigma}{}^{i}T^{\sigma j}{}_{kl}-R_{ab\sigma}{}^{j}T^{i\sigma}{}_{kl}+R_{abk}{}^{\sigma}T^{ij}{}_{\sigma l}+R_{abl}{}^{\sigma}T^{ij}{}_{k\sigma}$
for any tensor $T^{ij}{}_{kl}$, where $R$ is the \textbf{Riemann
curvature tensor }and $\nabla$ the Levi-Civita connection.\nomenclature{$R_{abcd}$}{Riemann curvature tensor\nomrefpage}
Similar identities hold for other types of tensors. When we do not
care about the exact indices and how they contract, we can just write
the \textbf{schematic identity }$\left(\nabla_{a}\nabla_{b}-\nabla_{b}\nabla_{a}\right)T^{ij}{}_{kl}=R*T.$
As $R$ is bounded on compact $M$, interchanging derivatives is a
zeroth-order operation on $M$.
\item For tensor $T_{a_{1}...a_{k}}$, define the \textbf{Weitzenbock curvature
operator} \textbf{\nomenclature{$\Ric$}{Weitzenbock curvature operator\nomrefpage}}
\[
\Ric(T)_{a_{1}...a_{k}}=2\sum_{j=1}^{k}\nabla_{[i}\nabla_{a_{j}]}T_{a_{1}...a_{j-1}}{}^{i}{}_{a_{j+1}...a_{k}}=\sum_{j}R_{a_{j}}{}^{\sigma}T_{a_{1}...a_{j-1}\sigma a_{j+1}...a_{k}}-\sum_{j\neq l}R_{a_{j}}{}^{\mu}{}_{a_{l}}{}^{\sigma}T_{a_{1}...\sigma...\mu...a_{k}}
\]
where $R_{ab}=R_{a\sigma b}{}^{\sigma}$ is the \textbf{Ricci tensor}.
The invariant form is 
\[
\Ric(T)(X_{1},...X_{k})=\sum_{a}\left(R(\partial_{i},X_{a})T\right)(X_{1},...,X_{a-1},\partial^{i},X_{a+1},...,X_{k})\;\forall X_{j}\in\mathfrak{X}M
\]
where $\partial^{i}=g^{ij}\partial_{j}$ and $R(\partial_{i},\partial_{j})=\nabla_{i}\nabla_{j}-\nabla_{j}\nabla_{i}$
(Penrose notation). Note that $\left\langle R(\partial_{a},\partial_{b})\partial_{d},\partial_{c}\right\rangle =R_{abcd}$.
Special cases include $\Ric(f)=0\;\forall f\in C^{\infty}(M)$ and
$\Ric(X)_{a}=R_{a}{}^{\sigma}X_{\sigma}\;\forall X\in\mathfrak{X}M$
(justifying the notation $\Ric$).\\
In local coordinates 
\[
\Ric\left(\omega\right)=dx^{j}\wedge\left(R(\partial_{i},\partial_{j})\omega\cdot\partial^{i}\right)\;\forall\omega\in\Omega^{k}(M),
\]
where $\cdot$ stands for contraction (interior product). Then we
have the \textbf{Weitzenbock formula}:
\[
\Delta\omega=\nabla_{i}\nabla^{i}\omega-\Ric(\omega)\;\forall\omega\in\Omega^{k}(M)
\]
where $\nabla_{i}\nabla^{i}\omega=\tr(\nabla^{2}\omega)$ is also
called the \textbf{connection Laplacian}, which differs from the Hodge
Laplacian by a zeroth-order term. The geometry of $M$ and differential
forms are more easily handled by the Hodge Laplacian, while the connection
Laplacian is more useful in calculations with tensors and the Penrose
notation.
\item For tensors $T_{a_{1}...a_{k}}$ and $Q_{a_{1}...a_{k}}$, the \textbf{tensor
inner product} is $\left\langle T,Q\right\rangle =T_{a_{1}...a_{k}}Q^{a_{1}...a_{k}}$.\nomenclature{$\left\langle T,Q\right\rangle$}{tensor inner product\nomrefpage}
But for $\omega,\eta\in\Omega^{k}(M),$ there is another dot product,
called the \textbf{Hodge inner product}, where 
\[
\left\langle \omega,\eta\right\rangle _{\Lambda}=\frac{1}{k!}\left\langle \omega,\eta\right\rangle 
\]
 So $\left|\omega\right|_{\Lambda}=\sqrt{\frac{1}{k!}}\left|\omega\right|.$
Then we define $\left\langle \left\langle \omega,\eta\right\rangle \right\rangle =\int_{M}\left\langle \omega,\eta\right\rangle \vol$
and $\left\langle \left\langle \omega,\eta\right\rangle \right\rangle _{\Lambda}=\int_{M}\left\langle \omega,\eta\right\rangle _{\Lambda}\vol$\nomenclature{$\left\langle \omega,\eta\right\rangle _{\Lambda},\left\langle \left\langle \omega,\;\eta\right\rangle \right\rangle _{\Lambda}$}{Hodge inner product\nomrefpage}.
Recall that $\omega\wedge\star\eta=\left\langle \omega,\eta\right\rangle _{\Lambda}\vol$
$\forall\omega\in\Omega^{k}(M),\forall\eta\in\Omega^{k}(M).$ Also
\[
\left\langle \left\langle d\omega,\eta\right\rangle \right\rangle _{\Lambda}=\left\langle \left\langle \omega,\delta\eta\right\rangle \right\rangle _{\Lambda}\;\forall\omega\in\Omega_{00}^{k}(M),\forall\eta\in\Omega_{00}^{k+1}(M)
\]
So $\left\langle \cdot,\cdot\right\rangle _{\Lambda}$ is more convenient
for integration by parts and the Hodge star. Nevertheless, as they
only differ up to a constant factor, we can still define $W^{m,p}\Omega^{k}(M)$
($m\in\mathbb{N}_{0},p\in[1,\infty))$ by $\left\langle \cdot,\cdot\right\rangle $
as in \Subsecref{The-setting}. Finally, by the Weitzenbock formula
and Penrose notation, we easily get the \textbf{Bochner formula}:
\[
\frac{1}{2}\Delta\left(|\omega|^{2}\right)=\frac{1}{2}\nabla_{i}\nabla^{i}\left(\left\langle \omega,\omega\right\rangle \right)=\left\langle \Delta\omega,\omega\right\rangle +|\nabla\omega|^{2}+\left\langle \Ric\left(\omega\right),\omega\right\rangle 
\]
\end{itemize}
\begin{rem*}
In \parencite{Schwarz1995}, the conventions are a bit different,
with $\Delta=\left(d\delta+\delta d\right),\Delta^{\Lambda}=-\nabla_{i}\nabla^{i},R^{W}=-\Ric$
and $\mathcal{N}$ the inwards unit normal vector field. Also the
difference between $\left\langle \left\langle \cdot,\cdot\right\rangle \right\rangle $
and $\left\langle \left\langle \cdot,\cdot\right\rangle \right\rangle _{\Lambda}$
is not made explicit in the book. We will not use such notation.
\end{rem*}
\begin{lem}
Some basic identities: \label{lem:basic-identities-forms-boundary}
\begin{enumerate}
\item $\forall\omega\in\Omega^{k}\left(M\right):\mathbf{t}\omega=0\iff\jmath^{*}\omega=0$
. Similarly, $\mathbf{n}\omega=0\iff\iota_{\nu}\omega=0$.
\item $\left(\mathbf{t}X\right)^{\flat}=\mathbf{t}(X^{\flat})\;\forall X\in\mathfrak{X}M$
\item $\jmath^{*}\mathbf{t}\omega=\jmath^{*}\omega$, $\mathbf{t}\omega=\iota_{\nu}(\nu^{\flat}\wedge\omega)$,
$\mathbf{n}\omega=\nu^{\flat}\wedge\iota_{\nu}\omega,$ $\mathbf{t}(\omega\wedge\eta)=\mathbf{t}\omega\wedge\mathbf{t}\eta$
$\forall\omega\in\Omega^{k}(M),\forall\eta\in\Omega^{l}(M)$
\item $\left\langle \left\langle \mathbf{t}\omega,\eta\right\rangle \right\rangle _{\Lambda}=\left\langle \left\langle \mathbf{t}\omega,\mathbf{t}\eta\right\rangle \right\rangle _{\Lambda}=\left\langle \left\langle \omega,\mathbf{t}\eta\right\rangle \right\rangle _{\Lambda}$
$\forall\omega,\eta\in\Omega^{k}(M)$
\item $\mathbf{t}\left(\star\omega\right)=\star\left(\mathbf{n}\omega\right)$,
$\mathbf{n}\left(\star\omega\right)=\star\left(\mathbf{t}\omega\right)$,
$\star d\omega=\left(-1\right)^{k+1}\delta\star\omega$, $\star\delta\omega=\left(-1\right)^{k}d\star\omega,$
$\star\Delta\omega=\Delta\star\omega$ $\forall\omega\in\Omega^{k}(M)$
\item $\jmath^{*}\mathbf{t}d\omega=\jmath^{*}d\omega=d^{\partial M}\jmath^{*}\omega=d^{\partial M}\jmath^{*}\mathbf{t}\omega$
$\forall\omega\in\Omega^{k}(M)$
\item Let $\omega\in\Omega^{k}(M)$. If $\mathbf{t}\omega=0$ then $\mathbf{t}d\omega=0$.
If $\mathbf{n}\omega=0$ then $\mathbf{n}\delta\omega=0$.
\item $\iota_{\nu}\omega=\mathbf{t}\left(\iota_{\nu}\omega\right)=\iota_{\nu}\mathbf{n}\omega$
$\forall\omega\in\Omega^{k}(M)$
\item $\jmath^{*}\left(\omega\wedge\star\eta\right)=\left\langle \jmath^{*}\omega,\jmath^{*}\iota_{\nu}\eta\right\rangle _{\Lambda}\vol_{\partial}$
$\forall\omega\in\Omega^{k}(M),\forall\eta\in\Omega^{k+1}(M)$
\end{enumerate}
\end{lem}

\begin{proof}
We will only prove the last assertion. Observe that $\jmath^{*}\left(\vol\right)=0$
so $\vol|_{\partial M}=\mathbf{n}\vol=\nu^{\flat}\wedge\iota_{\nu}\vol$.
Recall $\vol_{\partial}=\jmath^{*}(\iota_{\nu}\vol)$ and $\mathbf{t}\Omega^{k}\iso\jmath^{*}\Omega^{k}$,
so the problem is equivalent to proving
\[
\nu^{\flat}\wedge\mathbf{t}\omega\wedge\mathbf{t}\star\eta=\left\langle \mathbf{t}\omega,\mathbf{t}\iota_{\nu}\eta\right\rangle _{\Lambda}\mathrm{\mathrm{vol}}\text{ on }\partial M
\]
Simply observe that $\mathbf{t}\iota_{\nu}\eta=\iota_{\nu}\eta$ and
\[
\nu^{\flat}\wedge\mathbf{t}\omega\wedge\mathbf{t}\left(\star\eta\right)=\nu^{\flat}\wedge\mathbf{t}\omega\wedge\star\mathbf{n}\eta=\left\langle \nu^{\flat}\wedge\mathbf{t}\omega,\mathbf{n}\eta\right\rangle _{\Lambda}\mathrm{vol}=\left\langle \nu^{\flat}\wedge\mathbf{t}\omega,\nu^{\flat}\wedge\iota_{\nu}\eta\right\rangle _{\Lambda}\mathrm{vol}=\left\langle \mathbf{t}\omega,\iota_{\nu}\eta\right\rangle _{\Lambda}\mathrm{vol}
\]
\end{proof}
\begin{thm}[Integration of tensors and forms by parts]
$\;$\label{thm:integrate_tensor_forms_by_parts}
\begin{enumerate}
\item For tensors $T_{a_{1}...a_{k}}$ and $Q_{a_{1}...a_{k+1}}$,
\begin{equation}
\int_{M}\nabla_{i}\left(T_{a_{1}...a_{k}}Q^{ia_{1}...a_{k}}\right)=\int_{M}\nabla_{i}T_{a_{1}...a_{k}}Q^{ia_{1}...a_{k}}+\int_{M}T_{a_{1}...a_{k}}\nabla_{i}Q^{ia_{1}...a_{k}}=\int_{\partial M}\nu_{i}T_{a_{1}...a_{k}}Q^{ia_{1}...a_{k}}
\end{equation}
In other words, $\int_{M}\left\langle \nabla T,Q\right\rangle \vol+\int_{M}\left\langle T,\div Q\right\rangle \vol=\int_{\partial M}\left\langle \nu\otimes T,Q\right\rangle \vol_{\partial}$.
\item For $p\in(1,\infty),\omega\in\mathbb{R}W^{1,p}\Omega^{k},\eta\in\mathbb{R}W^{1,p'}\Omega^{k+1}:$
\begin{equation}
\left\langle \left\langle d\omega,\eta\right\rangle \right\rangle _{\Lambda}=\left\langle \left\langle \omega,\delta\eta\right\rangle \right\rangle _{\Lambda}+\left\langle \left\langle \jmath^{*}\omega,\jmath^{*}\iota_{\nu}\eta\right\rangle \right\rangle _{\Lambda}\label{eq:by_parts_d_delta}
\end{equation}
 where $\left\langle \left\langle \jmath^{*}\omega,\jmath^{*}\iota_{\nu}\eta\right\rangle \right\rangle _{\Lambda}=\int_{\partial M}\left\langle \jmath^{*}\omega,\jmath^{*}\iota_{\nu}\eta\right\rangle _{\Lambda}\vol_{\partial}$.
\item For $p\in(1,\infty),\omega\in\mathbb{R}W^{2,p}\Omega^{k}(M),\eta\in\mathbb{R}W^{1,p'}\Omega^{k}(M):$
\begin{equation}
\mathcal{D}(\omega,\eta)=\left\langle \left\langle -\Delta\omega,\eta\right\rangle \right\rangle _{\Lambda}+\left\langle \left\langle \jmath^{*}\iota_{\nu}d\omega,\jmath^{*}\eta\right\rangle \right\rangle _{\Lambda}-\left\langle \left\langle \jmath^{*}\delta\omega,\jmath^{*}\iota_{\nu}\eta\right\rangle \right\rangle _{\Lambda}
\end{equation}
where $\mathcal{D}(\omega,\eta):=\left\langle \left\langle d\omega,d\eta\right\rangle \right\rangle _{\Lambda}+\left\langle \left\langle \delta\omega,\delta\eta\right\rangle \right\rangle _{\Lambda}$
is called the \textbf{Dirichlet integral}. \nomenclature{$\mathcal{D}(\omega,\eta)$}{Dirichlet integral\nomrefpage}
\end{enumerate}
\end{thm}

\begin{proof}
$\;$
\begin{enumerate}
\item Let $X^{i}=T_{a_{1}...a_{k}}Q^{ia_{1}...a_{k}}.$ Then it is just
the divergence theorem.
\item By approximation, it is enough to prove the smooth case.
\begin{align*}
\int_{\partial M}\left\langle \jmath^{*}\omega,\jmath^{*}\iota_{\nu}\eta\right\rangle _{\Lambda}\vol_{\partial} & =\int_{\partial M}\jmath^{*}\left(\omega\wedge\star\eta\right)=\int_{M}d\left(\omega\wedge\star\eta\right)\\
 & =\int_{M}d\omega\wedge\star\eta+(-1)^{k}\int_{M}\omega\wedge d\star\eta=\left\langle \left\langle d\omega,\eta\right\rangle \right\rangle _{\Lambda}-\left\langle \left\langle \omega,\delta\eta\right\rangle \right\rangle _{\Lambda}
\end{align*}
\item Trivial. 
\end{enumerate}
\end{proof}

\subsection{Boundary conditions and potential theory}
\begin{defn}
We define:
\begin{itemize}
\item $\Omega_{D}^{k}(M)=\{\omega\in\Omega^{k}(M):\mathbf{t}\omega=0\}$
(\textbf{Dirichlet boundary condition})
\item $\Omega_{\text{hom}D}^{k}(M)=\{\omega\in\Omega^{k}(M):\mathbf{t}\omega=0,\mathbf{t}\delta\omega=0\}$
(\textbf{relative Dirichlet boundary condition})
\item $\Omega_{N}^{k}(M)=\{\omega\in\Omega^{k}(M):\mathbf{n}\omega=0\}$
(\textbf{Neumann boundary condition})
\item $\Omega_{\hom N}^{k}(M)=\{\omega\in\Omega^{k}(M):\mathbf{n}\omega=0,\mathbf{n}d\omega=0\}$
(\textbf{absolute Neumann boundary condition})
\item $\Omega_{0}^{k}\left(M\right)=\Omega_{D}^{k}\left(M\right)\cap\Omega_{N}^{k}\left(M\right)$
(\textbf{trace-zero boundary condition})
\item $\mathcal{H}^{k}(M)=\{\omega\in\Omega^{k}(M):d\omega=0,\delta\omega=0\}$
(\textbf{harmonic fields})
\item $\mathcal{H}_{D}^{k}(M)=\mathcal{H}^{k}(M)\cap\Omega_{D}^{k}(M)$
(\textbf{Dirichlet fields})
\item $\mathcal{H}_{N}^{k}(M)=\mathcal{H}^{k}(M)\cap\Omega_{N}^{k}(M)$
(\textbf{Neumann fields})
\end{itemize}
\end{defn}

\begin{rem*}
\nomenclature{$\Omega_{D}^{k},\Omega_{\text{hom}D}^{k}$}{different Dirichlet conditions for differential forms\nomrefpage}\nomenclature{$\Omega_{N}^{k},\Omega_{\hom N}^{k}$}{different Neumann conditions for differential forms\nomrefpage}\nomenclature{$\mathcal{H}^{k},\mathcal{H}_{D}^{k},\mathcal{H}_{N}^{k}$}{harmonic fields, then with Dirichlet and Neumann conditions\nomrefpage}

In writing the function spaces, we omit $M$ when there is no possible
confusion. Note that $\Omega_{00}^{k}$ (compact support in $\accentset{\circ}{M}$)
is different from $\Omega_{0}^{k}$.

We can readily extend these definitions to less regular spaces by
replacing $\omega\in\Omega^{k}$ with, for example, $\omega\in B_{3,1}^{\frac{1}{3}}\Omega^{k}$.
Boundary conditions are defined via the trace theorem, and therefore
require some regularity. For example, $B_{3,1}^{\frac{1}{3}}\Omega_{N}^{k}$
makes sense, while $L^{2}\Omega_{N}^{k}$ and $H^{1}\Omega_{\text{hom}N}^{k}$
do not make sense.

Observe that $L^{2}\text{-cl}\left(\Omega_{N}^{k}\right)$ (\textbf{closure}
in the $L^{2}$ norm)\nomenclature{$L^{2}\text{-cl}\left(\cdot\right)$}{closure under $L^2$ norm\nomrefpage}
is just $L^{2}\Omega^{k}$ since $\Omega_{00}^{k}$ is dense in $L^{2}\Omega^{k}$.

Most of these symbols come from \parencite{Schwarz1995}. Note that
in \parencite{Schwarz1995}, the difference between $L^{2}X$ and
$L^{2}\text{-cl}(X)$ (where $X$ is some space) is not made explicit.

Function spaces of type $p=\infty$ are problematic since the smooth
members are not dense (see \Corref{sobolev_and_density}). For instance,
$W^{m,\infty}\Omega^{k}\neq W^{m,\infty}\text{-cl}(\Omega^{k})$ in
general.

A special case is when $k=0$: $\Omega_{N}^{0}(M)=\Omega^{0}(M)=C^{\infty}(M)$
and $\Omega_{\text{hom}D}^{0}(M)=\Omega_{D}^{0}(M)$. Indeed, the
conditions for $\Omega_{\text{hom}D}^{0}$ and $\Omega_{\text{hom}N}^{0}$
are what analysts often call ``Dirichlet'' and ``Neumann'' boundary
conditions respectively.

In fluid dynamics, the condition for $\Omega_{N}^{1}$ is also called
``impermeable'', while $\Omega_{0}^{1}$ is ``no-slip''. On the
other hand, $\Omega_{\text{hom}N}^{1}$ is often given various names,
such as ``Navier-type'', ``free boundary'' or ``Hodge'' \parencite{Mitrea2009_Sylvie_Neumann_Laplacian,Monniaux_2013_Various_boundary_Stokes,Baba2016}.
The consensus, however, seems to be that $\Omega_{\text{hom}N}^{1}$
should be called the ``absolute boundary condition'' \parencite{wu1991_absolute,Hsu_2002_Multiplier_Heat_Absolute,Chen_2009_Stokes_absolute,baudoin2017stochastic_absolute,ouyang2017multiplicative},
which explains our choice of naming.
\end{rem*}
\begin{lem}
We have \textbf{Hodge duality}:
\begin{itemize}
\item $\star:\Omega_{D}^{k}(M)\iso\Omega_{N}^{n-k}(M)$, $\star:\Omega_{\hom D}^{k}(M)\iso\Omega_{\hom N}^{n-k}(M)$,
$\star:\mathcal{H}_{D}^{k}(M)\iso\mathcal{H}_{N}^{n-k}(M)$.
\item $\nabla_{X}(\star\omega)=\star\left(\nabla_{X}\omega\right)$, $\left|\star\omega\right|_{\Lambda}=\left|\omega\right|_{\Lambda}$
for $\omega\in\Omega^{k},X\in\mathfrak{X}M$.
\item For $m\in\mathbb{N}_{0},p\in[1,\infty)$, we have $\star:W^{m,p}\Omega_{D}^{k}(M)\iso W^{m,p}\Omega_{N}^{n-k}(M)$,
$\star:W^{m,p}\Omega_{\hom D}^{k}(M)\iso W^{m,p}\Omega_{\hom N}^{n-k}(M)$.
\end{itemize}
\end{lem}

We stress that harmonic fields are \textbf{harmonic forms}, i.e. $\Delta\omega=0$,
but the converse is not true in general.
\begin{thm}[4 versions]
\label{thm:4_version_gaffney} Let $\omega\in\Omega^{k}(M)$ be a
harmonic form. Then $\omega$ is a harmonic field if either
\begin{enumerate}
\item $\mathbf{t}\omega=0,\mathbf{n}\omega=0$ (trace-zero)
\item $\mathbf{t}\omega=0,\mathbf{t}\delta\omega=0$ (relative Dirichlet)
\item $\mathbf{n}\omega=0,\mathbf{n}d\omega=0$ (absolute Neumann)
\item $\mathbf{t}\delta\omega=0,\mathbf{n}d\omega=0$ (\textbf{Gaffney})
\end{enumerate}
\end{thm}

\begin{proof}
Trivial to show $\mathcal{D}(\omega,\omega)=0$ via integration by
parts.
\end{proof}
\begin{rem*}
The four conditions correspond to four different versions of the Poisson
equation $\Delta\omega=\eta$ (cf. \parencite[Section 3.4]{Schwarz1995}),
and four ways we can make $\Delta$ self-adjoint. In this paper, we
will just focus on the absolute Neumann Laplacian and the absolute
Neumann heat flow.

Gaffney, one of the earliest figures in the field, showed that the
Laplacian corresponding to the 4th boundary condition is self-adjoint
and called it the ``Neumann problem'' (cf. \parencite{Gaffney1954heat,Conner1954}).
We, however, feel the name ``Neumann'' should only be used when
its Hodge dual is Dirichlet-related (for instance, the Dirichlet potential
vs the Neumann potential, to be introduced shortly). Therefore, absent
a better rationalization or convention, we see no reason not to honor
the name of the mathematician.

In the same vein, some authors consider the 1st condition to be the
``Dirichlet boundary condition'' (following the intuition from the
scalar case, where the trace and the tangential part coincide). By
the same reasoning as above, we choose not to do so in this paper.
\end{rem*}
\begin{blackbox}[Dirichlet/Neumann fields]
 \label{blackbox:Dirichlet_Neumann_fields}$\mathcal{H}_{D}^{k}(M)$
and $\mathcal{H}_{N}^{k}(M)$ are finite-dimensional, and therefore
complemented in $\mathbb{R}W^{m,p}\Omega^{k}(M)$ $\forall m\in\mathbb{N}_{0},p\in[1,\infty]$.
\end{blackbox}

\begin{rem*}
All norms on $\mathcal{H}_{N}^{k}$ are equivalent, so we \uline{do
not need to specify which norm on \mbox{$\mathcal{H}_{N}^{k}$} we
are using} at any time.

These are very nice spaces, yet they often prevent uniqueness for
boundary value problems. We almost always want to work on their orthogonal
complements, where Hodge theory truly shines.
\end{rem*}
\begin{proof}
See \parencite[Theorem 2.2.6]{Schwarz1995}.
\end{proof}
\begin{cor}
$\forall m\in\mathbb{N}_{0},p\in[1,\infty]$, there is a continuous
projection $P_{m,p}:\mathbb{R}W^{m,p}\Omega^{k}\twoheadrightarrow\mathcal{H}_{N}^{k}$
such that\label{cor:ortho_projection}
\begin{itemize}
\item it is compatible across different Sobolev spaces, i.e. $P_{m_{0},p_{0}}(\omega)=P_{m_{1},p_{1}}(\omega)$
if $\omega\in W^{m_{0},p_{0}}\Omega^{k}\cap W^{m_{1},p_{1}}\Omega^{k}$.
\item $1-P_{m,p}:\mathbb{R}W^{m,p}\Omega^{k}\twoheadrightarrow W^{m,p}\left(\mathcal{H}_{N}^{k}\right)^{\perp}:=\{\omega\in W^{m,p}\Omega^{k}:\left\langle \left\langle \omega,\phi\right\rangle \right\rangle _{\Lambda}=0\;\forall\phi\in\mathcal{H}_{N}^{k}\}$
is also a compatible projection.
\end{itemize}
\end{cor}

\begin{proof}
Define the continuous linear map $\mathcal{I}_{m,p}:W^{m,p}\Omega^{k}\to\left(\mathcal{H}_{N}^{k}\right)^{*}$
where 
\[
\mathcal{I}_{m,p}\omega(\phi)=\left\langle \left\langle \omega,\phi\right\rangle \right\rangle _{\Lambda}\forall\phi\in\mathcal{H}_{N}^{k},\forall\omega\in W^{m,p}\Omega^{k}
\]
 Then note that $\left(\phi_{1},\phi_{2}\right)\mapsto\left\langle \left\langle \phi_{1},\phi_{2}\right\rangle \right\rangle _{\Lambda}$
is a positive-definite inner product on $\mathcal{H}_{N}^{k}$, so
$\restr{\mathcal{I}_{m,p}}{\mathcal{H}_{N}^{k}}:\mathcal{H}_{N}^{k}\iso\left(\mathcal{H}_{N}^{k}\right)^{*}$.
We also observe that $\restr{\mathcal{I}_{m,p}}{\mathcal{H}_{N}^{k}}$
does not depend on $m,p$, so we can define the continuous inverse
$\mathcal{J}:\left(\mathcal{H}_{N}^{k}\right)^{*}\iso\mathcal{H}_{N}^{k}$.
Then we can just set $P_{m,p}=\mathcal{J}\circ\mathcal{I}_{m,p}$.
As we defined $\mathcal{I}_{m,p}$ by $\left\langle \left\langle \cdot,\cdot\right\rangle \right\rangle _{\Lambda}$,
$P_{m,p}$ is compatible across different $m,p$.
\end{proof}
\begin{rem*}
From now on, for $\omega\in W^{m,p}\Omega^{k}$, we can decompose
$\boxed{\omega=\mathcal{P}^{N}\omega+\mathcal{P}^{N\perp}\omega}$
where $\mathcal{P}^{N}\omega=\omega_{\mathcal{H}_{N}^{k}}\in\mathcal{H}_{N}^{k}$
and $\mathcal{P}^{N\perp}\omega=\omega_{\left(\mathcal{H}_{N}^{k}\right)^{\perp}}\in W^{m,p}\left(\mathcal{H}_{N}^{k}\right)^{\perp}$.
The decomposition is \textbf{natural}, i.e. continuous and compatible
across different Sobolev spaces. By Hodge duality, similarly define
$\mathcal{P}^{D}$ and $\mathcal{P}^{D\perp}$ \nomenclature{$\mathcal{P}^{N},\;\mathcal{P}^{N\perp},\;\mathcal{P}^{D},\;\mathcal{P}^{D\perp}$}{natural orthogonal decomposition\nomrefpage}.
Note $\mathcal{P}^{N\perp}W^{1,p}\Omega_{N}^{k}\leq W^{1,p}\Omega_{N}^{k}$
and $\mathcal{P}^{N\perp}W^{2,p}\Omega_{\hom N}^{k}\leq W^{2,p}\Omega_{\hom N}^{k}$.
\end{rem*}
\begin{blackbox}[Potential theory]
\label{blackbox:potential_estimates} For $m\in\mathbb{N}_{0},p\in\left(1,\infty\right)$,
we define the \textbf{injective Neumann Laplacian} 
\[
\Delta_{N}:\mathcal{P}^{N\perp}W^{m+2,p}\Omega_{\hom N}^{k}\to\mathcal{P}^{N\perp}W^{m,p}\Omega^{k}
\]
\nomenclature{$\Delta_{N}$}{injective Neumann Laplacian\nomrefpage}
as simply $\Delta$ under domain restriction. Then $\left(-\Delta_{N}\right)^{-1}$
is called the \textbf{Neumann potential}, which is bounded (and actually
a Banach isomorphism). $\Delta_{N}$ can also be thought of as an
unbounded operator on $\mathcal{P}^{N\perp}W^{m,p}\Omega^{k}$.

By Hodge duality, we also define the Dirichlet counterparts $\Delta_{D}$
and $\left(-\Delta_{D}\right)^{-1}$.\nomenclature{$\left(-\Delta_{N}\right)^{-1},\;\left(-\Delta_{D}\right)^{-1}$}{Neumann and Dirichlet potentials\nomrefpage}
\end{blackbox}

\begin{proof}
See \parencite[Section 2.2, 2.3]{Schwarz1995}
\end{proof}
\begin{rem*}
Because duality is involved, we stay away from $p\in\{1,\infty\}$.
Amazingly enough, this is the only elliptic estimate we will need
for the rest of the paper. One could say the whole theory is a functional
analytic consequence of elliptic regularity (much like how the Nash
embedding theorem is a consequence of Schauder estimates, following
Günther's approach \parencite{tao_nash_notes}).

There are many identities which might seem complicated, but are actually
trivial to check and helpful for grasping the intuition behind routine
operations in Hodge theory, as well as its rich algebraic structure.
\end{rem*}
\begin{defn*}
We write $d_{c}$ as $d$ restricted to $W^{1,p}\Omega_{D}^{k}$ and
$\delta_{c}$ as $\delta$ restricted to $W^{1,p}\Omega_{N}^{k}$
for $p\in\left(1,\infty\right)$\nomenclature{$\delta_{c},d_{c}$}{adjoints of $d$ and $\delta$\nomrefpage}.
We will prove in \Subsecref{Distributions-and-adjoints} that they
are essentially adjoints of $\delta$ and $d$. Let us note that $\Delta_{N}=-\left(d\delta_{c}+\delta_{c}d\right)$
on $\mathcal{P}^{N\perp}W^{2,p}\Omega_{\hom N}^{k}$.
\end{defn*}
\begin{cor}
\label{cor:basics_of_potentials}Let $p\in(1,\infty)$. Some basic
properties:
\begin{enumerate}
\item $\mathcal{P}^{D\perp}\delta=\delta$ and $\mathcal{P}^{N\perp}d=d$
on $W^{1,p}\Omega^{k}$.\\
$\mathcal{P}^{N\perp}\delta_{c}=\delta_{c}$ on $W^{1,p}\Omega_{N}^{k}$
and $\mathcal{P}^{D\perp}d_{c}=d_{c}$ on $W^{1,p}\Omega_{D}^{k}$.
\item $\left(-\Delta_{D}\right)^{-1}\delta=\delta\left(-\Delta_{D}\right)^{-1}$
on $\mathcal{P}^{D\perp}W^{1,p}\Omega^{k}$ and $\left(-\Delta_{N}\right)^{-1}d=d\left(-\Delta_{N}\right)^{-1}$
on $\mathcal{P}^{N\perp}W^{1,p}\Omega^{k}$.\\
$\left(-\Delta_{N}\right)^{-1}\delta_{c}=\delta_{c}\left(-\Delta_{N}\right)^{-1}$
on $\mathcal{P}^{N\perp}W^{1,p}\Omega_{N}^{k}$ and $\left(-\Delta_{D}\right)^{-1}d_{c}=d_{c}\left(-\Delta_{D}\right)^{-1}$
on $\mathcal{P}^{D\perp}W^{1,p}\Omega_{D}^{k}$.
\item $\delta=\delta\mathcal{P}^{D\perp}=\delta\mathcal{P}^{N\perp}$ and
$d=d\mathcal{P}^{D\perp}=d\mathcal{P}^{N\perp}$ on $W^{1,p}\Omega^{k}$.
\item $d\delta d=d\left(\delta d+d\delta\right)=d\left(-\Delta\right)$.\\
$\delta d\delta\left(-\Delta_{D}\right)^{-1}=\delta$ on $\mathcal{P}^{D\perp}W^{1,p}\Omega^{k}$
and $d\delta d(-\Delta_{N})^{-1}=d$ on $\mathcal{P}^{N\perp}W^{1,p}\Omega^{k}$.
\item $d\left(W^{2,p}\Omega_{\hom N}^{k}\right)=d\left(W^{2,p}\Omega_{N}^{k}\right)\cap W^{1,p}\Omega_{N}^{k+1},$
$\delta\left(W^{2,p}\Omega_{\hom D}^{k}\right)=\delta\left(W^{2,p}\Omega_{D}^{k}\right)\cap W^{1,p}\Omega_{D}^{k-1}$.\\
$d\left(W^{3,p}\Omega_{\hom N}^{k}\right)\leq W^{2,p}\Omega_{\hom N}^{k+1}$,
$\delta\left(W^{3,p}\Omega_{\hom D}^{k}\right)\leq W^{2,p}\Omega_{\hom D}^{k-1}$.
\end{enumerate}
\end{cor}

\begin{rem*}
A good mnemonic device is that $\Delta_{N}$ is formed by $d$ and
$\delta_{c}$, so $\left(-\Delta_{N}\right)^{-1}$ commutes with $d$
and $\delta_{c}$.
\end{rem*}
\begin{proof}
$ $
\begin{enumerate}
\item Integration by parts.
\item Just check that the expressions are defined by using 1).
\end{enumerate}
The rest is trivial.
\end{proof}

\subsection{Hodge decomposition\label{subsec:Hodge-decomposition}}

We proceed differently from \parencite{Schwarz1995}, by using a more
algebraic approach in order to derive some results not found in the
book. There will be a lot of identities gathered through experience,
so their appearances can seem unmotivated at first. Hence, as motivation,
let's look at an example of a problem we will need Hodge theory for:
is it true that $W^{2,p}\Omega_{\mathrm{hom}N}^{k}$ is dense in $W^{1,p}\Omega_{N}^{k}$
for $p\in\left(1,\infty\right)$? The problem is more subtle than
it seems, and it is true that the heat flow, once constructed, will
imply the answer is yes. But we do not yet have the heat flow, and
it turns out this problem is needed for the $W^{1,p}$-analyticity
of the heat flow itself. This foundational approximation of boundary
conditions can be done easily once we understand Hodge theory and
the myriad connections between different boundary conditions.

Let $\omega\in W^{m,p}\Omega^{k}$ $(m\in\mathbb{N}_{0},p\in(1,\infty))$.
In one line, the \textbf{Hodge-Morrey decomposition algorithm} is

\[
\boxed{\omega=d_{c}\delta\left(-\Delta_{D}\right)^{-1}\mathcal{P}^{D\perp}\omega+\delta_{c}d\left(-\Delta_{N}\right)^{-1}\mathcal{P}^{N\perp}\omega+\omega_{\mathcal{H}^{k}}}
\]
where $\mathcal{P}^{D\perp}\omega=\omega_{\left(\mathcal{H}_{D}^{k}\right)^{\perp}},\mathcal{P}^{N\perp}\omega=\omega_{\left(\mathcal{H}_{N}^{k}\right)^{\perp}}$
are defined as in \Corref{ortho_projection}, and $\omega_{\mathcal{H}^{k}}$
is simply defined by subtraction. This is the heart of the matter,
and the rest is arguably just bookkeeping.

Note that if $\omega\in W^{1,p}\Omega^{k}$, $d\omega=d\delta d\left(-\Delta_{N}\right)^{-1}\mathcal{P}^{N\perp}\omega+d\omega_{\mathcal{H}^{k}}=d\mathcal{P}^{N\perp}\omega+d\omega_{\mathcal{H}^{k}}=d\omega+d\omega_{\mathcal{H}^{k}}$.
So $d\omega_{\mathcal{H}^{k}}=0$ and similarly $\delta\omega_{\mathcal{H}^{k}}=0$,
justifying the notation. A mild warning is that we do not yet have
$W^{1,p}\mathcal{H}^{k}=W^{1,p}\text{-}\mathrm{cl}\left(\mathcal{H}^{k}\right)$.

As we will keep referring to this decomposition, let us define
\begin{itemize}
\item $\mathcal{P}_{1}=d_{c}\delta\left(-\Delta_{D}\right)^{-1}\mathcal{P}^{D\perp}$.
Then $\mathcal{P}_{1}=d_{c}\left(-\Delta_{D}\right)^{-1}\delta\mathcal{P}^{D\perp}=d_{c}\left(-\Delta_{D}\right)^{-1}\delta$
on $W^{1,p}\Omega^{k}$.
\item $\mathcal{P}_{2}=\delta_{c}d\left(-\Delta_{N}\right)^{-1}\mathcal{P}^{N\perp}$
Then $\mathcal{P}_{2}=\delta_{c}\left(-\Delta_{N}\right)^{-1}d$ on
$W^{1,p}\Omega^{k}$.
\item $\mathcal{P}_{3}=1-\mathcal{P}_{1}-\mathcal{P}_{2}$.
\end{itemize}
We observe that the decomposition $1=\mathcal{P}_{1}+\mathcal{P}_{2}+\mathcal{P}_{3}$
is natural (continuous and compatible across different Sobolev spaces)
since all the operations are natural. In particular, $\mathcal{P}_{j}$
(for $j\in\{1,2,3\}$) is a zeroth-order operator, and if $\omega$
is smooth, so is $\mathcal{P}_{j}\omega$ by Sobolev embedding. Recall
that $\mathbf{t}\omega=0$ implies $\mathbf{t}d\omega=0$, while $\mathbf{n}\omega=0$
implies $\mathbf{n}\delta\omega=0$ (\Lemref{basic-identities-forms-boundary}).

\nomenclature{$\mathcal{P}_{1}\omega,\;\mathcal{P}_{2}\omega,\;\mathcal{P}_{3}\omega$}{the component projections in Hodge decomposition\nomrefpage}

\begin{thm}[Smooth decomposition]
 Some basic properties of $\mathcal{P}_{j}$ on $\Omega^{k}$:
\begin{enumerate}
\item $\mathcal{P}_{1}\delta=0$ on $\Omega^{k+1}$ and $\mathcal{P}_{2}d=0$
on $\Omega^{k-1}$.\\
$\mathcal{P}_{1}=\mathcal{P}_{2}=0$ on $\mathcal{H}^{k}$.
\item $\mathcal{P}_{3}\delta_{c}=0$ on $\Omega_{N}^{k+1}$ and $\mathcal{P}_{3}d_{c}=0$
on $\Omega_{D}^{k-1}$.
\item $\mathcal{P}_{j}\mathcal{P}_{i}=\delta_{ij}\mathcal{P}_{i}$. Therefore
$\Omega^{k}=\bigoplus_{j=1}^{3}\mathcal{P}_{j}\left(\Omega^{k}\right)$.
\item $\mathcal{P}_{1}\left(\Omega^{k}\right)=d_{c}\left(\Omega_{D}^{k-1}\right)=d_{c}\mathcal{P}^{D\perp}\left(\Omega_{D}^{k-1}\right)=d_{c}\delta\mathcal{P}^{D\perp}\left(\Omega_{\hom D}^{k}\right)\leq\Omega_{D}^{k}$.\\
$\mathcal{P}_{2}\left(\Omega^{k}\right)=\delta_{c}\left(\Omega_{N}^{k+1}\right)=\delta_{c}\mathcal{P}^{N\perp}\left(\Omega_{N}^{k+1}\right)=\delta_{c}d\mathcal{P}^{N\perp}\left(\Omega_{\hom N}^{k}\right)\leq\Omega_{N}^{k}$.\\
$\mathcal{P}_{3}\left(\Omega^{k}\right)=\mathcal{H}^{k}$.
\item $\Omega^{k}=\bigoplus_{j=1}^{3}\mathcal{P}_{j}\left(\Omega^{k}\right)$
is $\left\langle \left\langle \cdot,\cdot\right\rangle \right\rangle _{\Lambda}$-orthogonal
decomposition.
\end{enumerate}
\end{thm}

\begin{proof}
$ $
\begin{enumerate}
\item On $\Omega^{k+1}$, $\mathcal{P}_{1}\delta=d_{c}\left(-\Delta_{D}\right)^{-1}\delta\delta=0$.\\
Let $\eta\in\mathcal{H}^{k}.$ Then $\mathcal{P}_{1}\eta=d_{c}\left(-\Delta_{D}\right)^{-1}\delta\eta=0$.
\item We just need $\mathcal{P}_{2}\delta_{c}=\delta_{c}$ on $\Omega_{N}^{k+1}$.
Indeed, $\mathcal{P}_{2}\delta_{c}=\delta_{c}d\left(-\Delta_{N}\right)^{-1}\delta_{c}\mathcal{P}^{N\perp}=\delta_{c}d\delta_{c}\left(-\Delta_{N}\right)^{-1}\mathcal{P}^{N\perp}=\delta_{c}\mathcal{P}^{N\perp}=\delta_{c}$.
\item By 1), $\mathcal{P}_{2}\mathcal{P}_{1}=\mathcal{P}_{1}\mathcal{P}_{2}=\mathcal{P}_{1}\mathcal{P}_{3}=\mathcal{P}_{2}\mathcal{P}_{3}=0$.
By 2), $\mathcal{P}_{3}\mathcal{P}_{2}=\mathcal{P}_{3}\mathcal{P}_{1}=0$.
Then observe $\mathcal{P}_{2}=(\mathcal{P}_{1}+\mathcal{P}_{2}+\mathcal{P}_{3})\mathcal{P}_{2}=\mathcal{P}_{2}^{2}$.
Similarly, $\mathcal{P}_{1}^{2}=\mathcal{P}_{1}$ and $\mathcal{P}_{3}^{2}=\mathcal{P}_{3}$.
\item Recall $\mathcal{P}_{3}\left(\Omega^{k}\right)\leq\mathcal{H}^{k}.$
It becomes an equality since $\mathcal{P}_{2}\left(\mathcal{H}^{k}\right)=\mathcal{P}_{1}\left(\mathcal{H}^{k}\right)=0$.\\
Similarly, obviously $\mathcal{P}_{1}\left(\Omega^{k}\right)=d_{c}\delta\mathcal{P}^{D\perp}\left(\Omega_{\hom D}^{k}\right)\leq d_{c}\left(\Omega_{D}^{k-1}\right)$.
It becomes an equality since $\mathcal{P}_{2}d=0$ and $\mathcal{P}_{3}d_{c}=0$.
\item Trivial.
\end{enumerate}
\end{proof}
To extend this to Sobolev spaces, we will need to use distributions
and duality.

\begin{cor}[Sobolev version]
 Some basic properties of $\mathcal{P}_{j}$ on $W^{m,p}\Omega^{k}$
$(m\in\mathbb{N}_{0},p\in(1,\infty))$:\label{cor:Sobolev_Hodge_projections}
\begin{enumerate}
\item $\left\langle \left\langle \mathcal{P}_{j}\omega,\phi\right\rangle \right\rangle _{\Lambda}=\left\langle \left\langle \omega,\mathcal{P}_{j}\phi\right\rangle \right\rangle _{\Lambda}$
$\forall\omega\in W^{m,p}\Omega^{k},\forall\phi\in\Omega_{00}^{k},j=1,2,3$
\item $\mathcal{P}_{1}\delta=0$ on $W^{m+1,p}\Omega^{k+1}$ and $\mathcal{P}_{2}d=0$
on $W^{m+1,p}\Omega^{k-1}$.
\item $\mathcal{P}_{1}=\mathcal{P}_{2}=0$ on $W^{m+1,p}\mathcal{H}^{k}$
and $W^{m,p}\text{-}\mathrm{cl}\left(\mathcal{H}^{k}\right)$.
\item $\mathcal{P}_{3}\delta_{c}=0$ on $W^{m+1,p}\Omega_{N}^{k+1}$ and
$\mathcal{P}_{3}d_{c}=0$ on $W^{m+1,p}\Omega_{D}^{k-1}$.
\item $\mathcal{P}_{j}\mathcal{P}_{i}=\delta_{ij}\mathcal{P}_{i}$. Therefore
$W^{m,p}\Omega^{k}=\bigoplus_{j=1}^{3}\mathcal{P}_{j}\left(W^{m,p}\Omega^{k}\right)$.
\item $\mathcal{P}_{3}\left(W^{m,p}\Omega^{k}\right)=W^{m,p}\mathcal{H}^{k}$
for $m\geq1$ and $W^{m,p}\text{-}\mathrm{cl}\left(\mathcal{H}^{k}\right)$
for $m\geq0$.\\
$\mathcal{P}_{2}\left(W^{m,p}\Omega^{k}\right)=\delta_{c}\left(W^{m+1,p}\Omega_{N}^{k+1}\right)=\delta_{c}d\mathcal{P}^{N\perp}\left(W^{m+2,p}\Omega_{\hom N}^{k}\right)$
.\\
$\mathcal{P}_{1}\left(W^{m,p}\Omega^{k}\right)=d_{c}\left(W^{m+1,p}\Omega_{D}^{k-1}\right)=d_{c}\delta\mathcal{P}^{D\perp}\left(W^{m+2,p}\Omega_{\hom D}^{k}\right)$.
\item $\mathbf{t}\mathcal{P}_{1}=0$ and $\mathbf{n}\mathcal{P}_{2}=0$
on $W^{m+1,p}\Omega^{k}$.
\item For $p\geq2,$ $W^{m,p}\Omega^{k}=\bigoplus_{j=1}^{3}\mathcal{P}_{j}\left(W^{m,p}\Omega^{k}\right)$
is $\left\langle \left\langle \cdot,\cdot\right\rangle \right\rangle _{\Lambda}$-orthogonal
decomposition.
\item $W^{m,p}\text{-}\mathrm{cl}\left(d_{c}\left(\Omega_{D}^{k-1}\right)\right)=d_{c}\left(W^{m+1,p}\Omega_{D}^{k-1}\right)$
and $W^{m,p}\text{-}\mathrm{cl}\left(\delta_{c}\left(\Omega_{N}^{k+1}\right)\right)=\delta_{c}\left(W^{m+1,p}\Omega_{N}^{k+1}\right)$.\\
$W^{m+1,p}\text{-}\mathrm{cl}\left(\mathcal{H}^{k}\right)=W^{m+1,p}\mathcal{H}^{k}$.
\item $d=d\left(\mathcal{P}_{1}+\mathcal{P}_{2}+\mathcal{P}_{3}\right)=d\mathcal{P}_{2}=d\mathcal{P}^{N\perp}=\mathcal{P}^{N\perp}d$
on $W^{m+1,p}\Omega^{k}$.\\
Consequently, $\mathbf{n}d\mathcal{P}_{2}\left(W^{m+2,p}\Omega_{\hom N}^{k}\right)=\mathbf{n}d\left(W^{m+2,p}\Omega_{\hom N}^{k}\right)=0$,
and $\mathcal{P}_{2}\left(W^{m+2,p}\Omega_{\hom N}^{k}\right)\leq W^{m+2,p}\Omega_{\hom N}^{k}$.\\
We also have 
\[
d\left(W^{m+1,p}\Omega^{k-1}\right)=d\mathcal{P}_{2}\left(W^{m+1,p}\Omega^{k-1}\right)=d\left(W^{m+1,p}\Omega_{N}^{k-1}\right)=d\mathcal{P}^{N\perp}\left(W^{m+1,p}\Omega_{N}^{k-1}\right)
\]
\item $\delta_{c}=\mathcal{P}_{2}\delta_{c}$ on $W^{m+1,p}\Omega_{N}^{k}$
and
\[
\mathcal{P}_{2}\left(W^{m+1,p}\Omega^{k}\right)=\delta_{c}\left(W^{m+2,p}\Omega_{N}^{k+1}\right)=\delta_{c}d\mathcal{P}^{N\perp}\left(W^{m+3,p}\Omega_{\hom N}^{k}\right)=\delta_{c}\mathcal{P}^{N\perp}\left(W^{m+2,p}\Omega_{\hom N}^{k+1}\right)
\]
\end{enumerate}
\end{cor}

\begin{rem*}
Note that $L^{p}\text{-}\mathrm{cl}\left(\mathcal{H}^{k}\right)$
$(p\in(1,\infty))$ is defined, while $L^{p}\mathcal{H}^{k}$ is not.
\end{rem*}
\begin{proof}
$ $
\begin{enumerate}
\item Observe $\mathcal{P}_{1}\omega\in d_{c}\left(W^{m+1,p}\Omega_{D}^{k-1}\right),\mathcal{P}_{2}\omega\in\delta_{c}\left(W^{m+1,p}\Omega_{N}^{k+1}\right),\mathcal{P}_{3}\omega\in W^{m,p}\text{-}\mathrm{cl}\left(\mathcal{H}^{k}\right)$.
Simply show $d_{c}\left(W^{m+1,p}\Omega_{D}^{k-1}\right)\perp\delta_{c}\left(\Omega_{N}^{k+1}\right)$,
$W^{m,p}\text{-}\mathrm{cl}\left(\mathcal{H}^{k}\right)\perp d_{c}\left(\Omega_{D}^{k-1}\right)$,
and so forth via integration by parts.
\item $W^{m+1,p}\Omega^{k+1}=W^{m+1,p}\text{-}\mathrm{cl}\left(\Omega^{k+1}\right)$.
\item The case $W^{m,p}\text{-}\mathrm{cl}\left(\mathcal{H}^{k}\right)$
is trivial. For $\omega\in W^{m+1,p}\mathcal{H}^{k},$ $\left\langle \left\langle \mathcal{P}_{1}\omega,\phi\right\rangle \right\rangle _{\Lambda}=\left\langle \left\langle \omega,\mathcal{P}_{1}\phi\right\rangle \right\rangle _{\Lambda}=0\;\forall\phi\in\Omega_{00}^{k}$
since $W^{m+1,p}\mathcal{H}^{k}\perp d_{c}(\Omega_{D}^{k-1}$) (integration
by parts).
\item Let $\omega\in W^{m+1,p}\Omega_{N}^{k+1}$. Then $\left\langle \left\langle \mathcal{P}_{3}\delta_{c}\omega,\phi\right\rangle \right\rangle _{\Lambda}=\left\langle \left\langle \delta_{c}\omega,\mathcal{P}_{3}\phi\right\rangle \right\rangle _{\Lambda}=0\;\forall\phi\in\Omega_{00}^{k}$
since $\delta_{c}\left(W^{m+1,p}\Omega_{N}^{k+1}\right)\perp\mathcal{H}^{k}$.
\end{enumerate}
The rest is trivial.
\end{proof}
To connect Hodge decomposition to fluid dynamics, we will need the
\textbf{Friedrichs decomposition}: 
\[
\mathcal{P}_{3}=\left(\mathcal{P}^{N}+\mathcal{P}^{N\perp}\right)\mathcal{P}_{3}=\mathcal{P}_{3}^{N}+\mathcal{P}_{3}^{\mathrm{ex}}
\]
 \nomenclature{$\mathcal{P}_{3}^{N},\;\mathcal{P}_{3}^{\mathrm{ex}},\;\mathcal{P}_{3}^{D},\;\mathcal{P}_{3}^{\mathrm{co}}$}{Friedrichs decomposition\nomrefpage}
where
\begin{itemize}
\item $\mathcal{P}_{3}^{N}:=\mathcal{P}^{N}\mathcal{P}_{3}=\mathcal{P}^{N}=\mathcal{P}_{3}\mathcal{P}^{N}$
(as $\mathcal{P}^{N\perp}\mathcal{P}_{1}=\mathcal{P}_{1}$ and $\mathcal{P}^{N\perp}\mathcal{P}_{2}=\mathcal{P}_{2}$)
\item $\mathcal{P}_{3}^{\mathrm{ex}}:=\mathcal{P}^{N\perp}\mathcal{P}_{3}=\mathcal{P}_{3}\mathcal{P}^{N\perp}$
\end{itemize}
We similarly define $\mathcal{P}_{3}^{D},\mathcal{P}_{3}^{\mathrm{co}}$
via Hodge duality. Note that $\mathrm{ex}$ and $\mathrm{co}$ stand
for ``exact'' and ``coexact'' (and we will see why shortly).

Then we define $\mathbb{P}:=\mathcal{P}_{3}^{N}+\mathcal{P}_{2}$
as the \textbf{Leray projection}.\nomenclature{$\mathbb{P}$}{Leray projection\nomrefpage}
Then $1=\left(\mathcal{P}_{3}^{\text{ex}}+\mathcal{P}_{1}\right)+\left(\mathcal{P}_{3}^{N}+\mathcal{P}_{2}\right)=\left(\mathcal{P}_{3}^{\text{ex}}+\mathcal{P}_{1}\right)+\mathbb{P}$
is called the \textbf{Helmholtz decomposition}.
\begin{thm}[Friedrichs decomposition]
\label{thm:Friedrichs_Decomposition}Basic properties of $\mathcal{P}_{3}^{N},\mathcal{P}_{3}^{\mathrm{ex}}$
on $W^{m,p}\Omega^{k}$ $(m\in\mathbb{N}_{0},p\in(1,\infty))$:
\begin{enumerate}
\item $\mathcal{P}_{3}^{\mathrm{ex}}=d\delta(-\Delta_{N})^{-1}\mathcal{P}_{3}^{\mathrm{ex}}$
on $W^{m,p}\Omega^{k}$.
\item $\mathcal{P}_{3}^{\mathrm{ex}}\left(W^{m,p}\Omega^{k}\right)=W^{m,p}\text{-}\mathrm{cl}\left(\mathcal{H}^{k}\right)\cap d\left(W^{m+1,p}\Omega^{k-1}\right)$.
\item $\left(\mathcal{P}_{3}^{\mathrm{ex}}+\mathcal{P}_{1}\right)\left(W^{m,p}\Omega^{k}\right)=d\left(W^{m+1,p}\Omega^{k-1}\right)=d\left(W^{m+1,p}\Omega_{N}^{k-1}\right)=d\mathcal{P}^{N\perp}\left(W^{m+1,p}\Omega_{N}^{k-1}\right)$.
\item $\mathbb{P}\left(W^{m,p}\Omega^{k}\right)=\left(\mathcal{P}_{3}^{N}+\mathcal{P}_{2}\right)\left(W^{m,p}\Omega^{k}\right)=\Ker\left(\restr{\delta_{c}}{W^{m,q}\Omega_{N}^{k}}\right)$
when $m\geq1$ and $W^{m,p}\text{-}\mathrm{cl}\left(\Ker\left(\restr{\delta_{c}}{\Omega_{N}^{k}}\right)\right)$
when $m\geq0$.
\item $\left(\mathcal{P}_{3}+\mathcal{P}_{2}\right)\left(W^{m,p}\Omega^{k}\right)=\Ker\left(\restr{\delta}{W^{m,q}\Omega^{k}}\right)$
when $m\geq1$ and $W^{m,p}\text{-}\mathrm{cl}\left(\Ker\left(\restr{\delta}{\Omega^{k}}\right)\right)$
when $m\geq0$.
\item $\mathcal{P}^{N\perp}\mathbb{P}=\mathcal{P}_{2}=\mathbb{P}\mathcal{P}^{N\perp}$
on $W^{m,p}\Omega^{k}$.\\
Therefore $d\mathbb{P}=d\mathcal{P}^{N\perp}\mathbb{P}=d\mathcal{P}_{2}=d=d\mathcal{P}^{N\perp}=\mathcal{P}^{N\perp}d$
on $W^{m+1,p}\Omega^{k}$.
\item $\mathbb{P}\left(W^{m+2,p}\Omega_{\hom N}^{k}\right)\leq\mathcal{P}_{2}\left(W^{m+2,p}\Omega_{\hom N}^{k}\right)\oplus\mathcal{H}_{N}^{k}\leq W^{m+2,p}\Omega_{\hom N}^{k}$.
\end{enumerate}
\end{thm}

\begin{proof}
$ $
\begin{enumerate}
\item On $\Omega^{k}$: $\delta d(-\Delta_{N})^{-1}\mathcal{P}_{3}^{\mathrm{ex}}=\delta(-\Delta_{N})^{-1}d\mathcal{P}_{3}^{\mathrm{ex}}=0$,
so $\mathcal{P}_{3}^{\mathrm{ex}}=(-\Delta)(-\Delta_{N})^{-1}\mathcal{P}_{3}^{\mathrm{ex}}=d\delta(-\Delta_{N})^{-1}\mathcal{P}_{3}^{\mathrm{ex}}$.
Then we are done by density.
\item $\mathcal{P}_{3}^{N}d=\mathcal{P}_{3}\mathcal{P}^{N}d=0$ as $\mathcal{P}^{N\perp}d=d$.
\item $\mathcal{P}_{2}d=0$ and $\mathcal{P}_{3}^{N}d=0$.
\item We first prove the smooth version. Let $\omega\in\Ker\left(\restr{\delta_{c}}{\Omega_{N}^{k}}\right)$.
Then $\left\langle \left\langle \mathcal{P}_{1}\omega,\mathcal{P}_{1}\omega\right\rangle \right\rangle _{\Lambda}=\left\langle \left\langle \mathcal{P}_{1}\omega,\omega\right\rangle \right\rangle _{\Lambda}=0$
as $\Ker\left(\restr{\delta_{c}}{\Omega_{N}^{k}}\right)\perp d\left(\Omega^{k-1}\right)$,
so $\mathcal{P}_{1}\omega=0$. Similarly, $\mathcal{P}_{3}^{\mathrm{ex}}\omega=0$.
Then $\left(\mathcal{P}_{3}^{N}+\mathcal{P}_{2}\right)\Omega^{k}=\Ker\left(\restr{\delta_{c}}{\Omega_{N}^{k}}\right)$.\\
For $W^{m,p}\Omega^{k}$, the case $W^{m,p}\text{-}\mathrm{cl}\left(\Ker\left(\restr{\delta_{c}}{\Omega_{N}^{k}}\right)\right)$
is trivial. Then assume $m\geq1$ and $\omega\in\Ker\left(\restr{\delta_{c}}{W^{m,q}\Omega_{N}^{k}}\right).$
We can show $\mathcal{P}_{1}\omega=\mathcal{P}_{3}^{\mathrm{ex}}\omega=0$
as distributions since $\Ker\left(\restr{\delta_{c}}{W^{m,q}\Omega_{N}^{k}}\right)\perp d(\Omega^{k-1})$.
\item Just note that $\Ker\left(\restr{\delta}{W^{m,q}\Omega^{k}}\right)\perp d_{c}(\Omega_{D}^{k-1})$
and argue similarly.
\item Easy to check that $\mathcal{P}^{N\perp}\mathcal{P}_{3}^{N}=\mathcal{P}_{3}^{N}\mathcal{P}^{N\perp}=0$
and $\mathcal{P}^{N\perp}\mathcal{P}_{2}=\mathcal{P}_{2}\mathcal{P}^{N\perp}=\mathcal{P}_{2}$.
\item Trivial.
\end{enumerate}
\end{proof}
\begin{rem*}
Similar results for $\mathcal{P}_{3}^{D},\mathcal{P}_{3}^{\mathrm{co}}$
hold by Hodge duality. When $M$ has no boundary, $\mathcal{H}^{k}=\mathcal{H}_{D}^{k}=\mathcal{H}_{N}^{k}$
so $\mathcal{P}_{3}=\mathcal{P}_{3}^{N}=\mathcal{P}_{3}^{D}$.

A simple consequence of the Hodge-Helmholtz decomposition is that

\[
\frac{\Ker\left(\restr{\delta_{c}}{\Omega_{N}^{k}}\right)}{\delta_{c}\left(\Omega_{N}^{k+1}\right)}=\frac{\left(\mathcal{P}_{3}^{N}+\mathcal{P}_{2}\right)\left(\Omega^{k}\right)}{\mathcal{P}_{2}\left(\Omega^{k}\right)}=\mathcal{P}_{3}^{N}\left(\Omega^{k}\right)=\frac{\left(\mathcal{P}_{3}+\mathcal{P}_{1}\right)\left(\Omega^{k}\right)}{\left(\mathcal{P}_{3}^{\mathrm{ex}}+\mathcal{P}_{1}\right)\left(\Omega^{k}\right)}=\frac{\Ker\left(\restr{d}{\Omega^{k}}\right)}{d\left(\Omega^{k-1}\right)}
\]
This can be rewritten as $\boxed{\mathbb{H}_{a}^{k}\left(M\right)=\mathcal{H}_{N}^{k}\left(M\right)=\mathbb{H}_{\mathrm{dR}}^{k}\left(M,d\right)}$
(\textbf{Hodge isomorphism theorem}) where $\mathbb{H}_{\mathrm{dR}}^{k}\left(M,d\right):=\frac{\Ker\left(\restr{d}{\Omega^{k}}\right)}{d\left(\Omega^{k-1}\right)}$
is called the \textbf{$k$-th} \textbf{de Rham cohomology group},
and $\mathbb{H}_{a}^{k}\left(M\right):=\frac{\Ker\left(\restr{\delta_{c}}{\Omega_{N}^{k}}\right)}{\delta_{c}\left(\Omega_{N}^{k+1}\right)}$
is called the \textbf{$k$-th} \textbf{absolute de Rham cohomology
group}. In particular, $\beta^{k}\left(M\right):=\dim\mathcal{H}_{N}^{k}\left(M\right)=\dim\mathbb{H}_{\mathrm{dR}}^{k}\left(M,d\right)$
is called the \textbf{$k$-th Betti number} of $M$. Note that the
Hodge dual of $\mathbb{H}_{a}^{n-k}\left(M\right)$ is $\mathbb{H}_{r}^{k}\left(M\right):=\frac{\Ker\left(\restr{d_{c}}{\Omega_{D}^{k}}\right)}{d_{c}\left(\Omega_{D}^{k-1}\right)}$,
the \textbf{$k$-th} \textbf{relative de Rham cohomology group}.

We can also define right inverses (\emph{potentials}) for $d,\delta,\delta_{c},d_{c}$
(see \Subsecref{Hodge-Sobolev-spaces}).

In many ways, Hodge theory reduces otherwise complicated boundary
value problems into purely algebraic calculations. A standard Hodge-theoretic
calculation related to the Euler equation is given later in \Subsecref{Calc_pressure}.
We can also derive a general form of the Poincare inequality:
\end{rem*}
\begin{cor}[Poincare-Hodge-Dirac inequality]
\label{cor:poincare_ineq_Neumann} Let $\omega\in\mathcal{P}^{N\perp}W^{m+1,p}\Omega_{N}^{k}$
$(m\in\mathbb{N}_{0},p\in(1,\infty))$. Then 
\[
\left\Vert \omega\right\Vert _{W^{m+1,p}}\sim\left\Vert d\omega\right\Vert _{W^{m,p}}+\left\Vert \delta_{c}\omega\right\Vert _{W^{m,p}}
\]
 and we have a bijection 
\[
\mathcal{P}^{N\perp}W^{m+1,p}\Omega_{N}^{k}\xrightarrow{d\oplus\delta_{c}}d\left(W^{m+1,p}\Omega^{k}\right)\oplus\delta_{c}\left(W^{m+1,p}\Omega_{N}^{k}\right)=\left(\mathcal{P}_{1}+\mathcal{P}_{3}^{\text{ex}}\right)(W^{m,p}\Omega^{k+1})\oplus\mathcal{P}_{2}(W^{m,p}\Omega^{k-1})
\]

In particular, $\boxed{\left(d\oplus\delta_{c}\right)^{-1}\left(d\eta,\delta_{c}\upsilon\right)=\mathcal{P}_{2}\left(\eta-\upsilon\right)+\upsilon\;\forall\eta,\upsilon\in\mathcal{P}^{N\perp}W^{m+1,p}\Omega_{N}^{k}}$.
\end{cor}

\begin{proof}
Observe that
\begin{itemize}
\item $\mathcal{P}^{N\perp}W^{m+1,p}\Omega_{N}^{k}\xrightarrow{d\oplus\delta_{c}}d\mathcal{P}^{N\perp}\left(W^{m+1,p}\Omega_{N}^{k}\right)\oplus\delta_{c}\mathcal{P}^{N\perp}\left(W^{m+1,p}\Omega_{N}^{k}\right)$
is a continuous injection.
\item $d\mathcal{P}^{N\perp}\left(W^{m+1,p}\Omega_{N}^{k}\right)=d\left(W^{m+1,p}\Omega^{k}\right)=\left(\mathcal{P}_{1}+\mathcal{P}_{3}^{\mathrm{ex}}\right)(W^{m,p}\Omega^{k+1})$
by \Corref{Sobolev_Hodge_projections}.
\item $\delta_{c}\mathcal{P}^{N\perp}\left(W^{m+1,p}\Omega_{N}^{k}\right)=\delta_{c}\left(W^{m+1,p}\Omega_{N}^{k}\right)=\mathcal{P}_{2}(W^{m,p}\Omega^{k-1})$
by \Corref{basics_of_potentials} and \ref{cor:Sobolev_Hodge_projections}.
\end{itemize}
By open mapping, we only need to prove $d\oplus\delta_{c}$ (the \textbf{injective
Hodge-Dirac operator}) is surjective: let $\eta,\upsilon\in\mathcal{P}^{N\perp}W^{m+1,p}\Omega_{N}^{k}$.
We want to find $\omega\in\mathcal{P}^{N\perp}W^{m+1,p}\Omega_{N}^{k}$
such that $d\omega=d\eta,\delta_{c}\omega=\delta_{c}\upsilon$. By
the restriction $\delta_{c}\omega=\delta_{c}\upsilon$, the freedom
is in choosing 
\[
\vartheta:=\omega-\upsilon\in\mathcal{P}^{N\perp}\text{Ker \ensuremath{\left(\restr{\delta_{c}}{W^{m+1,p}\Omega_{N}^{k}}\right)}}=\mathcal{P}^{N\perp}\mathbb{P}(W^{m+1,p}\Omega^{k})=\mathcal{P}_{2}(W^{m+1,p}\Omega^{k})
\]
 such that $d\omega=d\upsilon+d\vartheta=d\eta$. In other words,
we want $\vartheta$ such that $d\vartheta=d\left(\eta-\upsilon\right)$
and $\mathcal{P}_{2}\vartheta=\vartheta$. Then we are done by setting
$\vartheta=\mathcal{P}_{2}\left(\eta-\upsilon\right)$.
\end{proof}
\begin{rem*}
We note that a less general version of the Poincare inequality was
used in \parencite{Schwarz1995} to establish the potential estimates
in \Blackboxref{potential_estimates} as well as \Blackboxref{Dirichlet_Neumann_fields}.
A more general version \parencite[Lemma 2.4.10]{Schwarz1995} deals
with the case $p\geq2$. Our version here only requires $p\in\left(1,\infty\right)$.

Among other things, the inequality allows the following approximation
of boundary conditions, which will play a crucial role for the $W^{1,p}$-analyticity
of the heat flow in \Subsecref{W1p-analyticity}.
\end{rem*}
\begin{cor}
\label{cor:crucial_boundary_approx_homN} Let $p\in\left(1,\infty\right).$
\begin{enumerate}
\item $W^{1,p}\Omega_{N}^{k}=d\left(W^{2,p}\Omega_{\mathrm{hom}N}^{k-1}\right)\oplus\mathrm{Ker}\left(\restr{\delta_{c}}{W^{1,p}\Omega_{N}^{k}}\right)$
and $\Omega_{N}^{k}=d\left(\Omega_{\mathrm{hom}N}^{k-1}\right)\oplus\mathrm{Ker}\left(\restr{\delta_{c}}{\Omega_{N}^{k}}\right)$.
\item $L^{p}\text{-}\mathrm{cl}\left(d\left(\Omega_{\mathrm{hom}N}^{k}\right)\right)=d\left(W^{1,p}\Omega_{N}^{k}\right)=d\left(W^{1,p}\Omega^{k}\right)$.
\item $W^{1,p}\text{-}\mathrm{cl}\left(W^{2,p}\Omega_{\mathrm{hom}N}^{k}\right)=W^{1,p}\Omega_{N}^{k}$.
\end{enumerate}
\end{cor}

\begin{proof}
$ $
\begin{enumerate}
\item Because $\mathbb{P}W^{1,p}\Omega^{k}\leq W^{1,p}\Omega_{N}^{k},$
we conclude $\mathbb{P}W^{1,p}\Omega^{k}=\mathbb{P}W^{1,p}\Omega_{N}^{k}$.
Meanwhile, $\left(\mathcal{P}_{1}+\mathcal{P}_{3}^{\mathrm{ex}}\right)W^{1,p}\Omega_{N}^{k}=\left(1-\mathbb{P}\right)W^{1,p}\Omega_{N}^{k}\leq W^{1,p}\Omega_{N}^{k}$,
so $\left(\mathcal{P}_{1}+\mathcal{P}_{3}^{\mathrm{ex}}\right)W^{1,p}\Omega_{N}^{k}\leq d\left(W^{2,p}\Omega_{N}^{k-1}\right)\cap W^{1,p}\Omega_{N}^{k}=d\left(W^{2,p}\Omega_{\mathrm{hom}N}^{k-1}\right)$.
\item $L^{p}\text{-}\mathrm{cl}\left(\left(\mathcal{P}_{1}+\mathcal{P}_{3}^{\mathrm{ex}}\right)\Omega_{N}^{k+1}\right)=\left(\mathcal{P}_{1}+\mathcal{P}_{3}^{\mathrm{ex}}\right)L^{p}\text{-}\mathrm{cl}\left(\Omega_{N}^{k+1}\right)=\left(\mathcal{P}_{1}+\mathcal{P}_{3}^{\mathrm{ex}}\right)L^{p}\Omega^{k+1}.$
\item We are done if $W^{1,p}\text{-}\mathrm{cl}\left(\mathcal{P}^{N\perp}\Omega_{\mathrm{hom}N}^{k}\right)=\mathcal{P}^{N\perp}W^{1,p}\Omega_{N}^{k}$.\\
Recall $\mathcal{P}_{2}\left(\Omega_{\mathrm{hom}N}^{k}\right)\leq\Omega_{\mathrm{hom}N}^{k}$
and $\delta_{c}\left(\Omega_{N}^{k}\right)=\delta_{c}\mathcal{P}^{N\perp}\left(\Omega_{\hom N}^{k}\right)$
by \Corref{Sobolev_Hodge_projections}, so by the formula of $\left(d\oplus\delta_{c}\right)^{-1}$
from \Corref{poincare_ineq_Neumann}:
\[
\left(d\oplus\delta_{c}\right)^{-1}\left[d\left(\Omega_{\mathrm{hom}N}^{k}\right)\oplus\delta_{c}\left(\Omega_{N}^{k}\right)\right]=\left(d\oplus\delta_{c}\right)^{-1}\left[d\mathcal{P}^{N\perp}\left(\Omega_{\mathrm{hom}N}^{k}\right)\oplus\delta_{c}\mathcal{P}^{N\perp}\left(\Omega_{\hom N}^{k}\right)\right]=\mathcal{P}^{N\perp}\Omega_{\mathrm{hom}N}^{k}
\]
So
\begin{align*}
W^{1,p}\text{-}\mathrm{cl}\left(\mathcal{P}^{N\perp}\Omega_{\mathrm{hom}N}^{k}\right) & =\left(d\oplus\delta_{c}\right)^{-1}\left[L^{p}\text{-}\mathrm{cl}\left(d\left(\Omega_{\mathrm{hom}N}^{k}\right)\right)\oplus L^{p}\text{-}\mathrm{cl}\left(\delta_{c}\left(\Omega_{N}^{k}\right)\right)\right]\\
 & =\left(d\oplus\delta_{c}\right)^{-1}\left[d\left(W^{1,p}\Omega^{k}\right)\oplus\delta_{c}\left(W^{1,p}\Omega_{N}^{k}\right)\right]\\
 & =\mathcal{P}^{N\perp}W^{1,p}\Omega_{N}^{k}
\end{align*}
\end{enumerate}
\end{proof}

\subsection{An easy mistake\label{subsec:An-easy-mistake}}

Let $p\in\left(1,\infty\right),\omega\in\Omega_{\mathrm{hom}N}^{k}$.
In other words, $\mathbf{n}\omega=0$ and $\mathbf{n}d\omega=0$.
Using intuition from Euclidean space, it is tempting to conclude $\nabla_{\nu}\omega=0$,
but this is not true in general.

We will not use Penrose notation but work in local coordinates on
$\partial M$, with $\partial_{1},...\partial_{n-1}$ for directions
on $\partial M$ and $\partial_{n}$ for the direction of $\widetilde{\nu}$.
Let $\{a_{1},...,a_{k}\}\subset\{1,...,n-1$\}. Observe that $\mathbf{n}d\omega=0$
implies

\[
0=\left(d\omega\right)_{na_{1}...a_{k}}=\partial_{n}\omega_{a_{1}...a_{k}}+\sum_{i}(\pm1)\partial_{a_{i}}\omega_{na_{1}...\widehat{a_{i}}...a_{k}}=\partial_{n}\omega_{a_{1}...a_{k}}
\]
since $\omega_{na_{1}...\widehat{a_{i}}...a_{k}}=0$ on $\partial M$.
Then recall $\partial_{n}\omega_{a_{1}...a_{k}}=\left(\nabla_{n}\omega\right)_{a_{1}...a_{k}}+\Gamma*\omega$
where $\Gamma*\omega$ is schematic for some terms with the Christoffel
symbols. As $\Gamma$ is bounded on $M$, we conclude $\left|\mathbf{t}\nabla_{\nu}\omega\right|\lesssim|\omega|$
and $\left|\mathbf{t}\nabla_{\nu}\omega\right|_{\Lambda}\lesssim|\omega|_{\Lambda}$
on $\partial M$. Then 
\[
\iota_{\nu}d\left(|\omega|^{2}\right)=\nabla_{\nu}\left\langle \omega,\omega\right\rangle =2\left\langle \nabla_{\nu}\omega,\omega\right\rangle =2\left\langle \mathbf{t}\nabla_{\nu}\omega,\omega\right\rangle 
\]
so $\left|\nabla_{\nu}\left(|\omega|^{2}\right)\right|\lesssim\left|\omega\right|^{2}$
on $\partial M$. This will be important in establishing the $L^{p}$-analyticity
of the heat flow in \Subsecref{Lp-analyticity}.

\section{Heat flow\label{sec:Heat-flow}}

As promised, we now obtain a simple construction of the heat flow.
We still work on the same setting as in \Subsecref{The-setting}.

\subsection{$L^{2}$-analyticity\label{subsec:L2-analyticity}}

Recall that $\Delta_{N}$ is an unbounded operator on $\mathbb{R}\mathcal{P}^{N\perp}L^{2}\Omega^{k}$
and $\left(-\Delta_{N}\right)^{-1}$ is bounded. It is trivial to
check that $\left(-\Delta_{N}\right)^{-1}$ is symmetric, therefore
self-adjoint. Then $\Delta_{N}$ is also self-adjoint. Then for $\omega\in D(\Delta_{N})=\mathcal{P}^{N\perp}H^{2}\Omega_{\mathrm{hom}N}^{k}$:
$\left\langle \left\langle \Delta_{N}\omega,\omega\right\rangle \right\rangle _{\Lambda}=-\mathcal{D}(\omega,\omega)\leq0.$
So $\Delta_{N}$ is dissipative. Therefore, by a complexification
argument , $\Delta_{N}^{\mathbb{C}}$ is acutely sectorial of angle
$0$ by \Thmref{Dissipative} and $\left(e^{t\Delta_{N}^{\mathbb{C}}}\right)_{t\geq0}$
is a $C_{0}$, analytic semigroup on $\mathbb{C}\mathcal{P}^{N\perp}L^{2}\Omega^{k}$.
By \Blackboxref{analyticity}, we can derive some basic facts about
$e^{t\Delta_{N}}:$
\begin{itemize}
\item For $m\in\mathbb{N}_{1},$ $D(\Delta_{N}^{m})\leq\mathcal{P}^{N\perp}H^{2m}\Omega^{k}$
and $\left\Vert \Delta_{N}^{m}\omega\right\Vert _{L^{2}}\sim\left\Vert \omega\right\Vert _{H^{2m}}\sim\left\Vert \omega\right\Vert _{D\left(\Delta_{N}^{m}\right)}\;\forall\omega\in D\left(\Delta_{N}^{m}\right)$
by potential estimates. Recall that $\left(e^{t\Delta_{N}}\right)_{t\geq0}$
on $\left(D(\Delta_{N}^{m}),\left\Vert \cdot\right\Vert _{H^{2m}}\right)$
is also a $C_{0}$ semigroup by Sobolev tower (\Thmref{Sobolev_tower}).
\item For $t>0$, by either the spectral theorem (with a complexification
step) or semigroup theory, $e^{t\Delta_{N}}$ is a self-adjoint contraction
on $\mathbb{R}\mathcal{P}^{N\perp}L^{2}\Omega^{k}$, with image in
$D(\Delta_{N}^{\infty})\leq\mathcal{P}^{N\perp}\Omega^{k}$ by the
analyticity of $\left(e^{s\Delta_{N}^{\mathbb{C}}}\right)_{s\geq0}$.
\item $\forall\omega\in\mathcal{P}^{N\perp}L^{2}\Omega^{k},$ $\left((0,\infty)\to\mathcal{P}^{N\perp}\Omega^{k},t\mapsto e^{t\Delta_{N}}\omega\right)$
is $C^{\infty}$-continuous by Sobolev tower. Let $m\in\mathbb{N}_{1},$
then $\partial_{t}^{m}\left(e^{t\Delta_{N}}\omega\right)=\Delta_{N}^{m}e^{t\Delta_{N}}\omega$
and $\left\Vert e^{t\Delta_{N}}\omega\right\Vert _{H^{2m}}\sim\left\Vert \Delta_{N}^{m}e^{t\Delta_{N}}\omega\right\Vert _{L^{2}}\lesssim_{\neg m,\neg t}\frac{m^{m}}{t^{m}}\text{\ensuremath{\left\Vert \omega\right\Vert }}_{L^{2}}$
\end{itemize}
Next we define the \textbf{non-injective Neumann Laplacian} $\widetilde{\Delta_{N}}$
\nomenclature{$\widetilde{\Delta_{N}}$}{non-injective Neumann Laplacian\nomrefpage}
as an unbounded operator on $L^{2}\Omega^{k}$ with $D(\widetilde{\Delta_{N}}^{m})=D(\Delta_{N}^{m})\oplus\mathcal{H}_{N}^{k}$
and $\widetilde{\Delta_{N}}^{m}=\Delta_{N}^{m}\oplus0$ $\forall m\in\mathbb{N}_{1}$.
By using either the spectral theorem or checking the definitions manually,
$\widetilde{\Delta_{N}}$ is also a self-adjoint, dissipative operator.
Then we also get an analytic heat flow, and $\widetilde{\Delta_{N}}=\Delta_{N}\oplus0_{\mathcal{H}_{N}^{k}}$
with $e^{t\widetilde{\Delta_{N}}}=e^{t\Delta_{N}}\oplus\mathrm{Id}_{\mathcal{H}_{N}^{k}}$.

Recall that for $m\in\mathbb{N}_{0},p\in\left(1,\infty\right),\omega\in W^{m,p}\Omega^{k}:$
$\left\Vert \omega\right\Vert _{W^{m,p}}\sim\left\Vert \mathcal{P}^{N\perp}\omega\right\Vert _{W^{m,p}}+\left\Vert \mathcal{P}^{N}\omega\right\Vert _{\mathcal{H}_{N}^{k}}$
where we do not need to specify the norm on $\mathcal{H}_{N}^{k}$
as they're all equivalent. Then the previous results for $\Delta_{N}$
can easily be extended to $\widetilde{\Delta_{N}}:$
\begin{itemize}
\item For $m\in\mathbb{N}_{1},$ $D(\widetilde{\Delta_{N}}^{m})\leq H^{2m}\Omega^{k}$
and $\forall\omega\in D(\widetilde{\Delta_{N}}^{m})$: $\left\Vert \widetilde{\Delta_{N}}^{m}\omega\right\Vert _{L^{2}}\sim\left\Vert \mathcal{P}^{N\perp}\omega\right\Vert _{H^{2m}}$
and $\left\Vert \omega\right\Vert _{D(\widetilde{\Delta_{N}}^{m})}\sim\left\Vert \omega\right\Vert _{H^{2m}}$.
Recall $\left(e^{t\widetilde{\Delta_{N}}}\right)_{t\geq0}$ on $D(\widetilde{\Delta_{N}}^{m})$
is also an a $C_{0}$ semigroup. (Sobolev tower)
\item For $t>0$, by either the spectral theorem (with a complexification
step) or semigroup theory, $e^{t\widetilde{\Delta_{N}}}$ is a self-adjoint
contraction on $\mathbb{R}L^{2}\Omega^{k}$, with image in $D\left(\widetilde{\Delta_{N}}^{\infty}\right)\leq\Omega^{k}$.
\item $\forall\omega\in L^{2}\Omega^{k},$ $\left((0,\infty)\to\Omega^{k},t\mapsto e^{t\widetilde{\Delta_{N}}}\omega\right)$
is $C^{\infty}$-continuous by Sobolev tower. Let $m\in\mathbb{N}_{1},$
then $\partial_{t}^{m}\left(e^{t\widetilde{\Delta_{N}}}\omega\right)=\widetilde{\Delta_{N}}^{m}e^{t\widetilde{\Delta_{N}}}\omega$
and 
\[
\left\Vert e^{t\widetilde{\Delta_{N}}}\omega\right\Vert _{H^{2m}}\sim\left\Vert e^{t\widetilde{\Delta_{N}}}\mathcal{P}^{N\perp}\omega\right\Vert _{H^{2m}}+\left\Vert \mathcal{P}^{N}\omega\right\Vert _{\mathcal{H}_{N}^{k}}\lesssim_{\neg m,\neg t}\frac{m^{m}}{t^{m}}\left\Vert \mathcal{P}^{N\perp}\omega\right\Vert _{L^{2}}+\left\Vert \mathcal{P}^{N}\omega\right\Vert _{\mathcal{H}_{N}^{k}}
\]
\end{itemize}
By these estimates, we conclude that $e^{t\widetilde{\Delta_{N}}}\xrightarrow{t\to\infty}\mathcal{P}^{N}$
in $\mathcal{L}\left(L^{2}\Omega^{k}\right)$ (\textbf{Kodaira projection}).
In fact, this is how Hodge decomposition was done historically.

\subsection{$L^{p}$-analyticity\label{subsec:Lp-analyticity}}

Though we could use the same symbols $\Delta_{N}$ and $\widetilde{\Delta_{N}}$
for the Neumann Laplacian on $L^{p}$, that can create confusion regarding
the domains. Let them still refer to the unbounded operators on $\mathbb{R}\mathcal{P}^{N\perp}L^{2}\Omega^{k}$
and $\mathbb{R}L^{2}\Omega^{k}$ as before. However, $e^{t\Delta_{N}}$
and $e^{t\widetilde{\Delta_{N}}}$ are compatible across all $L^{p}$
spaces (as we will see).

First we note that $\Omega_{00}^{k}\leq D\left(\widetilde{\Delta_{N}}^{\infty}\right)$
so $D\left(\widetilde{\Delta_{N}}^{\infty}\right)$ is dense in $L^{p}$
$\forall p\in\left(1,\infty\right)$.

Then for $L^{p}$-analyticity, we make a Gronwall-type argument (adapted
from \parencite[Appendix A]{Isett2015_heat} to handle the boundary).
\begin{thm}[Local boundedness]
 \label{thm:local_boundedness_Lp} For $p\in\left(1,\infty\right),s\in\left(0,1\right)$
and $u\in D\left(\widetilde{\Delta_{N}}^{\infty}\right):$
\[
\left\Vert e^{s\widetilde{\Delta_{N}}}u\right\Vert _{p}\lesssim_{p}\left\Vert u\right\Vert _{p}
\]
\end{thm}

\begin{proof}
By duality and the density of $D\left(\widetilde{\Delta_{N}}^{\infty}\right)$
in $L^{2}\cap L^{p}$, WLOG assume $p\geq2.$ By complex interpolation
(with a complexification step), WLOG assume $p=4K$ where $K$ is
a large natural number.

Let $U(s)=e^{s\Delta_{N}}u$, so $\partial_{s}U=\Delta U$ and 
\[
\partial_{s}\left(|U|^{4K}\right)=2K|U|^{4K-2}\left\langle 2\Delta U,U\right\rangle \stackrel{\text{Bochner}}{=}2K|U|^{4K-2}\left(\Delta\left(|U|^{2}\right)-2\left|\nabla U\right|^{2}-2\left\langle \Ric\left(U\right),U\right\rangle \right)
\]

So 
\[
\partial_{s}\int_{M}|U|^{4K}\leq2K\int_{M}|U|^{4K-2}\Delta\left(|U|^{2}\right)+\mathcal{O}_{M,K}\left(\int_{M}|U|^{4K}\right)
\]

Let $f=\left|U\right|^{2}$. As $U\in D\left(\widetilde{\Delta_{N}}^{\infty}\right)\leq\Omega_{\hom N}^{k}$,
$\left|\nabla_{\nu}f\right|\lesssim f$ on $\partial M$ by \Subsecref{An-easy-mistake}.
By Gronwall, we just need $\int_{M}f^{2K-1}\Delta f\lesssim\int_{M}f^{2K}$
(pseudo-dissipativity). Simply integrate by parts:
\begin{align*}
\left\langle \left\langle \Delta f,f^{2K-1}\right\rangle \right\rangle  & =-\left\langle \left\langle df,d\left(f^{2K-1}\right)\right\rangle \right\rangle +\left\langle \left\langle \nabla_{\nu}f,f^{2K-1}\right\rangle \right\rangle \\
 & =-(2K-1)\int_{M}\left|df\right|^{2}f^{2K-2}+\mathcal{O}_{M}\left(\int_{\partial M}f^{2K}\right)\\
 & =-\frac{2K-1}{K^{2}}\int_{M}\left|d\left(f^{K}\right)\right|^{2}+\mathcal{O}_{M}\left(\int_{\partial M}f^{2K}\right)
\end{align*}

Let $F=|f|^{K}$. So for any $\varepsilon>0,$ we want $C_{\varepsilon}>0$
such that $\int_{\partial M}F^{2}\leq\varepsilon\int_{M}\left|dF\right|^{2}+C_{\varepsilon}\int_{M}F^{2}$.
This follows from Ehrling's inequality, and the fact that $H^{1}(M)\to L^{2}(\partial M)$
is compact.
\end{proof}
So $\left(e^{t\widetilde{\Delta_{N}}}\right)_{t\geq0}$ can be uniquely
extended by density to $L^{2}\Omega^{k}+L^{p}\Omega^{k}$ and $\restr{e^{t\widetilde{\Delta_{N}}}}{L^{p}\Omega^{k}}\in\mathcal{L}\left(L^{p}\Omega^{k}\right)$.
With a complexification step and an appropriate core chosen by Sobolev
embedding, local boundedness on $L^{p}$ implies $L^{p}$-analyticity
for all $p\in\left(1,\infty\right)$ by \Thmref{simple_core_extrapol}.

Let $A_{p}$ be the generator of $\left(e^{t\widetilde{\Delta_{N}}}\right)_{t\geq0}$
on $L^{p}\Omega^{k}.$\nomenclature{$A_p$}{generator of heat flow on $L^p$\nomrefpage}
By the definition of generator, $A_{p}=\widetilde{\Delta_{N}}$ on
$D\left(\widetilde{\Delta_{N}}^{\infty}\right)$. In our terminology,
$A_{p}^{\mathbb{C}}$ is acutely quasi-sectorial. But we want a more
concrete description of $D(A_{p})$.
\begin{lem}
Let $p\in\left(1,\infty\right)$. Then $\left(D(A_{p}),\left\Vert \cdot\right\Vert _{D(A_{p})}\right)\sim\left(W^{2,p}\Omega_{\hom N}^{k},\left\Vert \cdot\right\Vert _{W^{2,p}}\right)\;$
and $W^{2,p}\text{-}\mathrm{cl}\left(D\left(\widetilde{\Delta_{N}}^{\infty}\right)\right)=W^{2,p}\Omega_{\hom N}^{k}$.\label{lem:D(A_p)}
\end{lem}

\begin{proof}
Observe that $\forall u\in D\left(\widetilde{\Delta_{N}}^{\infty}\right):$
$\mathcal{P}^{N\perp}u\in D(\Delta_{N}^{\infty})$ and 
\[
\left\Vert u\right\Vert _{D(A_{p})}=\left\Vert u\right\Vert _{p}+\left\Vert \widetilde{\Delta_{N}}u\right\Vert _{p}\sim\left\Vert \mathcal{P}^{N}u\right\Vert _{\mathcal{H}_{N}^{k}}+\left\Vert \mathcal{P}^{N\perp}u\right\Vert _{p}+\left\Vert \Delta_{N}\mathcal{P}^{N\perp}u\right\Vert _{p}\sim\left\Vert \mathcal{P}^{N}u\right\Vert _{\mathcal{H}_{N}^{k}}+\left\Vert \mathcal{P}^{N\perp}u\right\Vert _{W^{2,p}}\sim\left\Vert u\right\Vert _{W^{2,p}}
\]
Then $\left\Vert \cdot\right\Vert _{D(A_{p})}\sim\left\Vert \cdot\right\Vert _{W^{2,p}}$
since $D\left(\widetilde{\Delta_{N}}^{\infty}\right)$ is a dense
core in $\left(D(A_{p}),\left\Vert \cdot\right\Vert _{D(A_{p})}\right)$
(see \Lemref{core}). This also implies $D(A_{p})=W^{2,p}\text{-}\mathrm{cl}\left(D\left(\widetilde{\Delta_{N}}^{\infty}\right)\right)=W^{2,p}\text{-}\mathrm{cl}\left(D\left(\Delta_{N}^{\infty}\right)\right)\oplus\mathcal{H}_{N}^{k}$.

Recall that $D\left(\Delta_{N}^{\infty}\right)\leq\left(\mathcal{P}^{N\perp}W^{2,p}\Omega_{\hom N}^{k},\left\Vert \cdot\right\Vert _{W^{2,p}}\right)\stackrel{\Delta_{N}}{\iso}\left(\mathcal{P}^{N\perp}L^{p}\Omega^{k},\left\Vert \cdot\right\Vert _{L^{p}}\right)$.
Since $L^{p}\text{-}\mathrm{cl}\left(\Delta_{N}D\left(\Delta_{N}^{\infty}\right)\right)=L^{p}\text{-}\mathrm{cl}\left(D\left(\Delta_{N}^{\infty}\right)\right)=\mathcal{P}^{N\perp}L^{p}\Omega^{k},$
we conclude $W^{2,p}\text{-}\mathrm{cl}\left(D\left(\Delta_{N}^{\infty}\right)\right)=\left(-\Delta_{N}\right)^{-1}\left(\mathcal{P}^{N\perp}L^{p}\Omega^{k}\right)=\mathcal{P}^{N\perp}W^{2,p}\Omega_{\hom N}^{k}$
and we are done.
\end{proof}
So for $p\in\left(1,\infty\right),s\in\left(0,1\right)$ and $u\in L^{p}\Omega^{k}$:
$\left\Vert e^{s\widetilde{\Delta_{N}}}u\right\Vert _{W^{2,p}}\lesssim\frac{1}{s}\left\Vert u\right\Vert _{p}$.
That implies $\left\Vert e^{s\widetilde{\Delta_{N}}}u\right\Vert _{W^{1,p}}\lesssim\frac{1}{\sqrt{s}}\left\Vert u\right\Vert _{p}$
by complex interpolation (with complexification), using $\left[\mathbb{C}L^{p},\mathbb{C}W^{2,p}\right]_{\frac{1}{2}}=\left[\mathbb{C}F_{p,2}^{0},\mathbb{C}F_{p,2}^{2}\right]_{\frac{1}{2}}=\mathbb{C}F_{p,2}^{1}$.

Obviously, $D\left(A_{p}^{\infty}\right)=\{\omega\in\Omega_{\mathrm{\hom}N}^{k}:\Delta^{m}\omega\in W^{2,p}\Omega_{\hom N}^{k}\;\forall m\in\mathbb{N}_{0}\}=D\left(\widetilde{\Delta_{N}}^{\infty}\right)$
by Sobolev embedding.

Additionally, by the density of $D\left(A_{p}^{\infty}\right)$ in
$L^{p}$, we can show by approximation that 
\[
\left\langle \left\langle e^{t\widetilde{\Delta_{N}}}\omega,\eta\right\rangle \right\rangle _{\Lambda}=\left\langle \left\langle \omega,e^{t\widetilde{\Delta_{N}}}\eta\right\rangle \right\rangle _{\Lambda}\;\forall\omega\in L^{p}\Omega^{k},\eta\in L^{p'}\Omega^{k},p\in\left(1,\infty\right),t\geq0
\]

This implies that $e^{t\widetilde{\Delta_{N}}}\mathcal{P}^{N\perp}=\mathcal{P}^{N\perp}e^{t\widetilde{\Delta_{N}}}$
on $W^{m,p}\Omega^{k}$ $\forall m\in\mathbb{N}_{0},\forall p\in\left(1,\infty\right)$.

\subsection{$W^{1,p}$-analyticity\label{subsec:W1p-analyticity}}

We first observe that $W^{1,p}\text{-}\mathrm{cl}\left(D\left(\widetilde{\Delta_{N}}^{\infty}\right)\right)=W^{1,p}\text{-}\mathrm{cl}\left(W^{2,p}\text{-}\mathrm{cl}\left(D\left(\widetilde{\Delta_{N}}^{\infty}\right)\right)\right)=W^{1,p}\text{-}\mathrm{cl}\left(W^{2,p}\Omega_{\hom N}^{k}\right)=W^{1,p}\Omega_{N}^{k}$
by \Corref{crucial_boundary_approx_homN} and \Lemref{D(A_p)}.

Because we will soon be dealing with differential forms of different
degrees, define $\Omega(M)=\bigoplus_{k=0}^{n}\Omega^{k}(M)$ as the
\textbf{graded algebra} of differential forms where multiplication
is the wedge product. We simply define $W^{m,p}\Omega(M)=\bigoplus_{k=0}^{n}W^{m,p}\Omega^{k}(M)$,
and similarly for $B_{p,q}^{s},F_{p,q}^{s}$ spaces. Spaces like $\Omega_{D}\left(M\right)$,
$\Omega_{00}\left(M\right)$ or $W^{m,p}\Omega_{\mathrm{hom}N}$ are
also defined by direct sums. The dot products $\left\langle \cdot,\cdot\right\rangle _{\Lambda}$
and $\left\langle \left\langle \cdot,\cdot\right\rangle \right\rangle _{\Lambda}$
are also definable as the sum from each degree. Also define $\mathcal{H}(M)=\bigoplus_{k=0}^{n}\mathcal{H}^{k}(M)$.

As an example, $\omega\in L^{2}\Omega\left(M\right)$ and $\eta\in L^{2}\Omega\left(M\right)$
would imply $\omega\wedge\eta\in L^{1}\Omega\left(M\right)$. We also
recover integration by parts: 
\[
\left\langle \left\langle d\omega,\eta\right\rangle \right\rangle _{\Lambda}=\left\langle \left\langle \omega,\delta\eta\right\rangle \right\rangle _{\Lambda}+\left\langle \left\langle \jmath^{*}\omega,\jmath^{*}\iota_{\nu}\eta\right\rangle \right\rangle _{\Lambda}\;\forall\omega\in\mathbb{R}W^{1,p}\Omega\left(M\right),\forall\eta\in\mathbb{R}W^{1,p'}\Omega\left(M\right),p\in\left(1,\infty\right)
\]

Then we can set $D\left(\widetilde{\Delta_{N}}\right)=H^{2}\Omega_{\hom N}$
and $D\left(A_{p}\right)=W^{2,p}\Omega_{\hom N}$ ($p\in\left(1,\infty\right)$),
and previous results such as sectoriality or the Poincare inequality
still hold true in this new degree-independent framework, \emph{mutatis
mutandis}.
\begin{thm}[Commuting with derivatives I]
 \label{thm:commute_derivatives_I}Let $p\in\left(1,\infty\right).$
\begin{enumerate}
\item $\delta_{c}\left(D\left(A_{p}^{\infty}\right)\right)\leq D\left(A_{p}^{\infty}\right)$
and $d\left(D\left(A_{p}^{\infty}\right)\right)\leq D\left(A_{p}^{\infty}\right)$
\item Let $\omega\in D(A_{p})=W^{2,p}\Omega_{\hom N}$ and $\mathfrak{D}\in\{d,\delta_{c},\delta_{c}d,d\delta_{c}\}$.
Then for $t>0:$ $\mathfrak{D}e^{t\widetilde{\Delta_{N}}}\omega=e^{t\widetilde{\Delta_{N}}}\mathfrak{D}\omega$.
\end{enumerate}
\end{thm}

\begin{proof}
$ $
\begin{enumerate}
\item Let $\eta\in D\left(A_{p}^{\infty}\right)$. Obviously $d\eta\in W^{2,p}\Omega_{\hom N}$,
so $d\Delta^{m}\eta\in W^{2,p}\Omega_{\hom N}\;\forall m\in\mathbb{N}_{0}$.\\
Observe that $\mathbf{n}\eta=0$ implies $\mathbf{n}\delta\eta=0$,
and $\mathbf{n}d\eta=0$ implies $\mathbf{n}\delta d\eta=0.$ But
$\mathbf{n}\Delta\eta=0$ so $\mathbf{n}d\delta\eta=0$ and we conclude
$\delta_{c}\eta\in W^{2,p}\Omega_{\hom N}$. Similarly, $\delta_{c}\Delta^{m}\eta\in W^{2,p}\Omega_{\hom N}\;\forall m\in\mathbb{N}_{0}$.
\item Let $t>0$. Note that $\mathfrak{D}e^{t\widetilde{\Delta_{N}}}\omega\in D\left(A_{p}^{\infty}\right).$\\
Then $\frac{e^{h\widetilde{\Delta_{N}}}-1}{h}e^{t\widetilde{\Delta_{N}}}\omega\xrightarrow[h\downarrow0]{C^{\infty}}\widetilde{\Delta_{N}}e^{t\widetilde{\Delta_{N}}}\omega$
so $\partial_{t}\left(\mathfrak{D}e^{t\widetilde{\Delta_{N}}}\omega\right)=\mathfrak{D}\widetilde{\Delta_{N}}e^{t\widetilde{\Delta_{N}}}\omega=\widetilde{\Delta_{N}}\mathfrak{D}e^{t\widetilde{\Delta_{N}}}\omega$.
Therefore 
\[
e^{h\widetilde{\Delta_{N}}}\mathfrak{D}e^{t\widetilde{\Delta_{N}}}\omega=\mathfrak{D}e^{\left(t+h\right)\widetilde{\Delta_{N}}}\omega\;\forall t>0,\forall h>0
\]
Note that $\mathfrak{D}e^{\left(t+h\right)\widetilde{\Delta_{N}}}\omega\xrightarrow[t\downarrow0]{L^{p}}\mathfrak{D}e^{h\widetilde{\Delta_{N}}}\omega$
since $e^{\left(t+h\right)\widetilde{\Delta_{N}}}\omega\xrightarrow[t\downarrow0]{C^{\infty}}e^{h\widetilde{\Delta_{N}}}\omega$.\\
On the other hand, $e^{h\widetilde{\Delta_{N}}}\mathfrak{D}e^{t\widetilde{\Delta_{N}}}\omega\xrightarrow[t\downarrow0]{L^{p}}e^{h\widetilde{\Delta_{N}}}\mathfrak{D}\omega$
as $e^{t\widetilde{\Delta_{N}}}\omega\xrightarrow[t\downarrow0]{W^{2,p}}\omega$
(\uline{why we need \mbox{$\omega\in D(A_{p})$}}).\\
So $\mathfrak{D}e^{h\widetilde{\Delta_{N}}}\omega=e^{h\widetilde{\Delta_{N}}}\mathfrak{D}\omega\;\forall h>0$.
\end{enumerate}
\end{proof}
We can extend this via complexification. For $\omega\in\mathbb{C}W^{2,p}\Omega_{\mathrm{hom}N},$
$\mathfrak{D}^{\mathbb{C}}e^{t\widetilde{\Delta_{N}^{\mathbb{C}}}}\omega=e^{t\widetilde{\Delta_{N}^{\mathbb{C}}}}\mathfrak{D}^{\mathbb{C}}\omega\;\forall t>0$.

By $L^{p}$-analyticity, $\exists\alpha=\alpha(p)>0$ such that $\left(e^{z\widetilde{\Delta_{N}^{\mathbb{C}}}}\right)_{z\in\Sigma_{\alpha}^{+}\cup\{0\}}$
is a $C_{0}$, locally bounded, analytic semigroup on $\mathbb{C}L^{p}\Omega$.
Then by the identity theorem, $\mathfrak{D}^{\mathbb{C}}e^{z\widetilde{\Delta_{N}^{\mathbb{C}}}}\omega=e^{z\widetilde{\Delta_{N}^{\mathbb{C}}}}\mathfrak{D}^{\mathbb{C}}\omega\;\forall z\in\Sigma_{\alpha}^{+}$.
\begin{thm}[$W^{1,p}$-analyticity]
  $\left(e^{z\widetilde{\Delta_{N}^{\mathbb{C}}}}\right)_{z\in\Sigma_{\alpha}^{+}\cup\{0\}}$
is a $C_{0}$, analytic semigroup on $\mathbb{C}W^{1,p}\Omega_{N}$.
\end{thm}

\begin{proof}
Note that $\left(D(A_{p}^{\mathbb{C}}),\left\Vert \cdot\right\Vert _{W^{2,p}}\right)$
is dense in $\left(\mathbb{C}W^{1,p}\Omega_{N},\left\Vert \cdot\right\Vert _{W^{1,p}}\right)$
by \Corref{crucial_boundary_approx_homN}.

So by \Lemref{from_core_U}, we just need to show $\left(e^{z\widetilde{\Delta_{N}^{\mathbb{C}}}}\right)_{z\in\Sigma_{\alpha}^{+}\cup\{0\}}\subset\mathcal{L}\left(\mathbb{C}W^{1,p}\Omega_{N}\right)$
and is locally bounded. So it is enough to show 
\[
\left\Vert e^{z\widetilde{\Delta_{N}^{\mathbb{C}}}}u\right\Vert _{W^{1,p}}\lesssim\left\Vert u\right\Vert _{W^{1,p}}\;\forall u\in D\left(A_{p}^{\mathbb{C}}\right),\forall z\in\mathbb{D}\cap\Sigma_{\alpha}^{+}
\]
Consider $\mathcal{P}^{N\perp}u$, then we only need $\left\Vert e^{z\widetilde{\Delta_{N}^{\mathbb{C}}}}u\right\Vert _{W^{1,p}}\lesssim\left\Vert u\right\Vert _{W^{1,p}}\;\forall u\in\mathcal{P}^{N\perp}D\left(A_{p}^{\mathbb{C}}\right),\forall z\in\mathbb{D}\cap\Sigma_{\alpha}^{+}$.

Recall $e^{t\widetilde{\Delta_{N}}}\mathcal{P}^{N\perp}=\mathcal{P}^{N\perp}e^{t\widetilde{\Delta_{N}}}$
from \Subsecref{Lp-analyticity}. By the Poincare inequality (\Corref{poincare_ineq_Neumann}):
\begin{eqnarray*}
\left\Vert e^{z\widetilde{\Delta_{N}^{\mathbb{C}}}}u\right\Vert _{W^{1,p}} & \sim & \left\Vert d^{\mathbb{C}}e^{z\widetilde{\Delta_{N}^{\mathbb{C}}}}u\right\Vert _{p}+\left\Vert \delta_{c}^{\mathbb{C}}e^{z\widetilde{\Delta_{N}^{\mathbb{C}}}}u\right\Vert _{p}=\left\Vert e^{z\widetilde{\Delta_{N}^{\mathbb{C}}}}d^{\mathbb{C}}u\right\Vert _{p}+\left\Vert e^{z\widetilde{\Delta_{N}^{\mathbb{C}}}}\delta_{c}^{\mathbb{C}}u\right\Vert _{p}\\
 & \lesssim & \left\Vert d^{\mathbb{C}}u\right\Vert _{p}+\left\Vert \delta_{c}^{\mathbb{C}}u\right\Vert _{p}\sim\left\Vert u\right\Vert _{W^{1,p}}\;\forall u\in\mathcal{P}^{N\perp}D\left(A_{p}^{\mathbb{C}}\right),\forall z\in\mathbb{D}\cap\Sigma_{\alpha}^{+}
\end{eqnarray*}
\end{proof}
\begin{cor}
\label{cor:commute_derivatives_II}Let $\omega\in W^{1,p}\Omega_{N}$
and $\mathfrak{D}\in\{d,\delta_{c}\}$. Then for $t>0:$ $\mathfrak{D}e^{t\widetilde{\Delta_{N}}}\omega=e^{t\widetilde{\Delta_{N}}}\mathfrak{D}\omega$.
\end{cor}

\begin{proof}
Same as before, but with $e^{t\widetilde{\Delta_{N}}}\omega\xrightarrow[t\downarrow0]{W^{1,p}}\omega$.
\end{proof}
Let $A_{1,p}$ be the generator of $\left(e^{t\widetilde{\Delta_{N}}}\right)_{t\geq0}$
on $W^{1,p}\Omega_{N}$. \nomenclature{$A_{1,p}$}{generator of heat flow on $W^{1,p}$\nomrefpage}
Then $A_{1,p}$ and $A_{p}$ agree on $D\left(A_{p}^{2}\right)$ by
the definition of generators, so $A_{1,p}=\widetilde{\Delta_{N}}$
on $D\left(\widetilde{\Delta_{N}}^{\infty}\right)$. By potential
estimates, $\left\Vert \cdot\right\Vert _{D\left(A_{1,p}\right)}\sim\left\Vert \cdot\right\Vert _{W^{3,p}}$
on $D\left(\widetilde{\Delta_{N}}^{\infty}\right)$ and therefore
on $\left\Vert \cdot\right\Vert _{W^{3,p}}\mathrm{\text{-}\mathrm{cl}}\left(D\left(\widetilde{\Delta_{N}}^{\infty}\right)\right)=D\left(A_{1,p}\right)$.
By the same argument as in \Lemref{D(A_p)}, $D\left(A_{1,p}\right)=\left(-\Delta_{N}\right)^{-1}\left(\mathcal{P}^{N\perp}W^{1,p}\Omega_{N}\right)\oplus\mathcal{H}_{N}\geq D\left(\widetilde{\Delta_{N}}^{\infty}\right)$.
\begin{thm}[Compatibility with Hodge-Helmholtz]
 \label{thm:compat_Helmholtz_leray}Let $m\in\mathbb{N}_{0},p\in(1,\infty),t>0$.
By \Corref{commute_derivatives_II} and \Corref{Sobolev_Hodge_projections}:
\begin{itemize}
\item $e^{t\widetilde{\Delta_{N}}}d\left(W^{m+1,p}\Omega_{N}\right)=d\left(e^{t\widetilde{\Delta_{N}}}W^{m+1,p}\Omega_{N}\right)\leq d\left(\Omega_{N}\right)=d\left(\Omega\right)$.
\item $e^{t\widetilde{\Delta_{N}}}\delta_{c}\left(W^{m+1,p}\Omega_{N}\right)=\delta_{c}\left(e^{t\widetilde{\Delta_{N}}}W^{m+1,p}\Omega_{N}\right)\leq\delta_{c}\left(\Omega_{N}\right)=\delta_{c}\left(\Omega_{\mathrm{hom}N}\right)$.
\end{itemize}
As $e^{t\widetilde{\Delta_{N}}}=1$ on $\mathcal{H}_{N}$, we finally
conclude $e^{t\widetilde{\Delta_{N}}}\left(\mathcal{P}_{3}^{\mathrm{ex}}+\mathcal{P}_{1}\right)=\left(\mathcal{P}_{3}^{\mathrm{ex}}+\mathcal{P}_{1}\right)e^{t\widetilde{\Delta_{N}}},e^{t\widetilde{\Delta_{N}}}\mathcal{P}_{2}=\mathcal{P}_{2}e^{t\widetilde{\Delta_{N}}}$
and $e^{t\widetilde{\Delta_{N}}}\mathcal{P}_{3}^{N}=\mathcal{P}_{3}^{N}e^{t\widetilde{\Delta_{N}}}=\mathcal{P}_{3}^{N}$
on $W^{m,p}\Omega\left(M\right)$. Also, $e^{t\widetilde{\Delta_{N}}}\mathbb{P}=\mathbb{P}e^{t\widetilde{\Delta_{N}}}$
on $W^{m,p}\Omega\left(M\right)$ where $\mathbb{P}$ is the Leray
projection.

By the definition of generators, 
\[
\widetilde{\Delta_{N}}\left(\mathcal{P}_{3}^{\mathrm{ex}}+\mathcal{P}_{1}\right)=\left(\mathcal{P}_{3}^{\mathrm{ex}}+\mathcal{P}_{1}\right)\widetilde{\Delta_{N}},\mathcal{P}_{3}^{N}\widetilde{\Delta_{N}}=\widetilde{\Delta_{N}}\mathcal{P}_{3}^{N}=0,\mathcal{P}_{2}\widetilde{\Delta_{N}}=\widetilde{\Delta_{N}}\mathcal{P}_{2}=\widetilde{\Delta_{N}}\mathbb{P}=\mathbb{P}\widetilde{\Delta_{N}}
\]
 on $D(A_{p})=W^{2,p}\Omega_{\mathrm{hom}N}$.
\end{thm}

We briefly note that in the no-boundary case, we have $\Omega=\Omega_{N}=\Omega_{\mathrm{hom}N}$,
$\widetilde{\Delta_{N}}=\widetilde{\Delta_{D}}=\Delta$, $e^{t\Delta}\mathcal{P}_{1}=\mathcal{P}_{1}e^{t\Delta}$
on $W^{m,p}\Omega$, $\mathcal{P}_{1}\Delta=\Delta\mathcal{P}_{1}$
on $W^{2,p}\Omega$.
\begin{rem*}
The operator $\mathbb{P}\widetilde{\Delta_{N}}$, with the domain
$\mathbb{P}D(A_{p})$, is a well-defined unbounded operator on $\mathbb{P}L^{p}\Omega$.
By our arguments, its complexification is acutely sectorial, and $\mathbb{P}\widetilde{\Delta_{N}}=\widetilde{\Delta_{N}},e^{t\mathbb{P}\widetilde{\Delta_{N}}}=e^{t\widetilde{\Delta_{N}}}$
on $\mathbb{P}L^{p}\Omega$. Other authors call it the \textbf{Stokes
operator} corresponding to the ``Navier-type'' / ``free'' boundary
condition \parencite{Miyakawa_Lp_analyticity_Navier,Giga1982_Nonstationary_first_order,Mitrea2009_Sylvie_Neumann_Laplacian,mitrea2009_Nonlinear_hodge_NS,Baba2016}.
\end{rem*}

\subsection{Distributions and adjoints\label{subsec:Distributions-and-adjoints}}

Like the Littlewood-Paley projection, the heat flow does not preserve
compact supports in $\accentset{\circ}{M}$. So applying the heat
flow to a distribution is not well-defined. This can be a problem
as we will need to heat up the nonlinear term in the Euler equation
for Onsager's conjecture. For the Littlewood-Paley projection, we
fixed it by introducing tempered distributions. That in turn motivates
the following definition.
\begin{defn}
Let $I\subset\mathbb{R}$ be an open interval. Define
\begin{itemize}
\item $\mathscr{D}\Omega^{k}=\Omega_{00}^{k}=\colim\{\left(\Omega_{00}^{k}\left(K\right),C^{\infty}\text{ topo}\right):K\subset\accentset{\circ}{M}\text{ compact}\}$
as the space of \textbf{test $k$-forms} with Schwartz's topology
(colimit in the category of locally convex TVS).
\item $\mathscr{D}'\Omega^{k}=\left(\mathscr{D}\Omega^{k}\right)^{*}$ as
the space of\textbf{ $k\text{-}$currents} (or \textbf{distributional
$k$}-\textbf{forms}), equipped with the weak{*} topology.
\item $\mathscr{D}_{N}\Omega^{k}=D\left(\widetilde{\Delta_{N}}^{\infty}\right)$
as the space of \textbf{heated $k$-forms }with the Frechet $C^{\infty}$
topology and $\mathscr{D}'_{N}\Omega^{k}=\left(\mathscr{D}_{N}\Omega^{k}\right)^{*}$
as the space of\textbf{ heatable $k$-currents} (or \textbf{heatable
distributional $k$}-\textbf{forms}) with the weak{*} topology. \nomenclature{$\mathscr{D}_{N}\Omega^{k}, \mathscr{D}'_{N}\Omega^{k}$}{heated forms and heatable currents\nomrefpage}
\item \textbf{Spacetime test forms}: $\mathscr{D}\left(I,\Omega^{k}\right)=C_{c}^{\infty}\left(I,\Omega_{00}^{k}\right)=\colim\{\left(C_{c}^{\infty}\left(I_{1},\Omega_{00}^{k}(K)\right),C^{\infty}\text{ topo}\right):I_{1}\times K\subset I\times\accentset{\circ}{M}\text{ compact}\}$
and $\mathscr{D}_{N}\left(I,\Omega^{k}\right)=\colim\{\left(C_{c}^{\infty}\left(I_{1},\mathscr{D}_{N}\Omega^{k}\right),C^{\infty}\text{ topo}\right):I_{1}\subset I\text{ compact}\}$.
\item \textbf{Spacetime distributions} $\mathscr{D}'\left(I,\Omega^{k}\right)=\mathscr{D}\left(I,\Omega^{k}\right)^{*}$,
$\mathscr{D}'_{N}\left(I,\Omega^{k}\right)=\mathscr{D}_{N}\left(I,\Omega^{k}\right)^{*}.$
\end{itemize}
Obviously $\mathscr{D}\Omega^{k}\xhookrightarrow{i}\mathscr{D}_{N}\Omega^{k}$,
so there is an adjoint $\mathscr{D}'_{N}\Omega^{k}\xrightarrow{i^{*}}\mathscr{D}'\Omega^{k}.$
Unfortunately, $\Imm(i)$ is not dense so $i^{*}$ is not injective.
Nevertheless, we will make $i^{*}$ the implicit canonical map from
$\mathscr{D}'_{N}$ to $\mathscr{D}'$. In particular, $\omega_{j}\xrightarrow{\mathscr{D}'_{N}}0$
implies $\omega_{j}\xrightarrow{\mathscr{D'}}0$. Similarly, $\mathscr{D}\left(I,\Omega^{k}\right)\hookrightarrow\mathscr{D}_{N}\left(I,\Omega^{k}\right)$
and $\mathscr{D}'_{N}\left(I,\Omega^{k}\right)\to\mathscr{D}'\left(I,\Omega^{k}\right)$.

By Sobolev tower (\Thmref{Sobolev_tower}), we observe that $e^{t\widetilde{\Delta_{N}}}\phi\xrightarrow[t\downarrow0]{C^{\infty}}\phi\;\forall\phi\in\mathscr{D}{}_{N}\Omega^{k}$.

For $\Lambda\in\mathscr{D}'_{N}\Omega^{k}$, $t\geq0$ and $\phi\in\mathscr{D}{}_{N}\Omega^{k}$,
we define $e^{t\widetilde{\Delta_{N}}}\Lambda\left(\phi\right)=\Lambda\left(e^{t\widetilde{\Delta_{N}}}\phi\right)$.
As $\Lambda$ is continuous, $\exists m_{0},m_{1}\in\mathbb{N}_{0}$
such that $\left|\Lambda\left(\phi\right)\right|\lesssim\left\Vert \phi\right\Vert _{C^{m_{0}}}\lesssim\left\Vert \phi\right\Vert _{H^{m_{1}}}$.
Then for $t>0$ and $\phi\in\mathscr{D}{}_{N}\Omega^{k}$: $\left|e^{t\widetilde{\Delta_{N}}}\Lambda\left(\phi\right)\right|\lesssim\left\Vert e^{t\widetilde{\Delta_{N}}}\phi\right\Vert _{H^{m_{1}}}\lesssim_{t,m_{1}}\left\Vert \phi\right\Vert _{L^{2}}$
$\implies e^{t\widetilde{\Delta_{N}}}\Lambda\in L^{2}\Omega^{k}$
and $e^{t\widetilde{\Delta_{N}}}\Lambda=e^{\frac{t}{2}\widetilde{\Delta_{N}}}e^{\frac{t}{2}\widetilde{\Delta_{N}}}\Lambda\in\mathscr{D}{}_{N}\Omega^{k}$.

Also, for $p\in\left(1,\infty\right)$ and $\omega\in L^{p}\Omega^{k}$,
$e^{t\widetilde{\Delta_{N}}}\omega$ is the same in $L^{p}\Omega^{k}$
and $\mathscr{D}'_{N}\Omega^{k}$.
\end{defn}

\begin{rem*}
We note an important limitation: though heated forms are closed under
$d$ and $\delta$ by \Thmref{commute_derivatives_I}, because of
integration by parts, we cannot naively define $\delta$ or $\Delta$
on heatable currents.

Analogous concepts such as $\mathscr{D}_{D}$ and $\mathscr{D}'_{D}$
can be defined via Hodge duality for the relative Dirichlet heat flow.
\end{rem*}
Recall the graded algebra $\Omega(M)=\bigoplus_{k=0}^{n}\Omega^{k}(M)$
from \Subsecref{W1p-analyticity}. We can easily define $\mathscr{D}\Omega,\mathscr{D}_{N}\Omega$
etc. by direct sums.

For $\Lambda\in\mathscr{D}'_{N}\Omega$ and $\phi\in\mathscr{D}{}_{N}\Omega$,
we can define $\delta_{c}^{\mathscr{D}'_{N}}\Lambda\left(\phi\right)=\Lambda\left(d\phi\right)$
and $d^{\mathscr{D}'_{N}}\Lambda\left(\phi\right)=\Lambda\left(\delta_{c}\phi\right)$.
These will be consistent with the smooth versions, though we take
care to note that
\begin{equation}
\left\langle \left\langle \delta_{c}^{\mathscr{D}'_{N}}\omega,\phi\right\rangle \right\rangle _{\Lambda}=\left\langle \left\langle \omega,d\phi\right\rangle \right\rangle _{\Lambda}=\left\langle \left\langle \delta\omega,\phi\right\rangle \right\rangle _{\Lambda}+\left\langle \left\langle \jmath^{*}\iota_{\nu}\omega,\jmath^{*}\phi\right\rangle \right\rangle _{\Lambda}\;\forall\omega\in W^{1,p}\Omega,\phi\in\mathscr{D}_{N}\Omega,p\in\left(1,\infty\right)\label{eq:delta_D'_N}
\end{equation}
So $\delta_{c}^{\mathscr{D}'_{N}}$ agrees with $\delta_{c}$ on $W^{1,p}\Omega_{N}$
as defined previously. In particular, $\widetilde{\Delta_{N}}^{\mathscr{D}'_{N}}=-\left(d^{\mathscr{D}'_{N}}\delta_{c}^{\mathscr{D}'_{N}}+\delta_{c}^{\mathscr{D}'_{N}}d^{\mathscr{D}'_{N}}\right)$
is well-defined on $\mathscr{D}'_{N}\Omega$.

Note that $\delta^{\mathscr{D}'_{N}}\Lambda$ cannot be defined since
there is $\phi\in\mathscr{D}{}_{N}\Omega$ such that $d_{c}\phi$
is not defined.

For convenience, we also write $\Lambda\left(\phi\right)=\left\langle \left\langle \Lambda,\phi\right\rangle \right\rangle _{\Lambda}$
(abuse of notation) and $\Lambda^{\varepsilon}=e^{\varepsilon\widetilde{\Delta_{N}}}\Lambda$
for $\varepsilon>0$. Observe that for all $\Lambda\in\mathscr{D}'_{N}\Omega,\phi\in\mathscr{D}_{N}\Omega:$
\[
\left\langle \left\langle d\left(\Lambda^{\varepsilon}\right),\phi\right\rangle \right\rangle _{\Lambda}=\left\langle \left\langle \Lambda^{\varepsilon},\delta_{c}\phi\right\rangle \right\rangle _{\Lambda}=\left\langle \left\langle \Lambda,\left(\delta_{c}\phi\right)^{\varepsilon}\right\rangle \right\rangle _{\Lambda}=\left\langle \left\langle \Lambda,\delta_{c}\left(\phi^{\varepsilon}\right)\right\rangle \right\rangle _{\Lambda}=\left\langle \left\langle \left(d^{\mathscr{D}'_{N}}\Lambda\right)^{\varepsilon},\phi\right\rangle \right\rangle _{\Lambda}
\]
Then $d\left(\Lambda^{\varepsilon}\right)=\left(d^{\mathscr{D}'_{N}}\Lambda\right)^{\varepsilon}$
and similarly $\delta_{c}\left(\Lambda^{\varepsilon}\right)=\left(\delta_{c}^{\mathscr{D}'_{N}}\Lambda\right)^{\varepsilon}\forall\Lambda\in\mathscr{D}'_{N}\Omega$.
\begin{problem*}[Consistency problem]
 For $p\in\left(1,\infty\right)$, we have $L^{p}\Omega\hookrightarrow\mathscr{D}'_{N}\Omega$
and $L^{p}\Omega\hookrightarrow\mathscr{D}'\Omega$, and we can identify
$\mathscr{D}'_{N}\Omega\cap L^{p}\Omega=\mathscr{D}'\Omega\cap L^{p}\Omega=L^{p}\Omega$.
Let $d^{\mathscr{D'}}$ and $d^{\mathscr{D}'_{N}}$ be $d$ defined
on $\mathscr{D'}$ and $\mathscr{D}'_{N}$ respectively. For $\omega\in L^{p}\Omega,$
if $d^{\mathscr{D'}}\omega\in\mathscr{D}'\Omega\cap L^{p}\Omega$
, the question is whether we can say $d^{\mathscr{D}'_{N}}\omega\in\mathscr{D}'_{N}\Omega\cap L^{p}\Omega$.

More explicitly, if $\alpha,\omega\in L^{p}\Omega$ and $\left\langle \left\langle \alpha,\phi_{0}\right\rangle \right\rangle _{\Lambda}=\left\langle \left\langle \omega,\delta_{c}\phi_{0}\right\rangle \right\rangle _{\Lambda}\;\forall\phi_{0}\in\mathscr{D}\Omega,$
can we say $\left\langle \left\langle \alpha,\phi\right\rangle \right\rangle _{\Lambda}=\left\langle \left\langle \omega,\delta_{c}\phi\right\rangle \right\rangle _{\Lambda}\;\forall\phi\in\mathscr{D}_{N}\Omega$?
The answer is yes, and the method is analogous to some key steps in
\Subsecref{Justification} and \Subsecref{Interpolation-and-B-analyticity}.
\end{problem*}
Recall the cutoffs $\psi_{r}$ from \Eqref{cutoff}.
\begin{lem}
\label{lem:approx_consistency}Let $p\in\left(1,\infty\right)$ and
$\phi\in W^{1,p}\Omega_{N}^{k}.$ Then $\left(1-\psi_{r}\right)\phi\xrightarrow[r\downarrow0]{L^{p}}\phi$
and $\delta_{c}\left(\left(1-\psi_{r}\right)\phi\right)\xrightarrow[r\downarrow0]{L^{p}}\delta_{c}\phi$.
\end{lem}

\begin{proof}
In Penrose notation, 
\begin{align*}
\delta_{c}\left(\left(1-\psi_{r}\right)\phi\right)_{a_{1}...a_{k-1}} & =-\nabla^{i}\left(\left(1-\psi_{r}\right)\phi\right)_{ia_{1}...a_{k-1}}=\nabla^{i}\psi_{r}\phi_{ia_{1}...a_{k-1}}-\left(1-\psi_{r}\right)\nabla^{i}\phi_{ia_{1}...a_{k-1}}\\
\implies\delta_{c}\left(\left(1-\psi_{r}\right)\phi\right) & =\iota_{\nabla\psi_{r}}\phi+\left(1-\psi_{r}\right)\delta_{c}\phi=f_{r}\iota_{\widetilde{\nu}}\phi+\left(1-\psi_{r}\right)\delta_{c}\phi
\end{align*}
Then we only need $f_{r}\iota_{\widetilde{\nu}}\phi\xrightarrow[r\downarrow0]{L^{p}}0$.
As $\iota_{\widetilde{\nu}}\phi=0$ on $\partial M$, by \Thmref{coarea},
$\left\Vert f_{r}\iota_{\widetilde{\nu}}\phi\right\Vert _{L^{p}}\lesssim\frac{1}{r}\left\Vert \iota_{\widetilde{\nu}}\phi\right\Vert _{L^{p}\left(M_{<r}\right)}\lesssim\left\Vert \iota_{\widetilde{\nu}}\phi\right\Vert _{W^{1,p}\left(M_{<r}\right)}\xrightarrow{r\downarrow0}0$.
\end{proof}
Then we can conclude $\{\omega\in L^{p}\Omega(M):d^{\mathscr{D}'_{N}}\omega\in L^{p}\}=\{\omega\in L^{p}\Omega(M):d^{\mathscr{D}'}\omega\in L^{p}\}.$

Recall that for an unbounded operator $A$, we write $\left(A,D(A)\right)$
to specify its domain.
\begin{thm}[Adjoints of $d,\delta$]
 For $p\in\left(1,\infty\right),$ the closure of $\left(d,\Omega\left(M\right)\right)$
as well as $\left(d,\mathscr{D}_{N}\Omega\left(M\right)\right)$ on
$L^{p}\Omega\left(M\right)$ is $d_{L^{p}}$ where $D(d_{L^{p}})=\{\omega\in L^{p}\Omega(M):d^{\mathscr{D}'_{N}}\omega\in L^{p}\}=\{\omega\in L^{p}\Omega(M):d^{\mathscr{D}'}\omega\in L^{p}\}$.

By Hodge duality, the closure of $\left(\delta,\Omega\left(M\right)\right)$
as well as $\left(\delta,\mathscr{D}_{D}\Omega\left(M\right)\right)$
on $L^{p}\Omega\left(M\right)$ is $\delta_{L^{p}}$ where $D(\delta_{L^{p}})=\{\omega\in L^{p}\Omega(M):\delta^{\mathscr{D}'_{D}}\omega\in L^{p}\}=\{\omega\in L^{p}\Omega(M):\delta^{\mathscr{D}'}\omega\in L^{p}\}$.

Define $\boxed{\delta_{c,L^{p}}=d_{L^{p'}}^{*}}$ and $\boxed{d_{c,L^{p}}=\delta_{L^{p'}}^{*}}$.
Then $\delta_{c,L^{p}}$ is the closure of $\left(\delta,\mathscr{D}_{N}\Omega\left(M\right)\right)$
as well as $\left(\delta,\mathscr{D}\Omega\left(M\right)\right)$.
Also, $D\left(\delta_{c,L^{p}}\right)=\{\omega\in L^{p}\Omega(M):\delta_{c}^{\mathscr{D}'_{N}}\omega\in L^{p}\}$.

Similarly, $d_{c,L^{p}}$ is the closure of $\left(d,\mathscr{D}_{D}\Omega\left(M\right)\right)$
and $\left(d,\mathscr{D}\Omega\left(M\right)\right)$. Also, $D\left(d_{c,L^{p}}\right)=\{\omega\in L^{p}\Omega(M):d_{c}^{\mathscr{D}'_{D}}\omega\in L^{p}\}$.
\end{thm}

\begin{proof}
Firstly, it is trivial to check $d_{L^{p}}$ is closed and $\left(d,\Omega\left(M\right)\right)$
is closable $(\omega_{j}\xrightarrow{L^{p}}0\text{ }$ and $d\omega_{j}\xrightarrow{L^{p}}\eta$
would imply $\eta=0$ since $d\omega_{j}\xrightarrow{\mathscr{D}'}0$).
Then let $\omega\in D(d_{L^{p}})$. We can conclude $\left(\omega^{\varepsilon},d\left(\omega^{\varepsilon}\right)\right)=\left(\omega^{\varepsilon},\left(d^{\mathscr{D}'_{N}}\omega\right)^{\varepsilon}\right)\xrightarrow[\varepsilon\downarrow0]{L^{p}\oplus L^{p}}\left(\omega,d^{\mathscr{D}'_{N}}\omega\right)$.
This also gives the closure of $\left(d,\mathscr{D}_{N}\Omega\left(M\right)\right)$.

Then let $G\left(\delta_{c,L^{p}}\right)\leq L^{p}\Omega\oplus L^{p}\Omega$
be the graph of $\delta_{c,L^{p}}$. Similarly for $G\left(d_{L^{p'}}\right)\leq L^{p'}\Omega\oplus L^{p'}\Omega$.
Write $\mathfrak{J}(x,y)=(-y,x).$ By the definition of adjoints,
$\mathfrak{J}\left(G\left(\delta_{c,L^{p}}\right)\right)=G\left(d_{L^{p'}}\right)^{\perp}$.
Then observe that 
\begin{align*}
\Bigl(\left(L^{p}\oplus L^{p}\right)\text{-}\mathrm{cl}\left\{ \left(-\delta_{c}\phi,\phi\right):\phi\in\mathscr{D}\Omega\right\} \Bigr)^{\perp} & =\{\left(\omega_{1},\omega_{2}\right)\in L^{p'}\oplus L^{p'}:\left\langle \left\langle \omega_{1},\delta_{c}\phi\right\rangle \right\rangle _{\Lambda}=\left\langle \left\langle \omega_{2},\phi\right\rangle \right\rangle _{\Lambda}\;\forall\phi\in\mathscr{D}\Omega\}\\
 & =\{\left(\omega_{1},\omega_{2}\right)\in L^{p'}\oplus L^{p'}:\omega_{2}=d^{\mathscr{D}'}\omega_{1}\}=G\left(d_{L^{p'}}\right)
\end{align*}
Then $G\left(\delta_{c,L^{p}}\right)=\left(L^{p}\oplus L^{p}\right)\text{-}\mathrm{cl}\left\{ \left(\phi,\delta_{c}\phi\right):\phi\in\mathscr{D}\Omega\right\} $.
Do the same for $\phi\in\mathscr{D}_{N}\Omega$. Finally, by the definition
of adjoints: 
\begin{align*}
D\left(\delta_{c,L^{p}}\right) & =\left\{ \omega\in L^{p}\Omega(M):\left|\left\langle \left\langle \omega,d_{L^{p'}}\phi\right\rangle \right\rangle _{\Lambda}\right|\lesssim\left\Vert \phi\right\Vert _{L^{p'}}\;\forall\phi\in D\left(d_{L^{p'}}\right)\right\} \\
 & =\left\{ \omega\in L^{p}\Omega(M):\left|\delta_{c}^{\mathscr{D}_{N}^{'}}\omega\left(\phi^{\varepsilon}\right)\right|=\left|\left\langle \left\langle \omega,d\phi^{\varepsilon}\right\rangle \right\rangle _{\Lambda}\right|=\left|\left\langle \left\langle \omega,\left(d_{L^{p'}}\phi\right)^{\varepsilon}\right\rangle \right\rangle _{\Lambda}\right|\lesssim\left\Vert \phi^{\varepsilon}\right\Vert _{L^{p'}}\;\forall\phi\in D\left(d_{L^{p'}}\right),\forall\varepsilon>0\right\} \\
 & =\left\{ \omega\in L^{p}\Omega(M):\left|\delta_{c}^{\mathscr{D}_{N}^{'}}\omega\left(\phi\right)\right|\lesssim\left\Vert \phi\right\Vert _{L^{p'}}\;\forall\phi\in\mathscr{D}_{N}\Omega\right\} =\{\omega\in L^{p}\Omega(M):\delta_{c}^{\mathscr{D}'_{N}}\omega\in L^{p}\}
\end{align*}
For the third equal sign, we implicitly used the fact that $e^{t\widetilde{\Delta_{N}}}\phi\xrightarrow[t\downarrow0]{C^{\infty}}\phi\;\forall\phi\in\mathscr{D}{}_{N}\Omega^{k}$.
\end{proof}
In particular, $W^{1,p}\Omega_{N}=W^{1,p}\text{-}\mathrm{cl}\left(\mathscr{D}_{N}\Omega\right)\leq D\left(\delta_{c,L^{p}}\right)$.
Similarly, $W^{1,p}\Omega_{D}\leq D\left(d_{c,L^{p}}\right)$. This
makes our choice of notation consistent.

Interestingly, a literature search yields a similar result regarding
the adjoints of $d$ and $\delta$ in \parencite[Proposition 4.3]{AlanMcIntosh2004_Lip_domain},
where the authors used Lie flows \emph{on the domain M} which is bounded
in $\mathbb{R}^{n}$, as well as zero extensions to $\mathbb{R}^{n}$
to characterize $D(d_{L^{p}})$ and $D\left(d_{L^{p}}^{*}\right)$.
In \parencite[Equation 2.12]{Mitrea2009_Sylvie_Neumann_Laplacian},
for $\eta\in D(\delta_{L^{p}})$, the authors defined $\nu\vee\eta\in B_{p,p}^{-\frac{1}{p}}\Omega\left(\partial M\right)=\left(B_{p',p'}^{\frac{1}{p}}\Omega\left(\partial M\right)\right)^{*}$
($p\in\left(1,\infty\right)$) by 
\[
\left\langle \left\langle \nu\vee\eta,\jmath^{*}\omega\right\rangle \right\rangle _{\Lambda}=\left\langle \left\langle \eta,d\omega\right\rangle \right\rangle _{\Lambda}-\left\langle \left\langle \delta^{\mathscr{D}'}\eta,\omega\right\rangle \right\rangle _{\Lambda}\;\forall\omega\in\Omega\left(M\right)
\]
 which is reminiscent of \Eqref{delta_D'_N}. Note that $\left\langle \left\langle \nu\vee\eta,\jmath^{*}\omega\right\rangle \right\rangle _{\Lambda}$
is abuse of notation (referring to the natural pairing via duality).
Recall from \Blackboxref{trace_ext} that $W^{1,p'}\Omega\left(M\right)=F_{p',2}^{1}\Omega\left(M\right)\stackrel{\mathrm{Trace}}{\twoheadrightarrow}B_{p',p'}^{\frac{1}{p}}\restr{\Omega\left(M\right)}{\partial M}$
has a bounded linear section $\mathrm{Ext}$, so it is possible to
choose $\omega$ such that $\left\Vert \jmath^{*}\omega\right\Vert _{B_{p',p'}^{\frac{1}{p}}}\sim\left\Vert \omega\right\Vert _{W^{1,p'}}$
and therefore $\nu\vee\eta$ is well-defined with
\[
\left\Vert \nu\vee\eta\right\Vert _{B_{p,p}^{-\frac{1}{p}}}\sim\sup_{\substack{\omega\in W^{1,p'}\Omega\left(M\right)\\
\left\Vert \jmath^{*}\omega\right\Vert _{B_{p',p'}^{1/p}}=1
}
}\left|\left\langle \left\langle \nu\vee\eta,\jmath^{*}\omega\right\rangle \right\rangle _{\Lambda}\right|\lesssim\text{\ensuremath{\left\Vert \eta\right\Vert }}_{L^{p}}+\left\Vert \delta^{\mathscr{D}'}\eta\right\Vert _{L^{p}}
\]
Of course, for $\eta\in W^{1,p}\Omega$, $\nu\vee\eta=\jmath^{*}\iota_{\nu}\eta$.
We can now show an alternative description of $D\left(\delta_{c,L^{p}}\right)$:
\begin{thm}
For $p\in\left(1,\infty\right)$, $D\left(\delta_{c,L^{p}}\right)=\{\eta\in L^{p}\Omega(M):\delta_{c}^{\mathscr{D}'_{N}}\eta\in L^{p}\}=\{\eta\in L^{p}\Omega(M):\delta^{\mathscr{D}'}\eta\in L^{p}\;\text{and }\nu\vee\eta=0\}$.
\end{thm}

\begin{proof}
Assume $\eta\in L^{p}\Omega(M)$ and $\delta_{c}^{\mathscr{D}'_{N}}\eta\in L^{p}$.
Then $\exists\alpha\in L^{p}\Omega\left(M\right):$ $\alpha=\delta_{c}^{\mathscr{D}'_{N}}\eta=\delta^{\mathscr{D}'}\eta$.
By the definition of $\nu\vee\eta$, $\left\langle \left\langle \alpha,\omega\right\rangle \right\rangle _{\Lambda}+\left\langle \left\langle \nu\vee\eta,\jmath^{*}\omega\right\rangle \right\rangle _{\Lambda}=\left\langle \left\langle \eta,d\omega\right\rangle \right\rangle _{\Lambda}\;\forall\omega\in\Omega\left(M\right)$.
By the definition of $\delta_{c}^{\mathscr{D}'_{N}}\eta$, $\left\langle \left\langle \alpha,\omega\right\rangle \right\rangle _{\Lambda}=\left\langle \left\langle \eta,d\omega\right\rangle \right\rangle _{\Lambda}\;\forall\omega\in\mathscr{D}_{N}\Omega$.
So $\left\langle \left\langle \nu\vee\eta,\jmath^{*}\omega\right\rangle \right\rangle _{\Lambda}=0\;\forall\omega\in\mathscr{D}_{N}\Omega$.
Recall that $\mathrm{Ext}$ (the right inverse of $\mathrm{Trace}$)
is bounded, so $B_{p',p'}^{\frac{1}{p}}\text{-}\mathrm{cl}\left(\jmath^{*}\left(\mathscr{D}_{N}\Omega\right)\right)\stackrel{\mathrm{Ext}}{=}\jmath^{*}\left(W^{1,p}\text{-}\mathrm{cl}\left(\mathscr{D}_{N}\Omega\right)\right)=\jmath^{*}\left(W^{1,p}\Omega_{N}\left(M\right)\right)\stackrel{\mathrm{Ext}}{=}\jmath^{*}\left(W^{1,p}\Omega\left(M\right)\right)=B_{p',p'}^{\frac{1}{p}}\Omega\left(\partial M\right)$.
Therefore $\nu\vee\eta=0$.

Conversely, now assume $\eta\in L^{p}\Omega(M),$ $\delta^{\mathscr{D}'}\eta=\alpha\in L^{p}$
and $\nu\vee\eta=0$. Then by the definition of $\nu\vee\eta$ for
$\eta\in D(\delta_{L^{p}})$, $\left\langle \left\langle \alpha,\omega\right\rangle \right\rangle _{\Lambda}=\left\langle \left\langle \eta,d\omega\right\rangle \right\rangle _{\Lambda}\;\forall\omega\in\Omega\left(M\right)$.
The formula also holds for $\omega\in\mathscr{D}_{N}\Omega$, and
therefore $\delta_{c}^{\mathscr{D}_{N}'}\eta=\alpha\in L^{p}$.
\end{proof}
This result agrees with \parencite[Equation 2.17]{Mitrea2009_Sylvie_Neumann_Laplacian}.
Our characterization of the adjoints of $d$ and $\delta$ further
highlights how heatable currents are truly natural objects in Hodge
theory, independent of the theory of heat flows.

In particular, it is trivial to show $\mathbb{P}L^{p}\Omega=L^{p}\text{-}\mathrm{cl}\Ker\left(\restr{\delta_{c}}{\Omega_{N}}\right)=\{\eta\in D\left(\delta_{c,L^{p}}\right):\delta_{c}^{\mathscr{D}'_{N}}\eta=0\}$
for $p\in\left(1,\infty\right)$.
\begin{rem*}
The name ``heatable current'' simply refers to the largest topological
vector space of differential forms (and hence vector fields) for which
the heat equation can be solved (i.e. \emph{heatable}), and once we
apply the heat flow a heatable current becomes heated. The name ``current''
for distributional forms was introduced by Georges de Rham \parencite{rham1984differentiable},
likely with its physical equivalents in mind, and has since become
standard in various areas of mathematics such as geometric measure
theory and complex manifolds.

It is not easy to search for literature dealing with the subject and
how it relates to Hodge theory. They are mentioned in a couple of
papers \parencite{Braverman_1997_tempered_currents,Troyanov_2009_Hodge_temperate_currents}
dealing with ``tempered currents'' or ``temperate currents'' on
$\mathbb{R}^{n}$ -- differential forms with tempered-distributional
coefficients. Yet the notion of ``tempered'' -- not growing too
fast -- does not make sense on a compact manifold with boundary.
Arguably, it is the ability to facilitate the heat flow, or the Littlewood-Paley
projection, that most characterizes tempered distributions and makes
them ideal for harmonic analysis. For scalar functions, much more
is known (cf. \parencite{Kerkyacharian_2014_Heat_Dirichlet_decomposition,DuongThinh_Weighted_besov_operator,Taniguchi_2018_Neumann_Besov}
and their references). In the same vein, various results from harmonic
analysis should also hold for heatable currents.
\end{rem*}

\subsection{Square root\label{subsec:Square-root}}

We will not need this for the rest of the paper, but a popular question
is the characterization of the square root of the Laplacian.

By the Poincare inequality, $\mathcal{P}^{N\perp}H^{1}\Omega_{N}^{k}$
is a Hilbert space where the $H^{1}$-inner product can be replaced
by $\left(\omega,\eta\right)\mapsto\mathcal{D}(\omega,\eta)$ (the
Dirichlet integral). The space is dense in $\mathcal{P}^{N\perp}L^{2}\Omega^{k}$.
Define $\mathcal{A}$ as an unbounded operator on $\mathcal{P}^{N\perp}L^{2}\Omega^{k}$
where 
\[
D\left(\mathcal{A}\right)=\{\omega\in\mathcal{P}^{N\perp}H^{1}\Omega_{N}^{k}:\left|\mathcal{D}(\omega,\eta)\right|\lesssim_{\omega}\left\Vert \eta\right\Vert _{2}\;\forall\eta\in\mathcal{P}^{N\perp}H^{1}\Omega_{N}^{k}\}
\]
and $\left\langle \left\langle \mathcal{A}\omega,\eta\right\rangle \right\rangle _{\Lambda}=\mathcal{D}(\omega,\eta)\;\forall\omega\in D(\mathcal{A}),\forall\eta\in\mathcal{P}^{N\perp}H^{1}\Omega_{N}^{k}$.
Easy to check that $\left\langle \left\langle \mathcal{A}\omega,\eta\right\rangle \right\rangle _{\Lambda}=\mathcal{D}\left(\left(-\Delta_{N}\right)^{-1}\mathcal{A}\omega,\eta\right)$
$\forall\eta\in\mathcal{P}^{N\perp}H^{1}\Omega_{N}^{k}$. Therefore
$\omega=\left(-\Delta_{N}\right)^{-1}\mathcal{A}\omega\in\mathcal{P}^{N\perp}H^{2}\Omega_{\mathrm{hom}N}^{k}$
and $\mathcal{A}\omega=\left(-\Delta_{N}\right)\omega\;\forall\omega\in D(A)$,
so $\mathcal{A}\subset-\Delta_{N}$. It is trivial to check $D\left(-\Delta_{N}\right)\leq D\left(\mathcal{A}\right)$,
so $\mathcal{A}=-\Delta_{N}$.

By \textbf{Friedrichs extension} (cf. \parencite[Appendix A, Proposition 8.7]{Taylor_PDE1},
\parencite[Section 8, Proposition 2.2]{Taylor_PDE2}), we conclude
that 
\begin{eqnarray*}
\mathbb{C}\mathcal{P}^{N\perp}H^{1}\Omega_{N}^{k} & = & \left[\mathbb{C}\mathcal{P}^{N\perp}L^{2}\Omega^{k},\left(D\left(\Delta_{N}^{\mathbb{C}}\right),\left\Vert \cdot\right\Vert _{D\left(\Delta_{N}^{\mathbb{C}}\right)}\right)\right]_{\frac{1}{2}}=\left[\mathbb{C}\mathcal{P}^{N\perp}L^{2}\Omega^{k},\mathbb{C}\mathcal{P}^{N\perp}H^{2}\Omega_{\mathrm{hom}N}^{k}\right]_{\frac{1}{2}}\\
 & = & \left(D\left(\sqrt{-\Delta_{N}^{\mathbb{C}}}\right),\left\Vert \cdot\right\Vert _{D\left(\sqrt{-\Delta_{N}^{\mathbb{C}}}\right)}\right)
\end{eqnarray*}

By direct summing, we can extend the result to $\widetilde{\Delta_{N}}$
to get 
\[
\mathbb{C}H^{1}\Omega_{N}^{k}=\left[\mathbb{C}L^{2}\Omega^{k},\mathbb{C}H^{2}\Omega_{\mathrm{hom}N}^{k}\right]_{\frac{1}{2}}=\left(D\left(\sqrt{-\widetilde{\Delta_{N}^{\mathbb{C}}}}\right),\left\Vert \cdot\right\Vert _{D\left(\sqrt{-\widetilde{\Delta_{N}^{\mathbb{C}}}}\right)}\right)
\]
We note that the norms are only defined up to equivalent norms, and
$\left\Vert \cdot\right\Vert _{D(\mathcal{A})}$ is not the same as
$\left\Vert \cdot\right\Vert _{D(\mathcal{A})}^{*}$ (see \Secref{Common-notation}).
This difference is not always made explicit in \parencite{Taylor_PDE1,Taylor_PDE2}.

\subsection{Some trace-zero results}

Although we will not need them for the rest of the paper, let us briefly
delineate some results regarding the trace-zero Laplacian (cf. \Thmref{4_version_gaffney})
which are similar to those obtained above for the absolute Neumann
Laplacian. We begin by retracing our steps from \Corref{ortho_projection}.

Define $\mathcal{H}_{0}^{k}\left(M\right)=\mathcal{H}_{N}^{k}\left(M\right)\cap\mathcal{H}_{D}^{k}\left(M\right)$.
Obviously, $\mathcal{H}_{0}^{k}\left(M\right)$ is finite-dimensional
and we can define $\mathcal{P}^{0}$ and $\mathcal{P}^{0\perp}$ the
same way we did for $\mathcal{P}^{N}$ and $\mathcal{P}^{N\perp}$
in \Corref{ortho_projection}. When $M$ has no boundary, $\mathcal{P}^{0\perp}=\mathcal{P}^{N\perp}$
and $\mathcal{P}^{0}=\mathcal{P}^{N}=\mathcal{P}_{3}$.

It is a celebrated theorem, following from the \textbf{Aronszajn continuation}
\textbf{theorem} \parencite{Aronszajn1962_continuation}, that $\mathcal{H}_{0}^{k}\left(M\right)=0$
when every connected component of $M$ has nonempty boundary (cf.
\parencite[Theorem 3.4.4]{Schwarz1995}). When that happens, $\mathcal{P}^{0\perp}=1$
and $\mathcal{P}^{0}=0$.
\begin{blackbox}[Potential theory]
\label{blackbox:potential_estimates-trace_zero} For $m\in\mathbb{N}_{0},p\in\left(1,\infty\right)$,
we define the \textbf{injective trace-zero Laplacian} 
\[
\Delta_{0}:\mathcal{P}^{0\perp}W^{m+2,p}\Omega_{0}^{k}\to\mathcal{P}^{0\perp}W^{m,p}\Omega^{k}
\]
as simply $\Delta$ under domain restriction. Then $\left(-\Delta_{0}\right)^{-1}$
is called the \textbf{trace-zero potential}, which is bounded. $\Delta_{0}$
can also be thought of as an unbounded operator on $\mathcal{P}^{0\perp}W^{m,p}\Omega_{0}^{k}$.
\end{blackbox}

\begin{proof}
We only need to prove the theorem on each connected component of $M$.
So WLOG, $M$ is connected. If $\partial M=\emptyset$, we are back
to the absolute Neumann case in \Blackboxref{potential_estimates}.
When $\partial M\neq\emptyset$, $\mathcal{P}^{0\perp}=1$ and we
only need to show the trace-zero Poisson problem $\left(\Delta\omega,\restr{\omega}{\partial M}\right)=\left(\eta,0\right)$
is uniquely solvable for each $\eta\in W^{m,p}\Omega^{k}$. This is
\parencite[Theorem 3.4.10]{Schwarz1995}.
\end{proof}
Consequently, we have a trivial decomposition
\[
\omega=\mathcal{P}^{0\perp}\omega+\mathcal{P}^{0}\omega=d\delta\left(-\Delta_{0}\right)^{-1}\mathcal{P}^{0\perp}\omega+\delta d\left(-\Delta_{0}\right)^{-1}\mathcal{P}^{0\perp}\omega+\mathcal{P}^{0}\omega
\]
for $\omega\in W^{m,p}\Omega^{k}$, $m\in\mathbb{N}_{0}$, $p\in\left(1,\infty\right)$.
This decomposition is not as useful as the Hodge-Morrey decomposition
(\Subsecref{Hodge-decomposition}) since the the first two terms are
not orthogonal. However, it does mean that, when $\mathcal{P}^{0}=0$,
every differential form is a sum of exact and coexact forms.

For $\omega\in\mathcal{P}^{0\perp}W^{m+2,p}\Omega_{0}^{k},m\in\mathbb{N}_{0},p\in\left(1,\infty\right)$,
we also have $\omega=\left(-\Delta_{0}\right)^{-1}\left(-\Delta_{0}\right)\omega=\left(-\Delta_{0}\right)^{-1}\left(d\delta\omega+\delta d\omega\right),$
so $\left\Vert \omega\right\Vert _{W^{m+2,p}}\sim\left\Vert \delta\omega\right\Vert _{W^{m+1,p}}+\left\Vert d\omega\right\Vert _{W^{m+1,p}}$.
This trick is not enough to get the full Poincare inequality $\left\Vert \omega\right\Vert _{W^{1,p}}\sim\left\Vert \delta\omega\right\Vert _{p}+\left\Vert d\omega\right\Vert _{p}$,
and therefore \parencite[Lemma 2.4.10.iv]{Schwarz1995} might be wrong.

As $\left(-\Delta_{0}\right)^{-1}$ is symmetric and bounded on $\mathcal{P}^{0\perp}L^{2}\Omega^{k}$,
we conclude $\Delta_{0}$ is a self-adjoint and dissipative operator
on $\mathcal{P}^{0\perp}L^{2}\Omega^{k}$, with the domain $D\left(\Delta_{0}\right)=\mathcal{P}^{0\perp}H^{2}\Omega_{0}^{k}$.
This means $\Delta_{0}^{\mathbb{C}}$ is acutely sectorial on $\mathbb{C}\mathcal{P}^{0\perp}L^{2}\Omega^{k}$.

Next we define the \textbf{non-injective trace-zero Laplacian} $\widetilde{\Delta_{0}}$
as an unbounded operator on $L^{2}\Omega^{k}$ with $D\left(\widetilde{\Delta_{0}}^{m}\right)=D\left(\Delta_{0}^{m}\right)\oplus\mathcal{H}_{0}^{k}$
and $\widetilde{\Delta_{0}}^{m}=\Delta_{0}^{m}\oplus0$ $\forall m\in\mathbb{N}_{1}$.
Again, $\widetilde{\Delta_{0}^{\mathbb{C}}}$ is acutely sectorial
on $\mathbb{C}L^{2}\Omega^{k}$ and $\left\Vert \omega\right\Vert _{D\left(\widetilde{\Delta_{0}}^{m}\right)}\sim\left\Vert \omega\right\Vert _{H^{2m}}\;\forall\omega\in D\left(\widetilde{\Delta_{0}}^{m}\right),\forall m\in\mathbb{N}_{1}$.
In particular, $D\left(\widetilde{\Delta_{0}}\right)=\mathcal{P}^{0\perp}H^{2}\Omega_{0}^{k}\oplus\mathcal{H}_{0}^{k}=H^{2}\Omega_{0}^{k}$.

For $L^{p}$-analyticity, observe that on $\partial M$: $\left|\nabla_{\nu}\left(|\omega|^{2}\right)\right|=2\left|\left\langle \nabla_{\nu}\omega,\omega\right\rangle \right|=0\lesssim\left|\omega\right|^{2}$
$\forall\omega\in W^{2,p}\Omega_{0}^{k},\forall p\in\left(1,\infty\right)$.
So we argue as in \Thmref{local_boundedness_Lp}, and $L^{p}$-analyticity
follows.
\begin{rem*}
The operator $\mathbb{P}\widetilde{\Delta_{0}}$, with the domain
$H^{2}\Omega_{0}^{k}\cap\mathbb{P}L^{2}\Omega^{k}$, is a well-defined
unbounded operator on $\mathbb{P}L^{2}\Omega^{k}$. It is called the
\textbf{Stokes operator} corresponding to the trace-zero/no-slip boundary
condition, as discussed in \parencite{Fujita1964_Kato_Navier_Stokes,Giga1985_Lp_Stokes_solution,Mitrea2008_Stokes_Fujita_Kato_approach}
and others. It lies outside the scope of this paper. For more information,
see \parencite{Hieber2018_Stokes_operator_survey} and its references.
\end{rem*}

\section{Results related to the Euler equation}

\subsection{Hodge-Sobolev spaces\label{subsec:Hodge-Sobolev-spaces}}

We will have need of negative-order Sobolev spaces when we calculate
the pressure in the Euler equation.

Recall the space of heatable currents $\mathscr{D}'_{N}\Omega$ (defined
in \Subsecref{Distributions-and-adjoints}). Note that $\mathcal{P}^{N\perp}$
is well-defined on $\mathscr{D}_{N}^{'}\Omega$ by $\left\langle \left\langle \mathcal{P}^{N\perp}\Lambda,\phi\right\rangle \right\rangle _{\Lambda}=\left\langle \left\langle \Lambda,\mathcal{P}^{N\perp}\phi\right\rangle \right\rangle $
$\forall\Lambda\in\mathscr{D}'_{N}\Omega,\forall\phi\in\mathscr{D}_{N}\Omega$.
Same for $\mathcal{P}^{N}$, and we can uniquely identify $\mathcal{P}^{N}\Lambda\in\mathcal{H}_{N}$
$\forall\Lambda\in\mathscr{D}_{N}^{'}\Omega$.

Similarly, $\mathbb{P}\left(\mathscr{D}_{N}\Omega\right)\leq\mathscr{D}_{N}\Omega$
(use \Thmref{compat_Helmholtz_leray} and \Thmref{Friedrichs_Decomposition}),
so $\mathbb{P}$, $1-\mathbb{P}=\left(\mathcal{P}_{1}+\mathcal{P}_{3}^{\text{ex}}\right)$
and $\mathcal{P}_{2}=\mathbb{P}-\mathcal{P}^{N}$ are well-defined
on $\mathscr{D}'_{N}\Omega$.

For all $p\in\left(1,\infty\right)$, define $D_{N}=d^{\mathscr{D}'_{N}}+\delta_{c}^{\mathscr{D}'_{N}}$
on $\mathcal{P}^{N\perp}\mathscr{D}'_{N}\Omega$ and $\widetilde{D_{N}}=d^{\mathscr{D}'_{N}}+\delta_{c}^{\mathscr{D}'_{N}}$
on $\mathscr{D}'_{N}\Omega$ as the\textbf{ injective} and \textbf{non-injective
(Neumann) Hodge-Dirac operators}. \nomenclature{$D_{N}$, $\widetilde{D_{N}}$}{the injective and non-injective (Neumann) Hodge-Dirac operators\nomrefpage}

By the Poincare inequality (\Corref{poincare_ineq_Neumann}), it is
easy to check that $\restr{D_{N}}{\mathcal{P}^{N\perp}\mathscr{D}{}_{N}\Omega}:\mathcal{P}^{N\perp}\mathscr{D}{}_{N}\Omega\rightarrow\mathcal{P}^{N\perp}\mathscr{D}{}_{N}\Omega$
is bijective. Consequently, so is $D_{N}$ on $\mathcal{P}^{N\perp}\mathscr{D}'_{N}\Omega$.

Observe that $\forall m\in\mathbb{N}_{0},\forall p\in\left(1,\infty\right),\forall\alpha\in\mathcal{P}^{N\perp}W^{m,p}\Omega\left(M\right),\exists!\beta=\left(D_{N}\right)^{-1}\alpha\in\mathcal{P}^{N\perp}W^{m+1,p}\Omega_{N}$
and 
\begin{equation}
\left\Vert \beta\right\Vert _{W^{m+1,p}}\sim\left\Vert \alpha\right\Vert _{W^{m,p}}=\left\Vert d\beta+\delta_{c}\beta\right\Vert _{W^{m,p}}\sim\left\Vert d\beta\right\Vert _{W^{m,p}}+\left\Vert \delta_{c}\beta\right\Vert _{W^{m,p}}\label{eq:poincare_split_dirac}
\end{equation}
because $\mathcal{P}^{N\perp}W^{m,p}\Omega=d\left(W^{m+1,p}\Omega\right)\oplus\delta_{c}\left(W^{m+1,p}\Omega_{N}\right)$
is a direct sum of closed subspaces (corresponding to $\mathcal{P}_{1}+\mathcal{P}_{3}^{\mathrm{ex}}$
and $\mathcal{P}_{2}$).

Note that we do not have $d^{\mathscr{D}'_{N}}D_{N}=D_{N}d^{\mathscr{D}'_{N}}$,
but $d^{\mathscr{D}'_{N}}D_{N}^{2}=D_{N}^{2}d^{\mathscr{D}'_{N}}=-\Delta_{N}^{\mathscr{D}'_{N}}d^{\mathscr{D}'_{N}}$
is true.
\begin{defn}
For $m\in\mathbb{Z},p\in\left(1,\infty\right)$, let $W^{m,p}\left(D_{N}\right):=\left(D_{N}\right)^{-m}\left(\mathcal{P}^{N\perp}L^{p}\Omega\right)=\{\alpha\in\mathcal{P}^{N\perp}\mathscr{D}'_{N}\Omega:\left(D_{N}\right)^{m}\alpha\in L^{p}\Omega\}$
and $W^{m,p}\left(\widetilde{D_{N}}\right):=W^{m,p}\left(D_{N}\right)\oplus\mathcal{H}_{N}$.
They are Banach spaces under the norms $\left\Vert \alpha\right\Vert _{W^{m,p}\left(D_{N}\right)}:=\left\Vert \left(D_{N}\right)^{m}\alpha\right\Vert _{L^{p}\Omega}$
and $\left\Vert \beta\right\Vert _{W^{m,p}\left(\widetilde{D_{N}}\right)}:=\left\Vert \mathcal{P}^{N\perp}\beta\right\Vert _{W^{m,p}\left(D_{N}\right)}+\left\Vert \mathcal{P}^{N}\beta\right\Vert _{\mathcal{H}_{N}}$.
\nomenclature{$W^{m,p}\left(D_{N}\right)$, $W^{m,p}\left(\widetilde{D_{N}}\right)$}{Hodge-Sobolev spaces\nomrefpage}
\end{defn}

In a sense, these are comparable to homogeneous and inhomogeneous
Bessel potential spaces. We can extend the definitions to fractional
powers, but that is outside the scope of this paper.

It is trivial to check that $\left\Vert \alpha\right\Vert _{W^{m,p}\left(\widetilde{D_{N}}\right)}\sim\left\Vert \alpha\right\Vert _{W^{m,p}\Omega}$
$\forall\alpha\in\mathscr{D}_{N}\Omega,\forall m\in\mathbb{N}_{0},\forall p\in\left(1,\infty\right)$.
\begin{thm}
\label{thm:Hodge-Sobolev-basic-properties}Some basic properties of
$W^{m,p}\left(\widetilde{D_{N}}\right)$:
\begin{enumerate}
\item $\mathscr{D}_{N}\Omega$ is dense in $W^{m,p}\left(\widetilde{D_{N}}\right)$
$\forall m\in\mathbb{Z},\forall p\in\left(1,\infty\right)$.
\item $W^{m,p}\left(\widetilde{D_{N}}\right)=W^{m,p}\text{-}\mathrm{cl}\left(\mathscr{D}_{N}\Omega\right)$
$\forall m\in\mathbb{N}_{0},\forall p\in\left(1,\infty\right)$.
\item $\left\Vert d^{\mathscr{D}'_{N}}\beta\right\Vert _{W^{m,p}\left(\widetilde{D_{N}}\right)}+\left\Vert \delta_{c}^{\mathscr{D}'_{N}}\beta\right\Vert _{W^{m,p}\left(\widetilde{D_{N}}\right)}\lesssim\left\Vert \beta\right\Vert _{W^{m+1,p}\left(\widetilde{D_{N}}\right)}$
$\forall\beta\in W^{m+1,p}\left(\widetilde{D_{N}}\right),\forall m\in\mathbb{Z},\forall p\in\left(1,\infty\right)$\\
Then $\mathcal{P}_{2}=\delta_{c}^{\mathscr{D}'_{N}}d^{\mathscr{D}'_{N}}\left(-\Delta_{N}^{\mathscr{D}'_{N}}\right)^{-1}\mathcal{P}^{N\perp}=\delta_{c}^{\mathscr{D}'_{N}}\left(-\Delta_{N}^{\mathscr{D}'_{N}}\right)^{-1}d^{\mathscr{D}'_{N}}$
and $\mathbb{P}=\mathcal{P}_{2}+\mathcal{P}^{N}$ are of order 0 on
$W^{m,p}\left(\widetilde{D_{N}}\right)$.
\item $\left(W^{m,p}\left(\widetilde{D_{N}}\right)\right)^{*}=W^{-m,p'}\left(\widetilde{D_{N}}\right)$
$\forall m\in\mathbb{Z},\forall p\in\left(1,\infty\right)$ via the
pairing $\left\langle \alpha,\phi\right\rangle _{W^{-m,p}\left(\widetilde{D_{N}}\right),W^{m,p'}\left(\widetilde{D_{N}}\right)}=\left\langle \left\langle D_{N}^{-m}\mathcal{P}^{N\perp}\alpha,D_{N}^{m}\mathcal{P}^{N\perp}\phi\right\rangle \right\rangle _{\Lambda}+\left\langle \left\langle \mathcal{P}^{N}\alpha,\mathcal{P}^{N}\phi\right\rangle \right\rangle _{\Lambda}$
\end{enumerate}
\end{thm}

\begin{proof}
$ $
\begin{enumerate}
\item Because $D_{N}^{m}\left(\mathcal{P}^{N\perp}\mathscr{D}_{N}\Omega\right)=\mathcal{P}^{N\perp}\mathscr{D}_{N}\Omega$
is dense in $\mathcal{P}^{N\perp}L^{p}\Omega$.
\item We only need $W^{m,p}\left(D_{N}\right)=\mathcal{P}^{N\perp}W^{m,p}\left(\widetilde{D_{N}}\right)\leq W^{m,p}\text{-}\mathrm{cl}\left(\mathcal{P}^{N\perp}\mathscr{D}_{N}\Omega\right)$.
Let $\alpha\in\mathcal{P}^{N\perp}W^{m,p}\left(\widetilde{D_{N}}\right)$
and $\alpha^{\varepsilon}=e^{\varepsilon\widetilde{\Delta_{N}}}\alpha$
as usual. Then $D_{N}^{m}\left(\alpha^{\varepsilon}\right)=\left(D_{N}^{m}\alpha\right)^{\varepsilon}\xrightarrow[\varepsilon\downarrow0]{L^{p}}D_{N}^{m}\alpha$.
So $D_{N}^{-m}D_{N}^{m}\left(\alpha^{\varepsilon}\right)=\alpha^{\varepsilon}\xrightarrow[\varepsilon\downarrow0]{W^{m,p}}\alpha$
by \Eqref{poincare_split_dirac}. 
\item Let $D_{N}^{m+1}\mathcal{P}^{N\perp}\beta\in L^{p}$. Then $D_{N}^{m}\mathcal{P}^{N\perp}\beta\in\mathcal{P}^{N\perp}W^{1,p}\Omega_{N}$
by \Eqref{poincare_split_dirac}.\\
\\
When $m=2k$ $\left(k\in\mathbb{Z}\right)$: $\left\Vert dD_{N}^{2k}\mathcal{P}^{N\perp}\beta\right\Vert _{L^{p}}+\left\Vert \delta_{c}D_{N}^{2k}\mathcal{P}^{N\perp}\beta\right\Vert _{L^{p}}\sim\left\Vert dD_{N}^{2k}\mathcal{P}^{N\perp}\beta+\delta_{c}D_{N}^{2k}\mathcal{P}^{N\perp}\beta\right\Vert _{L^{p}}=\left\Vert D_{N}^{2k+1}\mathcal{P}^{N\perp}\beta\right\Vert _{L^{p}}$.\\
\\
When $m=2k+1$ $\left(k\in\mathbb{Z}\right)$: $D_{N}^{2k}\mathcal{P}^{N\perp}\beta\in\mathcal{P}^{N\perp}W^{2,p}\Omega_{\hom N}$
and 
\begin{align*}
 & \left\Vert D_{N}dD_{N}^{2k}\mathcal{P}^{N\perp}\beta\right\Vert _{L^{p}}+\left\Vert D_{N}\delta_{c}D_{N}^{2k}\mathcal{P}^{N\perp}\beta\right\Vert _{L^{p}}=\left\Vert \delta_{c}dD_{N}^{2k}\mathcal{P}^{N\perp}\beta\right\Vert _{L^{p}}+\left\Vert d\delta_{c}D_{N}^{2k}\mathcal{P}^{N\perp}\beta\right\Vert _{L^{p}}\\
\sim & \left\Vert \delta_{c}dD_{N}^{2k}\mathcal{P}^{N\perp}\beta+d\delta_{c}D_{N}^{2k}\mathcal{P}^{N\perp}\beta\right\Vert _{L^{p}}=\left\Vert D_{N}^{2k+2}\mathcal{P}^{N\perp}\beta\right\Vert _{L^{p}}
\end{align*}
\item Simply observe that $\left(W^{m,p}\left(D_{N}\right)\right)^{*}=W^{-m,p'}\left(D_{N}\right)$
via the isomorphisms $W^{m,p}\left(D_{N}\right)\stackrel{D_{N}^{m}}{\iso}\mathcal{P}^{N\perp}L^{p}\Omega$
and $W^{-m,p'}\left(D_{N}\right)\stackrel{D_{N}^{-m}}{\iso}\mathcal{P}^{N\perp}L^{p'}\Omega$.
\end{enumerate}
\end{proof}
\begin{rem*}
We briefly note that $D_{N}$ with the domain $\mathcal{P}^{N\perp}H^{1}\Omega_{N}$
is self-adjoint on $\mathcal{P}^{N\perp}L^{2}\Omega$ and its complexification
is therefore ``bisectorial''. For more on this, see \parencite{Alan_Hinf_calculus_86,mcintosh2010operator,mcintosh_hal_Hodge_dirac}.
\end{rem*}
\begin{cor}
\label{cor:div_negative_est} Assume $U\in\mathbb{P}L^{2}\mathfrak{X}$.
Define $\mathrm{div}(U\otimes U)\in\mathscr{D}'_{N}\mathfrak{X}$
by $\left\langle \left\langle \mathrm{div}(U\otimes U),X\right\rangle \right\rangle _{\Lambda}:=-\left\langle \left\langle U\otimes U,\nabla X\right\rangle \right\rangle $
$\forall X\in\mathscr{D}_{N}\mathfrak{X}$.

If $p\in\left(1,\infty\right)$ and $U\otimes U\in L^{p}\Gamma\left(TM\otimes TM\right)$,
then $\left\Vert \mathrm{div}(U\otimes U)^{\flat}\right\Vert _{W^{-1,p}\left(\widetilde{D_{N}}\right)}\lesssim\left\Vert U\otimes U\right\Vert _{L^{p}}$.
\end{cor}

\begin{proof}
For $\eta\in\Omega\left(M\right),$ write $\eta_{k}$ for the part
of $\eta$ in $\Omega^{k}$. Let $\phi\in\mathscr{D}_{N}\Omega$,
then
\[
\left|\left\langle \left\langle D_{N}^{-1}\mathcal{P}^{N\perp}\mathrm{div}(U\otimes U)^{\flat},\phi\right\rangle \right\rangle _{\Lambda}\right|=\left|\left\langle \left\langle U\otimes U,\nabla\left(D_{N}^{-1}\mathcal{P}^{N\perp}\phi\right)_{1}^{\sharp}\right\rangle \right\rangle \right|\lesssim\text{\ensuremath{\left\Vert U\otimes U\right\Vert _{L^{p}}}}\left\Vert \phi\right\Vert _{L^{p'}}
\]
This implies $\mathrm{div}(U\otimes U)^{\flat}\in W^{-1,p}\left(\widetilde{D_{N}}\right)$.
Then observe $\left|\left\langle \left\langle \mathrm{div}(U\otimes U)^{\flat},\phi\right\rangle \right\rangle _{\Lambda}\right|=\left|\left\langle \left\langle U\otimes U,\nabla\left(\phi\right)_{1}^{\sharp}\right\rangle \right\rangle \right|\lesssim\text{\ensuremath{\left\Vert U\otimes U\right\Vert _{L^{p}}}}\left\Vert \phi\right\Vert _{W^{1,p'}\left(\widetilde{D_{N}}\right)}$.
\end{proof}

\subsection{Calculating the pressure\label{subsec:Calc_pressure}}

In this subsection, we assume that $\partial_{t}\mathcal{V}+\mathrm{div}(\mathcal{V}\otimes\mathcal{V})+\grad\mathfrak{p}\stackrel{\mathscr{D}_{N}^{'}\left(I,\mathfrak{X}\right)}{=\joinrel=\joinrel}0$,
$\mathcal{V}\in L_{\mathrm{loc}}^{2}\left(I,\mathbb{P}L^{2}\mathfrak{X}\right)$,
$\mathfrak{p}\in L_{\mathrm{loc}}^{1}(I\times M)$. This is true,
for instance, in the case of Onsager's conjecture (see \Subsecref{Justification}
and \Subsecref{heating_nonlinear}).

We first note that $\mathcal{H}_{N}^{0}=\mathcal{H}^{0}=\{\text{locally constant functions}\}$.
Then we can show $\mathcal{V}$ uniquely determines $\mathfrak{p}$
by a formula, up to a difference in $\mathcal{H}_{N}^{0}$ ($d\mathfrak{p}$
is always unique). It is no loss of generality to \uline{set \mbox{$\mathfrak{p}=\mathcal{P}^{N\perp}\mathfrak{p}$}}
(implying $\int_{M}\mathfrak{p}=0$).
\begin{enumerate}
\item Assume $\mathcal{V}\otimes\mathcal{V}\in L_{t}^{q}W^{m+1,p}\Gamma\left(TM\otimes TM\right)$
for some $m\in\mathbb{N}_{0},p\in\left(1,\infty\right),q\in[1,\infty]$.

Let $\omega=\mathrm{div}(\mathcal{V}\otimes\mathcal{V})^{\flat}$.
Then $d^{\mathscr{D}'_{N}}\mathfrak{p}\stackrel{\mathscr{D}_{N}^{'}\left(I,\mathfrak{X}\right)}{=\joinrel=\joinrel}\left(\mathbb{P}-1\right)\omega\in L_{t}^{q}W^{m,p}\Omega^{1}$.
By the Poincare inequality (\Corref{poincare_ineq_Neumann}), there
is a unique $\mathfrak{f}\in L_{t}^{q}\mathcal{P}^{N\perp}W^{m+1,p}\Omega^{0}$
such that $d\mathfrak{f}=\left(\mathbb{P}-1\right)\omega\stackrel{\mathscr{D}_{N}^{'}\left(I,\mathfrak{X}\right)}{=\joinrel=\joinrel}d^{\mathscr{D}'_{N}}\mathfrak{p}$.
An explicit formula is $\mathfrak{f}=-R_{d}\omega$ where $R_{d}:=\mathcal{P}^{N\perp}\delta\left(-\Delta_{D}\right)^{-1}\mathcal{P}^{D\perp}+\mathcal{P}^{N\perp}\delta\left(-\Delta_{N}\right)^{-1}\mathcal{P}_{3}^{\mathrm{ex}}$
is the potential for $d$.

We aim to show $\mathfrak{f}=\mathfrak{p}$. Let $\psi\in C_{c}^{\infty}\left(I,\mathscr{D}_{N}\Omega^{0}\right)$.
Then because $\Omega^{0}=\mathcal{P}_{2}\left(\Omega^{0}\right)\oplus\mathcal{P}_{3}^{N}\left(\Omega^{0}\right),$
we conclude $\mathcal{P}^{N\perp}\psi=\delta_{c}\phi$ where $\phi:=d\left(-\Delta_{N}\right)^{-1}\mathcal{P}^{N\perp}\psi\in C_{c}^{\infty}\left(I,\mathscr{D}_{N}\Omega^{1}\right)$
and
\[
\int_{I}\left\langle \left\langle \mathfrak{f},\psi\right\rangle \right\rangle _{\Lambda}=\int_{I}\left\langle \left\langle \mathfrak{f},\mathcal{P}^{N\perp}\psi\right\rangle \right\rangle _{\Lambda}=\int_{I}\left\langle \left\langle \mathfrak{f},\delta_{c}\phi\right\rangle \right\rangle _{\Lambda}=\int_{I}\left\langle \left\langle d\mathfrak{f},\phi\right\rangle \right\rangle _{\Lambda}=\int_{I}\left\langle \left\langle d^{\mathscr{D}'_{N}}\mathfrak{p},\phi\right\rangle \right\rangle _{\Lambda}=\int_{I}\left\langle \left\langle \mathfrak{p},\psi\right\rangle \right\rangle _{\Lambda}
\]
Therefore $\mathfrak{p}=\mathfrak{f}$ and $\boxed{\left\Vert \mathfrak{p}\right\Vert _{L_{t}^{q}W^{m+1,p}}\lesssim\left\Vert \omega\right\Vert _{L_{t}^{q}W^{m,p}}\lesssim\left\Vert \mathcal{V}\otimes\mathcal{V}\right\Vert _{L_{t}^{q}W^{m+1,p}}}$.
\item Assume $\mathcal{V}\otimes\mathcal{V}\in L_{t}^{q}L^{p}\Gamma\left(TM\otimes TM\right)$
for some $p\in\left(1,\infty\right),q\in[1,\infty]$.

Let $\omega=\mathrm{div}(\mathcal{V}\otimes\mathcal{V})^{\flat}$.
Then $d^{\mathscr{D}'_{N}}\mathfrak{p}\stackrel{\mathscr{D}_{N}^{'}\left(I,\mathfrak{X}\right)}{=\joinrel=\joinrel}\left(\mathbb{P}-1\right)\omega\in L_{t}^{q}W^{-1,p}\left(\widetilde{D_{N}}\right)$
by \Corref{div_negative_est} and \Thmref{Hodge-Sobolev-basic-properties}.
Then $-\delta_{c}^{\mathscr{D}'_{N}}d^{\mathscr{D}'_{N}}\mathfrak{p}\stackrel{\mathscr{D}_{N}^{'}\left(I,\mathfrak{X}\right)}{=\joinrel=\joinrel}\delta_{c}^{\mathscr{D}'_{N}}\left(1-\mathbb{P}\right)\omega=\delta_{c}^{\mathscr{D}'_{N}}\omega\in L_{t}^{q}W^{-2,p}\left(\widetilde{D_{N}}\right)$
and $\mathfrak{p}=-D_{N}^{-2}\delta_{c}^{\mathscr{D}'_{N}}\omega$,
so $\left\Vert \mathfrak{p}\right\Vert _{L_{t}^{q}L^{p}}\lesssim\left\Vert \delta_{c}^{\mathscr{D}'_{N}}\omega\right\Vert _{L_{t}^{q}W^{-2,p}\left(\widetilde{D_{N}}\right)}\lesssim\left\Vert \omega\right\Vert _{L_{t}^{q}W^{-1,p}\left(\widetilde{D_{N}}\right)}\lesssim\left\Vert \mathcal{V}\otimes\mathcal{V}\right\Vert _{L_{t}^{q}L^{p}}$.

\end{enumerate}
\begin{rem*}
It is also possible to define $R_{\delta_{c}}:=d^{\mathscr{D}'_{N}}\left(-\Delta_{N}^{\mathscr{D}'_{N}}\right)^{-1}\mathcal{P}^{N\perp}$
on $\mathscr{D}'_{N}\Omega$ and have $R_{d}=\left(D_{N}^{-1}-R_{\delta_{c}}\right)\mathcal{P}^{N\perp}$
on $\mathscr{D}'_{N}\Omega$. This would then imply $\left\Vert R_{d}\alpha\right\Vert _{W^{m+1,p}\left(\widetilde{D_{N}}\right)}\lesssim\left\Vert \alpha\right\Vert _{W^{m,p}\left(\widetilde{D_{N}}\right)}$
$\forall\alpha\in W^{m,p}\left(\widetilde{D_{N}}\right),\forall m\in\mathbb{Z},\forall p\in\left(1,\infty\right)$.
\end{rem*}

\subsection{On an interpolation identity}

\label{subsec:Interpolation-and-B-analyticity}Let $p\in\left(1,\infty\right)$.
We are faced with the difficulty of finding a good interpolation characterization
for $B_{p,1}^{\frac{1}{p}}\Omega_{N}$. We do have $B_{p,1}^{\frac{1}{p}}\Omega=\left(L^{p}\Omega,W^{1,p}\Omega\right)_{\frac{1}{p},1}$
(complexification, then projection onto the real part), but our heat
flow is not analytic on $\mathbb{C}W^{1,p}\Omega$. The hope is that
$B_{p,1}^{\frac{1}{p}}\Omega_{N}=\left(L^{p}\Omega,W^{1,p}\Omega_{N}\right)_{\frac{1}{p},1}$,
and our first guess is to try to find some kind of projection. Indeed,
the Leray projection yields
\begin{equation}
\mathbb{P}B_{p,1}^{\frac{1}{p}}\Omega=\left(\mathbb{P}L^{p}\Omega,\mathbb{P}W^{1,p}\Omega\right)_{\frac{1}{p},1}\label{eq:interpolation_difficult_P}
\end{equation}
and the heat flow is well-behaved on $\mathbb{P}W^{1,p}\Omega=\mathbb{P}W^{1,p}\Omega_{N}$
(\Thmref{Friedrichs_Decomposition}, \Thmref{compat_Helmholtz_leray}).
By interpolation, $\mathbb{P}$ is $B_{p,1}^{\frac{1}{p}}$-continuous,
so $\mathbf{n}\mathbb{P}:B_{p,1}^{\frac{1}{p}}\Omega\to L^{p}\restr{\Omega}{\partial M}$
is continuous and $\mathbb{P}B_{p,1}^{\frac{1}{p}}\Omega=\mathbb{P}B_{p,1}^{\frac{1}{p}}\Omega_{N}$.

This is enough to get all the Besov estimates we will need for Onsager's
conjecture.

Additionally, it is true that the heat semigroup is also $C_{0}$
and analytic on $\mathbb{C}\mathbb{P}B_{p,1}^{\frac{1}{p}}\Omega_{N}$
by Yosida's half-plane criterion (\Thmref{yosida_half_plane}). Unlike
the $L^{p}$-analyticity case, here we already have analyticity on
the 2 endpoints, so the criterion simply follows by interpolation.
Alternatively, observe that there exists $C>0$ such that $\sup_{t>0}\left\Vert t\left(\widetilde{\Delta_{N}^{\mathbb{C}}}-C\right)e^{t\left(\widetilde{\Delta_{N}^{\mathbb{C}}}-C\right)}\right\Vert _{\mathcal{L}\left(V\right)}<\infty$
for $V\in\{\mathbb{C}\mathbb{P}L^{p}\Omega,\mathbb{C}\mathbb{P}W^{1,p}\Omega_{N}\}$.
Therefore it also holds for $V=\mathbb{C}\mathbb{P}B_{p,1}^{\frac{1}{p}}\Omega_{N}$
by interpolation, and that is another criterion for analyticity (\parencite[Section II, Theorem 4.6.c]{engel2000one-parameter}).

Unfortunately, this does not tell us about the relationship between
$\left(L^{p}\Omega,W^{1,p}\Omega_{N}\right)_{\frac{1}{p},1}$ and
$B_{p,1}^{\frac{1}{p}}\Omega_{N}$. Obviously $\left(L^{p}\Omega,W^{1,p}\Omega_{N}\right)_{\frac{1}{p},1}\hookrightarrow B_{p,1}^{\frac{1}{p}}\Omega_{N}$
by the density of $W^{1,p}\Omega_{N}$. The other direction is more
delicate. Interpolation involving boundary conditions is often nontrivial.
The reader can see \parencite{Guidetti1991_Interpolation_boundary,Lofstrm1992_interpol_neumann,Amann2019_vol2}
to get an idea of the challenges involved, especially at the critical
regularity levels $\mathbb{N}+\frac{1}{p}$.

Nevertheless, there are a few interesting things we can say about
these spaces.
\begin{defn}[Neumann condition on strip]
 For vector field $X$ and $r>0$ small$,$ with $\psi_{r}$ as in
\Eqref{cutoff}, define 
\[
\mathbf{n}_{r}X=\psi_{r}\left\langle X,\widetilde{\nu}\right\rangle \widetilde{\nu}\text{ and }\mathbf{t}_{r}X=X-\mathbf{n}_{r}X
\]
Then define $\mathfrak{X}_{N,r}=\{X\in\mathfrak{X}:\left\langle X,\widetilde{\nu}\right\rangle =0\text{ on }M_{<r}\}$.
Similarly we can define $W^{m,p}\mathfrak{X}_{N,r}$ and $B_{p,q}^{s}\mathfrak{X}_{N,r}$
by setting $\left\Vert \left\langle X,\widetilde{\nu}\right\rangle \right\Vert _{L^{1}\left(M_{<r}\right)}=0$.
We note that $L^{3}\mathfrak{X}_{N,r}$ makes sense since the definition
does not require the trace theorem, unlike $L^{3}\mathfrak{X}_{N}$
which is ill-defined.
\end{defn}

Some basic facts:
\begin{enumerate}
\item $\mathbf{t}_{r}\mathfrak{X}\leq\mathfrak{X}_{N,\frac{r}{2}}$
\item $\mathbf{t}_{r}=1$ and $\mathbf{n}_{r}=0$ on $\mathfrak{X}_{N,r}$
\item $\mathbf{t}_{\frac{r}{2}}\mathbf{t}_{r}=\mathbf{t}_{r}$
\item $\left\Vert \mathbf{t}_{r}X\right\Vert _{W^{m,p}}\lesssim_{r,m,p}\left\Vert X\right\Vert _{W^{m,p}}$
for $m\in\mathbb{N}_{0},p\in[1,\infty]$
\item $W^{m,p}\mathfrak{X}_{N,r}$ and $B_{p,q}^{s}\mathfrak{X}_{N,r}$
are Banach for $m\in\mathbb{N}_{0},p\in[1,\infty],s\geq0,q\in[1,\infty]$
\item $B_{p,q}^{m_{\theta}}\mathfrak{X}_{N,r}\xhookrightarrow{\mathbf{t}_{r}=1}\left(W^{m_{0},p}\mathfrak{X}_{N,\frac{r}{2}},W^{m_{1},p}\mathfrak{X}_{N,\frac{r}{2}}\right)_{\theta,q}\hookrightarrow B_{p,q}^{m_{\theta}}\mathfrak{X}_{N,\frac{r}{2}}$
for $\theta\in\left(0,1\right),m_{j}\in\mathbb{N}_{0},m_{0}\neq m_{1},p\in[1,\infty],q\in[1,\infty]$,
$m_{\theta}=\left(1-\theta\right)m_{0}+\theta m_{1}$.
\end{enumerate}
\begin{rem*}
The last assertion is proven by the definition of the $J$-method,
and it works like partial interpolation. The reader can notice the
similarity with the Littlewood-Paley projection $(P_{\leq N}P_{\leq\frac{N}{2}}=P_{\leq\frac{N}{2}})$.
The hope is that $\mathbf{t}_{r}X\xrightarrow{t\downarrow0}X$ in
a good way for $X\in\mathfrak{X}_{N}$.
\end{rem*}
A subtle issue is that for $X\in B_{p,q}^{m_{\theta}}\mathfrak{X}_{N,r}$,
$\left\Vert X\right\Vert _{\left(W^{m_{0},p}\mathfrak{X}_{N,\frac{r}{2}},W^{m_{1},p}\mathfrak{X}_{N,\frac{r}{2}}\right)_{\theta,q}}\lesssim_{r}\left\Vert X\right\Vert _{B_{p,q}^{m_{\theta}}\mathfrak{X}_{N,r}}$.
The implicit constant which depends on $r$ can blow up as $r\downarrow0$.

Define $B_{p,q}^{s}\mathfrak{X}_{N,0+}=B_{p,q}^{s}\text{-}\mathrm{cl}\left(\cup_{r>0\text{ small}}B_{p,q}^{s}\mathfrak{X}_{N,r}\right)$
and $W^{m,p}\mathfrak{X}_{N,0+}=W^{m,p}\text{-}\mathrm{cl}\left(\cup_{r>0\text{ small}}W^{m,p}\mathfrak{X}_{N,r}\right)$.

Then we recover the usual spaces by results from \Subsecref{Strip-decay}:
\begin{thm}
\label{thm:B1ppN0_is_B1ppN} Let $p\in(1,\infty)$:
\begin{enumerate}
\item $L^{p}\mathfrak{X}_{N,0+}=L^{p}\mathfrak{X},$ $W^{1,p}\mathfrak{X}_{N,0+}=W^{1,p}\mathfrak{X}_{N}$.
\item $B_{p,1}^{\frac{1}{p}}\mathfrak{X}_{N,0+}=B_{p,1}^{\frac{1}{p}}\mathfrak{X}_{N}$.
\end{enumerate}
\end{thm}

\begin{proof}
$\;$
\begin{enumerate}
\item Let $X\in L^{p}\mathfrak{X}$. Then $\mathbf{n}_{r}X\xrightarrow[r\downarrow0]{L^{p}}0$
by shrinking support. If $X\in W^{1,p}\mathfrak{X}_{N},$ then by
\Thmref{coarea} 
\begin{align*}
\left\Vert \mathbf{n}_{r}X\right\Vert _{W^{1,p}} & =\left\Vert \psi_{r}\left\langle X,\widetilde{\nu}\right\rangle \right\Vert _{W^{1,p}\left(M_{<r}\right)}\lesssim\left\Vert \psi_{r}\right\Vert _{W^{1,\infty}\left(M_{<r}\right)}\left\Vert \left\langle X,\widetilde{\nu}\right\rangle \right\Vert _{L^{p}\left(M_{<r}\right)}+\left\Vert \psi_{r}\right\Vert _{L^{\infty}}\left\Vert \left\langle X,\widetilde{\nu}\right\rangle \right\Vert _{W^{1,p}\left(M_{<r}\right)}\\
 & \lesssim\frac{1}{r}\left\Vert \left\langle X,\widetilde{\nu}\right\rangle \right\Vert _{L^{p}\left(M_{<r}\right)}+\left\Vert \left\langle X,\widetilde{\nu}\right\rangle \right\Vert _{W^{1,p}\left(M_{<r}\right)}\lesssim\left\Vert \left\langle X,\widetilde{\nu}\right\rangle \right\Vert _{W^{1,p}\left(M_{<r}\right)}\xrightarrow{r\downarrow0}0
\end{align*}
\item Let $Y\in B_{p,1}^{\frac{1}{p}}\mathfrak{X}.$ As $B_{\infty,\infty}^{\frac{1}{p}}=\left(L^{\infty},W^{1,\infty}\right)_{\frac{1}{p},\infty}$
and $\psi_{r}\in W^{1,\infty},$ we conclude $\left\Vert \psi_{r}\right\Vert _{B_{\infty,\infty}^{\frac{1}{p}}}\lesssim\left\Vert \psi_{r}\right\Vert _{L^{\infty}}^{\frac{1}{p'}}\left\Vert \psi_{r}\right\Vert _{W^{1,\infty}}^{\frac{1}{p}}\lesssim\left(\frac{1}{r}\right)^{\frac{1}{p}}.$\\
 Then by \Thmref{coarea} and \Thmref{product_estimate} : 
\begin{align*}
\left\Vert \mathbf{n}_{r}Y\right\Vert _{B_{p,1}^{\frac{1}{p}}} & \lesssim_{\neg r}\left\Vert \psi_{r}\right\Vert _{B_{\infty,1}^{\frac{1}{p}}\left(M\right)}\left\Vert \left\langle Y,\widetilde{\nu}\right\rangle \right\Vert _{L^{p}\left(M_{<4r}\right)}+\left\Vert \psi_{r}\right\Vert _{L^{\infty}}\left\Vert \left\langle Y,\widetilde{\nu}\right\rangle \right\Vert _{B_{p,1}^{\frac{1}{p}}\left(M\right)}\lesssim\left(\frac{1}{r}\right)^{\frac{1}{p}}\left\Vert \left\langle Y,\widetilde{\nu}\right\rangle \right\Vert _{L^{p}\left(M_{<4r}\right)}+\left\Vert Y\right\Vert _{B_{p,1}^{\frac{1}{p}}}\\
 & \lesssim\left\Vert \left\langle Y,\widetilde{\nu}\right\rangle \right\Vert _{L^{p}\left(M_{<4r},\mathrm{avg}\right)}+\left\Vert Y\right\Vert _{B_{p,1}^{\frac{1}{p}}\left(M\right)}\lesssim_{\neg r}\left\Vert Y\right\Vert _{B_{p,1}^{\frac{1}{p}}}
\end{align*}
Therefore $\left\Vert \mathbf{n}_{r}Y\right\Vert _{B_{p,1}^{1/p}}$
does not blow up as $r\downarrow0$. Then we make a dense convergence
argument: assume $X\in B_{p,1}^{\frac{1}{p}}\mathfrak{X}_{N}$ and
let $X_{j}\in\mathfrak{X}$ such that $X_{j}\xrightarrow{B_{p,1}^{1/p}}X$,
then $\left\Vert \left\langle X_{j},\nu\right\rangle \right\Vert _{L^{p}(\partial M)}\xrightarrow{j\to\infty}0$.
Note that we do not have $\mathbf{n}X_{j}=0$. By \Thmref{coarea}:
\begin{align*}
\left\Vert \mathbf{n}_{r}X_{j}\right\Vert _{B_{p,1}^{\frac{1}{p}}} & \lesssim\left\Vert \mathbf{n}_{r}X_{j}\right\Vert _{L^{p}}^{\frac{1}{p'}}\left\Vert \mathbf{n}_{r}X_{j}\right\Vert _{W^{1,p}}^{\frac{1}{p}}\\
 & \lesssim\left\Vert \left\langle X_{j},\widetilde{\nu}\right\rangle \right\Vert _{L^{p}\left(M_{<r}\right)}^{\frac{1}{p'}}\left(\left\Vert \psi_{r}\right\Vert _{W^{1,\infty}\left(M_{<r}\right)}^{\frac{1}{p}}\left\Vert \left\langle X_{j},\widetilde{\nu}\right\rangle \right\Vert _{L^{p}\left(M_{<r}\right)}^{\frac{1}{p}}+\left\Vert \psi_{r}\right\Vert _{L^{\infty}}^{\frac{1}{p}}\left\Vert \left\langle X_{j},\widetilde{\nu}\right\rangle \right\Vert _{W^{1,p}\left(M_{<r}\right)}^{\frac{1}{p}}\right)\\
 & \lesssim\left\Vert \left\langle X_{j},\widetilde{\nu}\right\rangle \right\Vert _{L^{p}\left(M_{<r}\right)}\left(\frac{1}{r}\right)^{\frac{1}{p}}+\left\Vert \left\langle X_{j},\widetilde{\nu}\right\rangle \right\Vert _{L^{p}\left(M_{<r}\right)}^{\frac{1}{p'}}\left\Vert \left\langle X_{j},\widetilde{\nu}\right\rangle \right\Vert _{W^{1,p}\left(M_{<r}\right)}^{\frac{1}{p}}\\
 & \lesssim r^{\frac{1}{p'}}\left\Vert \left\langle X_{j},\widetilde{\nu}\right\rangle \right\Vert _{W^{1,p}\left(M_{<r}\right)}+\left\Vert \left\langle X_{j},\nu\right\rangle \right\Vert _{L^{p}\left(\partial M\right)}+\left\Vert \left\langle X_{j},\widetilde{\nu}\right\rangle \right\Vert _{W^{1,p}\left(M_{<r}\right)}
\end{align*}
So $\limsup_{r\downarrow0}\left\Vert \mathbf{n}_{r}X_{j}\right\Vert _{B_{p,1}^{\frac{1}{p}}}\lesssim\left\Vert \left\langle X_{j},\nu\right\rangle \right\Vert _{L^{p}\left(\partial M\right)}$
and 
\begin{align*}
\limsup_{r\downarrow0}\left\Vert \mathbf{n}_{r}X\right\Vert _{B_{p,1}^{\frac{1}{p}}} & \lesssim\limsup_{r\downarrow0}\left\Vert \mathbf{n}_{r}\left(X-X_{j}\right)\right\Vert _{B_{p,1}^{\frac{1}{p}}}+\limsup_{r\downarrow0}\left\Vert \mathbf{n}_{r}X_{j}\right\Vert _{B_{p,1}^{\frac{1}{p}}}\\
 & \lesssim\left\Vert X-X_{j}\right\Vert _{B_{p,1}^{\frac{1}{p}}}+\left\Vert \left\langle X_{j},\nu\right\rangle \right\Vert _{L^{p}\left(\partial M\right)}
\end{align*}
As $j$ is arbitrary, let $j\to\infty$ and $\limsup_{r\downarrow0}\left\Vert \mathbf{n}_{r}X\right\Vert _{B_{p,1}^{\frac{1}{p}}}=0$.
\end{enumerate}
\end{proof}
These results hold not just for vector fields, but also for differential
forms once we perform the proper modifications: for differential form
$\omega$, define $\mathbf{n}_{r}\omega=\psi_{r}\widetilde{\nu}^{\flat}\wedge\left(\iota_{\widetilde{\nu}}\omega\right)$,
$\mathbf{t}_{r}\omega=\omega-\mathbf{n}_{r}\omega$, $W^{m,p}\Omega_{r}^{k}=\{\omega\in W^{m,p}\Omega^{k}:\iota_{\widetilde{\nu}}\omega=0\text{ on }M_{<r}\}$,
replace $\left\langle X,\widetilde{\nu}\right\rangle $ with $\iota_{\widetilde{\nu}}\omega$
in the proofs etc. In particular, $B_{p,1}^{\frac{1}{p}}\Omega_{N,0+}^{k}=B_{p,1}^{\frac{1}{p}}\Omega_{N}^{k}$
for $p\in\left(1,\infty\right)$.

\appendix

\section{Complexification\label{sec:Complexification}}

Throughout this small subsection, the \uline{overline always stands
for conjugation}, and not topological closure.

Let $\mathbb{R}X$ be a real NVS, then a \textbf{complexification}
of $\mathbb{R}X$ is a tuple $\left(\mathbb{C}X,\mathbb{R}X\xhookrightarrow{\phi}\mathbb{C}X\right)$
such that
\begin{enumerate}
\item $\mathbb{C}X$ is a complex NVS.
\item $\phi$ is a linear, continuous injection and $\phi(\mathbb{R}X)\oplus i\phi(\mathbb{R}X)=\mathbb{C}X$.
\item $\left\Vert \phi(x)\right\Vert _{\mathbb{C}X}=\left\Vert x\right\Vert _{\mathbb{R}X}$
and $\left\Vert \phi(x)+i\phi(y)\right\Vert _{\mathbb{C}X}=\left\Vert \phi(x)-i\phi(y)\right\Vert _{\mathbb{C}X}$
$\forall x,y\in\mathbb{R}X$.
\end{enumerate}
The last property says $\left\Vert \cdot\right\Vert _{\mathbb{C}X}$
is a \textbf{complexification norm}. By treating $\phi(\mathbb{R}X)$
as the real part, $\forall z\in\mathbb{C}X,$ we can define $\Re z,\Im z$
as the real and imaginary parts respectively, so $z=\Re z+i\Im z$.
Then define $\overline{z}=\Re z-i\Im z$. So $\overline{\lambda z}=\overline{\lambda}\overline{z}\;\forall z\in\mathbb{C}X,\forall\lambda\in\mathbb{C}$.

\paragraph{Construction}

A standard construction of such a complexification is $\mathbb{C}X=\mathbb{R}X\otimes_{\mathbb{R}}\mathbb{C}$.
As $\mathbb{R}X$ is a flat and free $\mathbb{R}$-module, $0\to\mathbb{R}\hookrightarrow\mathbb{C}\stackrel{\Im}{\twoheadrightarrow}\mathbb{R}\to0$
induces $0\to\mathbb{R}X\xhookrightarrow{\phi}\mathbb{C}X\stackrel{\Im}{\twoheadrightarrow}\mathbb{R}X\to0$
as a split short exact sequence and $\mathbb{C}X=\phi(\mathbb{R}X)\oplus i\phi(\mathbb{R}X)$.
Then we can make $\phi$ implicit and not write it again. The representation
$z=x+iy=\left(x,y\right)$\textbf{ }is unique. Easy to see that any
two complexifications of $\mathbb{R}X$ must be isomorphic as $\mathbb{C}$-modules.

We define the \textbf{minimal complexification norm }(also called
\textbf{Taylor norm}) 
\[
\left\Vert x+iy\right\Vert _{T}:=\sup_{\theta\in[0,2\pi]}\left\Vert x\cos\theta-y\sin\theta\right\Vert _{\mathbb{R}X}=\sup_{\theta\in[0,2\pi]}\left\Vert \Re e^{i\theta}\left(x+iy\right)\right\Vert _{\mathbb{R}X}\;\forall x,y\in\mathbb{R}X
\]

Any other complexification norm is equivalent to $\left\Vert \cdot\right\Vert _{T}$.
\begin{proof}
Let $\left\Vert \cdot\right\Vert _{B}$ be another complexification
norm. Then $\left\Vert \Re e^{i\theta}\left(x+iy\right)\right\Vert _{\mathbb{R}X}=\left\Vert \Re e^{i\theta}\left(x+iy\right)\right\Vert _{B}\leq\left\Vert x+iy\right\Vert _{B}$
(\uline{minimal}) and $\left\Vert x+iy\right\Vert _{B}\leq\left\Vert x\right\Vert _{\mathbb{R}X}+\left\Vert y\right\Vert _{\mathbb{R}X}=\left\Vert \Re\left(x+iy\right)\right\Vert _{\mathbb{R}X}+\left\Vert \Re\left(-i\left(x+iy\right)\right)\right\Vert _{\mathbb{R}X}\leq2\left\Vert x+iy\right\Vert _{T}$.
\end{proof}
So the topology of $\mathbb{C}X$ is unique. It is more convenient,
however, to set $\left\Vert x+iy\right\Vert _{\mathbb{C}X}=\left\Vert \left(x,y\right)\right\Vert _{\mathbb{R}X\oplus\mathbb{R}X}=\left(\left\Vert x\right\Vert _{\mathbb{R}X}^{2}+\left\Vert y\right\Vert _{\mathbb{R}X}^{2}\right)^{\frac{1}{2}}\;\forall x,y\in\mathbb{R}X$.
Easy to see that any two complexifications of $\mathbb{R}X$ must
be isomorphic as complex NVS, so we write $\boxed{\mathbb{C}X=\mathbb{R}X\otimes_{\mathbb{R}}\mathbb{C}}$
from this point on, and if $\mathbb{R}X$ is normed, so is $\mathbb{C}X$.
Obviously, if $\mathbb{R}X$ is Banach, so is $\mathbb{C}X$, and
when that happens, we call $\left(\mathbb{R}X,\mathbb{C}X\right)$
a \textbf{Banach complexification couple}.

\paragraph{Real operators}

Let $\left(\mathbb{R}X,\mathbb{C}X\right)$ and $\left(\mathbb{R}Y,\mathbb{C}Y\right)$
be 2 Banach complexification couples.
\begin{itemize}
\item An operator $A:D\left(A\right)\leq\mathbb{C}X\to\mathbb{C}Y$ is called
a \textbf{real operator} when $D\left(A\right)=\mathbb{C}\Re D(A)$
and $A\Re\left(D\left(A\right)\right)\leq\mathbb{R}Y.$ In particular,
$A(x,y)=\left(Ax,Ay\right)$ $\forall x,y\in\mathbb{R}X$.
\item An unbounded $\mathbb{R}$-linear operator $T:D(T)\leq\mathbb{R}X\to\mathbb{R}Y$
has a natural complexified version $T^{\mathbb{C}}=T\otimes_{\mathbb{R}}1_{\mathbb{C}}:\mathbb{C}X\to\mathbb{C}Y$
where $D\left(T^{\mathbb{C}}\right)=\mathbb{C}D\left(T\right)$. Obviously
$T^{\mathbb{C}}$ is a real operator and we write $\left(\mathbb{R}X,\mathbb{C}X\right)\xrightarrow{\left(T,T^{\mathbb{C}}\right)}\left(\mathbb{R}Y,\mathbb{C}Y\right)$.\nomenclature{$\mathbb{C}Y,T^{\mathbb{C}}$}{complexification of spaces and operators \nomrefpage}
\begin{itemize}
\item $\overline{D\left(T^{\mathbb{C}}\right)}=D\left(T^{\mathbb{C}}\right)$
and $T^{\mathbb{C}}\overline{z}=\overline{T^{\mathbb{C}}z}\;\forall z\in\mathbb{C}X$.
\item $T$ is closed $\iff T^{\mathbb{C}}$ is closed. Same for bounded,
compact, densely defined.
\end{itemize}
\item For any unbounded $\mathbb{C}$-linear operator $A:D\left(A\right)\leq\mathbb{C}X\to\mathbb{C}Y$
such that $D\left(A\right)=\mathbb{C}\Re\left(D\left(A\right)\right)$,
define 2 real operators $\left\{ \begin{aligned} & \Re A=\left(\Re\circ\left(\restr{A}{\Re D(A)}\right)\right)^{\mathbb{C}}\\
 & \Im A=\left(\Im\circ\left(\restr{A}{\Re D(A)}\right)\right)^{\mathbb{C}}
\end{aligned}
\right.$\\
Then $A=\Re A+i\Im A.$ We can see that $A$ is real $\iff\Re A=A\iff\Im A=0$.
Also, $A$ is bounded $\iff$ $\Re A,\Im A$ are bounded.
\end{itemize}

\paragraph{Spectrum}

For $\left(\mathbb{R}X,\mathbb{C}X\right)\xrightarrow{\left(T,T^{\mathbb{C}}\right)}\left(\mathbb{R}Y,\mathbb{C}Y\right)$,
define
\begin{itemize}
\item $\rho(T):=\rho\left(T^{\mathbb{C}}\right),\sigma(T):=\sigma\left(T^{\mathbb{C}}\right)$.
\item $\rho_{\mathbb{R}}(T):=\{\lambda\in\mathbb{R}:\lambda-T\text{ is boundedly invertible}\}$
and $\sigma_{\mathbb{R}}(T):=\mathbb{R}\backslash\rho_{\mathbb{R}}(T)$.
\end{itemize}
If $\zeta\in\mathbb{C}$ and $\zeta-T^{\mathbb{C}}$ is boundedly
invertible, so is $\overline{\zeta}-T^{\mathbb{C}}$. So $\overline{\sigma\left(T\right)}=\sigma\left(T\right)$
and $\overline{\rho(T)}=\rho(T)$.\\
For $\lambda\in\mathbb{R}$, $\lambda-T^{\mathbb{C}}$ is boundedly
invertible $\iff$ $\lambda-T$ is boundedly invertible. So $\rho_{\mathbb{R}}(T)=\rho(T)\cap\mathbb{R}$
and $\sigma_{\mathbb{R}}\left(T\right)=\sigma\left(T\right)\cap\mathbb{R}$.

\subparagraph{Semigroup}

\begin{onehalfspace}
$T$ generates an $\mathbb{R}$-linear $C_{0}$ semigroup $\iff$
$T^{\mathbb{C}}$ generates a $\mathbb{C}$-linear $C_{0}$ semigroup.
When that happens, $\left(e^{tT}\right)^{\mathbb{C}}=e^{tT^{\mathbb{C}}}$.
\end{onehalfspace}
\begin{proof}
When either happens, $T$ and $T^{\mathbb{C}}$ are densely defined.
Also, $T-j$ and $T^{\mathbb{C}}-j$ are boundedly invertible for
$j\in\mathbb{N}$ large enough, so $T$ and $T^{\mathbb{C}}$ are
closed. Easy to use Hille-Yosida to show both $T$ and $T^{\mathbb{C}}$
must generate $C_{0}$ semigroups.

As in the proof of Hille-Yosida, define the Yosida approximations
$T_{j}=T\frac{1}{1-\frac{1}{j}T}$, $T_{j}^{\mathbb{C}}=T^{\mathbb{C}}\frac{1}{1-\frac{1}{j}T^{\mathbb{C}}}=\left(T_{j}\right)^{\mathbb{C}}$.
As $T_{j}$ and $T_{j}^{\mathbb{C}}$ are bounded, $\left(e^{tT_{j}}\right)^{\mathbb{C}}=e^{tT_{j}^{\mathbb{C}}}$
by power series expansion. Then $\left(e^{tT}\right)^{\mathbb{C}}=e^{tT^{\mathbb{C}}}$
as $e^{tT}=\lim_{j\to\infty}e^{tT_{j}}$ pointwise.
\end{proof}

\paragraph{Hilbert spaces}

Let $\mathbb{R}H$ be a real Hilbert space with inner product $\left\langle \cdot,\cdot\right\rangle $.
Then $\mathbb{C}H$ is also Hilbert with the inner product
\[
\left\langle x_{1}+iy_{1},x_{2}+iy_{2}\right\rangle _{\mathbb{C}H}:=\left\langle x_{1},x_{2}\right\rangle +\left\langle y_{1},y_{2}\right\rangle +i\left(\left\langle y_{1},x_{2}\right\rangle -\left\langle x_{1},y_{2}\right\rangle \right)\;\forall x_{j},y_{j}\in\mathbb{R}H
\]

Then $\left\Vert x+iy\right\Vert _{\mathbb{C}H}=\left(\left\Vert x\right\Vert _{\mathbb{R}H}^{2}+\left\Vert y\right\Vert _{\mathbb{R}H}^{2}\right)^{\frac{1}{2}}\;\forall x,y\in\mathbb{R}H$,
consistent with our previously chosen norm.

Also, $\left\langle z_{1},z_{2}\right\rangle _{\mathbb{C}H}=\overline{\left\langle z_{2},z_{1}\right\rangle _{\mathbb{C}H}}\;\forall z_{1},z_{2}\in\mathbb{C}H.$

Let $\left(A,A^{\mathbb{C}}\right):\left(\mathbb{R}H,\mathbb{C}H\right)\to\left(\mathbb{R}H,\mathbb{C}H\right)$
be unbounded.
\begin{itemize}
\item $A$ is symmetric $\iff A^{\mathbb{C}}$ is symmetric. When that happens,
$\left\langle Ax+iAy,x+iy\right\rangle _{\mathbb{C}H}=\left\langle Ax,x\right\rangle +\left\langle Ay,y\right\rangle \;\forall x,y\in\mathbb{R}H$.
\item $\mathbb{C}\left(\mathbb{R}H\oplus\mathbb{R}H\right)=\mathbb{C}H\oplus\mathbb{C}H$
and $G\left(A^{\mathbb{C}}\right)=\mathbb{C}G\left(A\right)$ (graphs).
Also $\mathbb{C}\left(G\left(A\right)^{\perp}\right)=G\left(A^{\mathbb{C}}\right)^{\perp}$.
\item $A$ is self-adjoint $\iff A^{\mathbb{C}}$ is self-adjoint. When
this happens, $\sigma(A)=\sigma(A^{\mathbb{C}})\subset\mathbb{R}$.
\item $A$ is dissipative $\iff$ $A^{\mathbb{C}}$ is dissipative.
\end{itemize}
For more information on complexification, see \parencite[Appendix C]{Glueck_thesis}. 

\printnomenclature{}

\printbibliography
\end{document}